\documentclass[10pt]{article}

\usepackage{amsfonts}
\usepackage{amsmath}
\usepackage{amssymb}
\usepackage{xfrac}
\usepackage[dvipsnames]{xcolor}
\usepackage{amsthm}
\usepackage{fullpage}
\usepackage{dsfont}
\usepackage{enumerate}
\usepackage{mathtools}
\usepackage{tikz-cd}
\usepackage{bbm}
\usepackage[linktocpage]{hyperref}
\hypersetup{
    colorlinks=true,
    citecolor=red,
    linkcolor=blue,
    urlcolor=red,
}
\usepackage[capitalize, nameinlink]{cleveref}
\usepackage[final]{showlabels}
\usepackage{tikz}
\usetikzlibrary{shapes.geometric}
\usepackage{comment}
\usepackage{mathrsfs}
\usepackage{fourier}
\usepackage{graphicx}
\usepackage[all]{xy}
\usepackage[titletoc]{appendix}
\usepackage{titlesec}
 \usepackage{etoolbox, chngcntr}
\AtBeginEnvironment{appendices}{%
 \titleformat{\section}{\bfseries\Large}{\appendixname~\thesection:}{0.5em}{}%
 \titleformat{\subsection}{\bfseries\large}{\thesubsection}{0.5em}{}%
\counterwithin{equation}{section}
}
\usepackage[nottoc]{tocbibind}
\usepackage{stackengine}
\stackMath
\setstackgap{S}{0pt}

\usepackage[backend=bibtex, giveninits=true, doi=false,isbn=false,url=false,style=alphabetic, maxnames=8, maxalphanames=9]{biblatex}
\bibliography{Ref.bib}
\DeclareMathAlphabet{\mathcal}{OMS}{cmsy}{m}{n}

\graphicspath{ {fig/} }

\definecolor{myyellow}{RGB}{200, 200, 0}

\newtheorem{theorem}{Theorem}
\numberwithin{theorem}{section}
\newtheorem{prop}[theorem]{Proposition}
\newtheorem{lemma}[theorem]{Lemma}
\newtheorem{corollary}[theorem]{Corollary}

\theoremstyle{definition}
\newtheorem{definition}[theorem]{Definition}
\newtheorem*{remark}{Remark}
\newtheorem{example}{Example}
\newtheorem{assumption}{Assumption}

\newcommand{\ext}{\mathrm{Ext}}
\newcommand{\bs}{\mathrm{BS}}
\newcommand{\kb}{\mathbbm{k}}
\renewcommand{\hom}{\mathrm{Hom}}
\newcommand{\im}{\mathrm{im } \,}

\newcommand{\Ac}{\mathcal{A}}
\newcommand{\Cc}{\mathcal{C}}
\newcommand{\Zb}{\mathbb{Z}}
\newcommand{\hf}{\mathfrak{h}}
\newcommand{\id}{\mathrm{id}}
\newcommand{\grk}{\underline{\mathrm{rk}} \ }
\newcommand{\Db}{\mathbb{D}}
\newcommand{\jw}{\mathrm{JW}}
\newcommand{\Dext}{\mathscr{D}^{\mathrm{Ext}}}
\newcommand{\bsbimext}{\textnormal{BSBim}^{\mathrm{Ext}}}
\newcommand{\bsbim}{\textnormal{BSBim}}
\newcommand{\sbim}{\mathbb{S}\textnormal{Bim}}
\newcommand{\Fext}{\mathscr{F}^{\ext}}
\newcommand{\Bs}{\mathscr{B}}
\newcommand{\kar}{\textbf{Kar}}
\newcommand{\bsw}{\mathrm{BS}(\underline{w})}

\newcommand{\hh}{\mathrm{HH}}
\newcommand{\hhh}{\mathrm{HHH}}

\newcommand\scalemath[2]{\scalebox{#1}{\mbox{\ensuremath{\displaystyle #2}}}}
\newcommand{\sbrac}[1]{\left[#1 \right]}
\newcommand{\abrac}[1]{\left\langle#1\right\rangle}
\newcommand{\dsbrac}[1]{\llbracket #1 \rrbracket}
\newcommand{\paren}[1]{\left( #1 \right)}
\newcommand{\set}[1]{\left \{ #1 \right \}}
\newcommand{\rhos}[3]{\rho_s^e({}_{#1}\widehat{\underline{#2}}_{#3})}

\newcommand{\rhott}[3]{\rho_t t(\rho_t)^e({}_{#1}\widehat{#2}_{#3})}
\newcommand{\rhosw}{\rho_s^e(\underline{w})}
\newcommand{\rhottw}{\rho_t t(\rho_t)^e(\underline{w})}
\newcommand{\urtt}{\underline{\rho_t t(\rho_t)}}
\newcommand{\urhos}{\underline{\rho_s}}
\renewcommand{\ll}[1]{\mathrm{L}_{  \underline{w}, \underline{#1} }}
\newcommand{\usmt}[3]{{}_{#1}\widehat{\underline{#2}}_{#3}}
\newcommand{\smt}[3]{{}_{#1}\widehat{#2}_{#3}}
\newcommand{\s}[1]{\scalemath{0.8}{#1}}
\newcommand{\un}[1]{\underline{#1}}
\newcommand{\lfrown}{\mathbin{\acute{\frown}}}

\newcommand{\ch}{\mathrm{ch}}

\newcommand\dboxed[1]{
\raisebox{-1ex}{
\begin{tikzpicture}
         \node[draw,dashed] {\text{$#1$}}; 
   \end{tikzpicture}}
}

\newcommand\dbox[1]{
\begin{tikzpicture}
         \node[draw,dashed] {\text{$#1$}}; 
   \end{tikzpicture}
}

\makeatletter
\DeclareRobustCommand\widecheck[1]{{\mathpalette\@widecheck{#1}}}
\def\@widecheck#1#2{%
    \setbox\z@\hbox{\m@th$#1#2$}%
    \setbox\tw@\hbox{\m@th$#1%
       \widehat{%
          \vrule\@width\z@\@height\ht\z@
          \vrule\@height\z@\@width\wd\z@}$}%
    \dp\tw@-\ht\z@
    \@tempdima\ht\z@ \advance\@tempdima2\ht\tw@ \divide\@tempdima\thr@@
    \setbox\tw@\hbox{%
       \raise\@tempdima\hbox{\scalebox{1}[-1]{\lower\@tempdima\box
\tw@}}}%
    {\ooalign{\box\tw@ \cr \box\z@}}}
\makeatother

\tikzset{
uni/.style={circle,fill,draw,inner sep=0mm,minimum size=1mm},
  midarrow/.style={postaction={decorate,decoration={markings,mark=at position #1 with {\arrow{>}}}}},
  midarrow/.default=0.5,
  midarrowrev/.style={postaction={decorate,decoration={markings,mark=at position #1 with {\arrow{<}}}}},
  midarrowrev/.default=0.5,
  dot/.style={circle,fill,draw,inner sep=0mm,minimum size=1.3mm},  
  rdot/.style={circle,fill, color=red, draw,inner sep=0mm,minimum size=1.3mm},  
  bdot/.style={circle,fill, color=blue, draw,inner sep=0mm,minimum size=1.3mm},  
  pdot/.style={circle,color=Violet,fill=Violet,draw,inner sep=0mm,minimum size=1.3mm},  
  hdot/.style={circle,fill=white,draw,inner sep=0mm,minimum size=1.3mm}, 
  rhdot/.style={circle,color=red,fill=white,draw,inner sep=0mm,minimum size=1.3mm, }, 
  bhdot/.style={circle,color=blue,fill=white,draw,inner sep=0mm,minimum size=1.3mm, }, 
  yhdot/.style={circle,color=myyellow,fill=white,draw,inner sep=0mm,minimum size=1.3mm, }, 
  rkhdot/.style={circle,color=red,fill=white,draw,inner sep=0mm,minimum size=4mm, }, 
  rbkhdot/.style={circle,color=red,fill=white,draw,inner sep=0mm,minimum size=6mm, }, 
  rbbkhdot/.style={circle,color=red,fill=white,draw,inner sep=0mm,minimum size=6.7mm, }, 
  tridot/.style={square,fill=white,draw,inner sep=0mm,minimum size=1.3mm}, 
    btridot/.style={rectangle,color=blue,fill=white,draw,inner sep=0.3mm,minimum size=1.7mm}, 
    rtridot/.style={rectangle,color=red,fill=white,draw,inner sep=0.3mm,minimum size=1.7mm}, 
  every picture/.style=thick
}
\newsavebox\lowerdot
\savebox\lowerdot{%
\begin{tikzpicture}[scale=0.3,thick,baseline]
 \draw (0,-0.5) to (0,0.5);
 \node at (0,-0.5) {$\bullet$};
\end{tikzpicture}%
}
\newsavebox\upperdot
\savebox\upperdot{%
\begin{tikzpicture}[scale=0.3,thick,baseline]
 \draw (0,-0.5) to (0,0.5);
 \node at (0,0.5) {$\bullet$};
\end{tikzpicture}%
}

\newcommand{\rzero}{\ \raisebox{-2ex}{\begin{tikzpicture}[xscale=0.25,yscale=0.3,thick,baseline]
 \draw[red] (0,0) -- (0,2); 
\end{tikzpicture}}}

\newcommand{\bzero}{\ \raisebox{-2ex}{\begin{tikzpicture}[xscale=0.25,yscale=0.3,thick,baseline]
 \draw[blue] (0,0) -- (0,2); 
\end{tikzpicture}}}

\newcommand{\vzero}{\ \raisebox{-2ex}{\begin{tikzpicture}[xscale=0.25,yscale=0.3,thick,baseline]
 \draw[violet] (0,0) -- (0,3); 
\end{tikzpicture}}}

\newcommand{\yzero}{\ \raisebox{-2ex}{\begin{tikzpicture}[xscale=0.25,yscale=0.3,thick,baseline]
 \draw[myyellow] (0,0) -- (0,3); 
\end{tikzpicture}}}

\newcommand{\ds}{\ \raisebox{-2ex}{\begin{tikzpicture}[xscale=0.25,yscale=0.3,thick,baseline]
 \draw[red] (0,0) -- (0,2);
 \node[rtridot] at (0,1) {$\scalemath{0.9}{d_s}$};
\end{tikzpicture}}}

\newcommand{\dt}{\ \raisebox{-2ex}{\begin{tikzpicture}[xscale=0.25,yscale=0.3,thick,baseline]
 \draw[blue] (0,0) -- (0,2); 
 \node[btridot] at (0,1) {$\scalemath{0.9}{d_t}$};
\end{tikzpicture}}}

\newcommand{\hunit}{\ \raisebox{0ex}{\begin{tikzpicture}[xscale=0.25,yscale=0.3,thick,baseline]
 \draw (0,0) -- (0,1);
 \node[hdot] at (0,0) {};
\end{tikzpicture}}}

\newcommand{\unit}{\ \raisebox{0ex}{\begin{tikzpicture}[xscale=0.25,yscale=0.3,thick,baseline]
 \draw (0,0) -- (0,1);
 \node[dot] at (0,0) {};
\end{tikzpicture}}}

\newcommand{\rhunit}{\ \raisebox{0ex}{\begin{tikzpicture}[xscale=0.25,yscale=0.3,thick,baseline]
 \draw[red] (0,0) -- (0,1);
 \node[rhdot] at (0,0) {};
\end{tikzpicture}}}

\newcommand{\runit}{\ \raisebox{0ex}{\begin{tikzpicture}[xscale=0.25,yscale=0.3,thick,baseline]
 \draw[red] (0,0) -- (0,1);
 \node[dot, red] at (0,0) {};
\end{tikzpicture}}}

\newcommand{\bhunit}{\ \raisebox{0ex}{\begin{tikzpicture}[xscale=0.25,yscale=0.3,thick,baseline]
 \draw[blue] (0,0) -- (0,1);
 \node[bhdot] at (0,0) {};
\end{tikzpicture}}}

\newcommand{\rhzero}{\ \raisebox{0ex}{\begin{tikzpicture}[xscale=0.3,yscale=0.48,thick,baseline]
 \draw[red] (0,-1) -- (0,1);
 \node[rhdot] at (0,0) {};
\end{tikzpicture}}}

\newcommand{\bhzero}{\ \raisebox{0ex}{\begin{tikzpicture}[xscale=0.3,yscale=0.48,thick,baseline]
 \draw[blue] (0,-1) -- (0,1);
 \node[bhdot] at (0,0) {};
\end{tikzpicture}}}

\newcommand{\bunit}{\ \raisebox{0ex}{\begin{tikzpicture}[xscale=0.25,yscale=0.3,thick,baseline]
 \draw[blue] (0,0) -- (0,1);
 \node[dot, blue] at (0,0) {};
\end{tikzpicture}}}

\newcommand{\vunit}{\ \raisebox{0ex}{\begin{tikzpicture}[xscale=0.25,yscale=0.3,thick,baseline]
 \draw[violet] (0,0) -- (0,1);
 \node[dot, violet] at (0,0) {};
\end{tikzpicture}}}

\newcommand{\rcounit}{\ \raisebox{0ex}{\begin{tikzpicture}[xscale=0.25,yscale=0.3,thick,baseline]
 \draw[red] (0,-1) -- (0,0);
 \node[dot, red] at (0,0) {};
\end{tikzpicture}}}

\newcommand{\yunit}{\ \raisebox{0ex}{\begin{tikzpicture}[xscale=0.25,yscale=0.3,thick,baseline]
 \draw[myyellow] (0,0) -- (0,1);
 \node[dot, myyellow] at (0,0) {};
\end{tikzpicture}}}

\newcommand{\rhpitchcupin}{\ \raisebox{2ex}{\begin{tikzpicture}[xscale=0.25,yscale=-0.3,thick,baseline]
             \draw[red] (0,0) to[out=90,in=-180] (1,1.3);
	       \draw[red] (2,0) to[out=90,in=0] (1,1.3);
              \draw[blue] (1,0) to (1,0.7);
              \node[bhdot] at (1,0.7) {};
\end{tikzpicture}}}

\newcommand{\rhhpitchcupout}{\ \raisebox{2ex}{\begin{tikzpicture}[xscale=0.25,yscale=-0.3,thick,baseline]
             \draw[red] (0,0) to[out=90,in=-180] (1,1.3);
	       \draw[red] (2,0) to[out=90,in=0] (1,1.3);
              \node[rhdot] at (1,1.3) {};
             \draw[blue] (1,0) to (1,0.7);
	       \node[bhdot] at (1,0.7) {};
\end{tikzpicture}}}

\newcommand{\bhpitchcupin}{\ \raisebox{2ex}{\begin{tikzpicture}[xscale=0.25,yscale=-0.3,thick,baseline]
             \draw[blue] (0,0) to[out=90,in=-180] (1,1.3);
	       \draw[blue] (2,0) to[out=90,in=0] (1,1.3);
              \draw[red] (1,0) to (1,0.7);
              \node[rhdot] at (1,0.7) {};
\end{tikzpicture}}}

\newcommand{\bhpitchcupout}{\ \raisebox{2ex}{\begin{tikzpicture}[xscale=0.25,yscale=-0.3,thick,baseline]
             \draw[blue] (0,0) to[out=90,in=-180] (1,1.3);
	       \draw[blue] (2,0) to[out=90,in=0] (1,1.3);
              \node[bhdot] at (1,1.3) {};
	       \node[rdot] at (1,0.7) {};
        \draw[red] (1,0) to (1,0.7);
\end{tikzpicture}}}

\newcommand{\bhhpitchcupout}{\ \raisebox{2ex}{\begin{tikzpicture}[xscale=0.25,yscale=-0.3,thick,baseline]
             \draw[blue] (0,0) to[out=90,in=-180] (1,1.3);
	       \draw[blue] (2,0) to[out=90,in=0] (1,1.3);
              \node[bhdot] at (1,1.3) {};
             \draw[red] (1,0) to (1,0.7);
	       \node[rhdot] at (1,0.7) {};
\end{tikzpicture}}}

\newcommand{\rmult}{\ \raisebox{0ex}{\begin{tikzpicture}[xscale=0.25,yscale=0.3,thick,baseline]
 \draw[red] (-1,-0.25) -- (0,0.5);
 \draw[red] (1,-0.25) -- (0,0.5);
 \draw[red] (0,1.5) -- (0,0.5);
\end{tikzpicture}}}

\newcommand{\rcomult}{\ \raisebox{0ex}{\begin{tikzpicture}[xscale=0.25,yscale=0.3,thick,baseline]
 \draw[red] (-1,1.25) -- (0,0.5);
 \draw[red] (1,1.25) -- (0,0.5);
 \draw[red] (0,0.5) -- (0,-0.5);
\end{tikzpicture}}}

\newcommand{\bcounit}{\ \raisebox{0ex}{\begin{tikzpicture}[xscale=0.25,yscale=0.3,thick,baseline]
 \draw[blue] (0,-1) -- (0,0);
 \node[dot, blue] at (0,0) {};
\end{tikzpicture}}}

\newcommand{\0}{\ \raisebox{-2ex}{\begin{tikzpicture}[xscale=0.25,yscale=0.3,thick,baseline]
 \draw (0,0) -- (0,3); 
\end{tikzpicture}}}

\renewcommand{\1}{\ \raisebox{-2ex}{\begin{tikzpicture}[xscale=0.25,yscale=0.3,thick,baseline]
 \draw (0,0) -- (0,1);
 \draw (0,2) -- (0,3);
 \node[dot] at (0,1) {};
  \node[dot] at (0,2) {};
\end{tikzpicture}}}

\newcommand{\rect}{
\begin{tikzpicture}[scale=0.3]
   \draw (0,0) rectangle (2.1,0.6);
\end{tikzpicture}
}

\title{\textbf{The Two-Color Ext Soergel Calculus}}
\date{\relax }
\author{or: The Ext-Dihedral Cathedral\\ \vspace{-1ex} \\Cailan Li}

\begin{document}

\maketitle
\begin{abstract}
    We compute Ext groups between Soergel Bimodules associated to the infinite/finite dihedral group for a realization in characteristic 0 and show that they are free right $R-$modules. In particular, we obtain an explicit diagrammatic basis for the Hochschild cohomology of indecomposable Soergel Bimodules. We then give a diagrammatic presentation for the corresponding monoidal category of Ext-enhanced Soergel Bimodules.  \\
   \indent  As applications, we explicitly compute HOMFLY homology/triply graded link homology $\overline{\hhh}$ for the connect sum of two Hopf links and the negative torus link $T(3,-3)$ as right $R-$modules. Furthermore, we show that the Hochschild cohomology of Soergel Bimodules in finite dihedral type categorifies Gomi's trace, providing a $t-$analog of Soergel's Hom Formula in the dihedral setting.
\end{abstract}

\tableofcontents


\section{Introduction}

\indent The story of Soergel Bimodules had it's birthplace in the work of Wolfgang Soergel in the 1990s \cite{Soe90} as an alternative algebraic approach to proving the illustrious Kazhdan-Lusztig (KL) conjectures \cite{KL79} and to give a combinatorial description of Harish-Chandra bimodules \cite{Soe92}.  Soergel Bimodules have a relatively straightforward definition, yet they have had profound applications in representation theory, geometry and link homology, especially in recent years. Conceptually, this arises from the fact that Soergel Bimodules categorify the Hecke Algebra, and the Hecke Algebra is a fundamental object of study within these fields. 

\subsection{History}

Let $(W, S)$ be a Coxeter group and $\mathbf{H}_W$ be its associated Hecke Algebra. When $W$ is the Weyl group for a split torus $T$ of a reductive group $G$, It has two well-known bases over $\Zb[v^{\pm 1}]$, the \textit{standard basis} $\set{\delta_w}_{w\in W}$ and the \textit{Kazhdan-Lusztig basis} $\set{b_w}_{w\in W}$.  In \cite{KL79}, Kazhdan and Lusztig made their astonishing conjecture that the change of basis matrix between $\set{b_w}_{w\in W}$ and $\set{\delta_x}_{x\in W}$ at $v=1$ equals the matrix of Jordan-Hölder multiplicities of simples in Verma modules in $\mathcal{O}_0(\mathfrak{g})$, the principal block of Category $\mathcal{O}$ for $\mathfrak{g}$, the lie algebra of $G$. This conjecture was then proved by Beilinson–Bernstein \cite{BB81}, and Brylinski–Kashiwara \cite{BryKash81} by employing deep geometric techniques, such as the usage of the Decomposition Theorem.\\

A key role in the proofs is played by $\mathcal{H}_W$, the (geometric) Hecke category, a certain monoidal subcategory of $D^b_{B\times B}(G, \mathbb{C})$, the $B\times B$ equivariant derived category of sheaves on $G$ with  $\mathbb{C}$ coefficients. As the name suggests, $\mathcal{H}_W$ is a categorification of the Hecke algebra. By categorification, we mean that there is an isomorphism of algebras
\begin{align}
\label{heckecatiso}
    K_\oplus(\mathcal{H}_W)\cong \mathbf{H}_W 
\end{align} 
where $K_\oplus(\mathcal{H}_W)$ is the split Grothendieck group of $\mathcal{H}_W$. In the early 1990s, Soergel \cite{Soe92} gave another incarnation of the Hecke category $\mathcal{H}_W$ that was more "algebraic." Specifically, let $\mathfrak{t}=\mathrm{Lie}(T)$ and $R=\mathrm{Sym}(\mathfrak{t}^*)$. Because $ H^\bullet_{B\times B}(\mathrm{pt})\cong  R \otimes_\mathbb{C} R$, equivariant hypercohomology of any object in $D^b_{B\times B}(G, \mathbb{C})$ is naturally a graded module over $R \otimes_\mathbb{C} R$. Since $R$ is commutative, hypercohomology is therefore a functor $ \mathbb{H}^\bullet_{B\times B}: \mathcal{H}_W\to R-\mathrm{Bim} $. Soergel's key observation is that $ \mathbb{H}^\bullet_{B\times B}$ is fully-faithful and monoidal and as a result we have a monoidal equivalence
\[ \mathcal{H}_W \xrightarrow{\sim} \sbim(\mathfrak{t}, W)  \]
where $\sbim$ is the category of \textit{Soergel bimodules}, a monoidal subcategory of $R-\mathrm{Bim}$ given by the essential image of $\mathbb{H}^\bullet_{B\times B}$. Using this equivalence (or rather a similar equivalence with Soergel modules), Soergel \cite{Soe90} shows how the KL conjectures would follow from a certain desired property of $\sbim(\mathfrak{t}, W)$ called \textit{Soergel's Conjecture}. \\

In fact, Soergel gives a completely algebraic definition of $\sbim$ in his work above. Specifically, $W$ acts on $\mathfrak{t}$ and thus on $R$ via graded algebra automorphisms. Let $R^s\subset R$ be the subalgebra of elements of $R$ invariant under $s\in S$. Define the $R-$bimodule $B_s=R\otimes_{R^s}R (1)$. $\sbim$ is then equivalent to the smallest full additive monoidal Karoubian graded subcategory of $R-\mathrm{Bim}$ containing $B_s$ for each $s\in S$. Thus, the definition above allows us to define $\sbim (\mathfrak{h}, W)$ for \emph{any} Coxeter group $W$ and \emph{any} representation $\hf$. Moreover in \cite{Soe07}, Soergel shows that when $\hf$ is ``reflection faithful", $\sbim (\mathfrak{h}, W)$ categorifies $\mathbf{H}_W$. In other words, there is an isomorphism of algebras
\[ \mathbf{H}_W\cong K_\oplus (\sbim (\mathfrak{h}, W)) \]
In this setting, the statement of Soergel's conjecture still makes sense and Soergel shows that this would imply a major open question in combinatorics: the Kazhdan-Lusztig positivity conjecture for arbitrary Coxeter groups.

However, lacking the Decomposition Theorem in this setting, he was unable to prove his conjecture. \\

Fortunately $\sbim (\mathfrak{t}, W)$ is a monoidal category, and thus one can hope to give a presentation by generators modulo local relations. Using the language of planar diagrammatics, this was first done by Elias-Khovanov in \cite{EK10} for $W=S_n$, by Elias \cite{DC} in the dihedral case and finally by Elias-Williamson \cite{SC} for any Coxeter group $W$. Although the presentation was found using Soergel bimodules, the resulting monoidal category can be considered independently, thus giving a third incarnation of the Hecke category $\mathscr{D}(\hf, W)$, referred to as the diagrammatic Hecke category. This incarnation has the advantage that many complicated computations can be reduced to algorithmic manipulations of planar diagrams. Using the diagrammatic Hecke category, Elias and Williamson \cite{HodgeSoergel} were able to prove Soergel's conjecture for any Coxeter group $W$ thereby proving the Kazhdan-Lusztig positivity conjecture and completing Soergel's purely algebraic proof of the Kazhdan-Lusztig conjectures. The diagrammatic Hecke category has subsequently led to other breakthroughs in representation theory such as Williamson's counterexamples \cite{Wtorsionexplod} to the long-standing Lusztig's conjecture (the analogue of the KL conjectures for representations of reductive algebraic groups in characteristic $p$) and the subsequent new character formulas for irreducible and indecomposable tilting modules conjectured in \cite{RWtilt} and proved in \cite{RWtilt} and \cite{KoszulAMRW}.\\

In type $A$, the Hecke algebra $\mathbf{H}_{S_n}$ also plays an important role in knot theory in the construction of link invariants. Given a link $L$ written as the closure of a braid $\beta_L\in \mathrm{Br}_{n}$, Jones \cite{Jonesannals} shows that the HOMFLY polynomial of $L$ is ``essentially equal" to $\mathrm{Tr}(\pi(\beta_L))$ where $\pi: B_{n}\to \mathbf{H}_{S_{n}}$ and $\mathrm{Tr}$ is the Jones-Ocneanu trace on $\mathbf{H}_{S_n}$. Khovanov \cite{Kho07} would later categorify this construction to produce a triply graded link homology theory as follows. For a braid $\beta_L\in B_{n}$, there is an associated (Rouquier) complex $F(\beta_L)\in K^b( \sbim(\mathbb{C}^{n}, S_{n}))$ (categorification of $\mathbf{H}_{S_n}$) where $\mathbb{C}^{n}$ is the permutation representation of $S_{n}$. Khovanov then applies the functor of Hochschild cohomology (categorification of $\mathrm{Tr}$) to $F(\beta_L)$ and takes cohomology to produce the desired link homology $\mathrm{HHH}(\beta_L)$. In recent years there has been a great deal of interest surrounding $\mathrm{HHH}(\beta_L)$ as it turns out to have deep connections with Hilbert Schemes \cite{ORS12}, \cite{GNR}, \cite{GH17} rational Cherednik algebras \cite{GORS14} and rational Catalan combinatorics \cite{Hog17}, \cite{HM19}, \cite{GMV20}.
Needless to say, Soergel Bimodules and the diagrammatic Hecke category are at the center of modern geometric representation theory and categorification.

\subsection{Main Results}

In this paper we give a diagrammatic presentation for the monoidal categories $\sbim^\ext( \hf, W_{\infty}), \sbim^\ext( \hf, W_{m})$ of Ext-enhanced Soergel Bimodules for the infinite and finite dihedral group $W_{\infty}$ and $W_m$ (we will abbreviate both cases using the notation $W_{\infty/m}$) for a realization $\hf$ in characteristic 0 satisfying the usual 2-color assumptions + a linear independence condition. This paper builds upon the $A_1$ case done by Shotaro Makisumi in \cite{M}. We will give some clarifications on the monoidal structure later, but for now note that extending/replacing the graded Hom spaces in $\sbim(\mathfrak{h}, W_{\infty/m})$ via
\[ \hom_{R-\mathrm{Bim}}^\bullet(-, -) \leadsto \ext_{R-\mathrm{Bim}}^{\bullet, \bullet}(-,-) \]

still results in a category $\sbim^\ext( \hf, W_{\infty/m})$ with composition given by the Yoneda product. This still has a monoidal structure (or rather super-monoidal structure as explained in \cite{M}) and in addition to the Elias-Khovanov generators we have (essentially) three additional one-color generators pictured below using planar diagrammatics
\[  \rhzero \qquad \qquad \bhzero \qquad \qquad \raisebox{0ex}{\dboxed{\xi}} \quad \xi\in \Lambda^\bullet(V[1](-2))  \]
with many new relations (see \cref{Maksect}). We also have (essentially) one additional two-color generator
\[\raisebox{-4ex}{\begin{tikzpicture}[scale=0.7]
	       \draw[red] (1.1,-0.1) -- (2,-1);
	       \draw[blue] (1.5,-0.1) -- (2,-1);
	       \draw[violet] (2.9,-0.1) -- (2,-1);
	       \draw[blue] (1.1,-2) -- (2,-1);
	       \draw[red] (1.5,-2) -- (2,-1);
	       \draw[violet] (2.9,-2) -- (2,-1);
	       \node at (2.1,-0.4) {$\scalemath{0.8}{\ldots}$};
	       \node at (2.1,-1.6) {$\scalemath{0.8}{\ldots}$};
	       \node[rhdot, minimum size=3.4mm] at (2,-1) {};
	       \node at (2,-1) {$\scalemath{0.7}{2k}$};
	    \end{tikzpicture}}\quad k\ge 2\]
with many new relations (see \cref{diagraminftysect} and \cref{diagramfinitesect}). These generators and relations can be considered independently of $\sbim^\ext( \hf, W_{\infty/m})$ giving rise to the diagrammatic Ext enhanced Hecke categories $\mathscr{D}^\ext(\hf, W_{\infty/m})$. To be more precise, as in \cite{EK10} and \cite{SC} what we actually do is give a presentation for $\bsbim^\ext(\hf, W_{\infty/m})$, the full subcategory of $D^b(R-\mathrm{Bim})$\footnote{Minor technicality, Hom spaces in $\bsbim^\ext(\hf, W_{\infty/m})$ will actually be $\bigoplus_n \hom_{D^b(R-\mathrm{Bim})}(M, N[n])$.} with objects the Bott-Samelson bimodules from which $\sbim^\ext( \hf, W_{\infty/m})$ can be recovered by taking the Karoubian envelope. One main result of the paper will then be the following equivalences
\[ \mathscr{D}^\ext(\hf, W_\infty) \simeq \bsbim^\ext(\hf, W_\infty), \qquad  \mathscr{D}^\ext(\hf, W_{m}) \simeq \bsbim^\ext(\hf, W_{m}) \]
Note that although the categories $\mathscr{D}^\ext(\hf, W_{\infty / m})$ can be considered when $\hf$ has characteristic $p$, it is not clear to the author that the same equivalences above still hold, as there appears to be more generators in $\bsbim^\ext(\hf, W_{\infty / m})$. In the course of the proof we obtain a complete description of Hom spaces in $\bsbim^\ext(\hf, W_{\infty / m})$. Specifically for any two expressions $\un{v}, \un{w}$, we will compute  
\[ \ext_{R-\mathrm{Bim}}^{\bullet, \bullet}( \bs(\un{v}), \bs(\un{w}) ) \]
as a right $R$ module (although the results can be equivalently stated with the left $R-$module structure) (see \cref{maincohoiso}, \cref{maincohoisofin}). In particular the morphism space above will always be free as a right $R-$ module. When $m=2,3,4,6$ this is also a special case of \cite{WW11} which uses geometric methods that do not apply when $W_m$ is non-crystallographic. In fact, we are able to obtain an explicit basis for $\ext^{\bullet, \bullet}_{R-\mathrm{Bim}}(R, B_w)$ where $B_w$ is an indecomposable Soergel Bimodule (see \cref{algextindecomp}). \\

One thing to note is that throughout the algebraic portion of the paper, diagrammatics are already present, first as a visual aid to help explain the new relations, and then used in proofs of theorems which live in the algebraic category. This is allowed, however, as each time we are either using the equivalence $\mathscr{D}(\hf, W_{\infty / m}) \simeq \bsbim(\hf, W_{\infty / m})$ established in \cite{DC} or a previously established relation in $\bsbimext(\hf, W_{\infty/m})$ of which we had a diagrammatic interpretation of. When the expression for a relation in $\bsbimext(\hf, W_{\infty/m})$ is not too cumbersome, we will write it out explicitly, and accompany it with a diagrammatic description.

\subsubsection{Applications}

\begin{enumerate}[(a)]
    \item Because Hochschild cohomology of a $R-$bimodule $M$ is defined to be $\mathrm{HH}^k(M):=\ext_{R-\mathrm{Bim}}^k(R, M)$ one can apply our results to compute triply graded link homology $\mathrm{HHH}(\beta_L)$ where $\beta_L$ is a braid on 3 strands by setting $\hf=\mathbb{C}^3 (\mathbb{C}^2)$ to be the permutation (geometric) realization of $S_3$. This yields a new technique for computing $\mathrm{HHH}$, and can be visualized schematically as follows: $\hhh^{A=k}$ is cohomology of the following chain complex \vspace{-1ex}
 \[ \ldots\to \hh^k( F(\beta)_i) \xrightarrow{d_i} \hh^k( F(\beta)_{i+1})  \to \ldots \vspace{-3ex} \vspace{1ex} \qquad \qquad \text{(homological algebra)}  \]
\begin{center}
\raisebox{1.9ex}{\rotatebox{270}{$\leadsto$}}  \cref{algextindecomp} \vspace{-2ex}
\end{center}
 \[\qquad \qquad \qquad \qquad \ \ldots\to R^{\ell_i} \xrightarrow{\cref{finequiv}} R^{\ell_{i+1}}  \to \ldots \vspace{-3ex} \vspace{1ex} \qquad \qquad \text{(linear algebra over }R)\vspace{1ex} \]
  An upshot of this technique is that it doesn't require $\hhh$ to be parity\footnote{Concentrated in only even or only odd $T-$degrees.} which is crucial for the method in \cite{EH19} used to compute $\hhh$ of positive torus links. Thus our technique allows for new computations of $\hhh$ not in the literature. For example in \cref{appendixhomfly}, we compute $\hhh$ of the negative torus link $T(3,-3)$. \\
  
  In \cite{L3deg} we go beyond 3 strands by computing $\hhh$ of any positive or negative braid on $n$ strands in certain $T-$degrees. Also, unlike in \cite{EH19} which can only compute $\hhh$ as a vector space, our method computes $\hhh$ as $R-$modules. This capability is vital in our computation of $\hhh(T(3,-3))$. See \cite[Problem 1.12, Section 1.1.2]{HM19} for more on the importance of computing $\hhh$ as an $R-$module.
    
    \item In \cite{Gom06}, Gomi generalizes the two variable Jones-Ocneanu trace on $\mathbf{H}_{S_n}$ to a trace on $\mathbf{H}_W$ satisfying a ``Markov" type condition, where $W$ is any finite Coxeter group as follows. Because $\mathbf{H}_W$ is semisimple any trace $\tau:\mathbf{H}_W\to \Zb[q,q^{-1}][t]$ can be written as
    \[ \tau=\sum_{\chi\in \mathrm{Irr} W} w^\chi \chi_q \]
    Gomi then defines $w^\chi:\mathbf{H}_W\to \Zb[q,q^{-1}][t] $ using entries from Lusztig's Fourier transform matrix $S$. When $W$ is the Weyl group of a reductive algebraic group $G/\mathbb{F}_q$, $S$ is the change of basis matrix between unipotent characters and almost characters for the finite group of Lie type $G^F$ \cite{Luz84}. When $W$ is of dihedral type, $S$ is the Exotic Fourier transform of Lusztig \cite{Luz94}. In \cref{gomisect} we show that the Poincare series for the Hochschild cohomology of Soergel Bimodules for $W_m$ coincides with Gomi's trace defined above\footnote{We were informed by Minh-Tâm Trinh that a version of this was stated in the thesis of Lasy, with an incomplete proof.} which allows us to establish a $t-$analog of Soergel's Hom Formula. For $m=2,3,4,6$ this is a special case of \cite{WW11} which uses properties of unipotent character sheaves. From this we speculate that a suitable category of Ext-enhanced Soergel Bimodules is the right setting for the study of ''spetses" \cite{BMM99}.
    
    \item Soergel Bimodules corresponding to $\sbim(\mathbb{C}^n, S_n)$ have an alternative pictorial description that is 3D instead of 2D. Instead of depicting morphisms between Bott-Samelson bimodules using planar diagrams, we depict a Bott-Samelson bimodule $B_{i_1}\otimes_R \cdots \otimes_R B_{i_d}$ as a planar diagram, see \cite{Kho07} Figures 2,3. Morphisms can then be realized by linear combinations of foams$-$decorated two-dimensional CW-complexes embedded in $\mathbb{R}^3$ that arise as cobordisms between the planar diagrams described above. Foams are one of the fundamental objects used in constructing link homology theories. They show up prominently in the doubly graded $\mathfrak{sl}_n$ theories starting with \cite{Kho04} in $\mathfrak{sl}_3$ and \cite{MSV09}, \cite{QR16} for $\mathfrak{sl}_n$, and finally in \cite{RW20} for equivariant $\mathfrak{sl}_n$. \vspace{-1ex}\\
    
    One major step in the construction of all these link homology theories is to define some variant of foam evaluation for closed foams which is used to associate a vector space $F(\Gamma)$ to a web. In \cite{RW20b}, Robert and Wagner show that applying the variant, $F_\infty(\Gamma)$ to a special class of foams recovers $\mathscr{B}(\Gamma)$ the (singular) Soergel Bimodule corresponding to $\Gamma$. Moreover, they show that by closing the planar diagram first and then applying $F_\infty$ recovers $\mathrm{HH}_0(\mathscr{B}(\Gamma))$. In \cite{KRW} they extend this to $\mathrm{HH}_k(\mathscr{B}(\Gamma))$ for all $k\ge 0$. As $\mathrm{HH}_\bullet(M)\cong \mathrm{HH}^\bullet(M)$ for bimodules over polynomial rings, it would be interesting to see how our relations below can be realized in this framework.  
    
    \item A crucial step in the proof of the tilting character formula in \cite{KoszulAMRW} 
    was to relate two different derived categories of sheaves. To do this the authors first establish an equivalence between two different monoidal categories which act on the different categories above. Specifically, using the diagrammatics developed in \cite{SC}, they establish monoidal Koszul duality with modular coefficients for any Kac-Moody group $\mathscr{G}$. For more details there is a nice introduction to these ideas in \cite{M19}, but the key point right now is that there's an equivalence of monoidal categories
    \[  \kappa^{mon}: (\mathscr{D}(\hf, W), \star ) \xrightarrow{\sim} (\mathrm{TiltFM}(\hf^*, W), \widehat{\star}) \]
    giving rise to the doubly graded equivalence of categories
    \[  \kappa: K^b(\mathscr{D}(\hf, W)\otimes_R \kb) \xrightarrow{\sim} K^b(\kb\otimes_R \mathscr{D}(\hf^*, W))  \]
    interchanging indecomposable and tilting objects such that $\kappa\circ (-1)[1]=(1)\circ \kappa$. In \cite{HM20} Hogancamp and Makisumi conjecture that this can be extended to an equivalence of monoidal categories
    \[  (\kappa^{mon})^\ext: (\mathscr{D}^\ext(\hf, W), \star ) \overset{\sim}{\dashrightarrow} (\mathrm{TiltFM}^\ext(\hf^*, W), \widehat{\star}) \]
    giving rise to a triply graded equivalence of categories
    \[  \kappa^\ext: K^b(\mathscr{D}^\ext(\hf, W)\otimes_R \kb) \overset{\sim}{\dashrightarrow} K^b(\kb\otimes_R \mathscr{D}^\ext(\hf^*, W))  \]
    interchanging indecomposable and tilting objects such that $\kappa^\ext\circ (-1)[1]=(1)\circ \kappa^\ext$ and provide evidence in the $\mathfrak{gl}_2$ realization of $W=S_2$. The main results in this paper can therefore be used to help define $(\kappa^{mon})^\ext$ to prove such an equivalence. According to Makisumi, ''One may speculate that this triply-graded Koszul duality is connected to a duality for character sheaves."
\end{enumerate}

\subsection{Organization}

\begin{itemize}
    \item \textbf{Background}
    \begin{itemize}
        \item In \cref{prelimsect} we review notation and conventions. In particular our notation for Koszul complexes will be important for our computations later on. Throughout the paper we will \underline{always} assume the reader is familiar with Soergel bimodules and the two-color Soergel diagrammatics such as the light leaf morphisms. A good place to read up on this would be in \cite{EMTW20} Chapters 4, 5, 9, 10, 12. 
         \item In \cref{Maksect} we review the diagrammatics of the $A_1$ case $\sbim^\ext(\hf, S_2)$. For the computations done in this paper, it turns out that we only need that Theorem 5.2 in \cite{M} holds, i.e. the defining relations of $\Dext(\hf, S_2)$ hold in $\sbim^\ext(\hf, S_2)$.
    \end{itemize}
    \item \textbf{The Algebraic Category}
    \begin{itemize}
        \item In \cref{maincomptutesec} we compute $\ext_{R^e}^{\bullet, \bullet}(B_t, \bs(\underline{w}))$ as a right $R-$module when $m_{st}=\infty$ which will allow us to define the main new generating morphism $\Phi_t^{\underline{w}}$ in $\bsbimext(\hf, W_\infty)$ in \cref{newgensectgen}. We should note that $\Phi_t^{\underline{w}}$ doesn't quite correspond to the new generator in the diagrammatic category (they are off by a rotation+possibly a sign) but it's much easier to work with $\Phi_t^{\underline{w}}$ in the algebraic category.
        \item In \cref{extindecompsect} we give an explicit basis for $\ext_{R^e}^{\bullet, \bullet}(R, B_w)$ as a right $R$ module which shows that $\ext_{R^e}^{\bullet, \bullet}(R, B_w)$ is more or less controlled by $\hom_{R^e}^{ \bullet}(R, B_w)$.
        \item \cref{finitesect} proves a lot of the same theorems as above but now in the finite case $m_{st}<\infty$. In most cases, we just check that the argument for $m_{st}=\infty$ still holds when $m_{st}<\infty$. 
    \end{itemize}
    \item \textbf{The Diagrammatic Category}
    \begin{itemize}
        \item In \cref{diagraminftysect} and \cref{diagramfinitesect} we define the Ext-enhanced diagrammatic Hecke categories $\Dext_\infty$ and $\Dext_{m_{st}}$ associated to a (faithful) realization $\hf$ of the infinite and finite dihedral groups $W_\infty$, $W_{m_{st}}$ respectively. We will then establish the equivalences $\Dext_\infty\simeq \bsbimext(\hf, W_\infty)$ and $\Dext_{m_{st}}\simeq \bsbimext(\hf, W_{m_{st}})$. The proof bootstraps the proof of the dihedral equivalence in \cite{DC} in that it just suffices to show the isomorphism
        \[  \hom^{\bullet, \bullet}_{\Dext_\infty}(\Bs_\varnothing, \Bs_w) \cong \ext^{\bullet, \bullet}_{R^e}(R, B_w)  \]
        and that we have a very explicit description of the rank 2 projectors to the indecomposable summands $B_w$ given by the \textit{Jones-Wenzl} projectors. Unlike in \cite{DC} our proof relies on one of the main theorems in \cite{SC}, namely that double leaves form a right $R$ basis for hom spaces in the diagrammatic Hecke category.
    \end{itemize}
    \item \textbf{Appendices} 
    \begin{itemize}
        \item  In \cref{chainliftsection} we define chain lifts for morphisms in $\sbim(\hf, W)$. In \cref{appendixhomfly} we first compute HOMFLY homology for the connect sum of two Hopf links. Then we compute the HOMFLY homology for the negative torus link $T(3,-3)$. In \cref{gomisect} we show that the Poincare series for Hochschild homology of Soergel Bimodules agrees with Gomi's trace for $W=W_m$.  
    \end{itemize}

\end{itemize}

\subsection{Acknowledgements}

The author would like to thank Shotaro Makisumi for his guidance and support and for helping the author get over his life long fear of computing cohomology. The problem of giving a diagrammatic presentation for the Ext groups of Soergel Bimodules in rank 2 was suggested to the author by Makisumi and he very generously shared with the author a draft of his results in the $A_1$ case, laying the foundations for the work done here. We also thank Makisumi for reading through a very early and rough draft of the paper, as his comments helped improve the quality of the writing immensely. We also thank Mikhail Khovanov, Ben Elias, Eugene Gorsky and Ben Webster for helpful comments and discussions related to the paper. The author was supported by the U.S. Department of Defense (DoD) through the National Defense Science $\&$ Engineering Graduate (NDSEG) Fellowship.

\section{Preliminaries and Notation}
\label{prelimsect}

\subsection{Realizations and Gradings}

We first recall the definition of a realization of a Coxeter system $(W, S)$ as defined in \cite{SC} Section 3.

\begin{definition}
Let $\kb$ be a commutative ring. A realization of $(W,S)$ over $\kb$ is a triple $$\hf=(V, \set{\alpha_s}_{s\in S}\subset V^*, \set{\alpha_s^\vee}_{s\in S}\subset V)$$ where $V$ is a free, finite rank $\kb$ module such that for $\abrac{-,-}$ the natural pairing between $V, V^*$,                                                                                                                                                                                                                                                                    
\begin{enumerate}[(1)]
    \item $\abrac{\alpha_s^\vee, \alpha_s}=2$ for all $s\in S$.
    \item The assignment $s(v)=v-\abrac{v, \alpha_s} \alpha_s^\vee$ for all $v\in V$ yields a representation of $W$.
    \item For each $s,t\in S$ let $a_{st}:=\abrac{\alpha_s^\vee, \alpha_t}$ and $m_{st}$ the order of $st$ in $W$. Then $[m_{st}]_{a_{st}}=[m_{st}]_{a_{ts}}=0$ where $[k]_x$ are the $2-$colored quantum numbers defined in Section 3 of \cite{SC}.
\end{enumerate}
\end{definition}

For the rest of the paper we will assume the following. 

\begin{assumption}
\label{assump0}
$\hf$ is a balanced realization of rank $n$ that satisfies Demazure surjectivity over an integral domain $\kb$.
\end{assumption}

All of our algebras and modules will be bigraded with the cohomological grading $[1]$ and Soergel/internal grading $(1)$ that shifts degree down 1, e.g. $M(1)_d=M_{d+1}$. If an element has degree $(a,b)$ the first coordinate $a$ will always be the cohomological degree while the second coordinate $b$ will always be the internal degree. If an algebra/module has no grading shift such as $V$, it is in degree $(0,0)$. We will frequently use the combined shift $\dsbrac{1}:=[1](-2)$.

\begin{definition}
Let $R:=\mathrm{Sym}^\bullet (V^*(-2))$ and let $R^e=R\otimes_\kb R$. Let $R^{ee}=R\otimes_\kb R\otimes_\kb R$ and so on with $R^{eee}$, etc.
\end{definition}

This paper deals with the case when $W=W_{m_{st}}$ where $W_{m_{st}}$ is the Coxeter group generated by simple reflections $s,t$ with relation $(st)^{m_{st}}=1$. $W_\infty$ will be the infinite dihedral group. 

\begin{definition}
$\otimes$ will always mean $\otimes_\kb$ while $\otimes_s$ will mean $\otimes_{R^s}$.
\end{definition}

\begin{definition}
Given two complexes $A^\bullet=\oplus_p A^p, B^\bullet=\oplus_p B^p$ of $R^e-$ graded modules, define the bigraded Hom complex as
\[ \un{\hom}_{R^e}(A^\bullet, B^\bullet)=\bigoplus_{i,j}\un{\hom}^{i,j}_{R^e}(A^\bullet, B^\bullet)=\bigoplus_{i,j}\prod_p\hom_{R^e}(A^p, B^{p+i}(j)) \]
with differential $d^n: \un{\hom}^{n,j}_{R^e}(A^\bullet, B^\bullet)\to \un{\hom}^{n+1,j}_{R^e}(A^\bullet, B^\bullet)$ given by 
\[ d^n(f)=d_{B^\bullet}\circ f+(-1)^{n+1} f\circ d_{A^\bullet} \]
\end{definition}

\subsection{Koszul Complexes}
For $a\in R$, let $a^e:=a\otimes 1-1\otimes a\in R^e$. Given a sequence of homogeneous elements $a_1, \ldots, a_n\in R$ where $|a_i|$ is the internal degree of $a_i$, let $K(a_1^e, \ldots, a_n^e)=K(a_1^e)\otimes_{R^e} \ldots \otimes_{R^e} K(a_n^e)$ be the following graded Koszul complex where $K(a_i^e)$ is the complex
$$K(a_i^e)=[0\to R^e \underline{a_i} \xrightarrow{a_i^e} \boxed{R^e\underline{1}}\to 0]= [0\to R^e(-|a_i|) \xrightarrow{a_i^e} \boxed{R^e}\to 0]$$ 
where the boxed term is in cohomological degree 0 and we have underlined $a_i$ and $1$ as they are the ``exterior" part. Specifically, the graded Koszul complex has the structure of a bigraded dga as $K(a_i^e)=\Lambda^\bullet(\kb \underline{a_i} [1](-|a_i|) )\boxtimes R^e$ while $K(a_1^e, \ldots, a_n^e)= \Lambda^\bullet(\kb \underline{a_1} [1](-|a_1|)\oplus \ldots\kb \underline{a_n} [1](-|a_n|) )\boxtimes R^e $ where the differential is determined by $d(\underline{a_i} )=\underline{1}\boxtimes a_i^e$ and the graded Leibniz rule. Here we use the notation $\Lambda^\bullet(\un{V})\boxtimes R^e:= \Lambda^\bullet(\un{V})\otimes R^e$ to distinguish the exterior algebra part from $R^e$. All elements of $\Lambda^\bullet(\un{V})$ will be underlined as well. The $R-$bimodule structure is given by only acting on the $R^e$ part, i.e. $r\cdot(\underline{v}\boxtimes f\otimes g)\cdot  r^\prime=\underline{v}\boxtimes rf\otimes gr^\prime$ and when tensoring two Koszul complexes $\Lambda^\bullet(\un{V_1})\boxtimes R^e$, $\Lambda^\bullet(\un{V_2})\boxtimes R^e$ over $R$ (instead of $R^e$), we tensor the $\Lambda$ and $R^e$ parts separately, i.e
\[ \Lambda^\bullet(\un{V_1})\boxtimes R^e \otimes_R \Lambda^\bullet(\un{V_2})\boxtimes R^e=\Lambda^\bullet(\un{V_1}) \otimes \Lambda^\bullet(\un{V_2})\boxtimes R^e\otimes_R R^e=\Lambda^\bullet(\un{V_1}) \otimes \Lambda^\bullet(\un{V_2})\boxtimes R^{ee}\]
Now the $R-$bimodule structure is given by acting on the leftmost and rightmost tensor factors of $R^{ee}$. We will use the shorthand $\underline{v^{(1)}}:=\underline{v}\otimes \underline{1}$ for $v\in V_1$ and likewise $\underline{w^{(2)}}:=\underline{1}\otimes \underline{w}$ for $w\in V_2$.\\

Two Koszul complexes will be of great importance to us. As shown in \cite {M} $q_\varnothing: K_\varnothing\to R$ and $q_s: K_s\to B_s$ will be projective resolutions of $R$ and $B_s$ as a $R-$bimodule where
\begin{align}
\label{Kvarnothingeq}
    K_\varnothing:=K(e_1^e, \ldots, e_n^{e})&= \Lambda^\bullet(\un{V^*}\dsbrac{1})\boxtimes R^e \\
    \label{Kseq}
    K_s:= K(\rho_s s(\rho_s)^e, e_1^e, \ldots, e_{n-1}^e)(1)&=\Lambda^\bullet( \kb \underline{\rho_s s(\rho_s)} [1](-4) \oplus \un{(V^*)^s}\dsbrac{1} )\boxtimes R^e
\end{align}

where $\rho_s\in V^*$ satisfies $\abrac{\alpha_s^\vee, \rho_s}=1$ which exists by Demazure surjectivity and $\set{e_i}_{i=1}^{n-1}$ is a basis for $(V^*)^s$ while $\set{e_i}_{i=1}^{n}$ is a basis for $V^*$. $q_\varnothing$ is the map sending $\underline{1}\boxtimes 1\otimes 1\mapsto 1\otimes_R 1$ and everything else to 0, and similarly with $q_s$. \\

Koszul complexes have natural chain maps given by contraction. Specifically, following \cite {M}

\begin{definition}
Given $v\in V\dsbrac{-1}$, define $v\lfrown (-): \Lambda^\bullet(V^*\dsbrac{1})\to \Lambda^\bullet(V^*\dsbrac{1})$ by
\[ v \lfrown (r_1 \wedge \cdots \wedge r_k) = \sum_{i = 1}^k(-1)^{i+1} r_i(v) r_1 \wedge \cdots \wedge \widehat{r_i} \wedge \cdots \wedge r_k  \]
One can then check that $v_1 \lfrown (-)$ anticommutes with  $v_2 \lfrown(-)$ and so there is an induced map $\xi \lfrown(-):\Lambda^\bullet(V^*\dsbrac{1})\to \Lambda^\bullet(V^*\dsbrac{1}) $ for any $\xi\in \Lambda^\bullet(V\dsbrac{-1})$. Similarly, one also obtains a map $w \lfrown(-):\Lambda^\bullet(V\dsbrac{-1})\to \Lambda^\bullet(V\dsbrac{-1}) $ for any $w\in \Lambda^\bullet(V^*\dsbrac{1})$.
\end{definition}

Therefore for any $\xi\in \Lambda^\bullet(V\dsbrac{-1})$, by extending $R^e$ linearly, we obtain a map
\begin{align*}
 \widetilde{\iota_\xi}: K_\varnothing&\to K_\varnothing \\
 \un{z}\boxtimes 1\otimes 1 &\mapsto \xi \lfrown(\un{z})\boxtimes 1\otimes 1
\end{align*}
and because $\widetilde{\iota_\xi}$ is a derivation on $K_\varnothing$ it will automatically be a chain map. Let $\iota_\xi$ be the induced map on cohomology. We also have similar contraction maps for the Koszul complex $K_s$ and there is one in particular that will be of great importance, namely, 

\begin{definition}
Define $\gamma_s^\vee: \kb \underline{\rho_s s(\rho_s)} [1](-4) \oplus \un{(V^*)^s}\dsbrac{1}\to \kb $ by $\gamma_s^\vee(\underline{\rho_s s(\rho_s)})=1$ and $\gamma_s^\vee(\un{(V^*)^s})=0$. Then define the chain map
\[ \widetilde{\phi_s}=-\gamma_s^\vee\lfrown(-)\in \un{\hom}_{R^e}^{1,-4}(K_s, K_s) \]
and let $\phi_s$ be the induced map on cohomology. 
\end{definition}

In \cite {M} it was shown that $\alpha_s^\vee \lfrown(-): \Lambda^\bullet(V^*\dsbrac{1})\to \Lambda^\bullet(V^*\dsbrac{1})$ actually lands in $\Lambda^\bullet((V^*)^s\dsbrac{1})$. As a result, we can define

\begin{definition}
Define the chain map 
\[ \widetilde{\eta_s}^\ext= \alpha_s^\vee \lfrown(-)\in \un{\hom}_{R^e}^{1,-3}(K_\varnothing, K_s) \]
and let $\eta_s^\ext$ be the induced map on cohomology. 
\end{definition}

\begin{definition}
\label{extdemazuredef}
Let $s \in S$. Define the exterior Demazure operator $\partial_s: \Lambda^\bullet(V\dsbrac{-1})\to \Lambda^\bullet(V\dsbrac{-1}/\kb \alpha_s^\vee) $ as the composition
\[ \Lambda^\bullet(V\dsbrac{-1})\xrightarrow{\alpha_s \lfrown (-)} \Lambda^\bullet(V\dsbrac{-1})\twoheadrightarrow  \Lambda^\bullet(V\dsbrac{-1}/\kb \alpha_s^\vee) \]
\end{definition}

\subsection{Review of the $A_1$ Case}
\label{Maksect}
In this section we review the main results from \cite{M}. We first recall the definition of the algebraic category $\bsbimext(\hf, W)$.

\begin{definition}
Let $\mathbb{D}^b(R^e-\mathrm{gmod})$\footnote{According to 
Bernhard Keller, this is sometimes called the $\mathbb{Z}$-graded category associated with $D^b(R^e-\mathrm{gmod}).$}be the category with underlying objects $D^b(R^e-\mathrm{gmod})$ and morphism spaces
\[ \hom_{\mathbb{D}^b(R^e-\mathrm{gmod})}(A^\bullet, B^\bullet)=\bigoplus_{n\in \Zb}\hom_{D^b(R^e-\mathrm{gmod})}(A^\bullet, B^\bullet [n]) \]
\end{definition}

\begin{definition}
Given an expression $\underline{w}=(s_1, \ldots, s_m)$, define the corresponding Bott-Samelson complex to be
\[ \mathrm{KS}(\underline{w})=\mathrm{KS}(s_1, \ldots, s_m):=K_{s_1}\otimes_R \ldots\otimes_R K_{s_m} \]
 Then $\bsbimext(\hf, W)$ is the smallest full, additive, graded subcategory of $\mathbb{D}^b(R^e-\mathrm{gmod})$ consisting of complexes isomorphic to Bott-Samelson complexes.
\end{definition}

\begin{lemma}
$q_{\underline{w}}:\mathrm{KS}(s_1, \ldots, s_m)\xrightarrow{q_{s_1}\otimes_R \ldots\otimes_R q_{s_m}}\bs(s_1, \ldots, s_m)$ is a free $R^e$ resolution.
\end{lemma}
\begin{proof}
Proceed by induction on $m$. Note that any free $R^e$ resolution will also be a free left or right $R-$module resolution. So by induction we see that
\[ H^k (\mathrm{KS}(s_1, \ldots, s_{m-1})\otimes_R K_{s_m})=\mathrm{Tor}^R_k(\bs(s_1, \ldots, s_{m-1})_R, {}_RB_{s_m}) \]
where $\bs(s_1, \ldots, s_{m-1})_R$ is the left $R$ module given by the right action of $R$, aka $r\star m =m\cdot r$ for $m\in \bs(s_1, \ldots, s_{m-1})_R$. But since all Bott-Samelson bimodules are free left or right $R$ modules, we see that for $k>0$ the RHS is 0 and for $k=0$, we have
\[  H^0 (\mathrm{KS}(s_1, \ldots, s_{m-1})\otimes_R K_{s_m})=\bs(s_1, \ldots, s_{m-1})\otimes_R B_{s_m}  \]
as sets. But since the morphisms involved were $R^e-$linear it follows that the above is a bimodule isomorphism and so  $\mathrm{KS}(s_1, \ldots, s_m)$ is a free $R^e$ resolution of  $\bs(s_1, \ldots, s_m)$ as desired
\end{proof}

Thus as $\mathrm{KS}(s_1, \ldots, s_m)$ is a complex of projective $R^e$ modules that resolves the bimodule $\bs(s_1, \ldots, s_m)$ it follows that
\[ \hom_{\bsbimext(\hf, W)}(\mathrm{KS}(s_1, \ldots, s_m), \mathrm{KS}(r_1, \ldots, r_k))\cong \ext^{\bullet, \bullet}_{R^e}(\bs(s_1, \ldots, s_m), \bs(r_1, \ldots, r_k)) \]

As mentioned in the introduction, we could have equivalently defined $\bsbimext(\hf, W)$ as replacing $\hom$ spaces with $\ext$ groups. However, the (super)-monoidal structure will then be $\otimes^L_R$ of complexes. The upshot of the definition above\footnote{aka choosing a fixed projective resolution for each object.} is that the (super)monoidal structure is just simply the $\otimes_R$ of complexes and composition is just composition of chain complexes instead of the Yoneda product.\\

As each $\mathrm{KS}(s_1, \ldots, s_m)$ is a complex of projective $R^e-$modules any morphism in the Soergel category $\bsbim(\hf, W)$ will automatically lift to a morphism (unique up to homotopy) in $\bsbimext(\hf, W)$ and $\bsbim(\hf, W)$ embeds inside $\bsbimext(\hf, W)$ fully faithfully as the cohomological degree 0 part. \\

We now recall the definition of the diagrammatic category $\Dext(\hf,  S_2)$ as defined in \cite {M} Section 4.

\begin{definition}
\label{A_1diagdef}
Let $\Dext(\hf,  S_2)$ be the strict $\kb$ linear supermonoidal category associated to a realization $\hf$ of $S_2=\abrac{s}$ defined as follows. 
\begin{itemize}
    \item \textbf{Objects} of $\Dext(\hf,  S_2)$ are words in $S=\set{s}$, i.e. $\underline{w}=(s_1, \ldots, s_n)$ where $s_i\in S$ where the monoidal structure is given by concatenation.
    \item \textbf{Morphism spaces} in $\Dext(\hf,  S_2)$ are bigraded $\kb$ modules. For a morphism $\alpha$ homogeneous of total degree $(\ell, n)$, $\ell$ will be the \underline{cohomological degree} while $n$ will be the \underline{internal or Soergel degree}. Let $|\alpha|=\ell$ the cohomological degree. $\Dext_\infty$ will then be supermonoidal for the cohomological grading. Specifically, $\otimes$ will satisfy the following \underline{super exchange law}
    \[ (h\otimes k)\circ (f\otimes g)=(-1)^{|k||f|} (h\circ f)\otimes (k\circ g) \]
    $\hom_{\Dext(\hf,  S_2)}(\underline{v}, \underline{w})$ will be the free $\kb$ module generated by horizontally and vertically concatenating colored graphs built from certain generating morphisms, such that the bottom and top boundaries are $\underline{v}$ and $\underline{w}$. The generating morphisms will be the generating morphisms of the diagrammatic Hecke category $\mathscr{D}(\hf, S_2)$ for $A_1$ plus the additional two "Hochschild" generators
    \[
\begin{array}{c|cc}
\text{generator}
&
\begin{array}{c}\begin{tikzpicture}[scale=-0.5,thick,baseline]
 \draw (0,-1) to (0,1); \node[hdot] at (0,0) {};
\end{tikzpicture}\end{array}
&
\begin{array}{c}\begin{tikzpicture}[scale=0.5,thick,baseline]
 \draw[dashed] (-0.5,-0.5) rectangle (0.5,0.5); \node at (0,0) {$x$};
\end{tikzpicture}\end{array} \\
\text{bidegree} & (1,-4) & \deg x \\
\text{name} & \text{(bivalent) Hochschild dot} & \text{Exterior Box}
\end{array}
\]
Here, $x$ is a homogeneous element in $\Lambda^\bullet(V\dsbrac{-1})$, and $\deg x$ denotes its bidegree. We also define the following ``univalent Hochschild dots'' as shorthands:
\begin{equation} \label{eqn:hupdot-hdowndot-def}
\begin{array}{c}\begin{tikzpicture}[xscale=0.3,yscale=0.5,thick,baseline]
 \draw (0,0) -- (0,1); \node[hdot] at (0,1) {};
\end{tikzpicture}\end{array}
:=
\begin{array}{c}\begin{tikzpicture}[xscale=0.3,yscale=0.5,thick,baseline]
 \draw (0,-1) -- (0,1); \node[hdot] at (0,0) {}; \node[dot] at (0,1) {};
\end{tikzpicture}\end{array},
\qquad \qquad
\begin{array}{c}\begin{tikzpicture}[xscale=0.3,yscale=-0.5,thick,baseline]
 \draw (0,0) -- (0,1); \node[hdot] at (0,1) {};
\end{tikzpicture}\end{array}
:=
\begin{array}{c}\begin{tikzpicture}[xscale=0.3,yscale=-0.5,thick,baseline]
 \draw (0,-1) -- (0,1); \node[hdot] at (0,0) {}; \node[dot] at (0,1) {};
\end{tikzpicture}\end{array}.
\end{equation}
\item \textbf{Relations} in $\Dext(\hf,  S_2)$ are as follows. All the defining relations of $\mathscr{D}(\hf, S_2)$ in \cite{EK10} will be satisfied plus the relations invovling the "Hochschild'' generators in the subsection below.
\end{itemize}
\end{definition}

\subsubsection{$1-$color Relations}

\begin{enumerate}

\item \textbf{Hochschild dot slides past trivalent vertices:}
\begin{equation} \label{eqn:hdot-trivalent-a}
\begin{array}{c}\begin{tikzpicture}[xscale=0.6,yscale=0.7,thick,baseline]
 \draw (-1,-1) -- (0,0) -- (0,1); \draw (1,-1) -- (0,0); \node[hdot] at (-0.5,-0.5) {};
\end{tikzpicture}\end{array}
=
\begin{array}{c}\begin{tikzpicture}[xscale=0.6,yscale=0.7,thick,baseline]
 \draw (-1,-1) -- (0,0) -- (0,1); \draw (1,-1) -- (0,0); \node[hdot] at (0,0.5) {};
\end{tikzpicture}\end{array}
=
\begin{array}{c}\begin{tikzpicture}[xscale=0.6,yscale=0.7,thick,baseline]
 \draw (-1,-1) -- (0,0) -- (0,1); \draw (1,-1) -- (0,0); \node[hdot] at (0.5,-0.5) {};
\end{tikzpicture}\end{array}
\end{equation}
\begin{equation} \label{eqn:hdot-trivalent-b}
\begin{array}{c}\begin{tikzpicture}[xscale=0.6,yscale=-0.7,thick,baseline]
 \draw (-1,-1) -- (0,0) -- (0,1); \draw (1,-1) -- (0,0); \node[hdot] at (-0.5,-0.5) {};
\end{tikzpicture}\end{array}
=
\begin{array}{c}\begin{tikzpicture}[xscale=0.6,yscale=-0.7,thick,baseline]
 \draw (-1,-1) -- (0,0) -- (0,1); \draw (1,-1) -- (0,0); \node[hdot] at (0,0.5) {};
\end{tikzpicture}\end{array}
=
\begin{array}{c}\begin{tikzpicture}[xscale=0.6,yscale=-0.7,thick,baseline]
 \draw (-1,-1) -- (0,0) -- (0,1); \draw (1,-1) -- (0,0); \node[hdot] at (0.5,-0.5) {};
\end{tikzpicture}\end{array}
\end{equation}

\item \textbf{Hochschild barbell relation:}
\begin{equation} \label{eqn:hbarbell}
\begin{array}{c}\begin{tikzpicture}[xscale=0.3,yscale=0.3,thick,baseline]
 \draw (0,-1) -- (0,1); \node[dot] at (0,-1) {}; \node[hdot] at (0,0) {}; \node[dot] at (0,1) {};
\end{tikzpicture}\end{array}=
\begin{array}{c}\begin{tikzpicture}[scale=0.3,thick,baseline]
 \draw[dashed] (-1,-1) rectangle (1,1); \node at (0,0) {$\alpha_s^\vee$};
\end{tikzpicture}\end{array}
\end{equation}

\item \textbf{Hochschild dot annihilation:}
\begin{equation} \label{eqn:hdot-square}
\begin{array}{c}\begin{tikzpicture}[scale=0.5,thick,baseline]
 \draw (0,-1.5) -- (0,1.5); \node[hdot] at (0,-0.5) {}; \node[hdot] at (0,0.5) {};
\end{tikzpicture}\end{array}
=
0
\end{equation}

\item \textbf{Exterior boxes add and multiply:}
\begin{equation} \label{eqn:exterior-boxes-mult}
\begin{array}{c}\begin{tikzpicture}[scale=0.5,thick,baseline]
 \draw[dashed] (-0.5,-0.5) rectangle (0.5,0.5); \node at (0,0) {$x$};
\end{tikzpicture}\end{array}
+
\begin{array}{c}\begin{tikzpicture}[scale=0.5,thick,baseline]
 \draw[dashed] (-0.5,-0.5) rectangle (0.5,0.5); \node at (0,0) {$y$};
\end{tikzpicture}\end{array}
=
\begin{array}{c}\begin{tikzpicture}[scale=0.5,thick,baseline]
 \draw[dashed] (-1,-0.5) rectangle (1,0.5); \node at (0,0) {$x + y$};
\end{tikzpicture}\end{array}, \qquad
\begin{array}{c}\begin{tikzpicture}[scale=0.5,thick,baseline]
 \draw[dashed] (-0.5,0.5) rectangle (0.5,1.5); \node at (0,1) {$x$};
 \draw[dashed] (-0.5,-1.5) rectangle (0.5,-0.5); \node at (0,-1) {$y$};
\end{tikzpicture}\end{array}
=
\begin{array}{c}\begin{tikzpicture}[scale=0.5,thick,baseline]
 \draw[dashed] (-1,-0.5) rectangle (1,0.5); \node at (0,0) {$x \wedge y$};
\end{tikzpicture}\end{array}
\end{equation}
for $x, y \in  \Lambda^\bullet(V\dsbrac{-1})$.

\item \textbf{Exterior forcing relation:}
\begin{equation} \label{eqn:exterior-forcing}
\begin{array}{c}\begin{tikzpicture}[scale=0.5,thick,baseline]
 \draw (0,-1.5) -- (0,1.5);
 \draw[dashed] (0.5,-0.5) rectangle (1.5,0.5); \node at (1,0) {$x$};
\end{tikzpicture}\end{array}
=
\begin{array}{c}\begin{tikzpicture}[scale=0.5,thick,baseline]
 \draw (0,-1.5) -- (0,1.5);
 \draw[dashed] (-1.9,-0.5) rectangle (-0.5,0.5); \node at (-1.2,0) {$s(x)$};
\end{tikzpicture}\end{array}
+
\begin{array}{c}\begin{tikzpicture}[scale=0.5,thick,baseline]
 \draw (0,0.8) -- (0,1.5); \node [hdot] at (0,0.8) {};
 \draw[dashed] (-1,-0.5) rectangle (1,0.5); \node at (0,0) {$\partial_s(x)$};
 \draw (0,-1.5) -- (0,-0.8); \node[dot] at (0,-0.8) {};
\end{tikzpicture}\end{array}
\qquad \text{for } x \in \Lambda^\bullet(V\dsbrac{-1}),
\end{equation}
where $ \partial_s: \Lambda^\bullet(V\dsbrac{-1})\to \Lambda^\bullet(V\dsbrac{-1}/\kb \alpha_s^\vee) $ is the exterior Demazure operator defined in \cref{extdemazuredef}.
\end{enumerate}

\subsubsection{Further Relations}
The following relations follow from the defining relations above.

\begin{itemize}
    \item \emph{$1-$color Hochschild Jumping:}
    \begin{equation}
    \label{eqn:hochjump}
\begin{array}{c}\begin{tikzpicture}[xscale=0.4,yscale=0.6,thick,baseline]
 \draw (-1,-1) -- (-1,1); \draw (1,-1) -- (1,1); \node[hdot] at (-1,0) {};
\end{tikzpicture}\end{array}
=
\begin{array}{c}\begin{tikzpicture}[xscale=0.4,yscale=0.6,thick,baseline]
 \draw (-1,-1) -- (-1,1); \draw (1,-1) -- (1,1); \node[hdot] at (1,0) {};
\end{tikzpicture}\end{array}    
\end{equation}
\item \emph{$1-$color Cohomology:}
\begin{equation}
    \begin{array}{c}\begin{tikzpicture}[xscale=0.2,yscale=0.6,thick,baseline]
 \node at (-1,0) {\text{\scriptsize $\alpha_s^\vee$}}; \draw[dashed] (-2,-0.3) rectangle (0,0.3);
 \draw (1,-1) -- (1,0); \node[dot] at (1,0) {};
\end{tikzpicture}\end{array}
=
\begin{array}{c}\begin{tikzpicture}[xscale=0.2,yscale=0.6,thick,baseline]
 \node at (-1,0) {\text{\scriptsize $\alpha_s$}}; \draw(-2,-0.3) rectangle (0,0.3);
 \draw (1,-1) -- (1,0); \node[hdot] at (1,0) {};
\end{tikzpicture}\end{array}
\end{equation}
\item \emph{Hochschild coroot annihilation:}
\begin{equation} \label{eqn:coroot-hdot}
\begin{array}{c}\begin{tikzpicture}[xscale=0.2,yscale=0.6,thick,baseline]
 \node at (-1,0) {\text{\scriptsize $\alpha_s^\vee$}}; \draw[dashed] (-2,-0.3) rectangle (0,0.3);
 \draw (1,-1) -- (1,1); \node[hdot] at (1,0) {};
\end{tikzpicture}\end{array}
=
\begin{array}{c}\begin{tikzpicture}[xscale=-0.2,yscale=0.6,thick,baseline]
 \node at (-1,0) {\text{\scriptsize $\alpha_s^\vee$}}; \draw[dashed] (-2,-0.3) rectangle (0,0.3);
 \draw (1,-1) -- (1,1); \node[hdot] at (1,0) {};
\end{tikzpicture}\end{array}
= 0
\end{equation}
\end{itemize}

\subsubsection{Diagrammatics to Bimodules}

\begin{theorem}[Main Theorem of \cite{M}]
\label{Mmaintheorem}
There is a $\kb-$linear monodial functor $\mathcal{F}^\ext: \Dext(\hf, S_2)\to \bsbimext(\hf, S_2)$ extending the equivalence $\mathcal{F}: \mathscr{D}(\hf, S_2)\xrightarrow{\sim} \bsbim(\hf, S_2)$ defined on objects by sending $(s)\mapsto K_s$, $\varnothing\mapsto K_\varnothing$ and on morphisms by sending
\[ \mathcal{F}^\ext\paren{\raisebox{0.5ex}{\begin{tikzpicture}[xscale=0.4,yscale=0.4,thick,baseline]
 \draw (-1,-1) -- (-1,1); \node[hdot] at (-1,0) {};
\end{tikzpicture}} }=\phi_s, \qquad \mathcal{F}^\ext\paren{\raisebox{0ex}{\dboxed{\xi}}}=\iota_\xi \quad \textnormal{ for }\xi\in \Lambda^\bullet(V\dsbrac{-1})\]
and sends the generating morphisms of $\mathscr{D}(\hf, S_2)$ to their class in cohomology of their chain lifts as defined in \cref{chainliftsection}. Specifically,
\[ \mathcal{F}^\ext\paren{\unit}= \sbrac{\widetilde{\eta_s}} , \qquad  \mathcal{F}^\ext\paren{ 
\raisebox{1ex}{\begin{tikzpicture}[xscale=0.25,yscale=-0.3,thick,baseline]
 \draw (0,0) -- (0,1);
 \node[dot] at (0,0) {};
\end{tikzpicture}}}= \sbrac{\widetilde{\epsilon_s}}, \qquad \mathcal{F}^\ext\paren{\begin{array}{c}\begin{tikzpicture}[xscale=0.3,yscale=0.35,thick,baseline]
 \draw (-1,-1) -- (0,0) -- (0,1); \draw (1,-1) -- (0,0);
\end{tikzpicture}\end{array}}= \sbrac{\widetilde{\mu_s}}, \qquad  \mathcal{F}^\ext\paren{\begin{array}{c}\begin{tikzpicture}[xscale=0.3,yscale=-0.35,thick,baseline]
 \draw (-1,-1) -- (0,0) -- (0,1); \draw (1,-1) -- (0,0); 
\end{tikzpicture}\end{array}}=\sbrac{\widetilde{\delta_s}} \]
\end{theorem}

Under $\mathcal{F}^\ext$ the univalent Hochschild dots defined in \cref{eqn:hupdot-hdowndot-def} will be mapped to $\sbrac{\widetilde{\epsilon_s}}\circ \phi_s$ and $\phi_s \circ \sbrac{\widetilde{\eta_s}}$ respectively. In \cite {M} it was shown that $\phi_s \circ\sbrac{\widetilde{\eta_s}}=\eta_s^\ext$, in other words
\[\hunit =\eta_s^\ext \]

Thus we can and will use the RHS above for our computations involving $\hunit$ in the algebraic category instead of $\phi_s \circ\sbrac{\widetilde{\eta_s}}$. It is important to note that $\bsbim(\hf, S_2)$ only sits inside of $\bsbimext(\hf, S_2)$ isomorphically, with objects $B_s$ replaced by the dga $K_s$ and morphisms replaced by their cohomology class of their chain lifts. Likewise, throughout the rest of the paper diagrammatics traditionally representing morphisms in $\bsbim(\hf, W)$ will instead refer to their corresponding morphism in $\bsbimext(\hf, W)$.

\subsection{Action of Exterior Boxes}
Here we explain what diagrams with exterior boxes mean in $\bsbimext(\hf, W)$. Given $\xi\in \Lambda^\bullet(V\dsbrac{-1})$ we have that
\begin{center}
\begin{tikzcd}[row sep=0.4cm]
 \phantom{a}  &    \raisebox{-2ex}{$K_s$} \\
     &  K_\varnothing \otimes_R K_s \ar[u, "\widetilde{\lambda_s} \ "] \\
   &  K_\varnothing \otimes_R K_s  \ar[u, "\widetilde{\iota_\xi}\otimes_R\id \ "] \\ 
\ar[dash, uuu, "\dboxed{\scalemath{1.3}{\xi}} \ ", "\ \  ="']     & K_s \ar[u, "\widetilde{\tau_s} \ "]
\end{tikzcd}
\end{center}
where $\widetilde{\tau_s}$ is the chain lift for the inverse of the left unitor as defined in \cref{chainliftsection} and $\widetilde{\lambda_s}$ is the chain lift of the left unitor for $B_s$ (Specifically, $\widetilde{\lambda_s}= K_\varnothing \otimes_R K_s\xrightarrow{q_\varnothing\otimes_R \id} R\otimes_R K_s\xrightarrow{\mathrm{mult}}K_s $). If $f\in \un{\hom}^{1, \bullet}_{R^e}(K_\varnothing, K_s) $, one can also check that
\[ \dboxed{\xi} \ \  \raisebox{-2ex}{\begin{tikzpicture}[xscale=0.25,yscale=0.3,thick,baseline]
 \draw (0,0) -- (0,3);
 \node[rkhdot, color=black, fill=white] at (0,0){};
 \node at (0,0) {$\s{f}$};
\end{tikzpicture}}= \ \underset{\hspace{-0.4ex}\dboxed{\xi}}{\raisebox{-2ex}{\begin{tikzpicture}[xscale=0.25,yscale=0.3,thick,baseline]
 \draw (0,0) -- (0,3);
 \node[rkhdot, color=black, fill=white] at (0,0){};
 \node at (0,0) {$\s{f}$};
\end{tikzpicture}}} \ =  \raisebox{-2ex}{\begin{tikzpicture}[xscale=0.25,yscale=0.3,thick,baseline]
 \draw (0,0) -- (0,3);
 \node[rkhdot, color=black, fill=white] at (0,0){};
 \node at (0,0) {$\s{f}$};
\end{tikzpicture}}\ \dboxed{\xi} \]
as chain maps (each map has a different algebraic interpretation). Note that the relative positioning of the exterior boxes matter because of the Koszul sign rule. For example, suppose $f\in \un{\hom}^{1, \bullet}_{R^e}(K_s, K_s)$ and $y\in \Lambda^1(V\dsbrac{-1})$. Then we have that
\[ \dboxed{y} \ \raisebox{-3ex}{\begin{tikzpicture}[xscale=0.25,yscale=0.3,thick,baseline]
 \draw (0,0) -- (0,6);
 \node[rkhdot, color=black, fill=white] at (0,4){};
 \node at (0,4) {$\s{f}$};
\end{tikzpicture}}=- \raisebox{4ex}{$\dboxed{y}$} \ \raisebox{-3ex}{\begin{tikzpicture}[xscale=0.25,yscale=0.3,thick,baseline]
 \draw (0,0) -- (0,6);
 \node[rkhdot, color=black, fill=white] at (0,1.75){};
 \node at (0,1.75) {$\s{f}$};
\end{tikzpicture}} \]

\section{Computation of $\ext^{\bullet, \bullet}_{R^e}(B_t, \mathrm{BS}(\underline{w}))$ for $m_{st}=\infty$}
\label{maincomptutesec}

\subsection{Warm-Up computation of \(\ext^{\bullet, \bullet}_{R^e}(R, B_s)\)}

\begin{lemma}
\label{prehochbs}
Let $B$ be a $R^e-$module and suppose there is a subset $J\subset [n]$ s.t. $a_j^e=0 \text{ in }B \ \forall j\in J$, then there is an isomorphism of complexes as bigraded right $R-$modules .
\[ \underline{\hom}_{R^e}( K(a_1^e, \ldots, a_{n}^e) , B)\cong \underline{\hom}_{R^e}( \otimes_{i\not\in J} K(a_i^e) , B)\otimes_\kb \Lambda^\bullet\paren{ \oplus_{j\in J} \kb \underline{a_j} [1](-|a_j|) }^* \]
where the right $R-$module structure is given by $(f\cdot r)(x)=f(x)\cdot r$.
\end{lemma}
\begin{proof}
We have the following chain of isomorphism of complexes
\begin{equation*}
\begin{aligned}
\underline{\hom}_{R^e}( \otimes_{i\not\in J} K(a_i^e)\otimes_{R^e} \otimes_{j\in J} K(   a_j^e) , B) &\cong  \underline{\hom}_{R^e}( \otimes_{i\not\in J} K(a_i^e),\ \un{\hom}_{R^e} (\otimes_{j\in J} K(   a_j^e) , B))  \\
&\cong\underline{\hom}_{R^e}(\otimes_{i\not\in J} K(a_i^e), \Lambda^\bullet( \oplus_{j\in J} \kb \un{a_j} [1](-|a_j|))^*\otimes_\kb B )\\
&\cong\underline{\hom}_{R^e}( \otimes_{i\not\in J} K(a_i^e) , B)\otimes_\kb \Lambda^\bullet( \oplus_{j\in J} \kb \un{a_j} [1](-|a_j|))^*
\end{aligned} 
\end{equation*}
where the isomorphisms arise from the differential of the complex $\Lambda^\bullet( \oplus_{j\in J} \kb \un{a_j} [1](-|a_j|))^*\otimes_\kb B$ being 0 since $a_j^e=0$ in $B$. It is easy to check that these are all maps of right $R-$modules.
\end{proof}

\begin{corollary}
\label{hochbs}
We have an isomorphism of bigraded right $R$ modules
\[\ext_{R^e}^{\bullet, \bullet}( R, B_s )\cong H^\bullet(\underline{\hom}_{R^e}(K(\rho_s^e), B_s))\otimes_{\kb} \Lambda^\bullet(\un{V} \dsbrac{-1} /\kb \alpha_s^\vee )\]
More specifically we have an isomorphism
\[ \ext_{R^e}^{\bullet, \bullet}( R, B_s )\cong  \unit \ R \otimes_{\kb} \Lambda^\bullet(\un{V} \dsbrac{-1} /\kb \alpha_s^\vee ) \bigoplus  \hunit \ R \otimes_\kb \Lambda^\bullet(\un{V} \dsbrac{-1} /\kb \alpha_s^\vee )   \]
where $\hunit $ is the class of the map in $\underline{\hom}^{1,-3}_{R^e}(K(\rho_s^e), B_s)$ sending $\underline{\rho_s}\boxtimes 1\otimes 1\to  1\otimes_{s} 1$, and $\underline{1}\boxtimes 1\otimes 1$ to $0$. 
\end{corollary}
\begin{proof}
The LHS above is the cohomology of the complex $ \un{\hom}_{R^e}(K_\varnothing, B_s)$. From the decomposition $V^*= \kb \rho_s\oplus (V^*)^s$, we see that we can write $K_\varnothing=K( \rho_s^e, e_1^e, \ldots, e_{n-1}^e )$ in \cref{Kvarnothingeq}. Here we use the isomorphism $\Lambda^\bullet((V^*)^s\dsbrac{1})^*\cong \Lambda^\bullet(V \dsbrac{-1} /\kb \alpha_s^\vee )$ as explained in \cite {M} Section 2. Now apply \cref{prehochbs} as $e_i^e=0$ in $B_s$ and we see that
\begin{equation}
\label{homkbs}
    \un{\hom}_{R^e}(K_\varnothing, B_s)\cong \underline{\hom}_{R^e}(K(\rho_s^e), B_s)\otimes_{\kb} \Lambda^\bullet(\un{V} \dsbrac{-1} /\kb \alpha_s^\vee )
\end{equation}
It remains to compute the cohomology of $\underline{\hom}_{R^e}(K(\rho_s^e), B_s) $ 
which is the cohomology of the following complex of right $R-$modules
\[ 0\to \boxed{R(1)\oplus R(-1)}\xrightarrow{ \begin{bmatrix}
-\alpha_s & 0 \\
1 & 0
\end{bmatrix} }R(3)\oplus R(1)\xrightarrow{d^1} 0 \]
where we have decomposed $B_s\cong R(1)\oplus R(-1)= c_{id} R \oplus  c_{s} R $ as a right $R-$module where 
\begin{equation}
\label{01diagram}
    c_{id}=\0 \   \qquad \qquad c_{s}=\1 \
\end{equation}    
Here we are identifying a map in $\hom_{R^e}(B_s, B_s)$ with it's image after being applied to $1\otimes_{s} 1$. This is exactly the $01-$basis for $B_s$, for more information see Section 12.1 of \cite{EMTW20}. The differential $d^0=\rho_s^e$ will be the bimodule map
\[ \boxed{\rho_s} \0- \0 \ \boxed{\rho_s}= \0 \ \boxed{s(\rho_s) -\rho_s} + \1 \ \boxed{\partial_s(\rho_s)}= \0 \ \boxed{-\alpha_s} +\1\]
Applying this to the two basis vectors $c_{id}, c_s$ by stacking it above the diagrams in \cref{01diagram} and simplifying using diagrammatics yields the matrix above . We know $\hom_{R^e}(R, B_s )=R \unit$ so it remains to compute $\ext^1_{R^e}(R, B_s )$. We have
\[ \ker d^1=R\begin{bmatrix}
1 \\
0
\end{bmatrix}\oplus R\begin{bmatrix}
0 \\
1
\end{bmatrix}, \quad \im d^0=R\begin{bmatrix}
-\alpha_s \\
1
\end{bmatrix} \]
Applying the invertible matrix to $\ker d^1$ below 
\[Q=\begin{bmatrix}
1 & -\alpha_s\\
0 & 1
\end{bmatrix}\]
we see that $\ker d^1\cong R c_{id} \oplus \im d^0 \implies \ext^1_{R^e}(R, B_s )= R \hunit$ as desired.
\end{proof}

Explicity the isomorphism in \cref{homkbs} is given on homogeneous components by
\begin{equation}
\label{explicitisoeq}
\begin{aligned}
\underline{\hom}^{k, \bullet}_{R^e}( K_\varnothing , B_s)&\longrightarrow \underline{\hom}_{R^e}(K(\rho_s^e), B_s)\otimes_\kb \Lambda^\bullet( \un{(V^*)^s} \dsbrac{1})^*\\
\psi&\longrightarrow \sum_{1\le i_1<\ldots< i_k \le n-1} \paren{\underline{1}\boxtimes 1\otimes 1\mapsto \psi( \underline{e_{i_1}}\wedge\ldots\wedge \underline{e_{i_{k}}}  ) }\otimes \paren{\underline{e_{i_1}}\wedge\ldots\wedge \underline{e_{i_{k}}} }^*\\
+&\sum_{1\le i_1<\ldots< i_{k-1}\le n-1} \paren{\underline{\rho_s}\boxtimes 1\otimes 1\mapsto \psi( \underline{\rho_s}\wedge \underline{e_{i_1}}\wedge\ldots\wedge \underline{e_{i_{k-1}}}  ) }\otimes \paren{\underline{e_{i_1}}\wedge\ldots\wedge \underline{e_{i_{k-1}}} }^*
\end{aligned}
\end{equation}
where $v^*$ is the dual basis vector of $v\in \Lambda^\bullet( (V^*)^s \dsbrac{1})$.

\begin{lemma}
For any $f\in \underline{\hom}_{R^e}(K(\rho_s^e), B_s)$ and $\xi\in \Lambda^\bullet((V^*)^s\dsbrac{-1})^* $ the following relation holds at chain level
\begin{equation}
\underset{\hspace{-0.4ex}\dboxed{\xi}}{\raisebox{-2ex}{\begin{tikzpicture}[xscale=0.25,yscale=0.3,thick,baseline]
 \draw (0,0) -- (0,3);
 \node[rbbkhdot, color=black, fill=white] at (0,0){};
 \node at (0,0) {$\scalemath{0.7}{f\otimes \un{1}^*}$};
\end{tikzpicture}}}=f\otimes \underline{\xi} 
\end{equation}
where $f\otimes \un{1}^*\in \underline{\hom}_{R^e}( K_\varnothing , B_s) $ under the isomorphism in \cref{homkbs}.
\end{lemma}
\begin{proof}
This holds by direct computation. 
\end{proof}

In other words, the $\Lambda^\bullet((V^*)^s\dsbrac{-1})^* $ part of $\underline{\hom}_{R^e}(K_\varnothing, B_s)$ can always be extracted out as an exterior box leaving just $\underline{\hom}_{R^e}(K(\rho_s^e), B_s)$ which has the advantage of being much easier to work with as it's a smaller complex. We will sometimes abuse notation and write $f\in \underline{\hom}_{R^e}(K(\rho_s^e), B_s)$ when we really mean $f\otimes \un{1}^*\in \underline{\hom}_{R^e}( K_\varnothing , B_s) $.

\begin{remark}
Recall that the morphism  $\eta_s^{\ext}$ was defined to be the class of the map $\alpha_s^\vee \lfrown(-)\in \un{\hom}_{R^e}^{1,-3}(K_\varnothing, K_s)$. Postcomposing the resolution $q_{s}: K_s\to B_s$ with the isomorphism in \cref{explicitisoeq}, we see that
\begin{align*}
\un{\hom}_{R^e}^{1,\bullet}(K_\varnothing, K_s)\to \un{\hom}_{R^e}^{1,\bullet}(K_\varnothing, B_s)\to&  \underline{\hom}_{R^e}(K(\rho_s^e), B_s)\otimes_\kb \Lambda^\bullet( \un{(V^*)^s} \dsbrac{1})^*\\
    \alpha_s^\vee \lfrown(-)\mapsto  \alpha_s^\vee \lfrown(-)|_{\Lambda^1(\un{V^*}\dsbrac{1})\boxtimes R^e}\mapsto& \sum_{i=1}^{n-1} \paren{\un{1}\boxtimes 1\otimes 1\mapsto \abrac{\alpha_s^\vee, e_i} \otimes_{R^s} 1} \otimes (\un{e_i})^*\\ &+\paren{\un{\rho_s}\boxtimes 1\otimes 1\mapsto \abrac{\alpha_s^\vee, \rho_s} \otimes_{R^s} 1} \otimes \un{1}^*= \hunit\otimes \un{1}^*
\end{align*}

as $\abrac{\alpha_s^\vee, -}$ is 0 on $(V^*)^s$. The first map is an isomorphism after taking cohomology. As \hunit \  was also used to denote $\eta_s^{\ext}$ in \cite {M}, the above calculation and paragraph above this remark justifies our notation.
\end{remark}

\subsection{Main Computation}

Let red correspond to the simple reflection $s$ and blue correspond to the simple reflection $t$. Let $\underline{w}$ be an arbitrary expression in $s$ and $t$. For the rest of the paper we will also assume 

\begin{assumption}
\label{assump1}
$\exists \rho_s, \rho_t\in V^*$ such that $\abrac{\alpha_s^\vee, \rho_s}=1=\abrac{\alpha_t^\vee, \rho_t}$ and $\abrac{\alpha_s^\vee, \rho_t}=0=\abrac{\alpha_t^\vee, \rho_s}$.
\end{assumption}

In particular, this means that $\rho_s\in (V^*)^t$ and that $\rho_t\in (V^*)^s$.

\begin{remark}
This assumption is satisfied for both $\mathbb{C}^n$, the permutation realization, and the geometric realization for $W_{m_{st}}$. This assumption is also satisfied for the Kac-Moody realization for $W_\infty$, as the coroots $\set{\alpha_s^\vee}$ of the realization are linearly independent.
\end{remark}

\begin{lemma}
Let $F: \Ac\to \mathrm{End}(\Cc)$ be a monoidal functor from $(\Ac, \otimes)$ a monoidal category to the category of endofunctors of a category $\Cc$. Suppose we have objects $a_1, a_2\in \Ac$ such that $a_1\otimes(-)$ is left adjoint to $a_2\otimes(-)$ in $\Ac$, then 
\[ \hom_{\Cc}(F_{a_1}(c), c^\prime)\cong \hom_{\Cc}(c, F_{a_2}(c^\prime) ) \]
\end{lemma}
\begin{proof}
We need to give a unit and counit map satisfying the triangle identities. Since $F$ is monoidal, we see that $F_{a_2}\circ F_{a_1}=F_{a_2\otimes a_1}$ and thus the unit map will just come from applying $F$ to the unit map in $\Ac$, $\mathds{1}\to a_2\otimes a_1$ and similarly with the counit map and because $F$ is a functor these will satisfy the triangle identities.
\end{proof}

\begin{corollary}
\label{cor:adjunct}
Let $M, N\in \bsbim(\hf, W)$. Then we have a natural isomorphism of right $R-$modules, 
\[ \ext^{\bullet, \bullet}_{R^e}(B_t\otimes_R M, N)\xrightarrow{\sim} \ext^{\bullet, \bullet}_{R^e}(M,B_t\otimes_R N) \]
\end{corollary}
\begin{proof}
Apply the above lemma where $\Ac=\bsbim(\hf, W)$, $\Cc=\bsbimext(\hf, W)$ and $F$ is defined on objects as $F(B_s)=K_s\otimes_R (-)$ and then extended monoidally. On morphisms $F(f)=\widetilde{f}$, where $\widetilde{f}$ is any chain lift of $f$. Now use that $B_t\otimes(-)$ is biadjoint to itself with unit $u_t$. Then the isomorphism above is given by is given by $f\mapsto f\circ (u_t \otimes_R \id_M)$ which is clearly right $R-$linear.
\end{proof}

Our goal is to compute $\ext^{\bullet, \bullet}_{R^e}( \bs(\underline{v}), \bs(\underline{w}))$ but as seen above it suffices to compute Hochschild cohomology $\ext^{\bullet, \bullet}_{R^e}(R, \bs(\underline{w}))$. However it also suffices to compute $\ext^{\bullet, \bullet}_{R^e}(B_t,  \bs(\underline{w}))$ for all $\underline{w}$ and this turns out to be easier. Because of \cref{assump1}, $\rho_s, \rho_t$ will be linearly independent, so we can decompose $V^*=\kb \rho_s \oplus \kb \rho_t \oplus (V^*)^{s,t}$ as $s,t$ are reflections and thus fix a subspace of codimension 1. Let $\Lambda_{st}=\Lambda^\bullet((V^*)^{s,t} \dsbrac{1}  )^*\cong \Lambda^\bullet(  V \dsbrac{-1}/(\kb \alpha_s^\vee\oplus \kb \alpha_t^\vee)  )  $. As such the resolution of $B_t$ given in \cref{Kseq} can be written as
\[ K_t=K(\rho_t t(\rho_t)^e, \rho_s^e, e_1^e, \ldots, e_{n-2}^e)(1)  \]
where $\set{e_i}_{i=1}^{n-2}$ is a basis for $(V^*)^{s,t} $. As $e_i^e=0$ on $ \mathrm{BS}(\underline{w})$ when $e_i\in (V^*)^{s,t}$ we can use \cref{prehochbs} to obtain

\begin{theorem}
\label{prehochBSW}
We have an isomorphism of bigraded right $R$ modules, 
$$\ext^{\bullet, \bullet}_{R^e}(B_t, \mathrm{BS}(\underline{w} ))\cong H^\bullet(\underline{\hom}_{R^e}(K( \rho_t t(\rho_t)^e , \rho_s^e)(1),  \mathrm{BS}(\underline{w} )))\otimes_{\kb} \Lambda_{st}$$
\end{theorem}

As in the previous section, this decomposition shows that elements in $\Lambda_{st}$ act freely by contraction and so WLOG we will \fbox{assume $ \Lambda_{st}=\kb$} (and thus $K_t=K( \rho_t t(\rho_t)^e , \rho_s^e)(1)$) for the remainder of this paper to make notation easier.\\

To distinguish the terms $\rho_s^e$ and $\rho_t t(\rho_t)^e$ in the differential for $\underline{\hom}_{R^e}(K(\rho_t t(\rho_t)^e, \rho_s^e)(1), \mathrm{BS}(\underline{w} ))$ as $\underline{w}$ varies, let 
$$\rhosw: \mathrm{BS}(\underline{w} )\to \mathrm{BS}(\underline{w} )\quad \quad \rhottw: \mathrm{BS}(\underline{w} )\to \mathrm{BS}(\underline{w} ) $$ 

Explicitly, $K( \rho_t t(\rho_t)^e , \rho_s^e)(1)$ is the complex 
\begin{equation}
\label{ktcomplex}
0\to \underline{\rho_s\wedge\rho_t t(\rho_t)}(1)\boxtimes R^e\xrightarrow{ \begin{bmatrix}
-\rho_t t(\rho_t)^e  \\
\rho_s^e
\end{bmatrix} }\underline{\rho_s}(1)\boxtimes R^e\oplus \underline{\rho_t t(\rho_t)}(1) \boxtimes R^e \xrightarrow{ \begin{bmatrix}
\rho_s^e\ & \rho_tt(\rho_t) ^e
\end{bmatrix} } \boxed{\underline{1}(1)\boxtimes R^e} \to  0 
\end{equation}

And thus it follows that $\underline{\hom}_{R^e}(K( \rho_t t(\rho_t)^e , \rho_s^e),  \mathrm{BS}(\underline{w}))$ will be the total complex of the following double complex
\begin{equation}
\label{keydoublecomplex}
\begin{tikzcd}
& \mathrm{BS}(\underline{w} ) \underline{\rho_tt(\rho_t)} (-1) \arrow[r, "\rho_s^e" ]&  \mathrm{BS}(\underline{w} ) \underline{\rho_s\wedge\rho_tt(\rho_t)} (-1)\\
 &\boxed{\mathrm{BS}(\underline{w} )}\, \underline{1} (-1) \arrow[u, "\rho_tt(\rho_t)^e"] \arrow[r, "\rho_s^e" ] & \mathrm{BS}(\underline{w} ) \underline{\rho_s}(-1) \arrow[u, "\rho_t t(\rho_t)^e", swap]
\end{tikzcd}=\begin{tikzcd}
& \mathrm{BS}(\underline{w} ) (3) \arrow[r, "\rho_s^e" ]&  \mathrm{BS}(\underline{w} ) (5)\\
 &\boxed{\mathrm{BS}(\underline{w} )}\, (-1) \arrow[u, "\rho_tt(\rho_t)^e"] \arrow[r, "\rho_s^e" ] & \mathrm{BS}(\underline{w} )(1)  \arrow[u, "\rho_t t(\rho_t)^e", swap]
\end{tikzcd}
\end{equation}
where the boxed term is in cohomological degree 0 and on the LHS an element  $b\underline{1}$ for $b\in\mathrm{BS}(\underline{w} ) $ corresponds to the map sending $\underline{1}\boxtimes 1\otimes 1\in K(\rho_tt(\rho_t)^e, \rho_s^e)$ to $b$ (likewise with $b\un{\rho_s}$, etc) and on the RHS we have just replaced this by the corresponding internal degree, etc. It's clear that the corresponding spectral sequence of the double complex degenerates at the $E_2$ page and so taking horizontal cohomology first the $E_1$ page will look like

\begin{equation}
\label{keyeq}
\begin{tikzcd}
& H^0(\underline{\hom}_{R^e}(K( \rho_s^e), \mathrm{BS}(\underline{w} ))  ) (3) &  H^1(\underline{\hom}_{R^e}(K( \rho_s^e), \mathrm{BS}(\underline{w} ))  ) (5)\\
 &\boxed{H^0(\underline{\hom}_{R^e}(K( \rho_s^e), \mathrm{BS}(\underline{w} ))  )\, (-1)} \arrow[u, "\rho_tt(\rho_t)^e"]& H^1(\underline{\hom}_{R^e}(K( \rho_s^e), \mathrm{BS}(\underline{w} ))  )(1)  \arrow[u, "\rho_t t(\rho_t)^e", swap]
\end{tikzcd}
\end{equation}
In fact, we will show in \cref{mainsection} that both arrows in \cref{keyeq} are 0. The bottom left corner of the $E_2$ page will compute $\ext^{0, \bullet}_{R^e}(B_t, \mathrm{BS}(\underline{w} ))$ and the top right corner will compute $\ext^{2, \bullet}_{R^e}(B_t, \mathrm{BS}(\underline{w} ))$. The diagonal will give us a filtration on $\ext^{1, \bullet}_{R^e}(B_t, \mathrm{BS}(\underline{w} ))$. However, we will show that both groups are free right $R$ modules, and thus the filtration splits and the $E_2$ page will also compute $\ext^{1, \bullet}_{R^e}(B_t, \mathrm{BS}(\underline{w} ))$ as well. For the remainder of this section we will work in the case $W=W_\infty$ and will show later that the results also hold in the finite case $W=W_{m_{st}}$.

\subsubsection{Computation of $H^0(\underline{\hom}_{R^e}(K( \rho_s^e), \mathrm{BS}(\underline{w} ))  )$}
\label{h0section}

Throughout we will fix the total ordering on $W_\infty$
\[ \id< s< t<ts<st<\ldots \]
which refines the Bruhat order. For this section we will make the additional assumption 
\begin{assumption}
\label{assump2} $\set{\alpha_s, \alpha_t}$ is linearly independent, $\mathrm{char}\ \kb=0$ and $\hf$ is a symmetric realization, i.e. $a_{st}=a_{ts}$.
\end{assumption}

Following Section 3 in \cite{DC}, set $[2]=q+q^{-1}=-a_{st}=-a_{ts}$ and let $[m]=\tfrac{q^m-q^{-m}}{q-q^{-1}}$ for $m\in \mathbb{Z}^{\ge 0}$. It was shown in \cite{DC} that $[m]$ is in fact a polynomial in $[2]$, and thus $[m]\in \kb$ is well defined given $a_{st}$. \\

Our computations below will use the light leaves basis for $\mathrm{BS}(\underline{w} )$ as a free right $R-$module, more information can be found in Section 12.4 of \cite{EMTW20}. Here is a brief summary. For any subexpression $\underline{f}\subset \underline{w}$ (A subexpression for $\underline{w}=(w_1, \ldots w_m)$ is a string $\underline{f}=(f_1, \ldots,f_m)$ where $f_i\in\set{0,1}$), let $r(\underline{f})=\underline{w}^{\underline{f}}\in W_\infty$. Then there is a (flipped) light leaf map
\[ \overline{\mathrm{LL}}_{  \underline{w}, \underline{f}}: \bs(r(\underline{f}))\to \bs(\underline{w})  \]
One technically needs to choose a reduced expression for $r(\underline{f})\in W$ to make sense of above. However this isn't that big of an issue. For the affine case there is only one reduced expression for any $x\in W_\infty$ so $r(\underline{f})$ is a well defined expression. For the finite case any two reduced expressions are related by braid moves and the only effect on the light leaf map will be the addition of various $2m_{st}-$valent morphisms. \\  

Now let $c_{bot}=c_{id}\otimes_{R}\ldots \otimes_{R}c_{id}\in \mathrm{BS}(r(\underline{f}))$\footnote{We will often abuse notation and let $c_{bot}$ denote the corresponding element $c_{id}\otimes_R \ldots$ for any choice of $\underline{f}$.}. Then the light leaves basis for $\bsw$ is the set
\[ \set{\ll{f}=\overline{\mathrm{LL}}_{  \underline{w}, \underline{f} }(c_{bot}) \ | \ \underline{f}\subset \underline{w}} \]

\begin{lemma}
\label{keylemma}
Arrange the light leaves basis based on $r(\underline{f})$ from least to greatest in the total ordering on $W$ indicated above (can choose an arbitrary order on those elements with the same $r(\underline{f})$). In this basis, the matrix for $\rhosw$ as a right $R$ module will be upper triangular with diagonal entries equal to $r(\underline{e})^{-1}(\rho_s)-\rho_s$.
\end{lemma}
\begin{proof}
The proof of Proposition 12.26 in \cite{EMTW20} shows 
\begin{equation}
\rho_s \cdot \ll{e}= \ll{e}\cdot r(\underline{e}) ^{-1}(\rho_s)+\sum_{ \underline{w} ^{\underline{f^\prime}}<r(\underline{e}) } \ll{f^\prime}a_{f^\prime}     \quad\quad a_{f^\prime}\in R
\end{equation}  
and thus the lemma follows.
\end{proof}

\begin{corollary}
\label{keycor}
When $m_{st}=\infty$, the image of $\rhosw: \bsw\to \bsw$ is a free right $R$ module and the kernel is also a free right $R$ module with basis given by $\set{\ll{f} \, | r(\underline{f})=\id \textnormal{ or }t }$.
\end{corollary}
\begin{proof}
By \cref{keylemma} we can find a right $R$ basis for $\bsw$ such that $\rhosw$ is upper triangular with diagonal entries equal to $r(\underline{e})^{-1}(\rho_s)-\rho_s$. It is clear that when $r(\underline{e})=id \textnormal{ or }t $ that this expression is 0; we claim in all other cases, this will not be zero. One can easily check by induction that
\begin{align}
s(ts)^m(\rho_s) &=\rho_s-\sum_{k=0}^m (st)^k(\alpha_s)   \\ 
(ts)^m(\rho_s) &=\rho_s-\sum_{k=0}^{m-1} t(st)^k(\alpha_s) 
\end{align}
By Claim 3.5 in \cite{DC}, one then computes that
\begin{align}
\label{stsrhos}
s(ts)^m(\rho_s)-\rho_s &=-\sum_{k=0}^m\paren{ [2k+1]\alpha_s+[2k]\alpha_t }  \\\label{tsrhos}
(ts)^m(\rho_s)-\rho_s &=-\sum_{k=0}^{m-1} \paren{[2k+1]\alpha_s+[2k+2]\alpha_t}
\end{align}
To show \cref{stsrhos} and \cref{tsrhos} are nonzero, as $\alpha_s, \alpha_t$ are linearly independent, it suffices to show the sums
\[O_m=\sum_{k=0}^m [2k+1]\quad \quad E_m=\sum_{k=0}^m [2k] \]
don't simultaneously vanish. Similarly, to show \cref{tsrhos} is nonzero it suffices to show $O_{m-1}$ and $E_m$ don't simultaneously vanish. Depending on the realization there are two cases. First if $q=\pm 1$, it's clear that $O_m, E_m$ is never zero for $m\ge 0, m\ge 1$ respectively, in characteristic zero. Now for $q\neq \pm 1$, one computes that
\begin{align*}
O_m&= \frac{q+q^3+\ldots +q^{2m+1}-q^{-1}-q^{-3}-\ldots-q^{-(2m+1)}}{q-q^{-1}}    \\
&=\frac{1-q^{2m+2}+1-q^{-(2m+2)}}{(1-q^{-2})(1-q^2)} =-\frac{(q^{m+1}-q^{-(m+1)})^2}{(1-q^{-2})(1-q^2)} \\
E_m&= \frac{q^2+q^4+\ldots +q^{2m}-q^{-2}-q^{-4}-\ldots-q^{-(2m)}}{q-q^{-1}}=\frac{q^2-q^{-2m}+1-q^{2m+2}}{(q-q^{-1})(1-q^2)}     \\
&=\frac{q^{-2m}(q^{2m+2}-1)-(q^{2m+2}-1)}{(q-q^{-1})(1-q^2)}=\frac{(q^{-2m}-1)(q^{2m+2}-1)}{(q-q^{-1})(1-q^2)}
\end{align*}  
It follows that $O_m=0\iff q^{2m+2}=1$ and $E_m=0\iff q^{2m}=1 \textnormal{ or }q^{2m+2}=1$. It follows that
\begin{equation}
\label{vanishingeq}
\cref{stsrhos}=0\iff q^{2m+2}=1 \quad \quad \cref{tsrhos}=0\iff q^{2m}=1
\end{equation} 
But, as noted in \cite{DC}, a realization for $W_\infty$ will be a realization for $W_k$ $\iff$ $q^{2k}=1$. Thus for $\hf$ to be a (faithful )realization for $W_\infty$ $q^{2m}\neq 1 \ \forall m\in \mathbb{Z}^+$ and thus \cref{stsrhos} is never 0 for $m\ge 0$ and \cref{tsrhos} is never zero for $m\ge 1$. \\

Now note that when $r(\underline{e})=id \textnormal{ or }t $ not only does $r(\underline{e})^{-1}(\rho_s)-\rho_s=0$, but in fact $\ll{e}\in \ker \rho_s^e(\underline{w})$. This is clear from the diagrammatic picture as $r(\underline{e})$ corresponds to the red or blue lines protruding below $\ll{e}$ and it is precisely these lines which we need to slide $\rho_s$ from left to right over.\\

Therefore the image of $\rho_s^e(\underline{w})$ will be the span of $\set{ \rho_s^e(\underline{w})(\ll{f}) \, | r(\underline{f})\neq id \textnormal{ or }t }$. As shown above, we can arrange these vectors into a matrix that is upper triangular with nonzero entries on the diagonal. 
Because $R$ is a domain this will imply that $\set{ \rho_s^e(\underline{w})(\ll{f}) \, | r(\underline{f})\neq id \textnormal{ or }t }$ is linearly independent over $R$ and therefore free. It then follows that $\ker \rho_s^e(\underline{w})$ is generated by $\set{\ll{f} \, | r(\underline{f})=id \textnormal{ or }t }$ and thus is a basis.
\end{proof}

\subsubsection{Computation of $H^1(\underline{\hom}_{R^e}(K( \rho_s^e), \mathrm{BS}(\underline{w} ))  )$}

In this section we will use the 01-basis for $\bsw$ as defined in Chapter 12 of \cite{EMTW20}. In particular for $\underline{w}=(s_1, \ldots, s_m)$, $c_{bot}$ as defined in \cref{h0section} and
\[c_{top}=c_{s_1}\otimes_{R} \ldots\otimes_{R} c_{s_m} \in \bsw\]
are two special elements in this basis.

\begin{definition}
Define $\mathrm{Tr}:\mathrm{BS}(\underline{w})\to R$ by sending any element $b$ to the coefficient of $c_{top}$ when $b$ is expressed in the 01-basis as a right $R-$module.
\end{definition}

\begin{definition}
The global intersection form on $\mathrm{BS}(\underline{w})$ is the $R-$valued pairing 
\begin{equation*}
    \abrac{-,-}: \mathrm{BS}(\underline{w})\times \mathrm{BS}(\underline{w})\to R
\end{equation*}
defined by
\[ \abrac{a,b}=\mathrm{Tr}(ab) \]
\end{definition}

The global intersection form on $\mathrm{BS}(\underline{w})$ is non-degenerate as seen in \cite[Section 18.2.2.2]{EMTW20} and gives an isomorphism $D_{\underline{w}}:\bsw\cong \Db(\bsw)$ sending $v\mapsto\abrac{v,-}$.\\

Recall that in \cref{keycor} we found a right $R$ basis for $\ker \rhosw$ given by $\set{\ll{f} \, | r(\underline{f})=\id \textnormal{ or }t }$. For $\underline{f}=00\ldots$, we clearly have $r(\underline{f})=\id $ and the corresponding light leaves will be all start dots, i.e. the element $c_{s_1}\otimes_R \ldots \otimes_R c_{s_m}$ which is exactly $c_{top}$ in the 01-basis. One can easily check that the dual basis vector under the global intersection form will be $c_{bot}$ which we now denote 
$$1(\underline{w}):=c_{id}\otimes_R \ldots \otimes_R c_{id} \in \Db( \ker\rhosw )$$

\begin{definition}
Let $\Db: R^e-\mathrm{gmod}\to R^e-\mathrm{gmod}$ be the functor
\[ \Db(N):=\hom_{\_ R}(N, R) \]
where $\_ R$ means we take right $R$ module homomorphisms. This has a $R^e-\mathrm{gmod}$ structure defined as $(r\cdot f \cdot r^\prime)(b):=r \cdot f(b) \cdot r^\prime$.
\end{definition}

\begin{theorem}
\label{h1rhos}
We have an isomorphism of right $R$ modules
\[  H^1(\underline{\hom}_{R^e}(K( \rho_s^e), \mathrm{BS}(\underline{w} ))  ) \cong \Db(\ker \rhosw)  \]
\end{theorem}
\begin{proof}
We have a commutative diagram
\begin{equation}
\label{globalcommdiag}
    \begin{tikzcd}
    \bsw\arrow[r, "\rhosw"]\arrow[d, "D_{\underline{w}}", swap]  & \bsw \arrow[d, "D_{\underline{w}}"] \\
    \Db(\bsw)\arrow[r, "\Db(\rhosw)" ] & \Db(\bsw)
    \end{tikzcd}
\end{equation}
which follows from adjointness of the global intersection form with multiplication by elements of $R$. As $D_{\underline{w}}$ is an isomorphism it follows that 
$H^1(\underline{\hom}_{R^e}(K( \rho_s^e), \mathrm{BS}(\underline{w} ))  ) = \mathrm{coker \ } \rhosw\cong \mathrm{coker \ } \Db(\rhosw) $. From \cref{keycor} $\im \rhosw$ is a free right $R-$module and so we have a SES
\[ 0\to \Db(\im \rhosw)\xrightarrow{\Db(\rhosw)} \Db(\bsw)\to \Db(\ker \rhosw)\to 0 \]
and thus it follows that $\mathrm{coker \ } \Db(\rhosw)\cong \Db(\ker \rhosw)$.
\end{proof}

\subsubsection{Computation of $H^\bullet(\underline{\hom}_{R^e}(K( \rho_t t(\rho_t)^e , \rho_s^e),  \mathrm{BS}(\underline{w})))$}
\label{mainsection}

Using \cref{h1rhos}, \cref{keyeq} which is the $E_1$ page of the spectral sequence computing \\ $H^\bullet(\underline{\hom}_{R^e}(K( \rho_t t(\rho_t)^e , \rho_s^e),  \mathrm{BS}(\underline{w}) ))$ as a right $R-$module now reads

\begin{equation}
\label{keyeq2}
\begin{tikzcd}
& \ker \rhosw (3) & \Db(\ker \rhosw) (5)\\
 &\boxed{\ker \rhosw\, (-1)} \arrow[u, "\rho_t t(\rho_t)^e(\underline{w})"]& \Db(\ker \rhosw)(1)  \arrow[u, "D_{\underline{w}}(\rho_t t(\rho_t)^e(\underline{w}))|_{\ker \rhosw}", swap]
\end{tikzcd}
\end{equation}
We claim that both vertical arrows above are 0. From \cref{keycor} we know that $\ker \rhosw$ has a basis given by
 \[\set{\ll{f}=\overline{LL}_{\underline{w}, \underline{f}}(c_{bot}) \, | r(\underline{f})=id \textnormal{ or }t }\]

It follows that $\ker \rhosw \subseteq \ker \rhottw$ aka $\rho_t t(\rho_t)^e(\underline{w})|_{\ker \rhosw}=0$ as we can always slide $\rho_t t(\rho_t)$ over to the other side.\\

For the right hand $\rho_t t(\rho_t)^e$ note that replacing $\bsw$ with $\ker \rhosw$ in \cref{globalcommdiag} still produces a commutative diagram. It follows that for any $v\in \ker\rhosw$ we have
\[ D_{\underline{w}}(\rho_t t(\rho_t)^e)|_{\ker \rhosw}(\abrac{v, -})=\abrac{\rho_t t(\rho_t)^e(v), -  }|_{\ker \rhosw}=\abrac{v, \rho_t t(\rho_t)^e(-)  }|_{\ker \rhosw}=0 \]
where the last equality follows from the previous paragraph. Thus from \cref{prehochBSW} we see that

\begin{theorem}
\label{maincohoiso}
When $m_{st}=\infty$ and $\Lambda_{st}=\kb$ we have an isomorphism of right $R-$modules, 
\begin{align*}
\ext^{0, \bullet}_{R^e}(B_t, \mathrm{BS}(\underline{w}))&\cong \ker\rhosw \underline{1 } (-1)\cong \ker\rhosw (-1) \\
\ext^{1, \bullet}_{R^e}(B_t, \mathrm{BS}(\underline{w}))&\cong
\ker\rhosw\underline{\rho_t t(\rho_t)} (-1)\oplus  \widetilde{\Db(\ker\rhosw )}\underline{\rho_s}(-1)\\
&\cong  \ker\rhosw (3)\oplus  \widetilde{\Db(\ker\rhosw )}(1)  \\
\ext^{2, \bullet}_{R^e}(B_t, \mathrm{BS}(\underline{w}))&\cong \Db(\ker\rhosw) \underline{\rho_s \wedge \rho_t t(\rho_t)}(-1) \cong \Db(\ker\rhosw) (5) 
\end{align*}
\end{theorem}
\begin{proof}
The fact that both arrows going up in \cref{keyeq2} are 0 immediately gives the result for $\ext^{0, \bullet}_{R^e}(B_t, \mathrm{BS}(\underline{w}))$ and $\ext^{2, \bullet}_{R^e}(B_t, \mathrm{BS}(\underline{w}))$. For $\ext^{1, \bullet}_{R^e}(B_t, \mathrm{BS}(\underline{w}))$, the spectral sequence gives the filtration 
\[ 0\to \ker\rhosw (3)\to \ext^{1, \bullet}_{R^e}(B_t, \mathrm{BS}(\underline{w}))\to  \Db(\ker\rhosw )(1)\to 0 \]
 and since $\Db(\ker\rhosw )(1)$ is a free right $R-$module the SES splits.  
\end{proof}

\begin{remark}
 The maps $\Db(\ker\rhosw )\underline{\rho_s}$ are not chain maps and thus our notation $\widetilde{\Db(\ker\rhosw )}\underline{\rho_s}$ means we need to choose a $v\in \bs(\un{w})$ s.t. $\rho_tt(\rho_t)^e(w)=\rho_s^e(v)$ for $w\in \Db(\ker\rhosw )$. However after we introduce $\Phi_t^{\un{w}}$ in \cref{newgensectgen}, we will have a preferred choice. Namely, if $\mathrm{LL}_{w}$ is a light leaf morphism sending $c_{bot}\in \bs(\un{v})$ to $w$, then $\widetilde{w}\underline{\rho_s}$ will be the morphism $\mathrm{LL}_{w}\circ \Phi_t^{\un{v}}$.
\end{remark}

\section{New Generator and Relations: $m_{st}=\infty$}
\label{newgensect}


We will continue to assume that \cref{assump0}, \cref{assump1}, and \cref{assump2} hold in this section.

\subsection{Affine Dimension 1 Calculations}

\begin{definition}
An expression is called non repeating if there are no subexpressions of the form $ss\ldots$ or $tt\ldots$.
\end{definition}

\begin{definition}
Let $|\underline{w}|$ be the number of elements in the expression $\underline{w}$.
\end{definition}

\begin{lemma}
\label{lowestdeglem}
Suppose $m_{st}=\infty$ and $\underline{w}$ is a non repeating expression. If $|\underline{w}|$ is odd, then the lowest internal degree element in $\hom_{R^e}(R, \bs(\underline{w}))$ is of degree 1. If $|\underline{w}|$ is even, then the lowest internal degree element in $\hom_{R^e}(R, \bs(\underline{w}))$ is of degree 2.
\end{lemma}
\begin{proof}
It suffices to show the case when $|\underline{w}|$ is even as when $|\underline{w}|$ is odd $\underline{w}=(t,\ldots, t)$ or $(s,\ldots ,s)$ and we have the following natural isomorphism of graded vector spaces
\begin{align*}
\hom_{R^e}(R, \bs(t, \ldots, t))&\cong \hom_{R^e}(B_t, \bs(t, \ldots))\cong \hom_{R^e}(B_t\otimes_R B_t,  \bs( \ldots))\\
&\cong \hom_{R^e}(B_t,  \bs(t, \ldots)(1))\oplus \hom_{R^e}(B_t,  \bs(t, \ldots)(-1))   
\end{align*}
and so the lowest internal degree element in $\hom_{R^e}(R, \bs(t, \ldots, t))$ is precisely 1 less than the lowest internal degree element in $\hom_{R^e}(B_t,  \bs(t, \ldots))$,\\

We now proceed by induction on $|\underline{w}|$ and use the diagrammatic description of $\hom_{R^e}(R, \bs(\underline{w}))$. As $|\underline{w}|$ is even the first and last element of $\underline{w}$ are not the same and so the lowest internal degree element in $\hom_{R^e}(R, \bs(\underline{w}))$ must decompose as
\[
\ \boxed{\phantom{a}a\phantom{a}} \ \  \boxed{\phantom{a}c\phantom{a}} 
\]
where $\ \boxed{\phantom{a}a\phantom{a}} \in \hom_{R^e}(R, \bs(\underline{w^\prime}))$ and $\ \boxed{\phantom{a}c\phantom{a}} \in\hom_{R^e}(R, \bs(\underline{w^{\prime \prime}})) $ such that $(\underline{w^\prime}, \underline{w^{\prime \prime}})=\underline{w}$ because morphisms are only generated by red and blue trivalent vertices and dots when $m_{st}=\infty$. First assume that both $|\underline{w^\prime}|$ and $|\underline{w^{\prime \prime}}|$ are odd. Notice that $\ \boxed{\phantom{a}a\phantom{a}}$ has to be the lowest internal degree element in $\hom_{R^e}(R, \bs(\underline{w^\prime}))$ as otherwise we can replace it with the lowest. Same for $\ \boxed{\phantom{a}c\phantom{a}}$. Therefore by induction they both have degree 1 and thus the lowest degree element in $\hom_{R^e}(R, \bs(\underline{w}))$ has degree 2. \\

Now if $|\underline{w^\prime}|$ and $|\underline{w^{\prime \prime}}|$ are both even, one can show that it can't be the lowest internal degree element as there will exist dots (because $m_{st}=\infty$)  which can be turned into trivalent vertices, which lowers the degree. 
\end{proof}

\begin{lemma}
\label{lowestdeglem2}
For $m_{st}=\infty$ let $\underline{w}$ be any expression and let $m(t, \underline{w})$ be any non repeating subexpression of $t\underline{w}$ such that $|m(t, \underline{w})|$ is maximal among all non repeating subexpressions of $t\underline{w}$. If $|m(t, \underline{w})|$ is odd, then the lowest internal degree element of $\ker \rhosw (-1)$ is $-|\underline{w}|+|m(t, \underline{w})|$. If $|m(t, \underline{w})|$ is even, then the lowest internal degree element of $\ker \rhosw (-1)$ is $1-|\underline{w}|+|m(t, \underline{w})|$.   
\end{lemma}
\begin{proof}
As $\ker \rhosw (-1)\cong \hom_{R^e}(B_t, \bsw)$ it suffices to show the lowest degree in $\hom_{R^e}(B_t, \bsw)$ is as prescribed above. Note
\[ \hom_{R^e}(B_t, \bsw)\cong \hom_{R^e}(R,B_t\otimes_R \bsw) \]
Choose some $m(t, \underline{w})$ as defined above. Then the part of $t\underline{w}$ excluding $m(t, \underline{w})$ has to be of the form $ss\ldots$ or $tt\ldots$. As $B_s\otimes_R B_s\cong B_s(1)\oplus B_s(-1)$, we can then simplify $B_t\otimes_R \bsw$ to $\bs(m(t, \underline{w}))$ at the cost of some grading shifts. As there are exactly $|t\underline{w}|-|m(t, \underline{w})|=1+|\underline{w}|-|m(t, \underline{w})|$ places in $t\underline{w}$ where we need to apply this relation, it follows that
\[ \grk \ \hom_{R^e}(B_t, \bsw)= \grk \ (v+v^{-1})^{1+|\underline{w}|-|m(t, \underline{w})|}\hom_{R^e}(R,\bs(m(t, \underline{w}))) \]
Now apply the previous lemma. 
\end{proof}

\begin{prop}
\label{1dimcriterion}
For $m_{st}=\infty$, $\ext^{1, -(|\underline{w}|+1)}_{R^e}(B_t, \mathrm{BS}(\underline{w}))$ is a 1 dimensional $\kb$ module when $|m(t, \underline{w})|\ge 4$.
\end{prop}
\begin{proof}
Under the decomposition from \cref{maincohoiso}
\[\ext^{1, -(|\underline{w}|+1)}_{R^e}(B_t, \mathrm{BS}(\underline{w})))\cong \ker\rhosw (3)_{-|\underline{w}|-1}\oplus  \widetilde{\Db(\ker\rhosw )}(1)_{-|\underline{w}|-1}\]
$ \Db(\ker\rhosw )(1)_{-|\underline{w}|-1}$ is a $1$ dimensional $\kb$ module, as $1(\underline{w})$ is the lowest degree element in $\Db(\bs(\underline{w}))\cong \bs(\underline{w})$ and we showed above that $1(\underline{w})\in \Db(\ker\rhosw)$. Thus it suffices to show the lowest degree element in $\ker \rhosw (3)$ is greater than this. By \cref{lowestdeglem2}, we see that when $|m(t, \underline{w})|$ is odd, the lowest degree element in $\ker \rhosw (3)$ will be $-|\underline{w}|+|m(t, \underline{w})|-4$ and thus we want 
\[ -|\underline{w}|+|m(t, \underline{w})|-4>-|\underline{w}|-1\implies |m(t, \underline{w})|>3 \]
Similarily we find that when $|m(t, \underline{w})|$ is even, we want 
\[ -|\underline{w}|+|m(t, \underline{w})|-3>-|\underline{w}|-1\implies |m(t, \underline{w})|>2 \]
and thus $|m(t, \underline{w})|\ge 4$ encompasses both cases above. 
\end{proof}


\subsection{The New Generator}
\label{newgensectgen}

\begin{lemma}
\label{clarifyinglem}
Suppose we have a double complex $C^{\bullet, \bullet}$ with differentials and terms pictured below
\begin{center}
\begin{tikzcd}
    Z \ar[r, "f^\prime"] & W \\
    \boxed{X} \ar[r, "f"] \ar[u, "g"] & Y\arrow[u, "g^\prime", swap] 
\end{tikzcd}
\end{center}
(where $C^{0,0}=X$ is boxed) such that
\begin{enumerate}[(1)]
    \item $\ker f \subset \ker g$
    \item $\im g^\prime \subset \im f^\prime$
    \item $\ker f^\prime=0$
\end{enumerate}
Then $H^1(\mathrm{Tot}(C^{\bullet, \bullet}))\cong Y/\im f$. Precisely this means that
\begin{itemize}
    \item Any $y \in Y$ extends to a unique 1-cocycle $(y, z) \in Y \times Z$ in $\mathrm{Tot}(C^{\bullet, \bullet})$
    \item two 1-cocycles $(y, z)$ and $(y', z')$ define the same class in $H^1(\mathrm{Tot}(C^{\bullet, \bullet}))\iff y \equiv y'$ in $Y/im(f)$
\end{itemize}
\end{lemma}
\begin{proof}
Straightforward from assumptions. 
\end{proof}

\begin{remark}
Even without the assumptions $(1)-(3)$ above, if $y\in \im (f)$ and $(y, z)$ is any cocycle, then $(y, z)\equiv (0, 0)$ in $H^1(\mathrm{Tot}(C^{\bullet, \bullet}))$ and similarly if $z\in \im (g)$ and $(y, z)$ is any cocycle, then $(y, z)\equiv (0, 0)$ in $H^1(\mathrm{Tot}(C^{\bullet, \bullet}))$. This follows from the definition of $(y,z)$ being a cocycle. 
\end{remark}

\begin{corollary}
\label{clarifyinglemcor}
Suppose $|m(t, \underline{w})|\ge 4$. Then for any $u\in \bsw_{-|\underline{w}|}$, $\exists ! v\in  \bs(\underline{w})_{-|\underline{w}|+2}$ such that the map ${}_{tt}^v\psi_s^u: \underline{\rho_s}(1) \boxtimes R^e\oplus \underline{\rho_t t(\rho_t)}(1)\boxtimes R^e  \to \bsw$ defined by
\[ {}_{tt}^v\psi_s^u(\underline{\rho_t t(\rho_t)}\boxtimes 1\otimes 1)=v\quad \quad {}_{tt}^v\psi_s^u(\underline{\rho_s}\boxtimes 1\otimes 1)=u \]
is a cocycle in $\underline{\hom}_{R^e}^{1, -(|\underline{w}|+1)}(K_t, \bsw)$ and any relation in $\ext_{R^e}^{1, -(|\underline{w}|+1)}(B_t, \bsw)$ is determined by evaluating on $\underline{\rho_s}\boxtimes 1\otimes 1$.
\end{corollary}
\begin{proof}
The results will follow from applying \cref{clarifyinglem} to \cref{keydoublecomplex} in internal degree $-(|\underline{w}|+1)$ with $f, f^\prime=\rhosw$ and $g, g^\prime=\rhottw$. Condition $(1)$ was shown in \cref{mainsection}, $(2)$ follows from self-duality and $(3)$ follows from the proof of \cref{1dimcriterion} where it was shown that $\ker \rhosw (3)_{-(|\underline{w}|+1)}=0$ when $|m(t, \underline{w})|\ge 4$.
\end{proof}

\begin{definition}
Suppose $|m(t, \underline{w})|\ge 4$. Let $u=1(\underline{w})\in \bsw_{-|\underline{w}|}$ in \cref{clarifyinglem} and define \\
$\Phi_t^{\underline{w}}\in \ext^{1, -(|\underline{w}|+1)}_{R^e}(B_t, \mathrm{BS}(\underline{w}))$ to be
\[\Phi_t^{\underline{w}}\:=\sbrac{{}_{tt}^{v}\psi_{s}^{1(\underline{w})}  }\quad v \textnormal{ the unique element in }\bsw_{-|\underline{w}|+2} \textnormal{ such that } \rhosw(v)=\rhottw(1(\underline{w}))\]
\end{definition}

Except when $|\underline{w}|$ is small, we will not give a description for what the corresponding $v$ should be above, and \cref{clarifyinglem} essentially tells us we don't need to. We will denote $\Phi_t^{\underline{w}}$ diagrammatically as 

\begin{center}
\begin{tikzpicture}[scale=0.6]
	       \draw[violet] (1.1,-0.1) to[out=-90,in=-180] (2,-1);
	       \draw[violet] (2.9,-0.1) to[out=-90,in=0] (2,-1);
	       \foreach \x in {2} \draw[blue] (\x,-1) -- (\x,-2);
	       \node[rhdot] at (2,-1) {};
	       \node at (2,-0.4) {$\underline{w}$};
	    \end{tikzpicture} or
    \begin{tikzpicture}[scale=0.6]
	       \draw[violet] (1.1,-0.1) to[out=-90,in=-180] (2,-1);
	       \draw[violet] (2.9,-0.1) to[out=-90,in=0] (2,-1);
	       \foreach \x in {2} \draw[blue] (\x,-1) -- (\x,-2);
	       \node[rhdot] at (2,-1) {};
	       \node at (2,-0.4) {$\ldots$};
	    \end{tikzpicture}
\end{center}
where the purple lines indicate either $s$ or $t$. The red hollow dot is used to indicate that we will typically only know what $\Phi_t^{\underline{w}}$ does on $\underline{\rho_s}\boxtimes 1\otimes 1$. Likwise for $\Phi_s^{\underline{w}}$ we will color the hollow dot in the middle blue to denote as seen in the left below. In the case when the color of the bottom strand is not stated, we will color the hollow dot in the middle yellow as seen on the right below
\begin{center}
\begin{tikzpicture}[scale=0.6]
	       \draw[violet] (1.1,-0.1) to[out=-90,in=-180] (2,-1);
	       \draw[violet] (2.9,-0.1) to[out=-90,in=0] (2,-1);
	       \foreach \x in {2} \draw[red] (\x,-1) -- (\x,-2);
	       \node[bhdot] at (2,-1) {};
	       \node at (2,-0.4) {$\underline{w}$};
	    \end{tikzpicture}\qquad 
    \begin{tikzpicture}[scale=0.6]
	       \draw[violet] (1.1,-0.1) to[out=-90,in=-180] (2,-1);
	       \draw[violet] (2.9,-0.1) to[out=-90,in=0] (2,-1);
	       \foreach \x in {2} \draw[violet] (\x,-1) -- (\x,-2);
	       \node[yhdot] at (2,-1) {};
	       \node at (2,-0.4) {$\underline{w}$};
	    \end{tikzpicture}
\end{center}

\begin{remark}
The isomorphism
\begin{equation}
\label{kwresolution}
\hom_{K^b(\mathrm{Proj}(R^e))}(K_t, \mathrm{KS}(\underline{w}))\xrightarrow{q_{\underline{w}}\circ -}\hom_{D^b(R^e)}(K_t, \bsw)     
\end{equation} 
allows us to lift $\Phi_t^{\underline{w}}$ to a morphism $\widetilde{\Phi_t^{\underline{w}}}\in \underline{\hom}^{1,-(|\underline{w}+1)}_{R^e}( K_t , \mathrm{KS}(\underline{w}) )$ which one needs in order to compute relations in $\bsbimext(\hf, W_\infty)$. However like in \cref{chainliftsection} we don't need to find the entire chain lift and it suffices to work with just  $\widetilde{\Phi_t^{\underline{w}}}^0: K_t^{[-1]}\to \mathrm{KS}(\underline{w})^{[0]}$ (the notation here means that the map goes from cohomological degree $-1$ of $K_t$ to cohomological degree 0 of $K(\underline{w})$) which is defined as
\[\widetilde{\Phi_t^{\underline{w}}}^0:={}_{tt}^{ a \otimes \ldots \otimes d+\ldots}\psi_{s}^{1\otimes \ldots \otimes 1}: \urtt\boxtimes R^e \oplus \urhos \boxtimes R^e \to \underline{1}\boxtimes  R^{\otimes |\underline{w}|+1}  \]
 where $v=a\otimes_{s_1}\ldots\otimes_{s_m}d+\ldots$, 
In other words, we have replaced an expression $a\otimes_{s_1} \ldots \otimes_{s_m} d \in \bsw$ with $a\otimes \ldots \otimes d  \in R^{\otimes |\underline{w}|+1}$.
\end{remark}

\begin{remark}
To prevent notation overload, some of the relations in the following subsections are ``incorrectly written". Specifically, on one side of a relation will end up in $\mathrm{KS}(\underline{w})$ for some $\underline{w}=(s_1, \ldots, s_m)$ and the other side will end up in $\mathrm{KS}(\underline{w})\otimes_R K_\varnothing$ or $ K_\varnothing \otimes_R \mathrm{KS}(\underline{w})$, etc. This happens when the relation involves the counit $\widetilde{\epsilon_s}: K_s\to K_\varnothing$ for instance. One then needs to apply a chain lift of the right unitor $\widetilde{\theta_s}$ or the left unitor $\widetilde{\lambda_s}$ to one side to have actual equality. We will do this in the calculations by replacing $ \id_{\underline{w}}\otimes_R\widetilde{\epsilon_s}$ with 
$$  \id_{\underline{w}}\otimes_R\widetilde{\epsilon_s^\prime}:=(\id_{(s_1, \ldots, s_{m_1})}\otimes_R \widetilde{\theta_{s_m}})\circ(\id_{\underline{w}}\otimes_R\widetilde{\epsilon_s}):\mathrm{KS}(\underline{w})\otimes_R K_s\to \mathrm{KS}(\underline{w})\otimes_R K_\varnothing \to \mathrm{KS}(\underline{w})$$
in the case of $\mathrm{KS}(\underline{w})\otimes_R K_\varnothing$ in our calculations. In fact, we will always end up in the cohomological degree 0 part of $\mathrm{KS}(\underline{w})\otimes_R K_\varnothing$ after applying $\id_{\underline{w}}\otimes_R\widetilde{\epsilon_s}$ and so explicitly we then apply  
$$(\id_{(s_1, \ldots, s_{m_1})}\otimes_R \widetilde{\theta_{s_m}}^0)(\underline{1}\boxtimes a_1\otimes \ldots \otimes a_{m+1}\otimes  a_{m+2})=\underline{1}\boxtimes a_1\otimes \ldots \otimes a_{m+1}a_{m+2}$$
and similarily with $ K_\varnothing \otimes_R \mathrm{KS}(\underline{w})$, etc.
\end{remark}

\subsection{Low Strand Relations}

Recall $c_s=\frac{1}{2}(\alpha_s\otimes_s 1+1\otimes_s \alpha_s)\in B_s$. For the expression $tsts$ we have $m(tsts)=tsts$ and $\ell(tsts)\ge 4$. Therefore $\Phi_t^{(s,t,s)}$ is defined. 

\begin{lemma}
\label{cocyclerep}
$ {}_{tt}^{\scalemath{0.75}{-a_{st}\otimes_s \rho_t\otimes_t 1\otimes_s 1+c_s\otimes_t 1\otimes_s 1+1\otimes_s 1\otimes_t 1\otimes_s \alpha_s-1\otimes_s 1\otimes_t c_s} }\psi_{s}^{1\otimes_s 1\otimes_t 1 \otimes_s 1}$ is a cocycle representative for $\Phi_{t}^{(s,t,s)}$ 
\begin{center}
    \begin{tikzpicture}[scale=0.6]
	       \draw[red] (1.1,0.1) to[out=-90, in=180] (2,-1);
	       \draw[blue] (2,0.1) -- (2,-1);
	       \draw[red] (2.9,0.1) to[out=-90, in=0](2,-1);
	       \draw[blue] (2,-2) -- (2,-1);
	       \node[rhdot] at (2,-1) {};
	\end{tikzpicture}
\end{center}
\end{lemma}
\begin{proof}
By \cref{clarifyinglem} it suffices to show
\[ \rhos{s}{3}{}(-a_{st} \otimes_s \rho_t\otimes_t 1\otimes_s 1+c_s\otimes_t 1\otimes_s 1+1\otimes_s 1\otimes_t 1\otimes_s \alpha_s-1\otimes_s 1\otimes_t c_s)=\rhott{s}{3}{} (1\otimes_s 1\otimes_t 1 \otimes_s 1) \]
which one can check by a brute force calculation.
\end{proof}

\begin{lemma}
\label{easyobserv}
$s(\rho_s)+a_{ts}\rho_t+\rho_s\in (V^*)^{s,t}$
\end{lemma}
\begin{proof}
Easy check. 
\end{proof}

\begin{prop}
\label{4extrotation}
$\Phi_{t}^{(s,t,s)}$ is signed rotation invariant. Specifically,
\[ (\id_t\otimes_R \id_{(s, t)} \otimes_R \epsilon_{s})\circ (\id_t\otimes_R \id_{(s, t)} \otimes_R \mu_{s}) \circ (\id_t \otimes_R \Phi_t^{(s,t,s)}\otimes_R \id_{s} )\circ (\delta_t\otimes_R \id_{s})\circ (\eta_t\otimes_R \id_{s}) =-\Phi_{s}^{(t, s, t)}  \]
And similarly, 
\[ (\epsilon_{s}\otimes_R \id_{(t,  s)} \otimes_R \id_{t})\circ (\mu_{s}\otimes_R \id_{(t,s)} \otimes_R \id_{t}) \circ (\id_{s} \otimes_R \Phi_t^{(s,t,s)}\otimes_R \id_{t} )\circ ( \id_{s}\otimes_R \delta_t)\circ (\id_{s}\otimes_R \eta_t) =-\Phi_{s}^{(t,s,t)}  \]
\end{prop}
\begin{proof}
We prove the first equality as the second is quite similar (and in fact follows from the first). Diagrammatically we want to show
\begin{center}
    \begin{tikzpicture}[scale=0.6]
	       \draw[red] (1.1,0.1) to[out=-90,in=180](2,-1);
	       \draw[blue] (2,0.1) -- (2,-1);
	       \draw[red] (2.9,0.1) to[out=-90,in=0] (2,-1);
	       \draw[red] (2.9,0.1) to[out=90,in=180] (3.2,0.5) to[out=0,in=90] (3.5,-0.1);
	       \draw[red] (3.5,-0.1) -- (3.5, -2);
	       \draw[blue] (0.8,0.1) to[out=-90,in=-180] (1.5,-1.75) to[out=0, in=90] (2,-1);
	       \node[rhdot] at (2,-1) {};
	    \end{tikzpicture}\ \raisebox{0.6cm}{$= \ - \hspace{-1ex}$}  \begin{tikzpicture}[scale=0.6]
	       \draw[blue] (1.1,0.1) to[out=-90, in=180] (2,-1);
	       \draw[red] (2,0.1) -- (2,-1);
	       \draw[blue] (2.9,0.1) to[out=-90, in=0] (2,-1);
	       \draw[red] (2,-2) -- (2,-1);
	       \node[bhdot] at (2,-1) {};
	\end{tikzpicture}
\end{center}
From \cref{clarifyinglemcor} it suffices to show both sides agree on $\underline{\rho_t}\boxtimes 1\otimes 1$. As $(V^*)^t=\kb \rho_s \oplus (V^*)^{s,t}$, From \cref{easyobserv} we have that $\rho_t+t(\rho_t)=-a_{st}\rho_s+r$ where $r= \rho_t+t(\rho_t)+a_{st}\rho_s\in (V^*)^{s,t} $. The first part of the diagram will be the series of maps
\begin{center}
\begin{tikzcd}
K_t\otimes_R K_t \otimes_R K_s &  
\begin{tabular}{c}
$(-a_{st}\underline{\rho_s^{ (1)}}-a_{st}\underline{\rho_s^{ (2)}})\boxtimes  \rho_t\otimes 1\otimes 1\otimes  1 +(\underline{r^{ (1)}}+\underline{r^{ (2)}})\boxtimes \rho_t\otimes 1\otimes 1\otimes  1 $ \\
$- (\urtt^{ (1)}+ \urtt^{ (2)})\boxtimes 1\otimes 1 \otimes 1\otimes 1+\underline{\rho_t^{(3)}}\boxtimes ( \rho_t \otimes 1 \otimes 1 \otimes 1- 1\otimes 1\otimes t(\rho_t)\otimes 1)$
\end{tabular}
\\ 
K_t\otimes_R K_s \arrow[u, "\widetilde{\delta_t}\otimes_R \id_s"]& \underline{\rho_t+t(\rho_t)^{(1)}}\boxtimes \rho_t\otimes 1\otimes  1- \urtt^{(1)}\boxtimes 1\otimes 1\otimes 1 +\underline{\rho_t^{(2)}}\boxtimes \rho_t\otimes 1 \otimes 1 -\underline{\rho_t^{(2)}}\boxtimes 1\otimes t(\rho_t) \otimes 1  \arrow[u] \\
K_\varnothing \otimes_R K_s \arrow[u, "\widetilde{\eta_t}\otimes_R \id_s"]& \underline{\rho_t^{(1)}}\boxtimes 1\otimes 1\otimes 1+\underline{\rho_t^{(2)}}\boxtimes 1\otimes 1 \otimes 1\arrow[u]  \\
K_s\arrow[u, "\widetilde{\tau_s}^1"]& \underline{\rho_t}\boxtimes 1\otimes 1 \arrow[u]
\end{tikzcd}    
\end{center}
where we have used the chain lifts defined in \cref{chainliftsection}. Be aware that since $\underline{\rho_t^{(2)}}= \underline{1}\otimes_R \un{\rho_t}$, we have to apply $\widetilde{\eta_t}^0\otimes_R \id_s$ to $\underline{\rho_t^{(2)}}\boxtimes 1\otimes 1\otimes 1$ when going from the third line to the the second line from the top. At the next step we apply
\[ \id_t\otimes_R  {}_{t}^{-a_{st}\rho_t \otimes 1\otimes 1\otimes 1+c_s\otimes 1\otimes 1+1\otimes 1\otimes 1\otimes \alpha_s-1\otimes 1\otimes c_s }\psi_{s}^{1\otimes 1\otimes 1 \otimes 1}\otimes_R \id_s \]
so only the $\underline{{\phantom{\rho_s}}^{(2)}}$ terms survive. Because $r\in (V^*)^{s,t} $, the $\underline{r^{(2)}}$ term also disappears. Thus, the next series of maps will be 
\begin{center}
\begin{tikzcd}
K_t\otimes_R  \mathrm{KS}(s,t)  & \underline{1}\boxtimes \left( -1\otimes (1\otimes 1\otimes 2)+1\otimes (1\otimes 1\otimes 1)\right)=-\underline{1}\boxtimes 1\otimes 1\otimes 1\otimes 1\\ 
K_t\otimes_R  \mathrm{KS}(s,t) \otimes_R K_s \arrow[u, "\id_t\otimes_R \id_{(s,t)}\otimes_R \widetilde{\epsilon_s^\prime}^0"]& \underline{1}\boxtimes \left(0+0+0-1\otimes (1\otimes 1\otimes  \partial_s(\alpha_s))\otimes 1+1\otimes ( 1\otimes 1\otimes  \partial_s(-s(\rho_s)) ) \otimes 1\right)
\arrow[u] \\
K_t\otimes_R  \mathrm{KS}(s,t,s) \otimes_R K_s \arrow[u, "\id_t\otimes_R \id_{(s,t)}\otimes_R \widetilde{\mu_s}^0"]&  \arrow[u] 
\begin{tabular}{c}
$\underline{1}\boxtimes \left(-a_{st} \rho_t\otimes  (1 \otimes 1\otimes 1\otimes 1)\otimes  1\right.$\\
$- \left.1\otimes (-a_{st} \otimes \rho_t\otimes 1\otimes 1+c_s\otimes 1\otimes 1+1\otimes 1\otimes 1\otimes \alpha_s-1\otimes 1\otimes c_s)\otimes 1\right)$
\end{tabular}
\\
K_t\otimes_R K_t \otimes_R K_s\arrow[u, "\id_t \otimes_R \widetilde{\Phi_t^{(s,t,s)}}^0\otimes_R\id"]& -a_{st}\underline{\rho_s^{ (2)}}\boxtimes  \rho_t\otimes (1\otimes 1)\otimes  1 - \urtt^{ (2)}\boxtimes 1\otimes (1 \otimes 1)\otimes 1\arrow[u]
\end{tikzcd}    
\end{center}
which is exactly equal to $-\widetilde{\Phi_s^{(t,s,t)}}^0(\underline{\rho_t}\boxtimes 1\otimes 1)$ as desired.
\end{proof}

\begin{prop}
\label{4extreduct}
We have the following reduction identity in $\ext^{1, -3}_{R^e}(B_t, B_s B_t)$.
\[ (\id_s\otimes_R \id_t\otimes_R \epsilon_s)\circ \Phi_t^{(s,t,s)}= (\eta^{\ext}_s\otimes_R \id_t)\circ \tau_t -(\eta_s\otimes_R \id_t)\circ (\id_\varnothing\otimes_R \phi_t)\circ \tau_t \]
Diagrammatically this will be of the form,
\begin{center}
    \begin{tikzpicture}[scale=0.6]
	       \draw[red] (1.1,0.1) to[out=-90, in=180] (2,-1);
	       \draw[blue] (2,0.1) -- (2,-1);
	       \draw[red] (2.9,0.1) to[out=-90, in=0](2,-1);
	       \draw[blue] (2,-2) -- (2,-1);
	       \node[rhdot] at (2,-1) {};
	       \node[rdot] at (2.9,0.1) {};
	\end{tikzpicture} \raisebox{0.5cm}{$= \ $}
 \begin{tikzpicture}[scale=0.6]
	       \draw[red] (1.1,0.1) -- (1.1,-1);
	       \draw[blue] (2,0.1) -- (2,-1);
	       \draw[blue] (2,-2) -- (2,-1);
	       \node[rhdot] at (1.1,-1) {};
	\end{tikzpicture}\raisebox{0.6cm}{$\ \ - \ \ $}
\begin{tikzpicture}[scale=0.6]
	       \draw[red] (1.1,0.1) -- (1.1,-1);
	       \draw[blue] (2,0.1) -- (2,-1);
	       \draw[blue] (2,-2) -- (2,-1);
	       \node[rdot] at (1.1,-1) {};
	       \node[bhdot] at (2,-1.3) {};
	\end{tikzpicture}
\end{center}
\end{prop}
\begin{proof}
We have to check that both sides agree on both $\underline{\rho_s}\boxtimes 1\otimes 1$ and $\urtt\boxtimes 1\otimes 1$ as $\dim_\kb \ext^{1, -3}_{R^e}(B_t, B_s B_t)\neq 1$. The RHS applied to $\underline{\rho_s}\boxtimes 1\otimes 1$ will be 
\begin{align*}
\paren{\widetilde{\eta_s^{\ext}}^1\otimes_R \id_t}\circ \widetilde{\tau_t}^1(\underline{\rho_s}\boxtimes 1\otimes 1)&= \paren{\widetilde{\eta_s^{\ext}}^1\otimes_R \id_t}(\underline{\rho_s^{(1)}}\boxtimes 1\otimes 1\otimes 1+\underline{\rho_s^{(2)}}\boxtimes 1\otimes 1\otimes 1) =\underline{1}\boxtimes 1\otimes 1\otimes 1  \\
(\widetilde{\eta_s}^0\otimes_R \id_t)\circ (\id_\varnothing\otimes_R \widetilde{\phi_t}^1)\circ \widetilde{\tau_t}^1(\underline{\rho_s}\boxtimes 1\otimes 1)&=(\widetilde{\eta_s}^0\otimes_R \id_t)\circ (\id_\varnothing\otimes_R \widetilde{\phi_t}^1)(\underline{\rho_s^{(1)}}\boxtimes 1\otimes 1\otimes 1+\underline{\rho_s^{(2)}}\boxtimes 1\otimes 1\otimes 1)\\
&=0
\end{align*}
and thus the sum agrees with $(\id_s\otimes_R \id_t\otimes_R \widetilde{\lambda_s}^0)\circ(\id_s\otimes_R \id_t\otimes_R \widetilde{\epsilon_s}^0)\circ \widetilde{\Phi_t^{(s,t,s)}}^0(\urhos\boxtimes 1\otimes 1)$. Similarly we see that 
\begin{align*}
\paren{\widetilde{\eta_s^{\ext}}^1\otimes_R \id_t}\circ \widetilde{\tau_t}^1(\urtt\boxtimes 1\otimes 1)&= \paren{\widetilde{\eta_s^{\ext}}^1\otimes_R \id_t}(\underline{\rho_t^{(1)}}\boxtimes t(\rho_t)\otimes 1\otimes 1+ \underline{t(\rho_t)^{(1)}}\boxtimes 1\otimes \rho_t \otimes 1+\underline{\gamma_t^{(2)}}\boxtimes 1\otimes 1 \otimes 1) \\
&=\abrac{\alpha_s^\vee, t(\rho_t)}\underline{1}\boxtimes 1\otimes \rho_t\otimes 1 =-a_{st}\underline{1}\boxtimes 1\otimes \rho_t\otimes 1
\end{align*}
while 
\begin{align*}
&(\widetilde{\eta_s}^0\otimes_R \id_t)\circ (\id_\varnothing\otimes_R \widetilde{\phi_t}^1)\circ \widetilde{\tau_t}^1(\urtt\boxtimes 1\otimes 1)\\
=&(\widetilde{\eta_s}^0\otimes_R \id_t)\circ (\id_\varnothing\otimes_R \widetilde{\phi_t}^1)(\underline{\rho_t^{(1)}}\boxtimes t(\rho_t)\otimes 1\otimes 1+ \underline{t(\rho_t)^{(1)}}\boxtimes 1\otimes \rho_t \otimes 1+\underline{\gamma_t^{(2)}}\boxtimes 1\otimes 1 \otimes 1)\\
=&(\widetilde{\eta_s}^0\otimes_R \id_t)(-\underline{1}\boxtimes 1\otimes 1\otimes 1)=-\underline{1}\boxtimes c_s\otimes 1
\end{align*}
On the other hand the LHS applied to $\urtt\boxtimes 1\otimes 1$ is
\begin{align*}
& (\id_s\otimes_R \id_t\otimes_R \widetilde{\epsilon_s^\prime}^0)\circ \widetilde{\Phi_t^{(s,t,s)}}^0(\urtt\boxtimes 1\otimes 1)\\
=&  (\id_s\otimes_R \id_t\otimes_R \widetilde{\epsilon_s^\prime}^0)\paren{\underline{1}\boxtimes -a_{st} \otimes \rho_t\otimes 1\otimes 1+c_s\otimes 1\otimes 1+1\otimes 1\otimes 1\otimes \alpha_s-1\otimes 1\otimes c_s}\\
=&\underline{1}\boxtimes-a_{st} \otimes \rho_t\otimes 1+c_s\otimes 1
\end{align*}
as $\widetilde{\epsilon_s^\prime}^0(c_s)=\rho_s-s(\rho_s)=\alpha_s$. But this is exactly 
\[ \paren{\widetilde{\eta_s^{\ext}}^1\otimes_R \id_t}\circ \widetilde{\tau_t}^1(\urtt\boxtimes 1\otimes 1) -(\widetilde{\eta_s}^0\otimes_R \id_t)\circ (\id_\varnothing\otimes_R \widetilde{\phi_t}^1)\circ \widetilde{\tau_t}^1(\urtt\boxtimes 1\otimes 1) \]
and thus we are done.
\end{proof}

\begin{corollary}
\label{4extreductcor}
We have the following reduction identity in $\ext^{1, -3}_{R^e}(B_t, B_t B_s)$.
\[ (\epsilon_s\otimes_R \id_t\otimes_R \id_s)\circ \Phi_t^{(s,t,s)}= (\id_t\otimes_R \eta^{\ext}_s)\circ \sigma_t -( \id_t\otimes_R \eta_s)\circ ( \phi_t\otimes_R \id_\varnothing)\circ \sigma_t \]
Diagrammatically this will be of the form,
\begin{center}
    \begin{tikzpicture}[scale=0.6]
	       \draw[red] (1.1,0.1) to[out=-90, in=180] (2,-1);
	       \draw[blue] (2,0.1) -- (2,-1);
	       \draw[red] (2.9,0.1) to[out=-90, in=0] (2,-1);
	       \draw[blue] (2,-2) -- (2,-1);
	       \node[rhdot] at (2,-1) {};
	       \node[rdot] at (1.1,0.1) {};
	\end{tikzpicture} \raisebox{0.5cm}{$= \ $}
\begin{tikzpicture}[scale=0.6]
	       \draw[blue] (1.1,0.1) -- (1.1,-2);
	       \draw[red] (2,0.1) -- (2,-1);
	       \node[rhdot] at (2,-1) {};
	\end{tikzpicture}\raisebox{0.5cm}{$\ - \ $}
\begin{tikzpicture}[scale=0.6]
	       \draw[blue] (1.1,0.1) -- (1.1,-2);
	       \draw[red] (2,0.1) -- (2,-1);
	       \node[rdot] at (2,-1) {};
	       \node[bhdot] at (1.1,-1.4) {};
	\end{tikzpicture}
\end{center}
\end{corollary}

\begin{corollary}
\label{4extreductcor2}
We have the following reduction identity in $\ext^{1, -3}_{R^e}(B_t, B_s B_s)$.
\[ (\id_s\otimes_R \epsilon_t\otimes_R \id_s)\circ \Phi_t^{(s,t,s)}=  -\delta_s\circ \eta_s\circ \epsilon_t^{\ext} +\delta_s\circ \eta_s^{\ext}\circ \epsilon_t  \]
Diagrammatically this will be of the form,
\begin{center}
    \begin{tikzpicture}[scale=0.6]
	       \draw[red] (1.1,0.1) to[out=-90, in=180] (2,-1);
	       \draw[blue] (2,0.1) -- (2,-1);
	       \draw[red] (2.9,0.1) to[out=-90, in=0] (2,-1);
	       \draw[blue] (2,-2) -- (2,-1);
	       \node[rhdot] at (2,-1) {};
	       \node[bdot] at (2,0.1) {};
	\end{tikzpicture} \raisebox{0.5cm}{$= \ -$}
\begin{tikzpicture}[scale=0.6]
	       \draw[blue] (1.5,-1.3) -- (1.5,-2);
	       \draw[red] (1.5,-0.4) -- (1.5,-0.9);
	       \draw[red] (1.5,-0.4) -- (0.9,0);
	       \draw[red] (1.5,-0.4) -- (2.1,0);
	       \node[rdot] at (1.5,-0.9) {};
	       \node[bhdot] at (1.5,-1.3) {};
	\end{tikzpicture}\raisebox{0.5cm}{$\ + \ $}
\begin{tikzpicture}[scale=0.6]
	       \draw[blue] (1.5,-1.3) -- (1.5,-2);
	       \draw[red] (1.5,-0.4) -- (1.5,-0.9);
	       \draw[red] (1.5,-0.4) -- (0.9,0);
	       \draw[red] (1.5,-0.4) -- (2.1,0);
	       \node[rhdot] at (1.5,-0.9) {};
	       \node[bdot] at (1.5,-1.3) {};
	\end{tikzpicture}
\end{center}
\end{corollary}

\begin{corollary}
\label{4extreductcor3}
We have the following reduction identity in $\ext^{1, -3}_{R^e}(R, B_s B_t B_s)$.
\[ \Phi_t^{(s,t,s)}\circ \eta_t  =  - (\id_s \otimes_R \eta_t^{\ext}\otimes_R \id_s)\circ \delta_s\circ \eta_s +(\id_s \otimes_R \eta_t\otimes_R \id_s)\circ \delta_s\circ \eta_s^{\ext}\ \]
Diagrammatically this will be of the form,
\begin{center}
    \begin{tikzpicture}[scale=0.6]
	       \draw[red] (1.1,0.1) to[out=-90, in=180] (2,-1);
	       \draw[blue] (2,0.1) -- (2,-1);
	       \draw[red] (2.9,0.1) to[out=-90, in=0](2,-1);
	       \draw[blue] (2,-2) -- (2,-1);
	       \node[rhdot] at (2,-1) {};
	       \node[bdot] at (2,-2) {};
	\end{tikzpicture} \raisebox{0.5cm}{$= \ -$}
	\raisebox{0.05cm}{
\begin{tikzpicture}[scale=0.6]
	       \draw[red] (1.1,-0.1) to[out=-90, in=-180] (2,-1.5) to[out=0, in=-90] (2.9, -0.1);
	       \draw[blue] (2,-0.1) -- (2,-0.9);
	       \node[bhdot] at (2,-0.9) {};
	\end{tikzpicture}} \raisebox{0.5cm}{$ \ + \ $}
\begin{tikzpicture}[scale=0.6]
	       \draw[red] (1.1,-0.1) to[out=-90, in=-180] (2,-1.5) to[out=0, in=-90] (2.9, -0.1);
	       \draw[blue] (2,-0.1) -- (2,-0.9);
	       \node[rhdot] at (2,-1.5) {};
	       \node[bdot] at (2,-0.9) {};
	\end{tikzpicture}
\end{center}
\end{corollary}

The above corollaries all follow by rotating \cref{4extreduct} and using \cref{4extrotation}. Because  \cref{4extrotation} says that $\Phi_s^{(s,t,s)}$ is only signed rotation invariant the signs switch in the last two identities. 

\subsection{High Strand Relations}

\begin{definition}
Given an expression $\underline{w}=(s_1, \ldots, s_i, \ldots, s_m)$, define the expresssions 
\begin{align*}
    \underline{\widehat{w}^i}&=(s_1, \ldots, s_{i-1}, s_{i+1}, \ldots, s_m) & |\underline{\widehat{w}^i} |=m-1\\
    \underline{\widecheck{w}^i}&=(s_1, \ldots,s_{i-1}, s_i, s_{i},s_{i+1} \ldots, s_m) & |\underline{\widecheck{w}^i}|=m+1
\end{align*}
\end{definition}

\begin{lemma}
\label{newgenreduct}
Let $\underline{w}=(s_1, \ldots, s_m)$ and suppose that $\ell(m(t \underline{\widehat{w}^i}) )\ge 4$, then 
\[  (\id_{(s_1, \ldots, s_{i-1})}\otimes_R \epsilon_{s_i}  \otimes_R \id_{(s_{i+1}, \ldots, s_m)} )\circ \Phi_{t}^{(s_1, \ldots, s_{i-1}, s_i, s_{i+1}, \ldots, s_m)}= \Phi_{t}^{\underline{\widehat{w}^i}}  \]

Diagrammatically, this will be of the form
\begin{center}
    \begin{tikzpicture}[scale=0.6]
	       \draw[violet] (1.1,-0.1) to[out=-90,in=-180] (2,-1);
	       \draw[violet] (2.9,-0.1) to[out=-90,in=0] (2,-1);
	       \foreach \x in {2} \draw[blue] (\x,-1) -- (\x,-2);
	       \draw[violet] (2,-0.1)--(2,-1);
	       \node[rhdot] at (2,-1) {};
	       \node[pdot] at (2,-0.1) {};
	       \node at (1.55,-0.4) {$\ldots$};
	       \node at (2.4,-0.4) {$\ldots$};
	    \end{tikzpicture}\raisebox{0.5cm}{$=$}\begin{tikzpicture}[scale=0.6]
	       \draw[violet] (1.1,-0.1) to[out=-90,in=-180] (2,-1);
	       \draw[violet] (2.9,-0.1) to[out=-90,in=0] (2,-1);
	       \foreach \x in {2} \draw[blue] (\x,-1) -- (\x,-2);
	       \node[rhdot] at (2,-1) {};
	       \node at (2,-0.4) {$\underline{\widehat{w}^i}$};
	    \end{tikzpicture}
\end{center}
\end{lemma}
\begin{proof}
Since $\dim_\kb \ext^{1, -(|\underline{\widehat{w}^i}| +1) }_{R^e}(B_t, \bs(\underline{\widehat{w}^i} ))=1 $ so it suffices to check both sides agree on $\underline{\rho_s}\boxtimes 1\otimes 1$. This will then follow from the fact that $\widetilde{\epsilon_{s_i}^\prime}^0(1\otimes 1)=1$.
\end{proof}

\begin{remark}
The lemma above doesn't apply to $\Phi_t^{(s,t,s)}$, as seen in \cref{4extreduct}.
\end{remark}

\begin{lemma}
\label{newgenmult}
Let $\underline{w}=(s_1, \ldots, s_m)$ and suppose $|m(t, \underline{w})|\ge 4$, then
\[ (\id_{(s_1, \ldots, s_{i-1})}\otimes_R \delta_{s_i}  \otimes_R \id_{(s_{i+1}, \ldots, s_m)} )\circ \Phi_t^{\underline{w}}=\Phi_t^{\underline{\widecheck{w}^i}}  \]
Diagrammatically this will be of the form
\begin{center}
    \begin{tikzpicture}[scale=0.6]
	       \draw[violet] (1.1,-0.1) to[out=-90,in=-180] (2,-1);
	       \draw[violet] (2.9,-0.1) to[out=-90,in=0] (2,-1);
	       \foreach \x in {2} \draw[blue] (\x,-1) -- (\x,-2);
	       \draw[violet] (2,-0.1)--(2,-1);
	       \draw[violet] (2,-0.1)--(1.4,0.6);
	       \draw[violet] (2,-0.1)--(2.6,0.6);
	       \node[rhdot] at (2,-1) {};
	       \node at (1.55,-0.4) {$\ldots$};
	       \node at (2.4,-0.4) {$\ldots$};
	    \end{tikzpicture}\raisebox{0.5cm}{$=$}\begin{tikzpicture}[scale=0.6]
	       \draw[violet] (1.1,-0.1) to[out=-90,in=-180] (2,-1);
	       \draw[violet] (2.9,-0.1) to[out=-90,in=0] (2,-1);
	       \foreach \x in {2} \draw[blue] (\x,-1) -- (\x,-2);
	       \node[rhdot] at (2,-1) {};
	       \node at (2,-0.4) {$\underline{\widecheck{w}^i}$};
	    \end{tikzpicture}
\end{center}
\end{lemma}
\begin{proof}
Proceeds similarly to \cref{newgenreduct}.
\end{proof}

\begin{remark}
The expression $(-1)^{|m(ts_m)|-1}$ appearing below is 0 if $s_m=t$ and 1 if $s_m=s$. 
\end{remark}

\begin{theorem}
\label{newgenrotation}
$\Phi_t^{\underline{w}}$ is signed rotation invariant when $|m(t, \underline{w})|\ge 4$. Specifically, letting $\underline{w}=(s_1, \ldots, s_m)$ we have 
\begin{align*}
&(\id_t\otimes_R \id_{(s_1, \ldots, s_{m-1})} \otimes_R \epsilon_{s_m})\circ (\id_t\otimes_R \id_{(s_1, \ldots, s_{m-1})} \otimes_R \mu_{s_m}) \circ (\id_t \otimes_R \Phi_t^{\underline{w}}\otimes_R \id_{s_m} )\circ (\delta_t\otimes_R \id_{s_m})\circ (\eta_t\otimes_R \id_{s_m})\\ =&(-1)^{|m(ts_m)|-1}\Phi_{s_m}^{(t, s_1, \ldots, s_{m-1})}
\end{align*} 
In other words, if $s_m\neq t$, then rotating $\Phi_t^{\underline{w}}$ clockwise will pick up a negative sign. And similarly, 
\begin{align*}
&(\epsilon_{s_1}\otimes_R \id_{(s_2, \ldots, s_{m})} \otimes_R \id_{t})\circ (\mu_{s_1}\otimes_R \id_{(s_2, \ldots, s_{m})} \otimes_R \id_{t}) \circ (\id_{s_1} \otimes_R \Phi_t^{\underline{w}}\otimes_R \id_{t} )\circ ( \id_{s_1}\otimes_R \delta_t)\circ (\id_{s_1}\otimes_R \eta_t)\\ =&(-1)^{|m(ts_1)|-1}\Phi_{s_1}^{(s_2, \ldots, s_{m}, t)}     
\end{align*}  
In other words, if $s_1\neq t$, then rotating $\Phi_t^{\underline{w}}$ counterclockwise will pick up a negative sign.
\end{theorem}
\begin{proof}
We will prove the first equality as the second is quite similar. First suppose $s_m=t$. Diagrammatically we want to show
\begin{center}
    \begin{tikzpicture}[scale=0.6]
	       \draw[violet] (1.1,-0.1) to[out=-90,in=-180] (2,-1);
	       \draw[blue] (2.9,-0.1) to[out=-90,in=0] (2,-1);
	       \draw[blue] (2.9,-0.1) to[out=90,in=180] (3.2,0.5) to[out=0,in=90] (3.5,-0.1);
	       \draw[blue] (3.5,-0.1) -- (3.5, -2);
	       \draw[blue] (0.9,-0.1) to[out=-90,in=-180] (1.5,-1.75) to[out=0, in=90] (2,-1);
	       \node[rhdot] at (2,-1) {};
	       \node at (2,-0.4) {$\ldots$};
	    \end{tikzpicture}\ \raisebox{0.6cm}{=} \ \begin{tikzpicture}[scale=0.6]
	       \draw[blue] (1.1,-0.1) to[out=-90,in=-180] (2,-1);
	       \draw[violet] (2.9,-0.1) to[out=-90,in=0] (2,-1);
	       \foreach \x in {2} \draw[blue] (\x,-1) -- (\x,-2);
	       \node[rhdot] at (2,-1) {};
	       \node at (2,-0.4) {$\ldots$};
	    \end{tikzpicture}
\end{center}
As $\dim_\kb \ext_{R^e}^{1, -(|\underline{w}|+1)}(B_t, \bsw)=1$ it suffices to check both sides agree on $\underline{\rho_s}\boxtimes 1\otimes 1$. The first part of the diagram will be the series of maps
\begin{center}
\begin{tikzcd}
K_t\otimes_R K_t \otimes_R K_t &  
\begin{tabular}{c}
$\underline{\rho_s^{ (1)}}\boxtimes \paren{ \rho_t\otimes 1\otimes 1\otimes  1 -1\otimes 1\otimes t(\rho_t)\otimes 1}+\underline{\rho_s^{ (2)}}\boxtimes \paren{ \rho_t\otimes 1\otimes 1\otimes  1 -1\otimes 1\otimes t(\rho_t)\otimes 1}$ \\
$+\underline{\rho_s^{(3)}}\boxtimes \paren{ \rho_t\otimes 1\otimes 1\otimes  1 -1\otimes 1\otimes t(\rho_t)\otimes 1}$
\end{tabular}
\\ 
K_t\otimes_R K_t \arrow[u, "\widetilde{\delta_t}^1\otimes_R \id_t"]& \underline{\rho_s^{(1)}}\boxtimes \paren{ \rho_t\otimes 1\otimes  1 -1\otimes t(\rho_t)\otimes 1}+\underline{\rho_s^{(2)}}\boxtimes(\rho_t\otimes 1 \otimes 1-1\otimes t(\rho_t)\otimes 1) \arrow[u] \\
K_\varnothing \otimes_R K_t \arrow[u, "\widetilde{\eta_t}^1\otimes_R \id_t"]& \underline{\rho_s^{(1)}}\boxtimes 1\otimes 1\otimes 1+\underline{\rho_s^{(2)}}\boxtimes 1\otimes 1\otimes 1 \arrow[u]  \\
K_t\arrow[u, "\widetilde{\tau_t}^1"]& \underline{\rho_s}\boxtimes 1\otimes 1 \arrow[u]
\end{tikzcd}    
\end{center}
where we have used the chain lifts defined in \cref{chainliftsection} and that $\rho_s\in (V^*)^t$.\\

At the next step we apply $\id \otimes_R \widetilde{\Phi_t^{\underline{w}}}^0\otimes_R\id $ and therefore only the $\underline{\rho_s^{(2)}}$ term survives. The next part of the diagram will be the series of maps
\begin{center}
\begin{tikzcd}
K_t\otimes_R  \mathrm{KS}(s_1, \ldots, s_{m-1})  & \underline{1}\boxtimes  1\otimes 1^{\otimes |\underline{w}|-1}\otimes 1\\ 
K_t\otimes_R  \mathrm{KS}(s_1, \ldots, s_{m-1}) \otimes_R K_t \arrow[u, "\id_t\otimes_R \id_{(s_1, \ldots, s_{m-1})}\otimes_R \widetilde{\epsilon_t^\prime}^0"]& \underline{1}\boxtimes \paren{0-1\otimes 1^{\otimes |\underline{w}|-1}\otimes \partial_t(t(\rho_t))\otimes 1} \arrow[u] \\
K_t\otimes_R  \mathrm{KS}(\underline{w}) \otimes_R K_t \arrow[u, "\id_t\otimes_R \id_{(s_1, \ldots, s_{m-1})}\otimes_R \widetilde{\mu_t}^0"]&  \arrow[u]\underline{1}\boxtimes \paren{ \rho_t \otimes 1^{\otimes |\underline{w}|+1} \otimes 1-1 \otimes 1^{\otimes |\underline{w}|}\otimes t(\rho_t)\otimes 1} \\
K_t\otimes_R K_t \otimes_R K_t\arrow[u, "\id_t \otimes_R \widetilde{\Phi_t^{\underline{w}}}^0\otimes_R\id"]& \underline{\rho_s^{ (2)}}\boxtimes \paren{ \rho_t\otimes 1\otimes 1\otimes  1 -1\otimes 1\otimes t(\rho_t)\otimes 1} \arrow[u]
\end{tikzcd}    
\end{center}
which is exactly equal to $\widetilde{\Phi_t^{(t, s_1, \ldots, s_{m-1})}}^0( \underline{\rho_s}\boxtimes 1\otimes 1 )$. \\

Now suppose $s_m=s$. Diagrammatically we want to show
\begin{center}
    \begin{tikzpicture}[scale=0.6]
	       \draw[violet] (1.1,-0.1) to[out=-90,in=-180] (2,-1);
	       \draw[red] (2.9,-0.1) to[out=-90,in=0] (2,-1);
	       \draw[red] (2.9,-0.1) to[out=90,in=180] (3.2,0.5) to[out=0,in=90] (3.5,-0.1);
	       \draw[red] (3.5,-0.1) -- (3.5, -2);
	       \draw[blue] (0.9,-0.1) to[out=-90,in=-180] (1.5,-1.75) to[out=0, in=90] (2,-1);
	       \node[rhdot] at (2,-1) {};
	       \node at (2,-0.4) {$\ldots$};
	    \end{tikzpicture}\ \raisebox{0.6cm}{$= \ -$}  \begin{tikzpicture}[scale=0.6]
	       \draw[blue] (1.1,-0.1) to[out=-90,in=-180] (2,-1);
	       \draw[violet] (2.9,-0.1) to[out=-90,in=0] (2,-1);
	       \foreach \x in {2} \draw[red] (\x,-1) -- (\x,-2);
	       \node[bhdot] at (2,-1) {};
	       \node at (2,-0.4) {$\ldots$};
	    \end{tikzpicture}
\end{center}
One can check that $|m(s, t, s_1, \ldots, s_{m-1})|\ge 4$ and thus $\dim_\kb \ext_{R^e}^{1, -(|\underline{w}|+1)}(B_{s_m}, \bs(t, s_1, \ldots, s_{m-1}))=1$. It follows that 
\begin{center}
    \begin{tikzpicture}[scale=0.6]
	       \draw[violet] (1.1,-0.1) to[out=-90,in=-180] (2,-1);
	       \draw[red] (2.9,-0.1) to[out=-90,in=0] (2,-1);
	       \draw[red] (2.9,-0.1) to[out=90,in=180] (3.2,0.5) to[out=0,in=90] (3.5,-0.1);
	       \draw[red] (3.5,-0.1) -- (3.5, -2);
	       \draw[blue] (0.9,-0.1) to[out=-90,in=-180] (1.5,-1.75) to[out=0, in=90] (2,-1);
	       \node[rhdot] at (2,-1) {};
	       \node at (2,-0.4) {$\ldots$};
	    \end{tikzpicture}\ \raisebox{0.6cm}{$= \ c$} \ \begin{tikzpicture}[scale=0.6]
	       \draw[blue] (1.1,-0.1) to[out=-90,in=-180] (2,-1);
	       \draw[violet] (2.9,-0.1) to[out=-90,in=0] (2,-1);
	       \foreach \x in {2} \draw[red] (\x,-1) -- (\x,-2);
	       \node[bhdot] at (2,-1) {};
	       \node at (2,-0.4) {$\ldots$};
	    \end{tikzpicture}
\end{center}

Because $|m(s, t, s_1, \ldots, s_{m-1})|\ge 4$, there is a subexpression of the form $stst$. Therefore there must be a subexpression $st$ in $(s_1, \ldots, s_{m-1})$ say at $(\ldots, s_i, \ldots, s_j, \ldots)$. Now apply 
$$\id_t\otimes_R\epsilon_{s_1}\otimes_R \epsilon_{s_2} \otimes_R\ldots \otimes_R\id_{s_i}\otimes_R\ldots \otimes_R \id_{s_j} \otimes_R \ldots \otimes_R \epsilon_{s_{m-1}} $$ 
to both sides. Using \cref{newgenreduct} we end up with the equation
\begin{center}
    \begin{tikzpicture}[scale=0.6]
	       \draw[red] (1.1,0.1)to[out=-90, in=180](2,-1);
	       \draw[blue] (2,0.1) -- (2,-1);
	       \draw[red] (2.9,0.1) to[out=-90, in=0] (2,-1);
	       \draw[red] (2.9,0.1) to[out=90,in=180] (3.2,0.5) to[out=0,in=90] (3.5,-0.1);
	       \draw[red] (3.5,-0.1) -- (3.5, -2);
	       \draw[blue] (0.7,0.1) to[out=-90,in=-180] (1.5,-1.75) to[out=0, in=90] (2,-1);
	       \node[rhdot] at (2,-1) {};
	    \end{tikzpicture}\ \raisebox{0.6cm}{$= c$} \ \begin{tikzpicture}[scale=0.6]
	       \draw[blue] (1.1,0.1) to[out=-90, in=180] (2,-1);
	       \draw[red] (2,0.1) -- (2,-1);
	       \draw[blue] (2.9,0.1) to[out=-90, in=0] (2,-1);
	       \draw[red] (2,-2) -- (2,-1);
	       \node[bhdot] at (2,-1) {};
	\end{tikzpicture}
\end{center}
And from \cref{4extrotation} it follows that $c=-1$ as desired.
\end{proof}

\begin{corollary}
\label{newgencor}
We have the following equality when $|m(t, \underline{w})|\ge 4$.
\[  (\id_t \otimes_R \Phi_t^{\underline{w}} )\circ (\delta_t)\circ (\eta_t) = (-1)^{|m(ts_m)|-1} ( \Phi_{s_m}^{(t, s_1, \ldots, s_{m-1})} \otimes_R \id_{s_m})\circ (\delta_{s_m})\circ (\eta_{s_m})  \]
Diagrammatically this will be
\begin{center}
    \begin{tikzpicture}[scale=0.6]
	       \draw[violet] (1.1,-0.1) to[out=-90,in=-180] (2,-1);
	       \draw[violet] (2.9,-0.1) to[out=-90,in=0] (2,-1);
	       \draw[blue] (0.9,-0.1) to[out=-90,in=-180] (1.45,-1.75) to[out=0, in=-90] (2,-1);
	       \node[rhdot] at (2,-1) {};
	       \node at (2,-0.4) {$\ldots$};
	    \end{tikzpicture}\ \raisebox{0.6cm}{$= (-1)^{|m(ts_m)|-1}$} \ \begin{tikzpicture}[scale=0.6]
	       \draw[blue] (1.1,-0.1) to[out=-90,in=-180] (2,-1);
	       \draw[violet] (2.9,-0.1) to[out=-90,in=0] (2,-1);
	       \draw[violet] (2,-1) to[out=-90,in=-180] (2.45,-1.75) to[out=0, in=-90] (3.1,-0.1);
	       \node[yhdot] at (2,-1) {};
	       \node at (2,-0.4) {$\ldots$};
	       \node at (3.5,-1.5) {$s_m$};
	    \end{tikzpicture}
\end{center}
\end{corollary}
\begin{proof}
Follows from adding a cup on the bottom right in \cref{newgenrotation}.
\end{proof}

\begin{corollary}
\label{newgencyclic}
Assume that $|m(t, \underline{w})|\ge 4$. Then $\Phi_t^{\underline{w}}$ is cyclic. Diagramatically this means
\begin{equation}
\label{newgencycliceq}
\raisebox{-0.8cm}{
    \begin{tikzpicture}[scale=0.6]
	       \draw[blue] (0.8,1.3) to[out=-85,in=-180] (1.4,-1.75) to[out=0, in=-90] (2,-1);
	       \draw[violet] (1.1,-0.1) to[out=-90,in=-180] (2,-1);
	       \draw[violet] (1.1,-0.1) to[out=90,in=-180] (2.8,1.4) to[out=0,in=90] (4.5,-0.1);
	       \draw[violet] (4.5,-0.1) -- (4.5, -2);
	       \draw[violet] (2.9,-0.1) to[out=-90,in=0] (2,-1);
	       \draw[violet] (2.9,-0.1) to[out=90,in=180] (3.2,0.5) to[out=0,in=90] (3.5,-0.1);
	       \node at (3.5,-2.3) {$s_m$};
	       \draw[violet] (3.5,-0.1) -- (3.5, -2);
	       \node[rhdot] at (2,-1) {};
	       \node at (2,-0.4) {$\ldots$};
	       \node at (4,-1) {$\ldots$};
	       \node at (4.5,-2.3) {$ s_1$};
	    \end{tikzpicture}\ \raisebox{1cm}{$= $} \ \begin{tikzpicture}[scale=0.6]
	       \draw[violet] (2.9,-0.1) to[out=90,in=0] (1.2,1.4) to[out=180,in=90] (-0.5, -0.1);
	       \draw[violet] (-0.5,-0.1) -- (-0.5, -2);
	       \draw[violet] (1.1,-0.1) to[out=-90,in=-180] (2,-1);
	       \draw[violet] (1.1,-0.1) to[out=90,in=0] (0.8,0.5) to[out=-180,in=90] (0.5, -0.1) ;
	       \draw[violet] (2.9,-0.1) to[out=-90,in=0] (2,-1);
	       \draw[violet] (0.5,-0.1) -- (0.5, -2);
	       \draw[blue] (2,-1) to[out=-90,in=-180] (2.6,-1.75) to[out=0, in=-90] (3.2,1.3);
	       \node[rhdot] at (2,-1) {};
	       \node at (2,-0.4) {$\ldots$};
	       \node at (0,-0.9) {$\ldots$};
	       \node at (-0.5,-2.3) {$s_m$};
	       \node at (0.5,-2.3) {$ s_1$};
	    \end{tikzpicture}}
\end{equation}
\end{corollary}
\begin{proof}
We can apply \cref{newgenrotation} repeatedly to obtain 
\begin{center}
    \begin{tikzpicture}[scale=0.6]
	       \draw[blue] (0.8,1.3) to[out=-85,in=-180] (1.4,-1.75) to[out=0, in=-90] (2,-1);
	       \draw[violet] (1.1,-0.1) to[out=-90,in=-180] (2,-1);
	       \draw[violet] (1.1,-0.1) to[out=90,in=-180] (2.8,1.4) to[out=0,in=90] (4.5,-0.1);
	       \draw[violet] (4.5,-0.1) -- (4.5, -2);
	       \draw[violet] (2.9,-0.1) to[out=-90,in=0] (2,-1);
	       \draw[violet] (2.9,-0.1) to[out=90,in=180] (3.2,0.5) to[out=0,in=90] (3.5,-0.1);
	       \draw[violet] (3.5,-0.1) -- (3.5, -2);
	       \draw[violet] (3.5, -2) to[out=-90, in=0] (2,-3) to[out=-180, in=-90] (0.5,-2);
	       \draw[violet] (0.5,-2) -- (0.5, 1.3);
	       \draw[violet] (4.5, -2) to[out=-90, in=0] (2,-3.5) to[out=-180, in=-90] (-0.5,-2);
	       \draw[violet] (-0.5,-2) -- (-0.5, 1.3);
	       \node[rhdot] at (2,-1) {};
	       \node at (2,-0.4) {$\ldots$};
	       \node at (4,-1) {$\ldots$};
	       \node at (0,-1) {$\ldots$};
	       \node at (1.1,-2.1) {$s_m$};
	       \node at (-1,-2.1) {$ s_1$};
	    \end{tikzpicture}\ \  \raisebox{1.5cm}{$= (-1)^{\ell(|t\underline{w}^{-1})|-1}$} \
	    \raisebox{0.5cm}{\begin{tikzpicture}[scale=1.1]
	       \draw[violet] (1.1,-0.1) to[out=-90,in=-180] (2,-1);
	       \draw[blue] (2.9,-0.1) to[out=-90,in=0] (2,-1);
	       \draw[violet] (0.9,-0.1) to[out=-90,in=-180] (1.45,-1.75) to[out=0, in=-90] (2,-1);
	       \node[yhdot] at (2,-1) {};
	       \node at (2,-0.4) {$\ldots$};
	       \node at (0.7,-0.5) {$s_1$};
	    \end{tikzpicture}}
\end{center}
(Note $|m(t\ldots s_i)|-1+|m(s_i s_{i-1})|-1=|m( t\ldots s_i s_{i-1})|-1$.) Now apply \cref{newgencor}, cap off the strands $s_1, \ldots, s_m$ to the left and use isotopy to arrive at
\begin{equation*}
\raisebox{-0.8cm}{
    \begin{tikzpicture}[scale=0.6]
	       \draw[blue] (0.8,1.3) to[out=-85,in=-180] (1.4,-1.75) to[out=0, in=-90] (2,-1);
	       \draw[violet] (1.1,-0.1) to[out=-90,in=-180] (2,-1);
	       \draw[violet] (1.1,-0.1) to[out=90,in=-180] (2.8,1.4) to[out=0,in=90] (4.5,-0.1);
	       \draw[violet] (4.5,-0.1) -- (4.5, -2);
	       \draw[violet] (2.9,-0.1) to[out=-90,in=0] (2,-1);
	       \draw[violet] (2.9,-0.1) to[out=90,in=180] (3.2,0.5) to[out=0,in=90] (3.5,-0.1);
	       \node at (3.5,-2.3) {$s_m$};
	       \draw[violet] (3.5,-0.1) -- (3.5, -2);
	       \node[rhdot] at (2,-1) {};
	       \node at (2,-0.4) {$\ldots$};
	       \node at (4,-1) {$\ldots$};
	       \node at (4.5,-2.3) {$ s_1$};
	    \end{tikzpicture}\ \raisebox{1cm}{$= (-1)^{|m(t\underline{w}^{-1}t )|-1}$} \ \begin{tikzpicture}[scale=0.6]
	       \draw[violet] (2.9,-0.1) to[out=90,in=0] (1.2,1.4) to[out=180,in=90] (-0.5, -0.1);
	       \draw[violet] (-0.5,-0.1) -- (-0.5, -2);
	       \draw[violet] (1.1,-0.1) to[out=-90,in=-180] (2,-1);
	       \draw[violet] (1.1,-0.1) to[out=90,in=0] (0.8,0.5) to[out=-180,in=90] (0.5, -0.1) ;
	       \draw[violet] (2.9,-0.1) to[out=-90,in=0] (2,-1);
	       \draw[violet] (0.5,-0.1) -- (0.5, -2);
	       \draw[blue] (2,-1) to[out=-90,in=-180] (2.6,-1.75) to[out=0, in=-90] (3.2,1.3);
	       \node[rhdot] at (2,-1) {};
	       \node at (2,-0.4) {$\ldots$};
	       \node at (0,-0.9) {$\ldots$};
	       \node at (-0.5,-2.3) {$s_m$};
	       \node at (0.5,-2.3) {$ s_1$};
	    \end{tikzpicture}}
\end{equation*}
One can then check that in all cases, $|m(t\underline{w}^{-1}t )|-1$ is always even.
\end{proof}

For more background on  cyclicity we refer the reader to \cite[Chapter 7.5]{EMTW20} or \cite[Section 4.4]{Lau10}. The main upshot of having a cyclic morphism is the following theorem

\begin{theorem}[{Cockett–Koslowski—Seely \cite{CKS}}]
\label{cks}
Given a string diagram $D$ representing a cyclic 2-morphism between 1-morphisms with chosen biadjoints, any isotopy of $D$ fixing endpoints represents the same 2-morphism.
\end{theorem}
We should be careful in applying this theorem however. Let $\partial$ denote the diagram on the RHS below and note that
\begin{equation}
\label{isoeq}
    \begin{tikzpicture}[xscale=0.4,yscale=0.4,thick,baseline]
\draw[red] (-1.8,-0.5) to[out=90, in=135] (-1,1.5);
 \draw[red] (-1,1.5) -- (0,0.5);
 \draw[red] (1,1.5) -- (0,0.5);
 \draw[red] (0,0.5) -- (0,-0.5);
 \draw[red] (0,-0.5) to[out=-90, in=180] (1, -1) to[out=0, in=-90]  (2,1.5);
\end{tikzpicture} \ \textnormal{ is isotopic ``fixing endpoints" to }  \begin{tikzpicture}[xscale=0.4,yscale=0.4,thick,baseline]
 \draw[red] (-1,1.5) -- (0,0.5);
 \draw[red] (1,1.5) -- (0,0.5);
 \draw[red] (0,0.5) -- (0,-0.5);
\end{tikzpicture} 
\end{equation} 
and these do turn out to be equal in $\mathscr{D}(\hf, S_2)$. However, this is not a consequence of cyclicity! The diagram on the RHS is technically not a string diagram, so \cref{cks} doesn't apply. A string diagram representing the RHS is drawn below \vspace{-2ex}

\begin{center}
     \begin{tikzpicture}[xscale=0.4,yscale=0.4,thick,baseline]
  \draw[red] (0,0.5) -- (0,1.5);
 \draw[red] (0,0.5) -- (0,-0.5);
 \node[dot] at (0,0.5) {};
  \node at (0,-0.9) {$s$};
 \node at (0.6,0.5) {$\partial$};
  \node at (0,1.8) {$ss$};
\end{tikzpicture} 
\end{center} 
and now it is clear that the LHS of \cref{isoeq} isn't an isotopy of the above string diagram. In general if adding a cap on top and a cup on bottom ``doesn't change" the morphism, we will say it is rotation invariant. For example, \cref{isoeq} are equal due to rotation invariance.  Rotation invariance of a morphism greatly simplifies the graphical calculus for a morphism because it implies we don't need to do this bookkeeping above on how the morphism was constructed, we only need to care about the color(s) of the morphism (See the proof of \cref{welldeflem}).

\begin{theorem}
\label{newgensquare}
Let $\underline{w}=(s_1, \ldots, s_m)$ be such that $|m(t, \underline{w})|\ge 4$ and for a fixed $i$, consider expresssions $\underline{v}$ such that $\ell(m(s_i\underline{v}))\ge 4$. Then the following relation holds in $\ext^{2,-|\underline{w}|-|\underline{v}|-2}_{R^e}(B_t, \bs (s_1,\ldots s_{i-1}, \underline{v}, s_{i+1}, \ldots, s_m ) )$
\[ (\id_{s_1} \otimes_R \cdots \otimes_R \id_{s_{i-1}}\otimes_R \Phi_{s_i}^{\underline{v}} \otimes_R \id_{s_{i+1}}\otimes_R \cdots \otimes_R \id_{s_m}   ) \circ \Phi_t^{\underline{w}} =0\]
\end{theorem}
WLOG assume that $s_i=t$. Diagramatically, this is of the form
\begin{center}
    \begin{tikzpicture}[scale=0.6]
	       \draw[violet] (1.1,-0.1) to[out=-90,in=-180] (2,-1);
	       \draw[violet] (2.9,-0.1) to[out=-90,in=0] (2,-1);
	       \foreach \x in {2} \draw[blue] (\x,-1) -- (\x,-2);
	       \draw[blue] (2,-0.1)--(2,-1);
	       \draw[violet] (1.4,0.6) to [out=-90, in=-180] (2,-0.1) ;
	       \draw[violet] (2.6,0.6) to [out=-90, in=0] (2,-0.1);
	        \node[rhdot] at (2,-0.1) {};
	       \node[rhdot] at (2,-1) {};
	       \node at (2,0.3) {$\underline{v}$};
	       \node at (1.55,-0.4) {$\ldots$};
	       \node at (2.4,-0.4) {$\ldots$};
	    \end{tikzpicture}\raisebox{0.4cm}{$=0$}
\end{center}
\begin{proof}
This will be true by degree reasons. Namely from \cref{maincohoiso} we know that 
\[ \ext^{2,-|\underline{w}|-|\underline{v}|-2}_{R^e}(B_t, \bs (s_1,\ldots s_{i-1}, \underline{v}, s_{i+1}, \ldots, s_m ) )\cong \Db( \ker \rho_s^e(s_1,\ldots s_{i-1}, \underline{v}, s_{i+1}, \ldots, s_m) )_{-|\underline{w}|-|\underline{v}|-2} \]
But notice that the lowest degree element on the RHS is $1(s_1,\ldots s_{i-1}, \underline{v}, s_{i+1}, \ldots, s_m)$ which has degree $-|\underline{w}|-|\underline{v}|$ and thus the RHS above is 0
\end{proof}

\begin{theorem}
Let $\underline{w}=(s_1, \ldots, s_m)$ be such that $|m(t, \underline{w})|\ge 4$, then the following relation holds in $\ext_{R^e}^{2, -|\underline{w}|-4}(B_t, \bsw)$
\[ (\phi_{s_1}\otimes_R\id_{s_2} \otimes_R \cdots \otimes_R \id_{s_m}) \circ \Phi_t^{\underline{w}}= (\id_{s_1}\otimes_R \phi_{s_2} \otimes_R \cdots \otimes_R \id_{s_m}) \circ \Phi_t^{\underline{w}} \]
When $s_1=s$ and $s_2=t$, diagrammatically this will be of the form
\[  \begin{tikzpicture}[scale=0.7]
	       \draw[red] (1,0.2) to[out=-90,in=-180] (2,-1);
	       \node[rhdot] at (1.05,-0.2) {};
	       \draw[violet] (3,0.2) to[out=-90,in=0] (2,-1);
	       \draw[blue] (2,0.2) -- (2,-1);
	       \foreach \x in {2} \draw[blue] (\x,-1) -- (\x,-2);
	       \node[rhdot] at (2,-1) {};
	       \node at (2.4,-0.4) {$\ldots$};
	    \end{tikzpicture}  \raisebox{0.5cm}{ \ = \ } 
	    \begin{tikzpicture}[scale=0.7]
	       \draw[red] (1,0.2) to[out=-90,in=-180] (2,-1);
	       \draw[violet] (3,0.2) to[out=-90,in=0] (2,-1);
	       \draw[blue] (2,0.2) -- (2,-1);
	       \foreach \x in {2} \draw[blue] (\x,-1) -- (\x,-2);
	       \node[bhdot] at (2,-0.2) {};
	       \node[rhdot] at (2,-1) {};
	       \node at (2.4,-0.4) {$\ldots$};
	    \end{tikzpicture}\]
\end{theorem}
\begin{proof}
If $s_1=s_2$ then this follows from 1 color Hochschild jumping. Otherwise wlog assume that $s_1=s$ and $s_2=t$. Then popping the red Hochschild out and applying \cref{4extreduct} and \cref{newgensquare} it follows that
\begin{equation}
\label{2colorjumpeq}
    \begin{tikzpicture}[scale=0.7]
	       \draw[red] (1,0.2) to[out=-90,in=-180] (2,-1);
	       \node[rhdot] at (1.05,-0.2) {};
	       \draw[violet] (3,0.2) to[out=-90,in=0] (2,-1);
	       \draw[blue] (2,0.2) -- (2,-1);
	       \foreach \x in {2} \draw[blue] (\x,-1) -- (\x,-2);
	       \node[rhdot] at (2,-1) {};
	       \node at (2.4,-0.4) {$\ldots$};
	    \end{tikzpicture} \raisebox{0.5cm}{ \ = \ }\begin{tikzpicture}[scale=0.7]
	       \draw[red] (1,0.2) to[out=-90,in=-180] (2,-1);
	       \draw[red] (1.05,-0.1) -- (1.6,-0.4);
	       \node[rhdot] at (1.6,-0.4) {};
	       \draw[violet] (3,0.2) to[out=-90,in=0] (2,-1);
	       \draw[blue] (2,0.2) -- (2,-1);
	       \foreach \x in {2} \draw[blue] (\x,-1) -- (\x,-2);
	       \node[rhdot] at (2,-1) {};
	       \node at (2.4,-0.4) {$\ldots$};
	    \end{tikzpicture}  \raisebox{0.5cm}{ \ = \ } \begin{tikzpicture}[scale=0.7]
	       \draw[red] (1,0.2) to[out=-90,in=-180] (2,-1);
	       \draw[red] (1.05,-0.1) -- (1.6,-0.4);
	       \node[rdot] at (1.6,-0.4) {};
	       \draw[violet] (3,0.2) to[out=-90,in=0] (2,-1);
	       \draw[blue] (2,0.2) -- (2,-1);
	       \foreach \x in {2} \draw[blue] (\x,-1) -- (\x,-2);
	       \node[rhdot] at (2,-1) {};
	       \node[bhdot] at (2,-0.2) {};
	       \node at (2.4,-0.4) {$\ldots$};
	    \end{tikzpicture} \raisebox{0.5cm}{ \ $+$ \ }
	    \begin{tikzpicture}[scale=0.7]
	       \draw[red] (1,0.2) to[out=-90,in=-180] (2,-1);
	       \draw[red] (1.05,-0.1) -- (2,-0.5);
	       \draw[red] (2.3,0.1) -- (2,-0.5);
	       \node[rdot] at (2.3, 0.1) {};
	       \draw[violet] (3,0.2) to[out=-90,in=0] (2,-1);
	       \draw[blue] (2,0.2) -- (2,-1);
	       \foreach \x in {2} \draw[blue] (\x,-1) -- (\x,-2);
	       \node[rhdot] at (2,-1) {};
	       \node[rhdot] at (2,-0.5) {};
	       \node at (2.4,-0.4) {$\ldots$};
	    \end{tikzpicture}\raisebox{0.5cm}{ \ $=$ \ }
	    \begin{tikzpicture}[scale=0.7]
	       \draw[red] (1,0.2) to[out=-90,in=-180] (2,-1);
	       \draw[violet] (3,0.2) to[out=-90,in=0] (2,-1);
	       \draw[blue] (2,0.2) -- (2,-1);
	       \foreach \x in {2} \draw[blue] (\x,-1) -- (\x,-2);
	       \node[bhdot] at (2,-0.2) {};
	       \node[rhdot] at (2,-1) {};
	       \node at (2.4,-0.4) {$\ldots$};
	    \end{tikzpicture}
\end{equation}

\end{proof}

\subsection{Ext Valent Morphisms}
\begin{definition}
For any two expressions $\underline{w}=(w_1, \ldots, w_m)$ and $\underline{v}=(v_1, \ldots, v_i, \ldots, v_j, \ldots,  v_n)$ in $s$ and $t$ such that $|m(\underline{v}^{-1}\underline{w})|\ge 4$ define the morphism  $\Omega_{\underline{v}}^{\underline{w}} $ as follows. First choose an anchor, which can be any term $v_i$ in the expression $\underline{v}$ such that $v_i=t$. Then twist $\Phi_t^{ (v_{i-1}, \ldots, v_1, \underline{w}, v_n, \ldots, v_{i+1})}$ using cups and caps until you end up with a morphism in $\ext^{1, -(n+m)}_{R^e}(\bs(\underline{v}),\bs(\underline{w}))$. Diagrammatically, we let
\begin{center}
\raisebox{0.6cm}{$\Omega_{\underline{v}}^{\underline{w}}  =$}
     \begin{tikzpicture}[scale=0.75]
	       \draw[violet] (1.1,-0.1) -- (2,-1);
	       \draw[violet] (2.9,-0.1) -- (2,-1);
	       \draw[violet] (1.1,-2) -- (2,-1);
	       \draw[violet] (2.9,-2) -- (2,-1);
	       \node at (2,-0.4) {$\underline{w}$};
	       \node at (2,-1.6) {$\underline{v}$};
	       \node[rhdot] at (2,-1) {};
	    \end{tikzpicture}
\end{center}
Similarly define $\Upsilon_{\underline{v}}^{\underline{w}}$ by first choosing an anchor, which instead is now any term $v_j$ in $\underline{v}$ such that $v_j=s$ and then twist $\Phi_s^{ (v_{j-1}, \ldots, v_1, \underline{w}, v_n, \ldots, v_{j+1})}$ until you end up with a morphism in $\ext^{1, -(n+m)}_{R^e}(\bs(\underline{v}),\bs(\underline{w}))$. 
Diagrammatically, we let
\begin{center}
\raisebox{0.6cm}{$\Upsilon_{\underline{v}}^{\underline{w}}  =$}
     \begin{tikzpicture}[scale=0.75]
	       \draw[violet] (1.1,-0.1) -- (2,-1);
	       \draw[violet] (2.9,-0.1) -- (2,-1);
	       \draw[violet] (1.1,-2) -- (2,-1);
	       \draw[violet] (2.9,-2) -- (2,-1);
	       \node at (2,-0.4) {$\underline{w}$};
	       \node at (2,-1.6) {$\underline{v}$};
	       \node[bhdot] at (2,-1) {};
	    \end{tikzpicture}
\end{center}
\end{definition}

\begin{lemma}
\label{welldeflem}
$\Omega_{\underline{v}}^{\underline{w}} $ and $\Upsilon_{\underline{v}}^{\underline{w}}$ are well defined, i.e. is independent of the choice of the anchor. (We emphasize that the anchor for $ \Omega_{\underline{v}}^{\underline{w}} $ has to be blue while the anchor for $\Upsilon_{\underline{v}}^{\underline{w}}$ has to be red.)
\end{lemma}
\begin{proof}
We prove this for $ \Omega_{\underline{v}}^{\underline{w}} $ as $\Upsilon_{\underline{v}}^{\underline{w}}$ is quite similar. WLOG (add appropriate caps and cups and use adjunction) assume that $\un{v}=(t, \ldots, t)$ and that our two different anchors are the first and last $t$ in $\un{v}$ and let $(\Omega_{\underline{v}}^{\underline{w}})_1$ and $(\Omega_{\underline{v}}^{\underline{w}})_2$ be the associated morphisms. Since $\Phi_t^{\un{w}}$ is cyclic by \cref{newgencyclic}, we can use \cref{cks} to bring $(\Omega_{\underline{v}}^{\underline{w}})_1$ to the LHS below and  $(\Omega_{\underline{v}}^{\underline{w}})_2$ to the RHS below 
\begin{equation}
\label{anchoreq}
\raisebox{0.8cm}{$(\Omega_{\underline{v}}^{\underline{w}})_1=$} 
\begin{tikzpicture}[scale=0.6]
	       \draw[blue] (2,-1)-- (2,-2);
	       \draw[violet] (1.1,1.3) to[out=-90,in=-180] (2,-1);
	       \draw[violet] (2,1.3) -- (2, -1);
	       \draw[blue] (2,-1) to[out=80,in=-180] (3.4,1.8) to[out=0,in=90] (4.7,-0.1);
	       \draw[blue] (4.7,-0.1) -- (4.7, -2);
	       \draw[violet] (2.9,-0.1) to[out=-90,in=0] (2,-1);
	       \draw[violet] (2.9,-0.1) to[out=90,in=180] (3.2,0.5) to[out=0,in=90] (3.5,-0.1);
	       \draw[violet] (2,-1) to[out=50,in=180] (3.2,1.4) to[out=0,in=90] (4.4, -0.1);
	       \draw[violet] (4.4,-0.1) -- (4.4, -2);
	       \draw[violet] (3.5,-0.1) -- (3.5, -2);
	       \node[rhdot] at (2,-1) {};
	       \node at (1.6,0) {$\underline{w}$};
	       \node at (3.9,-1) {$\ldots $};
	    \end{tikzpicture} \ \ \raisebox{0.8cm}{$ \overset{?}{=} $} \ \ 
\begin{tikzpicture}[scale=0.6]
            \draw[blue] (2,-1) to[out=100,in=0] (0.8,1.8) to[out=180,in=90] (-0.7, -0.1);
	       \draw[blue] (-0.7,-0.1) -- (-0.7, -2);
	       \draw[violet] (2,-1) to[out=120,in=0] (0.8,1.4) to[out=180,in=90] (-0.4, -0.1);
	       \draw[violet] (-0.4,-0.1) -- (-0.4, -2);
	       \draw[violet] (1.1,-0.1) to[out=-90,in=-180] (2,-1);
	       \draw[violet] (1.1,-0.1) to[out=90,in=0] (0.8,0.5) to[out=-180,in=90] (0.5, -0.1) ;
	       \draw[violet] (2,1.3) -- (2, -1);
	       \draw[violet] (2.9,1.3) to[out=-90,in=0] (2,-1);
	       \draw[violet] (0.5,-0.1) -- (0.5, -2);
	       \draw[blue] (2,-1) -- (2,-2);
	       \node[rhdot] at (2,-1) {};
	       \node at (2.35,0) {$\underline{w}$};
	       \node at (0,-1) {$\ldots $};
	    \end{tikzpicture}
\raisebox{0.8cm}{$=(\Omega_{\underline{v}}^{\underline{w}})_2$} 	    
\end{equation}
        If $\Phi_t^{\un{w}}$ were rotation invariant, then we could conclude that both sides were equal above. But \cref{newgenrotation} only says that $\Phi_t^{\un{w}}$ is signed rotation invariant so the two sides are the same up to a sign. However, note that there must be an even number of color changes in the expression $\un{v}=(t, \ldots, t)$ and so the sign must be $+1$ and thus we have equality as desired.  
\end{proof}

\begin{remark}
$\Upsilon_{\underline{v}}^{\underline{w}}=-\Omega_{\underline{v}}^{\underline{w}}$ as a consequence of \cref{newgenrotation}. For the most part we will work with $\Omega_{\underline{v}}^{\underline{w}}$ and occasionally use $\Upsilon_{\underline{v}}^{\underline{w}}$ as needed. Also note that the same argument shows that one can choose an anchor from a spot in the top expression $\underline{w}$ as well and $\underline{v}$ or $\underline{w}$ can be the empty expression as well.
\end{remark}

\begin{remark}
Up until now we were free to interchange $t$ with $s$ in all of our results but as seen in the definition of $\Omega_{\underline{v}}^{\underline{w}}$ and $\Upsilon_{\underline{v}}^{\underline{w}}$ we \underline{cannot} interchange $t$ and $s$. Also note that we have used curved lines with a red dot with a red hollow dot to diagrammatically denote $\Phi_t^{\underline{w}}$ and straight lines with a red hollow dot for $\Omega_{\underline{t}}^{\underline{w}}$. However since
\begin{center}
\raisebox{0.6cm}{$\Omega_{\underline{t}}^{\underline{w}}  =$}
     \begin{tikzpicture}[scale=0.75]
	       \draw[violet] (1.1,-0.1) -- (2,-1);
	       \draw[violet] (2.9,-0.1) -- (2,-1);
	       \draw[blue] (2,-2) -- (2,-1);
	       \node at (2,-0.4) {$\underline{w}$};
	       \node[rhdot] at (2,-1) {};
	    \end{tikzpicture}\raisebox{0.6cm}{$  = \  \hspace{-1ex}$}
	\begin{tikzpicture}[scale=0.75]
	       \draw[violet] (1.1,-0.1) to[out=-90,in=-180] (2,-1);
	       \draw[violet] (2.9,-0.1) to[out=-90,in=0] (2,-1);
	       \foreach \x in {2} \draw[blue] (\x,-1) -- (\x,-2);
	       \node[rhdot] at (2,-1) {};
	       \node at (2,-0.4) {$\underline{w}$};
	    \end{tikzpicture}\raisebox{0.6cm}{$  =\Phi_{t}^{\underline{w}} $}
\end{center}
there is no ambiguity in the diagrammatic picture if one cannot determine if a line is straight or not. 
\end{remark}

\begin{example}
If $v_n=t$ and $v_1=s$, then one possible presentation of $\Omega_{\underline{v}}^{\underline{w}}$ and $\Upsilon_{\underline{v}}^{\underline{w}}$ is given below to the left and right, respectively.
\begin{center}
    \raisebox{1.1cm}{$ \Omega_{\underline{v}}^{\underline{w}} =$} \ \begin{tikzpicture}[scale=0.6]
	       \draw[red] (2,-1) to[out=110,in=0] (0.75,1.4) to[out=180,in=90] (-0.5, -0.1);
	       \draw[red] (-0.5,-0.1) -- (-0.5, -2);
	       \draw[violet] (1.1,-0.1) to[out=-90,in=-180] (2,-1);
	       \draw[violet] (1.1,-0.1) to[out=90,in=0] (0.8,0.5) to[out=-180,in=90] (0.5, -0.1) ;
	       \draw[violet] (2,1.3) -- (2, -1);
	       \draw[violet] (2.9,1.3) to[out=-90,in=0] (2,-1);
	       \draw[violet] (0.5,-0.1) -- (0.5, -2);
	       \draw[blue] (2,-1) -- (2,-2);
	       \node[rhdot] at (2,-1) {};
	       \node at (2.4,-0.4) {$\ldots$};
	       \node at (0,-0.9) {$\ldots$};
	       \node at (2.1,1.6) {$\scalemath{0.9}{w_{1}}$};
	       \node at (2.9,1.6) {$\scalemath{0.9}{w_{m}}$};
	       \node at (-0.5,-2.3) {$s$};
	       \node at (0.5,-2.3) {$ v_{n-1}$};
	       \node at (2,-2.3) {$ t$};
	    \end{tikzpicture}\quad \quad \raisebox{1.1cm}{$ \Upsilon_{\underline{v}}^{\underline{w}} =$} \
	\begin{tikzpicture}[scale=0.6]
	       \draw[red] (2,-1)-- (2,-2);
	       \draw[violet] (1.1,1.3) to[out=-90,in=-180] (2,-1);
	       \draw[violet] (2,1.3) -- (2, -1);
	       \draw[blue] (2,-1) to[out=70,in=-180] (3.4,1.4) to[out=0,in=90] (4.5,-0.1);
	       \draw[blue] (4.5,-0.1) -- (4.5, -2);
	       \draw[violet] (2.9,-0.1) to[out=-90,in=0] (2,-1);
	       \draw[violet] (2.9,-0.1) to[out=90,in=180] (3.2,0.5) to[out=0,in=90] (3.5,-0.1);
	       \node at (3.5,-2.3) {$v_{n-1}$};
	       \draw[violet] (3.5,-0.1) -- (3.5, -2);
	       \node[bhdot] at (2,-1) {};
	       \node at (1.6,-0.4) {$\ldots$};
	       \node at (2,-2.3) {$s$};
	       \node at (1,1.6) {$\scalemath{0.9}{w_{1}}$};
	       \node at (1.9,1.6) {$\scalemath{0.9}{w_{m}}$};
	       \node at (4,-1) {$\ldots$};
	       \node at (4.5,-2.3) {$ t$};
	    \end{tikzpicture}
\end{center}

\end{example}

A huge advantage of working with $\Omega_{\underline{v}}^{\underline{w}}$ is that 

\begin{lemma}
\label{omegarotation}
$\Omega_{\underline{v}}^{\underline{w}}$ is rotation invariant. Let $\underline{v}=(v_1, \ldots, v_n)$ and $\underline{w}=(w_1, \ldots, w_n)$. Diagrammatically, this means
\begin{equation}
\raisebox{-6ex}{
 \begin{tikzpicture}[scale=0.75]
	       \draw[violet] (1.1,-0.1) -- (2,-1);
	       \draw[violet] (2.9,-0.1) -- (2,-1);
	       \draw[violet] (1.1,-2) -- (2,-1);
	       \draw[violet] (1.1,-0.1) to[out=90,in=0] (0.8,0.5) to[out=-180,in=90] (0.5, -0.1) ;
	       \draw[violet] (0.5,-0.1) -- (0.5,-2);
	       \draw[violet] (2.9,-2) -- (2,-1);
	       \draw[violet] (2.9, -2) to[out=-90,in=-180] (3.2,-2.6) to[out=0,in=-90] (3.5, -2);
	       \draw[violet] (3.5,-2) -- (3.5,-0.1);
	       \node at (2,-0.4) {$\underline{w}$};
	       \node at (2,-1.6) {$\underline{v}$};
	       \node[rhdot] at (2,-1) {};
	    \end{tikzpicture}\raisebox{1cm}{ \ \  = \ }\raisebox{0.3cm}{\begin{tikzpicture}[scale=0.8]
	       \draw[violet] (1.1,-0.1) -- (2,-1);
	       \draw[violet] (2.9,-0.1) -- (2,-1);
	       \draw[violet] (1.1,-2) -- (2,-1);
	       \draw[violet] (2.9,-2) -- (2,-1);
	       \node at (2,-0.2) {$\s{(\underline{\widehat{w}^1}, v_n)}$};
	       \node at (2,-1.8) {$\s{(w_1, \underline{\widehat{v}^n})}$};
	       \node[rhdot] at (2,-1) {};
	    \end{tikzpicture}}}
\end{equation}
\begin{equation}
    \raisebox{-6ex}{
 \begin{tikzpicture}[scale=0.75]
	       \draw[violet] (1.1,-0.1) -- (2,-1);
	       \draw[violet] (2.9,-0.1) -- (2,-1);
	       \draw[violet] (1.1,-2) -- (2,-1);
	       \draw[violet] (1.1,-2) to[out=-90,in=0] (0.8,-2.6) to[out=-180,in=-90] (0.5, -2) ;
	       \draw[violet] (0.5,-0.1) -- (0.5,-2);
	       \draw[violet] (2.9,-2) -- (2,-1);
	       \draw[violet] (2.9, -0.1) to[out=90,in=-180] (3.2,0.5) to[out=0,in=90] (3.5, -0.1);
	       \draw[violet] (3.5,-2) -- (3.5,-0.1);
	       \node at (2,-0.4) {$\underline{w}$};
	       \node at (2,-1.6) {$\underline{v}$};
	       \node[rhdot] at (2,-1) {};
	    \end{tikzpicture}\raisebox{1cm}{ \ \  = \ }\raisebox{0.3cm}{\begin{tikzpicture}[scale=0.8]
	       \draw[violet] (1.1,-0.1) -- (2,-1);
	       \draw[violet] (2.9,-0.1) -- (2,-1);
	       \draw[violet] (1.1,-2) -- (2,-1);
	       \draw[violet] (2.9,-2) -- (2,-1);
	       \node at (2,-0.2) {$\s{(v_1, \underline{\widehat{w}^n})}$};
	       \node at (2,-1.8) {$\s{(\underline{\widehat{v}^1}, w_n)}$};
	       \node[rhdot] at (2,-1) {};
	    \end{tikzpicture}}}
\end{equation}
\end{lemma}
\begin{proof}
This follows from the definition of $\Omega_{\underline{v}}^{\underline{w}}$ being well defined. Choose some blue strand to be the anchor for $\Omega_{\underline{v}}^{\underline{w}}$. Then the LHS above can be used as the definition of the RHS above.
\end{proof}

\begin{theorem}[$\Omega_{\underline{v}}^{\underline{w}}$ absorbs morphisms]
\label{newgenabsorb}
The following identities hold in $\bsbimext(\hf, W_\infty)$ whenever both sides of the equalities are defined. 
\begin{equation}
\raisebox{-1cm}{
     \begin{tikzpicture}[scale=0.75]
	       \draw[violet] (1,-0.1) -- (2,-1);
	       \draw[violet] (3,-0.1) -- (2,-1);
	       \draw[violet] (0.8,-2.2) -- (2,-1);
	       \draw[violet] (3.2,-2.2) -- (2,-1);
	       \draw[violet] (2,-1) -- (2,-1.7);
	       \draw[violet] (1.7,-2.2) -- (2,-1);
	       \draw[violet] (2,-1) -- (2.3,-2.2);
	       \node[pdot] at (2,-1.7) {};
	       \node[rhdot] at (2,-1) {};
	       \node at (2,-0.3) {$\underline{w}$};
	       \node at (1.4,-1.95) {$\ldots$};
	       \node at (2.6,-1.95) {$\ldots$};
	       \node at (2,-2.1) {$\scalemath{0.9}{v_i}$};
	       \node at (1.7,-2.4) {$\scalemath{0.7}{v_{i-1}}$};
	       \node at (2.5,-2.4) {$\scalemath{0.7}{v_{i+1}}$};
	    \end{tikzpicture}}=\raisebox{-1cm}{
     \begin{tikzpicture}[scale=0.75]
	       \draw[violet] (1,-0.1) -- (2,-1);
	       \draw[violet] (3,-0.1) -- (2,-1);
	       \draw[violet] (0.8,-2.2) -- (2,-1);
	       \draw[violet] (3.2,-2.2) -- (2,-1);
	       \draw[violet] (1.7,-2.2) -- (2,-1);
	       \draw[violet] (2,-1) -- (2.3,-2.2);
	       \node[rhdot] at (2,-1) {};
	       \node at (2,-0.3) {$\underline{w}$};
	       \node at (1.4,-1.95) {$\ldots$};
	       \node at (2.6,-1.95) {$\ldots$};
	       \node at (1.7,-2.4) {$\scalemath{0.7}{v_{i-1}}$};
	       \node at (2.5,-2.4) {$\scalemath{0.7}{v_{i+1}}$};
	    \end{tikzpicture}} \qquad \qquad \quad
	 \raisebox{-0.8cm}{
     \begin{tikzpicture}[scale=0.75]
	       \draw[violet] (0.8,0.2) -- (2,-1);
	       \draw[violet] (3.2,0.2) -- (2,-1);
	       \draw[violet] (1,-2) -- (2,-1);
	       \draw[violet] (3,-2) -- (2,-1);
	       \draw[violet] (2,-1) -- (2,-0.25);
	       \draw[violet] (1.65,0.2) -- (2,-1);
	       \draw[violet] (2,-1) -- (2.35,0.2);
	       \node[pdot] at (2,-0.25) {};
	       \node[rhdot] at (2,-1) {};
	       \node at (2,-1.7) {$\underline{v}$};
	       \node at (1.4,-0.1) {$\ldots$};
	       \node at (2.6,-0.1) {$\ldots$};
	       \node at (2,-0.05) {$\scalemath{0.8}{w_i}$};
	       \node at (1.7,0.35) {$\scalemath{0.7}{w_{i-1}}$};
	       \node at (2.5,0.35) {$\scalemath{0.7}{w_{i+1}}$};
	    \end{tikzpicture}}\raisebox{-0.2cm}{=}\raisebox{-0.8cm}{
     \begin{tikzpicture}[scale=0.75]
	       \draw[violet] (0.8,0.2) -- (2,-1);
	       \draw[violet] (3.2,0.2) -- (2,-1);
	       \draw[violet] (1,-2) -- (2,-1);
	       \draw[violet] (3,-2) -- (2,-1);
	       \draw[violet] (1.65,0.2) -- (2,-1);
	       \draw[violet] (2,-1) -- (2.35,0.2);
	       \node[rhdot] at (2,-1) {};
	       \node at (2,-1.7) {$\underline{v}$};
	       \node at (1.4,-0.1) {$\ldots$};
	       \node at (2.6,-0.1) {$\ldots$};
	       \node at (1.7,0.35) {$\scalemath{0.7}{w_{i-1}}$};
	       \node at (2.5,0.35) {$\scalemath{0.7}{w_{i+1}}$};
	    \end{tikzpicture}}
\end{equation}
\begin{equation}
\raisebox{-1cm}{
    \begin{tikzpicture}[scale=0.75]
	       \draw[violet] (1,-0.1) -- (2,-1);
	       \draw[violet] (3,-0.1) -- (2,-1);
	       \draw[violet] (0.8,-2.2) -- (2,-1);
	       \draw[violet] (3.2,-2.2) -- (2,-1);
	       \draw[violet] (2,-1) -- (2,-1.7);
	       \draw[violet] (1.7,-2.2) -- (2,-1.7);
	       \draw[violet] (2,-1.7) -- (2.3,-2.2);
	       \node[rhdot] at (2,-1) {};
	       \node at (2,-0.3) {$\underline{w}$};
	       \node at (1.4,-1.95) {$\ldots$};
	       \node at (2.6,-1.95) {$\ldots$};
	       \node at (1.7,-2.4) {$\scalemath{0.8}{v_{i}}$};
	       \node at (2.5,-2.4) {$\scalemath{0.8}{v_{i}}$};
	    \end{tikzpicture}}=\raisebox{-1cm}{
     \begin{tikzpicture}[scale=0.75]
	       \draw[violet] (1,-0.1) -- (2,-1);
	       \draw[violet] (3,-0.1) -- (2,-1);
	       \draw[violet] (0.8,-2.2) -- (2,-1);
	       \draw[violet] (3.2,-2.2) -- (2,-1);
	       \draw[violet] (1.7,-2.2) -- (2,-1);
	       \draw[violet] (2,-1) -- (2.3,-2.2);
	       \node[rhdot] at (2,-1) {};
	       \node at (2,-0.3) {$\underline{w}$};
	       \node at (1.4,-1.95) {$\ldots$};
	       \node at (2.6,-1.95) {$\ldots$};
	       \node at (1.7,-2.4) {$\scalemath{0.8}{v_{i}}$};
	       \node at (2.5,-2.4) {$\scalemath{0.8}{v_{i}}$};
	    \end{tikzpicture}} \qquad \qquad \quad
	 \raisebox{-0.8cm}{
     \begin{tikzpicture}[scale=0.75]
	       \draw[violet] (0.8,0.2) -- (2,-1);
	       \draw[violet] (3.2,0.2) -- (2,-1);
	       \draw[violet] (1,-2) -- (2,-1);
	       \draw[violet] (3,-2) -- (2,-1);
	       \draw[violet] (2,-1) -- (2,-0.25);
	       \draw[violet] (1.65,0.2) -- (2,-0.25);
	       \draw[violet] (2,-0.25) -- (2.35,0.2);
	       \node[rhdot] at (2,-1) {};
	       \node at (2,-1.7) {$\underline{v}$};
	       \node at (1.4,-0.1) {$\ldots$};
	       \node at (2.6,-0.1) {$\ldots$};
	       \node at (1.6,0.35) {$\scalemath{0.8}{w_{i}}$};
	       \node at (2.5,0.35) {$\scalemath{0.8}{w_{i}}$};
	    \end{tikzpicture}}\raisebox{-0.2cm}{=}\raisebox{-0.8cm}{
     \begin{tikzpicture}[scale=0.75]
	       \draw[violet] (0.8,0.2) -- (2,-1);
	       \draw[violet] (3.2,0.2) -- (2,-1);
	       \draw[violet] (1,-2) -- (2,-1);
	       \draw[violet] (3,-2) -- (2,-1);
	       \draw[violet] (1.65,0.2) -- (2,-1);
	       \draw[violet] (2,-1) -- (2.35,0.2);
	       \node[rhdot] at (2,-1) {};
	       \node at (2,-1.7) {$\underline{v}$};
	       \node at (1.4,-0.1) {$\ldots$};
	       \node at (2.6,-0.1) {$\ldots$};
	       \node at (1.6,0.35) {$\scalemath{0.8}{w_{i}}$};
	       \node at (2.5,0.35) {$\scalemath{0.8}{w_{i}}$};
	    \end{tikzpicture}}
\end{equation}
\begin{equation}
\raisebox{-1cm}{
    \begin{tikzpicture}[scale=0.75]
	       \draw[violet] (1.2,-0.2) -- (2,-1);
	       \draw[violet] (2.8,-0.2) -- (2,-1);
	       \draw[violet] (1.2,-1.8) -- (2,-1);
	       \draw[violet] (2.8,-1.8) -- (2,-1);
	       \draw[violet] (1.2,-1.8) -- (1.2,-2.4);
	       \draw[violet] (1.2,-1.8) -- (0.7,-1.3);
	       \draw[violet] (0.7,-0.2) -- (0.7,-1.3);
	       \draw[violet] (2.8,-1.8) -- (2.8, -2.4);
	       \node at (2,-0.4) {$\underline{w}$};
	       \node at (2,-1.6) {$\underline{v}$};
	       \node at (1.2,-2.7) {$\scalemath{0.9}{v_1}$};
	       \node[rhdot] at (2,-1) {};
	    \end{tikzpicture}} \raisebox{0.4cm}{=}
\raisebox{-0.3cm}{
\begin{tikzpicture}[scale=0.75]
	       \draw[violet] (1.1,-0.1) -- (2,-1);
	       \draw[violet] (2.9,-0.1) -- (2,-1);
	       \draw[violet] (1.1,-2) -- (2,-1);
	       \draw[violet] (2.9,-2) -- (2,-1);
	       \node at (2,-0.3) {$\s{(v_1, \underline{w})}$};
	       \node at (2,-1.6) {$\underline{v}$};
	       \node[rhdot] at (2,-1) {};
	    \end{tikzpicture}} \qquad \qquad \quad 
\raisebox{-0.8cm}{
    \begin{tikzpicture}[scale=0.75]
	       \draw[violet] (1.2,-0.2) -- (2,-1);
	       \draw[violet] (2.8,-0.2) -- (2,-1);
	       \draw[violet] (1.2,-1.8) -- (2,-1);
	       \draw[violet] (2.8,-1.8) -- (2,-1);
	       \draw[violet] (1.2,-0.2) -- (1.2,0.4);
	       \draw[violet] (1.2,-0.2) -- (0.7,-0.5);
	       \draw[violet] (0.7,-0.5) -- (0.7,-1.8);
	       \draw[violet] (2.8,-0.2) -- (2.8, 0.4);
	       \node at (2,-0.4) {$\underline{w}$};
	       \node at (2,-1.6) {$\underline{v}$};
	       \node at (1.2,0.6) {$\scalemath{0.9}{w_1}$};
	       \node[rhdot] at (2,-1) {};
	    \end{tikzpicture}} \raisebox{-0.1cm}{=}
\raisebox{-0.8cm}{
\begin{tikzpicture}[scale=0.75]
	       \draw[violet] (1.1,-0.2) -- (2,-1);
	       \draw[violet] (2.9,-0.2) -- (2,-1);
	       \draw[violet] (1.1,-2.1) -- (2,-1);
	       \draw[violet] (2.9,-2.1) -- (2,-1);
	       \node at (2,-1.8) {$\s{(w_1, \underline{v})}$};
	       \node at (2,-0.35) {$\underline{w}$};
	       \node[rhdot] at (2,-1) {};
	    \end{tikzpicture}}
\end{equation}
\end{theorem}
\begin{proof}
Follows from rotation invariance of $\Phi_s^{\underline{w}}$ along with \cref{newgenreduct} and \cref{newgenmult}.
\end{proof}

One particular choice of $\underline{w}$ and $\underline{v}$ in $\Omega_{\underline{v}}^{\underline{w}}$ will be essential to us, namely following \cite{DC} we define

\begin{definition}
Let $\usmt{s}{m}{}=\overbrace{\underline{st\cdots} }^{m \text{ terms}}$ where there are $m$ terms in the expression that alternate between $s$ and $t$ starting with $s$. Also let $\smt{s}{m}{}$ be the corresponding element in $W$. Define $\usmt{}{m}{s}$ and $\smt{}{m}{s}$ similarly. 
\end{definition}

\begin{definition}[$2k-$extvalent]
For $k\ge 2$, define the red $2k-$extvalent morphism to be \  $\Omega_{\usmt{t}{k}{}}^{\usmt{s}{k}{}}$ or \  $\Omega_{\usmt{s}{k}{}}^{\usmt{t}{k}{}}\in \ext^{1, -2k}_{R^e}(\bs( \usmt{t}{k}{}), \bs(\usmt{s}{k}{}) )$ which we will denote diagrammatically by either
\begin{center}
    \begin{tikzpicture}[scale=0.75]
	       \draw[red] (1.1,-0.1) -- (2,-1);
	       \draw[blue] (1.5,-0.1) -- (2,-1);
	       \draw[violet] (2.9,-0.1) -- (2,-1);
	       \draw[blue] (1.1,-2) -- (2,-1);
	       \draw[red] (1.5,-2) -- (2,-1);
	       \draw[violet] (2.9,-2) -- (2,-1);
	       \node at (2.1,-0.4) {$\ldots$};
	       \node at (2.1,-1.6) {$\ldots$};
	       \node[rhdot] at (2,-1) {};
	    \end{tikzpicture}\quad \quad  \begin{tikzpicture}[scale=0.75]
	       \draw[blue] (1.1,-0.1) -- (2,-1);
	       \draw[red] (1.5,-0.1) -- (2,-1);
	       \draw[violet] (2.9,-0.1) -- (2,-1);
	       \draw[red] (1.1,-2) -- (2,-1);
	       \draw[blue] (1.5,-2) -- (2,-1);
	       \draw[violet] (2.9,-2) -- (2,-1);
	       \node at (2.1,-0.4) {$\ldots$};
	       \node at (2.1,-1.6) {$\ldots$};
	       \node[rhdot] at (2,-1) {};
	    \end{tikzpicture}
\end{center}
This is well defined as $|m(\usmt{t}{k}{}^{-1}\usmt{s}{k}{})|=|\usmt{t}{k}{}^{-1}\usmt{s}{k}{}|=2k\ge 4$. Also for $k\ge 2$, define the blue $2k-$extvalent morphism to be \  $\Upsilon_{\usmt{t}{k}{}}^{\usmt{s}{k}{}}$ or $\Upsilon_{\usmt{s}{k}{}}^{\usmt{t}{k}{}}\in \ext^{1, -2k}_{R^e}(\bs( \usmt{t}{k}{}), \bs(\usmt{s}{k}{}) )$ which we will denote diagrammatically by
\begin{center}
    \begin{tikzpicture}[scale=0.75]
	       \draw[red] (1.1,-0.1) -- (2,-1);
	       \draw[blue] (1.5,-0.1) -- (2,-1);
	       \draw[violet] (2.9,-0.1) -- (2,-1);
	       \draw[blue] (1.1,-2) -- (2,-1);
	       \draw[red] (1.5,-2) -- (2,-1);
	       \draw[violet] (2.9,-2) -- (2,-1);
	       \node at (2.1,-0.4) {$\ldots$};
	       \node at (2.1,-1.6) {$\ldots$};
	       \node[bhdot] at (2,-1) {};
	    \end{tikzpicture}\quad \quad  \begin{tikzpicture}[scale=0.75]
	       \draw[blue] (1.1,-0.1) -- (2,-1);
	       \draw[red] (1.5,-0.1) -- (2,-1);
	       \draw[violet] (2.9,-0.1) -- (2,-1);
	       \draw[red] (1.1,-2) -- (2,-1);
	       \draw[blue] (1.5,-2) -- (2,-1);
	       \draw[violet] (2.9,-2) -- (2,-1);
	       \node at (2.1,-0.4) {$\ldots$};
	       \node at (2.1,-1.6) {$\ldots$};
	       \node[bhdot] at (2,-1) {};
	    \end{tikzpicture}
\end{center}
\end{definition}

The reason why the $2k$ extvalent morphisms are essential is because they can be used to generate the rest of the $\Omega_{\underline{v}}^{\underline{w}}$ morphisms. Namely,

\begin{lemma}
$\Omega_{\underline{v}}^{\underline{w}}$ can be constructed as the composition of $2k$ red extvalent morphisms, along with the generating morphisms of $\bsbim(\hf, W_\infty)$. 
\end{lemma}
\begin{proof}
Any two adjacent terms in $\underline{v}$ or $\underline{w}$ that are the same can be replaced using \cref{newgenabsorb}.
\end{proof}

\begin{example}
Let us give a partial description for two possible chain maps representing the $4$ extvalent morphism $\Omega_{(t,s)}^{(s,t)}\in \ext^{1, -4}_{R^e}(B_t B_s, B_s B_t)$. The complex $K_t K_s$ in homological degree 1 has a decomposition given by
\[ \underline{\rho_s^{(1)}}\boxtimes R^{ee}\bigoplus \underline{\rho_t t(\rho_t)^{(1)}}\boxtimes R^{ee}\bigoplus\underline{\rho_t^{(2)}}\boxtimes R^{ee}\bigoplus \underline{\rho_s s(\rho_s)^{(2)}}\boxtimes R^{ee} \]
By definition one possible presentation is given by
\begin{center}
    \raisebox{0.5cm}{$\Omega_{(t,s)}^{(s,t)}=$}\begin{tikzpicture}[scale=0.6]
	       \draw[red] (1.1,0.1) -- (2,-1);
	       \draw[blue] (2,0.3) -- (2,-1);
	       \draw[red] (2.9,-0.1) -- (2,-1);
	       \draw[red] (2.9,-0.1) to[out=90,in=180] (3.2,0.5) to[out=0,in=90] (3.5,-0.1);
	       \draw[red] (3.5,-0.1) -- (3.5, -2);
	       \foreach \x in {2} \draw[blue] (\x,-1) -- (\x,-2);
	       \node[rhdot] at (2,-1) {};
	    \end{tikzpicture}
\end{center}
so a possible partial chain lift of $\Omega_{(t,s)}^{(s,t)}$ is given by $\omega_{(t,s)}^{(s,t)}: K_tK_s^{[1]}\to K_s K_t^{[0]}$ 
\begin{align*}
\omega_{(t,s)}^{(s,t)}\paren{\underline{\rho_s^{(1)}}\boxtimes 1\otimes h \otimes 1}&=\underline{1}\boxtimes 1\otimes 1\otimes \partial_s(h)  \\
\omega_{(t,s)}^{(s,t)}\paren{\underline{\rho_t t(\rho_t)^{(1)}}\boxtimes 1\otimes h \otimes 1}&=\underline{1}\boxtimes -a_{st}\otimes \rho_t\otimes \partial_s(h)+c_s\otimes\partial_s(h)+1\otimes 1\otimes (h-\alpha_s\partial_s(h))\\
\omega_{(t,s)}^{(s,t)}\paren{\underline{\rho_t^{(2)}}\boxtimes 1\otimes h \otimes 1}&=0\\
 \omega_{(t,s)}^{(s,t)}\paren{\underline{\rho_s s(\rho_s)^{(2)}}\boxtimes 1\otimes h \otimes 1 }&=0 
\end{align*}
On the other hand, we have that $\Omega_{(t,s)}^{(s,t)}=-\Upsilon_{(t,s)}^{(s,t)}$ so another presentation is given by 

\begin{center}
    \raisebox{0.5cm}{$\Omega_{(t,s)}^{(s,t)}= - \ $}\begin{tikzpicture}[scale=0.6]
	       \draw[blue] (1.1,-0.1) -- (2,-1);
	       \draw[blue] (0.5,-0.1) to[out=90,in=180] (0.8,0.5) to[out=0,in=90] (1.1,-0.1);
	       \draw[blue] (0.5,-0.1) -- (0.5,-2);
	       \draw[red] (2,0.3) -- (2,-1);
	       \draw[blue] (2.9,0.1) -- (2,-1);
	       \foreach \x in {2} \draw[red] (\x,-1) -- (\x,-2);
	       \node[bhdot] at (2,-1) {};
	    \end{tikzpicture}
\end{center}
so another possible partial chain lift $-\upsilon_{(t,s)}^{(s,t)}: K_tK_s^{[1]}\to K_s K_t^{[0]}$ is given by
\begin{align*}
-\upsilon_{(t,s)}^{(s,t)}\paren{\underline{\rho_s^{(1)}}\boxtimes 1\otimes h \otimes 1}&=0  \\
-\upsilon_{(t,s)}^{(s,t)}\paren{\underline{\rho_t t(\rho_t)^{(1)}}\boxtimes 1\otimes h \otimes 1}&=0\\
-\upsilon_{(t,s)}^{(s,t)}\paren{\underline{\rho_t^{(2)}}\boxtimes 1\otimes h \otimes 1}&=-\underline{1}\boxtimes\partial_t(h)\otimes 1\otimes 1\\
-\upsilon_{(t,s)}^{(s,t)}\paren{\underline{\rho_s s(\rho_s)^{(2)}}\boxtimes 1\otimes h \otimes 1}&=-\underline{1}\boxtimes (h-a_{ts}\rho_s\partial_t(h))\otimes 1 \otimes 1+\partial_t(h)\otimes 1\otimes \alpha_t-\partial_t(h)\otimes c_t
\end{align*}

\end{example}

\subsection{Cohomology Relations}

\begin{theorem}
Suppose $k\ge 2$ and let $t_i=(\usmt{t}{k}{})_i$ the $i-$th spot in the expression $\usmt{t}{k}{}$ and similarly let $s_i=(\usmt{s}{k}{})_i$. Then we have the following relation in $\ext_{R^e}^{1, -2k+2 }(\bs( \usmt{t}{k}{}),\bs(\usmt{s}{k}{}) )$\vspace{-2ex}
\begin{equation}
\label{cohorelseq}
\raisebox{0.2cm}{\scalemath{1.05}{\boxed{[k]\alpha_s \ }}}
\raisebox{-0.5cm}{
\begin{tikzpicture}[scale=0.8]
	       \draw[red] (1.1,-0.1) -- (2,-1);
	       \draw[blue] (1.5,-0.1) -- (2,-1);
	       \draw[violet] (2.9,-0.1) -- (2,-1);
	       \draw[blue] (1.1,-2) -- (2,-1);
	       \draw[red] (1.5,-2) -- (2,-1);
	       \draw[violet] (2.9,-2) -- (2,-1);
	       \node[rhdot] at (2,-1) {};
	       \node at (2.1,-0.4) {$\ldots$};
	       \node at (2.1,-1.6) {$\ldots$};
	    \end{tikzpicture}}
\raisebox{0.2cm}{\scalemath{1.05}{=-\sum_{i=1}^{k-1} \boxed{ [k-i]} \ }}
\raisebox{-0.8cm}{
    \begin{tikzpicture}[scale=0.7]
	       \draw[red] (1,-0.1) -- (2,-1);
	       \draw[blue] (1.5,-0.1) -- (2,-1);
	       \draw[violet] (3,-0.1) -- (2,-1);
	       \draw[blue] (0.8,-2.2) -- (2,-1);
	       \draw[violet] (3.2,-2.2) -- (2,-1);
	       \draw[myyellow] (2,-1) -- (2,-1.5);
	       \draw[violet] (2,-2.2) -- (2,-1.8);
	       \draw[myyellow] (1.7,-2.2) to[out=90,in=180] (2,-1.5) to[out=0,in=90] (2.3,-2.2);
	       \node[pdot] at (2,-1.8) {};
	       \node[rhdot] at (2,-1) {};
	       \node at (2.1,-0.3) {$\ldots$};
	       \node at (1.4,-1.95) {$\ldots$};
	       \node at (2.6,-1.95) {$\ldots$};
	       \node at (2,-2.5) {$t_i$};
	    \end{tikzpicture}}\raisebox{0.2cm}{ \scalemath{1.05}{+\sum_{i=1}^{k}\boxed{[k+1-i]}}}
	    \raisebox{-0.4cm}{
	    \begin{tikzpicture}[scale=0.7]
	       \draw[red] (0.8,0.1) -- (2,-1);
	       \draw[violet] (3.2,0.1) -- (2,-1);
	       \draw[blue] (1,-2) -- (2,-1);
	       \draw[violet] (3,-2) -- (2,-1);
	       \draw[red] (2,-1) -- (1.5,-2);
	       \draw[myyellow] (2,-1) -- (2,-0.5);
	       \draw[violet] (2,0.2) -- (2,-0.2);
	       \draw[myyellow] (1.7,0.1) to[out=-90,in=180] (2,-0.5) to[out=0,in=-90] (2.3,0.1);
	       \node at (2,0.35) {$s_i$};
	       \node[pdot] at (2,-0.2) {};
	       \node[rhdot] at (2,-1) {};
	       \node at (2.6,-0.15) {$\ldots$};
	       \node at (1.4,-0.15) {$\ldots$};
	       \node at (2.1,-1.7) {$\ldots$};
	    \end{tikzpicture}} 
\end{equation}
\begin{equation}
\label{cohorelteq}
\raisebox{0.2cm}{\scalemath{1.05}{\boxed{[k]\alpha_t \ }}}
\raisebox{-0.5cm}{
\begin{tikzpicture}[scale=0.8]
	       \draw[red] (1.1,-0.1) -- (2,-1);
	       \draw[blue] (1.5,-0.1) -- (2,-1);
	       \draw[violet] (2.9,-0.1) -- (2,-1);
	       \draw[blue] (1.1,-2) -- (2,-1);
	       \draw[red] (1.5,-2) -- (2,-1);
	       \draw[violet] (2.9,-2) -- (2,-1);
	       \node[rhdot] at (2,-1) {};
	       \node at (2.1,-0.4) {$\ldots$};
	       \node at (2.1,-1.6) {$\ldots$};
	    \end{tikzpicture}}
\raisebox{0.2cm}{\scalemath{1.05}{=-\sum_{i=1}^{k-1} \boxed{ [k-i]} \ }}
\raisebox{-0.4cm}{
	    \begin{tikzpicture}[scale=0.7]
	       \draw[red] (0.8,0.1) -- (2,-1);
	       \draw[violet] (3.2,0.1) -- (2,-1);
	       \draw[blue] (1,-2) -- (2,-1);
	       \draw[violet] (3,-2) -- (2,-1);
	       \draw[red] (2,-1) -- (1.5,-2);
	       \draw[myyellow] (2,-1) -- (2,-0.5);
	       \draw[violet] (2,0.2) -- (2,-0.2);
	       \draw[myyellow] (1.7,0.1) to[out=-90,in=180] (2,-0.5) to[out=0,in=-90] (2.3,0.1);
	       \node at (2,0.35) {$s_i$};
	       \node[pdot] at (2,-0.2) {};
	       \node[rhdot] at (2,-1) {};
	       \node at (2.6,-0.15) {$\ldots$};
	       \node at (1.4,-0.15) {$\ldots$};
	       \node at (2.1,-1.7) {$\ldots$};
	    \end{tikzpicture}} \raisebox{0.2cm}{ \scalemath{1.05}{+\sum_{i=1}^{k}\boxed{[k+1-i]}}}
	    \raisebox{-0.8cm}{
    \begin{tikzpicture}[scale=0.7]
	       \draw[red] (1,-0.1) -- (2,-1);
	       \draw[blue] (1.5,-0.1) -- (2,-1);
	       \draw[violet] (3,-0.1) -- (2,-1);
	       \draw[blue] (0.8,-2.2) -- (2,-1);
	       \draw[violet] (3.2,-2.2) -- (2,-1);
	       \draw[myyellow] (2,-1) -- (2,-1.5);
	       \draw[violet] (2,-2.2) -- (2,-1.8);
	       \draw[myyellow] (1.7,-2.2) to[out=90,in=180] (2,-1.5) to[out=0,in=90] (2.3,-2.2);
	       \node[pdot] at (2,-1.8) {};
	       \node[rhdot] at (2,-1) {};
	       \node at (2.1,-0.3) {$\ldots$};
	       \node at (1.4,-1.95) {$\ldots$};
	       \node at (2.6,-1.95) {$\ldots$};
	       \node at (2,-2.5) {$t_i$};
	    \end{tikzpicture}}
\end{equation}
where the purple lines are either red or blue and the yellow lines will be the corresponding opposite color. 
\end{theorem}
\begin{proof}
Suppose $k$ is odd ($k$ even will proceed similarly except use the blue extvalent morphism instead) so that the bottom right strand above will be blue. Then we claim \cref{cohorelseq} can be obtained by rotating down the left $k-1$ strands of the following equation.
\begin{equation}
\label{cohoreleq1}
\raisebox{0.8cm}{$\boxed{\rho_s} \ $}
    \begin{tikzpicture}[scale=0.6]
	       \draw[red] (0.8,-0.1) to[out=-90,in=-180] (2,-1);
	       \draw[red] (3.2,-0.1) to[out=-90,in=0] (2,-1);
	       \foreach \x in {2} \draw[blue] (\x,-1) -- (\x,-2);
	       \node[rhdot] at (2,-1) {};
	       \node at (2,-0.3) {$\usmt{s}{2k-1}{}$};
	    \end{tikzpicture}\raisebox{0.5cm}{$- \ $}
	\begin{tikzpicture}[scale=0.6]
	       \draw[red] (0.8,-0.1) to[out=-90,in=-180] (2,-1);
	       \draw[red] (3.2,-0.1) to[out=-90,in=0] (2,-1);
	       \foreach \x in {2} \draw[blue] (\x,-1) -- (\x,-2);
	       \node[rhdot] at (2,-1) {};
	       \node at (2,-0.3) {$\usmt{s}{2k-1}{}$};
	    \end{tikzpicture}\raisebox{0.8cm}{$\ \boxed{\rho_s}  $}\raisebox{0.5cm}{$\ =0$}
\end{equation}
Specifically, apply polynomial forcing to move the left $\rho_s$ to the right $k-1$ times. The coefficient when we break the $2m+1$th line starting from the left will be of the form $\partial_t(s(ts)^m(\rho_s))$ and the coefficient when we break the $2m$th line starting from the left will be of the form $\partial_s((ts)^m(\rho_s))$. Similarly, apply polynomial forcing to move the right $\rho_s$ to the left $k$ times. The coefficients will be the same as above, except we have negative signs and we start counting from the right. Using \cref{stsrhos} and \cref{tsrhos} we see that
\begin{align*}
\partial_t(s(ts)^m(\rho_s))&=\frac{s(ts)^m(\rho_s)-(ts)^{m+1}(\rho_s)}{\alpha_t}=[2m+2]  \\
\partial_s((ts)^m(\rho_s))&=\frac{(ts)^m(\rho_s)-s(ts)^{m}(\rho_s)}{\alpha_s} =[2m+1] 
\end{align*}
and now rotating the left $k-1$ strands of the LHS of \cref{cohoreleq1} the only terms that are unbroken are of the form 
\begin{center}
\raisebox{0.5cm}{\scalemath{1.0}{\boxed{(ts)^{(k-1)/2}(\rho_s) }} \ }
\raisebox{0cm}{
\begin{tikzpicture}[scale=0.7]
	       \draw[red] (1.1,-0.1) -- (2,-1);
	       \draw[blue] (1.5,-0.1) -- (2,-1);
	       \draw[red] (2.9,-0.1) -- (2,-1);
	       \draw[blue] (1.1,-2) -- (2,-1);
	       \draw[red] (1.5,-2) -- (2,-1);
	       \draw[blue] (2.9,-2) -- (2,-1);
	       \node[rhdot] at (2,-1) {};
	       \node at (2.1,-0.4) {$\ldots$};
	       \node at (2.1,-1.6) {$\ldots$};
	    \end{tikzpicture}}\raisebox{0.5cm}{\scalemath{1.0}{- \, \boxed{  s(ts)^{(k-1)/2}(\rho_s) }} \ }
\raisebox{0cm}{
\begin{tikzpicture}[scale=0.7]
	       \draw[red] (1.1,-0.1) -- (2,-1);
	       \draw[blue] (1.5,-0.1) -- (2,-1);
	       \draw[red] (2.9,-0.1) -- (2,-1);
	       \draw[blue] (1.1,-2) -- (2,-1);
	       \draw[red] (1.5,-2) -- (2,-1);
	       \draw[blue] (2.9,-2) -- (2,-1);
	       \node[rhdot] at (2,-1) {};
	       \node at (2.1,-0.4) {$\ldots$};
	       \node at (2.1,-1.6) {$\ldots$};
	    \end{tikzpicture}}\raisebox{0.5cm}{\scalemath{1.0}{$= \ $\boxed{ [k]\alpha_s }} }
\raisebox{0cm}{
	   \begin{tikzpicture}[scale=0.7]
	       \draw[red] (1.1,-0.1) -- (2,-1);
	       \draw[blue] (1.5,-0.1) -- (2,-1);
	       \draw[red] (2.9,-0.1) -- (2,-1);
	       \draw[blue] (1.1,-2) -- (2,-1);
	       \draw[red] (1.5,-2) -- (2,-1);
	       \draw[blue] (2.9,-2) -- (2,-1);
	       \node[rhdot] at (2,-1) {};
	       \node at (2.1,-0.4) {$\ldots$};
	       \node at (2.1,-1.6) {$\ldots$};
	    \end{tikzpicture}}
\end{center}
Rearranging and applying \cref{newgenreduct}, \cref{newgenmult} will yield \cref{cohoreleq1}. A similar calculation applies when $k$ is even aka the bottom right strand will be red. \\

Now we will show why \cref{cohoreleq1} is true in $\ext_{R^e}^{1, -2k+2}(K_t, \bs(\usmt{s}{2k-1}{} ))$. Because $\bsbimext(\hf, W_\infty)$ is a supermonoidal category, we can move both boxed $\rho_s$ down in \cref{cohoreleq1} , and since $\rho_s\in (V^*)^t$ we have that
\begin{equation}
\label{rhotslideeq}
    \boxed{\rho_s} \bzero -\bzero \  \boxed{\rho_s}=0 
\end{equation} 
in $\bsbim(\hf, W_\infty)$. By the fully faithful embedding it follows that \cref{rhotslideeq} also holds in $\bsbimext(\hf, W_\infty)$, i.e there are chain homotopies $K_t\to K_t$ giving rise to \cref{rhotslideeq}. \\

\cref{cohorelteq} can also be obtained from \cref{cohoreleq1} by instead rotating the right $k-1$ strands down and then rotating the diagram by $180^\circ$ and applying \cref{omegarotation}.
\end{proof}

\begin{corollary}
For each $k$, adding a red and blue cap anywhere in \cref{cohorelseq} or \cref{cohorelteq} will give the cohomology relation for $k-1$.
\end{corollary}
\begin{proof}
Follows from \cref{cohoreleq1} and \cref{newgenreduct}.
\end{proof}

\begin{example}
For $k=2$ one can use \cref{4extreduct}, \cref{4extreductcor}, and \cref{4extreductcor2} to arrive at the relations
\begin{equation}
\raisebox{0.6cm}{$\ \  \boxed{[2]\alpha_s} \ \ $}
\begin{tikzpicture}[scale=0.6]
	       \draw[red] (1.2,0) -- (2,-1);
	       \draw[blue] (2.8,0) -- (2,-1);
	       \draw[blue] (1.2,-2) -- (2,-1);
	       \draw[red] (2.8,-2) -- (2,-1);
	       \node[rhdot] at (2,-1) {};
	\end{tikzpicture}
    \raisebox{0.6cm}{$\ \ =- [2] \ $}
	\begin{tikzpicture}[scale=0.6]
	       \draw[red] (1.3,0) -- (1.3,-1);
	       \draw[blue] (2.7,0) -- (2,-1);
	       \draw[blue] (1.3,-2) -- (2,-1);
	       \draw[red] (2.7,-2) -- (2.7,-1);
	       \node[bhdot] at (2,-1) {};
	       \node[rdot] at (2.7,-1) {};
	       \node[rdot] at (1.3,-1) {};
	\end{tikzpicture}
    \raisebox{0.6cm}{$\ \ + [2] \ \ $}
    \begin{tikzpicture}[scale=0.6]
	       \draw[red] (1.3,0) -- (1.3,-1);
	       \draw[blue] (2.7,0) -- (2,-1);
	       \draw[blue] (1.3,-2) -- (2,-1);
	       \draw[red] (2.7,-2) -- (2.7,-1);
	       \node[rhdot] at (2.7,-1) {};
	       \node[rdot] at (1.3,-1) {};
	\end{tikzpicture} 
	\raisebox{0.6cm}{$\ \  + \ \ $}
    \begin{tikzpicture}[scale=0.6]
	       \draw[red] (1.3,0) -- (2,-1);
	       \draw[blue] (2.7,0) -- (2.7,-1);
	       \draw[blue] (1.3,-2) -- (1.3,-1);
	       \draw[red] (2.7,-2) -- (2,-1);
	       \node[bhdot] at (2.7,-1) {};
	       \node[bdot] at (1.3,-1) {};
	\end{tikzpicture}\raisebox{0.6cm}{$\ \  - \ \ $}
	\begin{tikzpicture}[scale=0.6]
	       \draw[red] (1.3,0) -- (2,-1);
	       \draw[blue] (2.7,0) -- (2.7,-1);
	       \draw[blue] (1.3,-2) -- (1.3,-1);
	       \draw[red] (2.7,-2) -- (2,-1);
	       \node[bdot] at (2.7,-1) {};
	       \node[bhdot] at (1.3,-1) {};
	\end{tikzpicture}
\end{equation}
\begin{equation}
\label{4extcohorelt}
\raisebox{0.6cm}{$\ \  \boxed{[2]\alpha_t} \ \ $}
\begin{tikzpicture}[scale=0.6]
	       \draw[red] (1.2,0) -- (2,-1);
	       \draw[blue] (2.8,0) -- (2,-1);
	       \draw[blue] (1.2,-2) -- (2,-1);
	       \draw[red] (2.8,-2) -- (2,-1);
	       \node[rhdot] at (2,-1) {};
	\end{tikzpicture}
	\raisebox{0.6cm}{$\ \ = -[2]\ \ $}
    \begin{tikzpicture}[scale=0.6]
	       \draw[red] (1.3,0) -- (2,-1);
	       \draw[blue] (2.7,0) -- (2.7,-1);
	       \draw[blue] (1.3,-2) -- (1.3,-1);
	       \draw[red] (2.7,-2) -- (2,-1);
	       \node[bhdot] at (2.7,-1) {};
	       \node[bdot] at (1.3,-1) {};
	\end{tikzpicture}\raisebox{0.6cm}{$\ \  +[2] \ \ $}
	\begin{tikzpicture}[scale=0.6]
	       \draw[red] (1.3,0) -- (2,-1);
	       \draw[blue] (2.7,0) -- (2.7,-1);
	       \draw[blue] (1.3,-2) -- (1.3,-1);
	       \draw[red] (2.7,-2) -- (2,-1);
	       \node[bdot] at (2.7,-1) {};
	       \node[bdot] at (1.3,-1) {};
	       \node[rhdot] at (2,-1) {};
	\end{tikzpicture}
    \raisebox{0.6cm}{$\ \  + \ $}
	\begin{tikzpicture}[scale=0.6]
	       \draw[red] (1.3,0) -- (1.3,-1);
	       \draw[blue] (2.7,0) -- (2,-1);
	       \draw[blue] (1.3,-2) -- (2,-1);
	       \draw[red] (2.7,-2) -- (2.7,-1);
	       \node[rdot] at (2.7,-1) {};
	       \node[rhdot] at (1.3,-1) {};
	\end{tikzpicture}
    \raisebox{0.6cm}{$\ \ - \ \ $}
    \begin{tikzpicture}[scale=0.6]
	       \draw[red] (1.3,0) -- (1.3,-1);
	       \draw[blue] (2.7,0) -- (2,-1);
	       \draw[blue] (1.3,-2) -- (2,-1);
	       \draw[red] (2.7,-2) -- (2.7,-1);
	       \node[rhdot] at (2.7,-1) {};
	       \node[rdot] at (1.3,-1) {};
	\end{tikzpicture} 
\end{equation}
and the RHS looks similar to the Jones Wenzl projector (see Example 5.10 in Soergel Calculus). Moreover if we cap off appropriately, we can recover the rank 1 cohomology relation (assuming Hochschild jumping and barbell). 
\begin{center}
\raisebox{0.6cm}{$\ \   \ \ \boxed{[2]\alpha_s \alpha_t}$}
	\raisebox{0.4cm}{
	\begin{tikzpicture}[scale=0.6]
	       \draw[red] (2,1.2) -- (2,0);
	       \node[rhdot] at (2,0) {};
	\end{tikzpicture}}\raisebox{0.6cm}{$\ \  - \ \ \boxed{[2]\alpha_s}\dboxed{\alpha_t^\vee}$}
	\raisebox{0.4cm}{
	\begin{tikzpicture}[scale=0.6]
	       \draw[red] (2,1.2) -- (2,0);
	       \node[rdot] at (2,0) {};
	\end{tikzpicture}}
    \raisebox{0.6cm}{$\ \ = \boxed{[2]\alpha_s} \ \ $}
\begin{tikzpicture}[scale=0.6]
	       \draw[red] (1.2,0) -- (2,-1);
	       \draw[blue] (2.8,0) -- (2,-1);
	       \draw[blue] (1.2,-2) -- (2,-1);
	       \draw[red] (2.8,-2) -- (2,-1);
	       \node[rhdot] at (2,-1) {};
	       \node[bdot] at (1.2,-2) {};
	       \node[bdot] at (2.8,0) {};
	       \node[rdot] at (2.8,-2) {};
	\end{tikzpicture}\raisebox{0.6cm}{$\ \  = \ \ -\boxed{[2]\alpha_s}\dboxed{ \alpha_t^\vee}$}
	\raisebox{0.4cm}{
	\begin{tikzpicture}[scale=0.6]
	       \draw[red] (2,1.2) -- (2,0);
	       \node[rdot] at (2,0) {};
	\end{tikzpicture}}\raisebox{0.6cm}{$\ \  +\ \ \boxed{[2]\alpha_t}\dboxed{ \alpha_s^\vee}$}
	\raisebox{0.4cm}{
	\begin{tikzpicture}[scale=0.6]
	       \draw[red] (2,1.2) -- (2,0);
	       \node[rdot] at (2,0) {};
	\end{tikzpicture}}
\end{center}
where we have applied \cref{4extreductcor3} on the LHS. We can then cancel $[2]\alpha_t$ from both sides to arrive at the rank 1 cohomology relation.
\end{example}

\section{Computation of $\ext^{\bullet, \bullet}_{R^e}(R, B_w)$ for $m_{st}=\infty$}
\label{extindecompsect}
We will continue to assume that \cref{assump0}, \cref{assump1}, and \cref{assump2} hold in this section.

\begin{theorem}
\label{indecompcohograd}
Assume all quantum numbers are invertible. Then there is an isomorphism of right $R$ modules. 
\begin{align*}
\ext^{0, \bullet}_{R^e}(R, B_{\smt{t}{k}{}})&\cong R(-k) \\
\ext^{1, \bullet}_{R^e}(R, B_{\smt{t}{k}{}})&\cong R(4-k)\oplus R(k) \\
\ext^{2, \bullet}_{R^e}(R, B_{\smt{t}{k}{}})&\cong R(k+4)
\end{align*}
\end{theorem}
\begin{proof}
Proceed by induction on $k$. $k=1$ follows from  \cref{maincohoiso}. From \cref{maincohoiso} we know that
\begin{align}
\nonumber
\ext^{0, \bullet}_{R^e}(R, \mathrm{BS}(\usmt{t}{k}{}))&\cong  \ker\rhos{s}{k-1}{} (-1) \\
\label{ext1bw}
\ext^{1, \bullet}_{R^e}(R, \mathrm{BS}(\usmt{t}{k}{}))&\cong \ker\rhos{s}{k-1}{} (-1)(4)\oplus \Db(\ker\rhos{s}{k-1}{} (-1) ) \\
\label{ext2bw}
\ext^{2, \bullet}_{R^e}(R, \mathrm{BS}(\usmt{t}{k}{}))&\cong  \Db(\ker\rhos{s}{k-1}{} (-1) )(4)
\end{align}
In other words, $\ext^{\bullet, \bullet}_{R^e}(R, \mathrm{BS}(\usmt{t}{k}{}))$ satisfies the pattern as prescribed by the theorem. Because all quantum numbers are invertible, we have a decomposition 
\begin{equation}
\label{bsdecompeq}
\bs(\usmt{t}{k}{})\cong B_{\smt{t}{k}{}}\bigoplus_{y<\smt{t}{k}{}} B_{y}^{\oplus h_y}
\end{equation}

By Soergel's Hom formula we know that $\ext^{0, \bullet}_{R^e}(R, B_{\smt{t}{k}{}})\cong R(-k)$. Therefore
\[ \ker \rhos{s}{k-1}{}(-1)=\ext^{0, \bullet}_{R^e}(R, \bs(\usmt{t}{k}{})) \cong  R(-k)\bigoplus  \ext^{0, \bullet}_{R^e}(R, B_{y}^{\oplus h_y}) \]
By induction $\ext^{\bullet, \bullet}_{R^e}(R, B_{y}^{\oplus h_y})$ will also satisfy the pattern as prescribed by the theorem and so subtracting $\ext^{\bullet, \bullet}_{R^e}(R, B_{y}^{\oplus h_y})$ from both sides of \cref{ext1bw} and \cref{ext2bw} we complete the induction.
\end{proof}

Recall from \cref{maincohoiso} that we have the isomorphism
$$ \ext^{1, \bullet}_{R^e}(B_t, \mathrm{BS}(\underline{w}))\cong
\ker\rhosw\underline{\rho_t t(\rho_t)} (-1)\oplus  \widetilde{\Db(\ker\rhosw )}\underline{\rho_s}(-1) $$

\begin{lemma}
\label{ext1kerbasis}
The set of all morphisms of the form below
\begin{equation}
\raisebox{-1cm}{\begin{tikzpicture}
\node (pq1) [dashed, shape=rectangle,
                     draw=black, semithick,
                     text width=0.1\linewidth,
                     align=center] {$\mathrm{L}_{  \usmt{s}{k-1}{}, \underline{f} }$};
                     \draw[blue] (0,-1) -- (0,-0.6);
 \node[bhdot] at (0,-0.6) {};
\end{tikzpicture}} \  \ \textnormal{where} \ r(\underline{f})=\id  \qquad \qquad \quad  \raisebox{-1cm}{\ \  \begin{tikzpicture}
\draw[blue] (2.2,-1) to[out=-90,in=90] (2.2,-1.7) ;
\node[bhdot] at (2.2, -1.35) {};
\node at (2.2,-0.65) [dashed, shape=rectangle,
                     draw=black, semithick,
                     text width=0.1\linewidth,
                     align=center] {$\mathrm{L}_{  \usmt{s}{k-1}{}, \underline{f} }$};
\end{tikzpicture}} \ \  \textnormal{where} \  r(\underline{f})=t 
\end{equation}
give a right $R$ basis for the $\ker\rhosw\underline{\rho_t t(\rho_t)} (-1)$ part of $\ext^{1, \bullet}_{R^e}(B_t, \mathrm{BS}(\underline{w}))$. 
\end{lemma}
\begin{proof}
By definition, the $\ker\rhosw\underline{\rho_t t(\rho_t)} (-1)$ part of $\ext^{1, \bullet}_{R^e}(B_t, \mathrm{BS}(\underline{w}))$ in \cref{maincohoiso} consists of elements of the form ${}_{tt}^{a}\psi_s^0$ where $a\in \ker\rhosw\subset \bsw$ and thus a basis is given by $\set{{}_{tt}^{b}\psi_s^0}_{b\in \mathscr{B}}$ where $\mathscr{B}$ is a right $R$ basis for $ \ker\rhosw\subset \bsw$. Because 
\[  \bhzero= {}_{tt}^{1\otimes 1}\psi_s^0\]
and by definition $\ll{f}$ is the image of $c_{bot}$ and so it follows that the morphisms in the lemma are precisely given by ${}_{tt}^{\ll{f}}\psi_s^0$ and thus by \cref{keycor} we are done.
\end{proof}

\begin{lemma}
\label{ext1dkerbasis}
The set of all morphisms of the form below
\begin{equation*}
 \raisebox{-0.8cm}{\ \  
\begin{tikzpicture}
\draw[blue]  (2,-2.2) to[out=90, in=-90] (2,-1.6);
\draw[violet] (1.4, -1.1) to[out=-90, in=-180] (2, -1.6);
\draw[violet] (2.6, -1.1) to[out=-90, in=0] (2, -1.6);
\node at (2, -1.3) {$\scalemath{0.8}{r(\underline{f})}$};
\node[rhdot] at (2, -1.6) {};
\node at (2.1,-0.65) [dashed, shape=rectangle,
                     draw=black, semithick,
                     text width=0.1\linewidth,
                     align=center] {$\mathrm{L}_{  \usmt{s}{k-1}{}, \underline{f} }$};
\end{tikzpicture}}\quad r(\underline{f}) \textnormal{ can be anything}
\end{equation*}
span the $\widetilde{\Db(\ker\rhosw )}\underline{\rho_s}(-1)$ part of $\ext^{1, \bullet}_{R^e}(B_t, \mathrm{BS}(\underline{w}))$ as a right $R$ module. 
\end{lemma}
\begin{proof}
Similar to \cref{ext1kerbasis}. See the remark after \cref{maincohoiso} as well.
\end{proof}

However, note that when $r(\underline{f})\in\set{s,t, ts, st, tst}$ the first/bottom morphism $\Phi_t^{r(\underline{f})}$ in \cref{ext1dkerbasis} technically isn't defined! In this case, there is more than one possible choice of $v_f$ such that ${}_{tt}^{v_f}\psi_s^{1(r(\underline{f}))}$ is a cocycle. However the proof of \cref{ext1dkerbasis} only really needs that the bottom/first morphism sends $\underline{\rho_s}\boxtimes 1\otimes 1$ to  $\underline{1}\boxtimes 1\otimes 1(r(\underline{f}))$. Thus when $r(\underline{f})$ is in the set above we can define
\[  \Phi_t^{r(\underline{f})}:= \textnormal{ add dot morphisms to } \Phi_t^{\underline{w}} \textnormal{ until the top boundary is }r(\underline{f})  \]
For example, 
\begin{center}
\begin{tikzpicture}[scale=0.6]
	       \draw[red] (2,0.1) -- (2,-1);
	       \draw[blue] (2,-2) -- (2,-1);
	       \node[rhdot] at (2,-1) {};
	\end{tikzpicture} \raisebox{2.5ex}{ \ := \ }
\begin{tikzpicture}[scale=0.6]
	       \draw[red] (1.1,0.1) to[out=-90, in=180] (2,-1);
	       \draw[blue] (2,0.1) -- (2,-1);
	       \draw[red] (2.9,0.1) to[out=-90, in=0](2,-1);
	       \draw[blue] (2,-2) -- (2,-1);
	       \node[rhdot] at (2,-1) {};
	       \node[rdot] at (2.9,0.1) {};
	       \node[bdot] at (2,0.1) {};
	\end{tikzpicture} 
	\raisebox{2.5ex}{ \ = \ }
\begin{tikzpicture}[scale=0.6]
	       \draw[red] (1.1,0.1) to[out=-90, in=180] (2,-1);
	       \draw[blue] (2,0.1) -- (2,-1);
	       \draw[red] (2.9,0.1) to[out=-90, in=0](2,-1);
	       \draw[blue] (2,-2) -- (2,-1);
	       \node[rhdot] at (2,-1) {};
	       \node[rdot] at (1.1,0.1) {};
	       \node[bdot] at (2,0.1) {};
	\end{tikzpicture} 
\end{center}
One can check that this definition is well defined using \cref{4extreductcor}, \cref{4extreductcor2}, etc.

\begin{definition}
A pitchfork is a morphism depicted on the left below. A generalized pitchfork is looks like a pitchfork but the middle can have any number of dot morphisms. Such an example is depicted on the right below. \vspace{-4ex}
\begin{center}
    \includegraphics[scale=0.45]{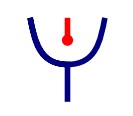}\qquad \qquad 
    \includegraphics[scale=0.4]{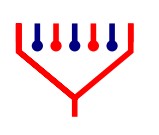}
\end{center}
\end{definition}

\begin{lemma}
\label{gpitchtopitch}
Generalized pitchforks can be written as the sum of morphisms that end with pitchforks somewhere on the top.
\end{lemma}
\begin{proof}
This was shown in the proof of Claim 5.26 in \cite{DC}.
\end{proof}

\begin{theorem}
\label{algextindecomp}
Assume all quantum numbers are invertible. Then $\forall k\ge 1$we have an isomorphism of right $R-$modules
\begin{align*}
\ext^{0, \bullet}_{R^e}(R, B_{\smt{t}{k}{}})&\cong  \overset{
\begin{tikzpicture}
        \node at (-0.5,0) [dashed, shape=rectangle,
                     draw=black, semithick,
                     text width=7ex,
                     align=center] {$\jw_{\usmt{t}{k}{} }$};
\end{tikzpicture}
} { \bunit \runit \ldots \vunit }\  R \\
\ext^{1, \bullet}_{R^e}(R, B_{\smt{t}{k}{}})&\cong \overset{
\begin{tikzpicture}
        \node at (-0.5,0) [dashed, shape=rectangle,
                     draw=black, semithick,
                     text width=7ex,
                     align=center] {$\jw_{\usmt{t}{k}{} }$};
\end{tikzpicture}
}  {\bhunit \runit \ldots \vunit } \ R\oplus   \ 
\raisebox{-0.2cm}{\begin{tikzpicture}[scale=0.6]
	       \draw[red] (1.1,-0.1) to[out=-90,in=-180] (2,-1);
	       \draw[violet] (2.9,-0.1) to[out=-90,in=0] (2,-1);
	       \draw[blue] (0.9,-0.1) to[out=-90,in=-180] (1.45,-1.35) to[out=0, in=-90] (2,-1);
	       \node[rhdot] at (2,-1) {};
	       \node at (2,-0.4) {$\scalemath{0.9}{\usmt{s}{k-1}{}}$};
        \node at (1.9,0.55) [dashed, shape=rectangle,
                     draw=black, semithick,
                     text width=7ex,
                     align=center] {$\jw_{\usmt{t}{k}{} }$};
	    \end{tikzpicture}} \ R\\
\ext^{2, \bullet}_{R^e}(R, B_{\smt{t}{k}{}})&\cong \  \raisebox{-0.2cm}{\begin{tikzpicture}[scale=0.6]
	       \draw[red] (1.1,-0.1) to[out=-90,in=-180] (2,-1);
	       \draw[violet] (2.9,-0.1) to[out=-90,in=0] (2,-1);
	       \draw[blue] (0.9,-0.1) to[out=-90,in=-180] (1.45,-1.35) to[out=0, in=-90] (2,-1);
	       \node[rhdot] at (2,-1) {};
	       \node[bhdot] at (1.45,-1.35) {};
	       \node at (2,-0.4) {$\scalemath{0.9}{\usmt{s}{k-1}{}}$};
        \node at (2,-0.4) {$\scalemath{0.9}{\usmt{s}{k-1}{}}$};
        \node at (1.9,0.55) [dashed, shape=rectangle,
                     draw=black, semithick,
                     text width=7ex,
                     align=center] {$\jw_{\usmt{t}{k}{} }$};
	    \end{tikzpicture}} \ R
\end{align*}
\end{theorem}
\begin{proof}
$\ext^{0, \bullet}_{R^e}(R, B_{\smt{t}{k}{}})$ was shown in the diagrammatic category in \cite{DC} Claim 5.26 and because of the equivalence, it also applies in the algebraic category. By \cite[Theorem 9.22]{EMTW20} $\jw_{\usmt{t}{k}{}}$ is the projector for $ B_{\smt{t}{k}{}}$ inside $\bs(\usmt{t}{k}{})$, by the definition of the Karoubian completion we have that 
\[ \ext^{i, \bullet}_{R^e}(R, B_{\smt{t}{k}{}})= \jw_{\usmt{t}{k}{} }\circ \ext^{i, \bullet}_{R^e}(R, \bs(\usmt{t}{k}{})) \]
Applying the adjunction
\[ \ext^{1, \bullet}_{R^e}(B_t, \bs(\usmt{s}{k-1}{})) \cong  \ext^{1, \bullet}_{R^e}(R, \bs(\usmt{t}{k}{})) \]
to \cref{ext1kerbasis} shows that the $``\ker \rho_s^e(\usmt{s}{k-1}{})"$ part of $\ext^{1, \bullet}_{R^e}(R, \bs(\usmt{t}{k}{}))$ has a right $R$ basis given by 
\begin{equation}
\label{kerindecompeq}
\bhunit \ \  \raisebox{-0.4cm}{\begin{tikzpicture}
\node (pq1) [dashed, shape=rectangle,
                     draw=black, semithick,
                     text width=0.1\linewidth,
                     align=center] {$\mathrm{L}_{  \usmt{s}{k-1}{}, \underline{f} }$};
\end{tikzpicture}} \  \ \textnormal{where} \ r(\underline{f})=\id \quad   \raisebox{-0.8cm}{\ \  \begin{tikzpicture}
\draw[blue] (0.9,-0.3) to[out=-90,in=-180] (1.45,-1.35) to[out=0, in=-90] (2,-1);
\node[bhdot] at (1.45, -1.35) {};
\node at (2.2,-0.65) [dashed, shape=rectangle,
                     draw=black, semithick,
                     text width=0.1\linewidth,
                     align=center] {$\mathrm{L}_{  \usmt{s}{k-1}{}, \underline{f} }$};
\end{tikzpicture}} \ \  \textnormal{where} \  r(\underline{f})=t
\end{equation}
We want to show that all such morphisms are in the $R$ span of  \bhunit \runit \ldots \vunit \ after postcomposing by $ \jw_{\usmt{t}{k}{}}$. \cref{maincohoiso} shows that $R$ acts freely on $\ext^{1, \bullet}_{R^e}(R, \bs(\usmt{t}{k}{}))$ and thus it suffices to show that 
\begin{center}
$ \jw_{\usmt{t}{k}{}}\circ \boxed{\alpha_t}$\bhunit \ \  \raisebox{-0.4cm}{\begin{tikzpicture}
\node (pq1) [dashed, shape=rectangle,
                     draw=black, semithick,
                     text width=0.1\linewidth,
                     align=center] {$\mathrm{L}_{  \usmt{s}{k-1}{}, \underline{f} }$};
\end{tikzpicture}} \, , \ \ 
$ \jw_{\usmt{t}{k}{}}\circ \boxed{\alpha_t}$\raisebox{-0.8cm}{\ \  \begin{tikzpicture}
\draw[blue] (0.9,-0.3) to[out=-90,in=-180] (1.45,-1.35) to[out=0, in=-90] (2,-1);
\node[bhdot] at (1.45, -1.35) {};
\node at (2.2,-0.65) [dashed, shape=rectangle,
                     draw=black, semithick,
                     text width=0.1\linewidth,
                     align=center] {$\mathrm{L}_{  \usmt{s}{k-1}{}, \underline{f} }$};
\end{tikzpicture}} \  $\in  \jw_{\usmt{t}{k}{}}\circ $\bhunit \runit \ldots \vunit \ $R$
\end{center}
But applying the $1-$color cohomology relation $\boxed{\alpha_t}$\bhunit=$\dboxed{\alpha_t^\vee} \bunit$, we see that all the morphisms above factor as 
\[   \jw_{\usmt{t}{k}{}}\circ \mathrm{L}_{  \usmt{t}{k}{}, \underline{f^\prime} }  \circ \alpha_t^\vee, \quad \quad r(\underline{f^\prime})=\id \]
Because $\set{\mathrm{L}_{  \usmt{t}{k}{}, \underline{f^\prime} }}_{r(\underline{f^\prime})=\id}$ is a right $R$ basis for $\ext^{0, \bullet}_{R^e}(R, \bs(\usmt{t}{k}{})) $, it follows that $\set{ \jw_{\usmt{t}{k}{}}\circ \mathrm{L}_{  \usmt{t}{k}{}, \underline{f^\prime} }}_{r(\underline{f^\prime})=\id}$ is spanned by $ \jw_{\usmt{t}{k}{}}\circ $ \bunit \runit \ldots \vunit \  as a right $R$ module from the previous calculation of $\ext^{0, \bullet}_{R^e}(R, B_{\smt{t}{k}{}})$. Thus we see that the $``\ker \rho_s^e(\usmt{s}{k-1}{})"$ part of $\ext^{1, \bullet}_{R^e}(R, \bs(\usmt{t}{k}{}))$ is spanned by $ \jw_{\usmt{t}{k}{}}\circ $ \bhunit \runit \ldots \vunit \  as a right $R$ module.\\

Similarily \cref{ext1dkerbasis} shows that the $``\Db(\ker \rho_s^e(\usmt{s}{k-1}{}))"$ part of $\ext^{1, \bullet}_{R^e}(R, \bs(\usmt{t}{k}{}))$ is generated as a right $R$ module by elements of the form
\begin{equation}
\label{Dindecompeq}
\Pi(\underline{f})=
    \raisebox{-0.8cm}{\ \  
\begin{tikzpicture}
\draw[blue] (0.9,-0.3) to[out=-90,in=-180] (1.55,-1.95) to[out=0, in=-90] (2,-1.6);
\draw[violet] (1.4, -1.1) to[out=-90, in=-180] (2, -1.6);
\draw[violet] (2.6, -1.1) to[out=-90, in=0] (2, -1.6);
\node at (2, -1.3) {$\scalemath{0.8}{r(\underline{f})}$};
\node[rhdot] at (2, -1.6) {};
\node at (2.2,-0.65) [dashed, shape=rectangle,
                     draw=black, semithick,
                     text width=0.1\linewidth,
                     align=center] {$\mathrm{L}_{  \usmt{s}{k-1}{}, \underline{f} }$};
\end{tikzpicture}}\quad r(\underline{f}) \textnormal{ can be anything}
\end{equation}
We want to show that all such morphisms are in the right $R$ span of \raisebox{-0.2cm}{\begin{tikzpicture}[scale=0.6]
	       \draw[red] (1.1,-0.1) to[out=-90,in=-180] (2,-1);
	       \draw[violet] (2.9,-0.1) to[out=-90,in=0] (2,-1);
	       \draw[blue] (0.9,-0.1) to[out=-90,in=-180] (1.45,-1.35) to[out=0, in=-90] (2,-1);
	       \node[rhdot] at (2,-1) {};
	       \node at (2,-0.4) {$\scalemath{0.9}{\usmt{s}{k-1}{}}$};
	    \end{tikzpicture}} 
after postcomposing by $ \jw_{\usmt{t}{k}{}}$. \\

\textbf{Step 1:} First notice that if $\mathrm{L}_{  \usmt{s}{k-1}{}, \underline{f} }$ contains a pitchfork (see \cite{DC} Section 5.3.4 for a picture), then $ \jw_{\usmt{t}{k}{}}\circ \Pi(\underline{f})=0$ and we are done. In fact, by \cref{gpitchtopitch} the same is true if $\mathrm{L}_{  \usmt{s}{k-1}{}, \underline{f} }$ contains a generalized pitchfork. \\
\textbf{Step 2:} \underline{In the affine case}, we claim that all light leaves without generalized pitchforks must be of the form 
\begin{equation}
\label{step2eq}
    \raisebox{0.25cm}{\runit \ \  $\bullet \bullet \bullet $ \vunit} \yzero \ \  \Diamond \Diamond \Diamond \vzero \quad \textnormal{or} \quad \raisebox{0.25cm}{\runit \ \  $\bullet \bullet \bullet $ \vunit} \yzero \ \ \Diamond \Diamond \Diamond \vzero  \raisebox{0.25cm}{\, \yunit}
\end{equation}
where $\bullet \bullet \bullet $ means the morphism consists of \underline{only dot morphisms} and $\Diamond \Diamond \Diamond$ means the morphism consists of \\
\underline{only identity morphisms} (aka straight lines). This is because in the light leaf algorithm, \underline{in the affine case}, we generate generalized pitchforks whenever a decoration is of the form $D0$ or $D1$. As a result, we only need to consider the light leaves that are made up of only dots ($U0$) and lines ($U1$). However, suppose that at step $i$ of our stroll, the corresponding decorations ends in $U1 U0$. It will follow the decoration at the next step has to be either $D0$ or $D1$ but this isn't allowed, or $i=k-2$ and we end the algorithm at the next step. In other words, we must have all lines after the appearance of the first line or all lines with a dot at the end, which is exactly what is depicted in \cref{step2eq}.\\
\textbf{Step 3:} For $\mathrm{L}_{  \usmt{s}{k-1}{}, \underline{f} }$ in the form in \cref{step2eq}, we claim that $\jw_{\usmt{t}{k}{}}\circ \Pi(\underline{f})$  is in the right $R$ span of
\begin{equation}
   \jw_{\usmt{t}{k}{}}\circ  \raisebox{-0.8cm}{\ \  
\begin{tikzpicture}
\draw[blue] (0.9,-0.75) to[out=-90,in=-180] (1.55,-1.95) to[out=0, in=-90] (2,-1.6);
\draw[violet] (1.4, -0.75) to[out=-90, in=-180] (2, -1.6);
\draw[violet] (2.6, -0.75) to[out=-90, in=0] (2, -1.6);
\node at (2, -1) {$\scalemath{1}{ \Diamond \Diamond \Diamond}$};
\node[rhdot] at (2, -1.6) {};
\end{tikzpicture}}
\raisebox{0.1cm}{\, \yunit $\ \bullet \bullet \bullet  $ \vunit }
\end{equation}
WLOG we will work with the first expression in \cref{step2eq}. Suppose that the first line in \cref{step2eq}(The yellow line)is colored blue, then $  \jw_{\usmt{t}{k}{}}\circ \Pi(\underline{f})$ will be equal to
\begin{equation}
   \jw_{\usmt{t}{k}{}}\circ  \raisebox{-0.8cm}{\ \  
\begin{tikzpicture}
\draw[blue] (0.9,-0.75) to[out=-90,in=-180] (1.55,-1.95) to[out=0, in=-90] (2,-1.6);
\draw[violet] (3, -0.75) to[out=-90, in=0] (2, -1.6);
\draw[blue] (2,-1.6)--(2,-0.75);
\node at (1.5, -1) {$\scalemath{1}{ \bullet \bullet \bullet  }$};
\node[rhdot] at (2, -1.6) {};
\node at (2.45,-1) {$\ \Diamond \Diamond \Diamond $ };
\end{tikzpicture}} \ =    \jw_{\usmt{t}{k}{}}\circ  \raisebox{-0.8cm}{\ \  
\begin{tikzpicture}
\draw[blue] (0.9,-0.75)  to[out=-90, in=180] (2,-1.4);
\draw[violet] (3, -0.75) to[out=-90, in=0] (2, -1.8);
\draw[blue] (2,-1.8)--(2,-0.75);
\node at (1.5, -1) {$\scalemath{1}{ \bullet \bullet \bullet  }$};
\node[rhdot] at (2, -1.8) {};
\node at (2.45,-1) {$\ \Diamond \Diamond \Diamond $ };
\end{tikzpicture}}=0
\end{equation}
where we have used \cref{omegarotation} and \cref{newgenabsorb}. Now suppose that the first line is colored red. If $\mathrm{L}_{  \usmt{s}{k-1}{}, \underline{f} }$ starts with a red line we are done. Otherwise $\mathrm{L}_{  \usmt{s}{k-1}{}, \underline{f} }$ starts with a red dot and so applying fusion and then polynomial forcing and eliminating all diagrams with a generalized pitchforks we see that

\begin{align*}
   \jw_{\usmt{t}{k}{}}\circ  \raisebox{-1cm}{\ \  
\begin{tikzpicture}[scale=1.2]
\draw[blue] (0.8,-0.75) to[out=-90,in=-180] (1.55,-1.95) to[out=0, in=-90] (2,-1.6);
\draw[red, thick] (1,-0.75)--(1, -1.1);
\node[rdot] at (1,-1.1) {};
\draw[violet] (3, -0.75) to[out=-90, in=0] (2, -1.6);
\draw[red] (2,-1.6)--(2,-0.75);
\node at (1.5, -0.9) {$\scalemath{1}{ \bullet \bullet \bullet  }$};
\node[rhdot] at (2, -1.6) {};
\node at (2.45,-1) {$\ \Diamond \Diamond \Diamond $ };
\end{tikzpicture}} \ &=    \jw_{\usmt{t}{k}{}}\circ  \raisebox{-1.2cm}{\ \ 
\begin{tikzpicture}[scale=1.2]
\draw[blue] (0.8,-0.75) to[out=-90,in=-180] (1.55,-2.15) to[out=0, in=-90] (2,-1.8);
\node at (1.7,-1.25) {$\scalemath{0.8}{\boxed{\rho_s}}$};
\draw[red] (0.95,-0.75) to[out=-90, in=180] (2,-1.6);
\draw[violet] (3, -0.75) to[out=-90, in=0] (2, -1.8);
\draw[red] (2,-1.8)--(2,-0.75);
\node at (1.5, -0.9) {$\scalemath{1}{ \bullet \bullet \bullet  }$};
\node[rhdot] at (2, -1.8) {};
\node at (2.45,-1) {$\ \Diamond \Diamond \Diamond $ };
\end{tikzpicture}}-\jw_{\usmt{t}{k}{}}\circ
\raisebox{-1.2cm}{\ \ 
\begin{tikzpicture}[scale=1.2]
\draw[blue] (0.8,-0.75) to[out=-90,in=-180] (1.55,-2.15) to[out=0, in=-90] (2,-1.8);
\node at (1.5,-1.7) {$\scalemath{0.8}{\boxed{s(\rho_s)}}$};
\draw[red] (0.95,-0.75) to[out=-90, in=180] (2,-1.4);
\draw[violet] (3, -0.75) to[out=-90, in=0] (2, -1.8);
\draw[red] (2,-1.8)--(2,-0.75);
\node at (1.5, -0.9) {$\scalemath{1}{ \bullet \bullet \bullet  }$};
\node[rhdot] at (2, -1.8) {};
\node at (2.45,-1) {$\ \Diamond \Diamond \Diamond $ };
\end{tikzpicture}} \\
&= \jw_{\usmt{t}{k}{}}\circ  \raisebox{-1.2cm}{\ \ 
\begin{tikzpicture}[scale=1.2]
\draw[blue] (0.8,-0.75) to[out=-90,in=-180] (1.55,-2.15) to[out=0, in=-90] (2,-1.8);
\draw[red] (0.95,-0.75) to[out=-90, in=180] (2,-1.8);
\draw[violet] (3, -0.75) to[out=-90, in=0] (2, -1.8);
\draw[red, thick] (2,-0.75)--(2, -1.1);
\node[rdot] at (2,-1.1) {};
\node at (1.5, -0.9) {$\scalemath{1}{ \bullet \bullet \bullet  }$};
\node[rhdot] at (2, -1.8) {};
\node at (2.45,-1) {$\ \Diamond \Diamond \Diamond $ };
\end{tikzpicture}}
\end{align*}
But the diagram above is essentially the same as what we started with except we have one more line to the left of $\bullet \bullet \bullet $. We can therefore repeat this process until all the lines $\Diamond \Diamond \Diamond$ are to the left of $\bullet \bullet \bullet $ as desired. \\
\textbf{Step 4: Cohomology Reduction} The previous step shows it suffices to prove morphisms of the form
\begin{center}
$\Pi( (\overbrace{1,\ldots, 1}^{j \  1's}, 0, \ldots, 0 ))=$
     \raisebox{-0.8cm}{\ \  
\begin{tikzpicture}
\draw[blue] (0.9,-0.85) to[out=-90,in=-180] (1.55,-1.95) to[out=0, in=-90] (2,-1.6);
\draw[red] (1.4, -0.85) to[out=-90, in=-180] (2, -1.6);
\draw[violet] (2.6, -0.85) to[out=-90, in=0] (2, -1.6);
\node at (2, -1.2) {$\usmt{s}{j}{}$};
\node[rhdot] at (2, -1.6) {};
\end{tikzpicture}}\yunit $\ \cdots$ \vunit
\end{center}
are in the right $R$ span of \raisebox{-0.2cm}{\begin{tikzpicture}[scale=0.6]
	       \draw[violet] (1.1,-0.1) to[out=-90,in=-180] (2,-1);
	       \draw[violet] (2.9,-0.1) to[out=-90,in=0] (2,-1);
	       \draw[blue] (0.9,-0.1) to[out=-90,in=-180] (1.45,-1.35) to[out=0, in=-90] (2,-1);
	       \node[rhdot] at (2,-1) {};
	       \node at (2,-0.4) {$\scalemath{0.9}{\usmt{s}{k-1}{}}$};
	    \end{tikzpicture}} 
after applying $ \jw_{\usmt{t}{k}{}}$. Note that \cref{cohoreleq1} will tell us 
\begin{equation}
\label{jeq}
     \raisebox{-0.8cm}{\ \  
\begin{tikzpicture}
\draw[blue] (0.9,-0.85) to[out=-90,in=-180] (1.55,-1.95) to[out=0, in=-90] (2,-1.6);
\draw[red] (1.4, -0.85) to[out=-90, in=-180] (2, -1.6);
\draw[violet] (2.6, -0.85) to[out=-90, in=0] (2, -1.6);
\node at (2, -1.2) {$\usmt{s}{j}{}$};
\node[rhdot] at (2, -1.6) {};
\end{tikzpicture}}\yunit = \textnormal{(Morphisms with a pitchfork)}-\raisebox{-0.8cm}{\ \  
\begin{tikzpicture}
\draw[blue] (0.9,-0.85) to[out=-90,in=-180] (1.55,-1.95) to[out=0, in=-90] (2,-1.6);
\draw[red] (1.4, -0.85) to[out=-90, in=-180] (2, -1.6);
\draw[violet] (2.6, -0.85) to[out=-90, in=0] (2, -1.6);
\draw[myyellow, thick] (2, -1.6) to[out=0, in=-90] (2.8, -0.85);
\node at (2, -1.2) {$\usmt{s}{j}{}$};
\node[rhdot] at (2, -1.6) {};
\end{tikzpicture}}\ \boxed{\usmt{s}{j}{}(\rho_s)-\rho_s }
\end{equation}
in cohomology. As before, the morphisms with a pitchfork go to 0 after applying $ \jw_{\usmt{t}{k}{} }$ and so the LHS of \cref{jeq} is in the right $R$ span of $\Pi( (\overbrace{1,\ldots, 1}^{j+1 \  1's}, 0, \ldots, 0 ))$. We can keep repeating this until $j=k-2$ which means everything is in the right $R$ span of \raisebox{-0.2cm}{\begin{tikzpicture}[scale=0.6]
	       \draw[violet] (1.1,-0.1) to[out=-90,in=-180] (2,-1);
	       \draw[violet] (2.9,-0.1) to[out=-90,in=0] (2,-1);
	       \draw[blue] (0.9,-0.1) to[out=-90,in=-180] (1.45,-1.35) to[out=0, in=-90] (2,-1);
	       \node[rhdot] at (2,-1) {};
	       \node at (2,-0.4) {$\scalemath{0.9}{\usmt{s}{k-1}{}}$};
	    \end{tikzpicture}} as desired.\\
	    
We have just shown that $\set{   \jw_{\usmt{t}{k}{} }\circ \bhunit \runit \ldots \vunit, \jw_{\usmt{t}{k}{} }\circ \raisebox{-0.2cm}{\begin{tikzpicture}[scale=0.6]
	       \draw[red] (1.1,-0.1) to[out=-90,in=-180] (2,-1);
	       \draw[violet] (2.9,-0.1) to[out=-90,in=0] (2,-1);
	       \draw[blue] (0.9,-0.1) to[out=-90,in=-180] (1.45,-1.35) to[out=0, in=-90] (2,-1);
	       \node[rhdot] at (2,-1) {};
	       \node at (2,-0.4) {$\scalemath{0.9}{\usmt{s}{k-1}{}}$};
	    \end{tikzpicture}} }$ spans $\ext^{1, \bullet}_{R^e}(R, B_{\smt{t}{k}{}})$. However, since they also have the correct degrees as specified in \cref{indecompcohograd}, they will also be linearly independent over $R$, and thus give a basis for $\ext^{1, \bullet}_{R^e}(R, B_{\smt{t}{k}{}})$.\\
	    
Similar reasoning works for the computation for $\ext^{2, \bullet}_{R^e}(R, B_{\smt{t}{k}{}})$; \cref{maincohoiso} shows that $\ext^{2, \bullet}_{R^e}(R, \bs(\usmt{t}{k}{}))$ is generated as a right $R$ module by elements of the form
\begin{equation}
\label{ext2lightleaf}
\Pi(\underline{f})=
    \raisebox{-0.8cm}{\ \  
\begin{tikzpicture}
\draw[blue] (0.9,-0.3) to[out=-90,in=-180] (1.55,-1.95) to[out=0, in=-90] (2,-1.6);
\draw[violet] (1.4, -1.1) to[out=-90, in=-180] (2, -1.6);
\draw[violet] (2.6, -1.1) to[out=-90, in=0] (2, -1.6);
\node at (2, -1.3) {$\scalemath{0.8}{r(\underline{f})}$};
\node[rhdot] at (2, -1.6) {};
\node[bhdot] at (1.55, -1.95) {};
\node at (2.2,-0.65) [dashed, shape=rectangle,
                     draw=black, semithick,
                     text width=0.1\linewidth,
                     align=center] {$\mathrm{L}_{  \usmt{s}{k-1}{}, \underline{f} }$};
\end{tikzpicture}}\quad r(\underline{f}) \textnormal{ can be anything}
\end{equation}
so we can proceed similarly as above. 
\end{proof}

\section{Ext Dihedral  Diagrammatics: $m_{st}=\infty$ }
\label{diagraminftysect}

For each color $s$ and $t$ let $\Dext_s=\Dext(\hf, S_2=\abrac{s})$ be the strict $\kb-$linear supermonoidal category as defined in \cref{A_1diagdef}.

\begin{definition}
\label{affinedef}
Let $\Dext_\infty:=\Dext(\hf, W_\infty)$ be the strict $\kb$ linear supermonoidal category associated to a realization $\hf$ of the infinite dihedral group $(W_\infty, S)$, where $S=\set{s,t}$ is the set of simple reflections, satisfying \cref{assump1} and \cref{assump2} defined as follows. 
\begin{itemize}
    \item \textbf{Objects} of $\Dext_\infty$ are expressions $\underline{w}=(s_1, \ldots, s_n)$ with $s_i\in S$ where the monoidal structure is given by concatenation. The color red will correspond to $s$ and the color blue will correspond to $t$.
    \item \textbf{Morphism spaces} in $\Dext_\infty$ are bigraded $\kb$ modules. For a morphism $\alpha$ homogeneous of total degree $(\ell, n)$, $\ell$ will be the \underline{cohomological degree} while $n$ will be the \underline{internal or Soergel degree}. Let $|\alpha|=\ell$ the cohomological degree. $\Dext_\infty$ will then be supermonoidal for the cohomological grading. Specifically, $\otimes$ will satisfy the following \underline{super exchange law}
    \[ (h\otimes k)\circ (f\otimes g)=(-1)^{|k||f|} (h\circ f)\otimes (k\circ g) \]
    $\hom_{\Dext_\infty}(\underline{v}, \underline{w})$ will be the free $\kb$ module generated by horizontally and vertically concatenating colored graphs built from certain generating morphisms, such that the bottom and top boundaries are $\underline{v}$ and $\underline{w}$\footnote{All generating morphisms will be properly embedded in the planar strip $\mathbb{R} \times [0,1]$ meaning each edge ends in a vertex or in the boundary of the strip. The bottom(top) boundary of the morphism will then be the intersection with $\mathbb{R}\times \set{0}$ ($\mathbb{R} \times \set{1}$). }, respectively. The generating morphisms will be: For the color red we have
    \begin{center}
        \begin{tabular}{c|c|c|c|c|c}
            generator &  \runit  & \raisebox{1ex}{\rcounit} & \rmult & \rcomult & \rhzero  \\[1ex]
            name & Startdot & Enddot & Merge & Split & (bivalent) Hochschild dot \\[1ex]
            bidegree& $(0,1)$  & $(0,1)$ & $(0, -1)$ & $(0,-1)$ & $(1, -4)$
        \end{tabular}
    \end{center}
    and likewise for the color blue. In addition we will have the generators
    \begin{center}
        \begin{tabular}{c|c|c|c|c}
            generator &  $\boxed{f}$  & $\dboxed{\xi }$ & \raisebox{-2ex}{\begin{tikzpicture}[scale=0.6]
	       \draw[red] (1.1,-0.1) -- (2,-1);
	       \draw[blue] (1.5,-0.1) -- (2,-1);
	       \draw[violet] (2.9,-0.1) -- (2,-1);
	       \draw[blue] (1.1,-2) -- (2,-1);
	       \draw[red] (1.5,-2) -- (2,-1);
	       \draw[violet] (2.9,-2) -- (2,-1);
	       \node at (2.1,-0.4) {$\scalemath{0.8}{\ldots}$};
	       \node at (2.1,-1.6) {$\scalemath{0.8}{\ldots}$};
	       \node[rhdot, minimum size=3.4mm] at (2,-1) {};
	       \node at (2,-1) {$\scalemath{0.7}{2k}$};
	    \end{tikzpicture}}  & \raisebox{-2ex}{\begin{tikzpicture}[scale=0.6]
	       \draw[blue] (1.1,-0.1) -- (2,-1);
	       \draw[red] (1.5,-0.1) -- (2,-1);
	       \draw[violet] (2.9,-0.1) -- (2,-1);
	       \draw[red] (1.1,-2) -- (2,-1);
	       \draw[blue] (1.5,-2) -- (2,-1);
	       \draw[violet] (2.9,-2) -- (2,-1);
	       \node at (2.1,-0.4) {$\scalemath{0.8}{\ldots}$};
	       \node at (2.1,-1.6) {$\scalemath{0.8}{\ldots}$};
	       \node[rhdot, minimum size=3.4mm] at (2,-1) {};
	       \node at (2,-1) {$\scalemath{0.7}{2k}$};
	    \end{tikzpicture}} \\[1ex]
            name & Box & Exterior Box & red $2k-$extvalent, $k\ge 2$  & red $2k-$extvalent, $k\ge 2$   \\[1ex]
            bidegree& $(0,\deg f)$  & $(|\xi|,-2|\xi|)$ & $(1, -2k)$  & $(1, -2k)$ 
        \end{tabular}
    \end{center}
    Here $f\in R$ while $\xi \in \Lambda^\vee$ are homogeneous elements. For example elements $\xi\in V\dsbrac{-1}$ have bidegree $(1, -2)$. The colors in the red $2k-$extvalent generator alternate between red and blue with $k$ strands on the bottom and $k$ strands on top. The circle in the middle will \underline{always} be red. 
    
    \item \textbf{Relations} in $\Dext_\infty$ are as follows. The generating morphisms for the color $s$ will satisfy the relations of $\Dext_s$ as specified in \cite {M} Section 4, and likewise for the color $t$ and $\Dext_t$. We will also have $2-$color relations involving the red $2k-$extvalent generator which will be given in the subsection below.
\end{itemize}
\end{definition}

\subsection{$2-$ color Relations}
\label{affinediagrel}

All relations in this subsection will hold with the color of the strands switched. 

\textbf{Rotation Invariance:}

\begin{equation}
\raisebox{-6ex}{
 \begin{tikzpicture}[scale=0.75]
	       \draw[red] (1.1,-0.1) -- (2,-1);
	       \draw[blue] (1.5,-0.1) -- (2,-1);
	       \draw[red] (2.9,-0.1) -- (2,-1);
	       \draw[blue] (1.1,-2) -- (2,-1);
	       \draw[red] (1.5,-2) -- (2,-1);
	       \draw[red] (1.1,-0.1) to[out=90,in=0] (0.8,0.5) to[out=-180,in=90] (0.5, -0.1) ;
	       \draw[red] (0.5,-0.1) -- (0.5,-2);
	       \draw[blue] (2.9,-2) -- (2,-1);
	       \draw[blue] (2.9, -2) to[out=-90,in=-180] (3.2,-2.6) to[out=0,in=-90] (3.5, -2);
	       \draw[blue] (3.5,-2) -- (3.5,-0.1);
	       \node at (2.1,-0.4) {$\ldots$};
	       \node at (2.1,-1.6) {$\ldots$};
	       \node[rkhdot] at (2,-1) {};
	       \node at (2,-1) {$\s{2k}$};
	    \end{tikzpicture}\raisebox{1cm}{ \ \  = \ }\raisebox{0.3cm}{\begin{tikzpicture}[scale=0.8]
	       \draw[blue] (1.1,-0.1) -- (2,-1);
	       \draw[red] (1.5,-0.1) -- (2,-1);
	       \draw[blue] (2.9,-0.1) -- (2,-1);
	       \draw[red] (1.1,-2) -- (2,-1);
	       \draw[blue] (1.5,-2) -- (2,-1);
	       \draw[red] (2.9,-2) -- (2,-1);
	       \node at (2.1,-0.4) {$\ldots$};
	       \node at (2.1,-1.6) {$\ldots$};
	       \node[rkhdot] at (2,-1) {};
	       \node at (2,-1) {$\s{2k}$};
	    \end{tikzpicture}}\raisebox{1cm}{ \ \  = \ } \begin{tikzpicture}[scale=0.75]
	       \draw[red] (1.1,-0.1) -- (2,-1);
	       \draw[blue] (1.5,-0.1) -- (2,-1);
	       \draw[red] (2.9,-0.1) -- (2,-1);
	       \draw[blue] (1.1,-2) -- (2,-1);
	       \draw[red] (1.5,-2) -- (2,-1);
	       \draw[blue] (1.1,-2) to[out=-90,in=0] (0.8,-2.6) to[out=-180,in=-90] (0.5, -2) ;
	       \draw[blue] (0.5,-0.1) -- (0.5,-2);
	       \draw[blue] (2.9,-2) -- (2,-1);
	       \draw[red] (2.9, -0.1) to[out=90,in=-180] (3.2,0.5) to[out=0,in=90] (3.5, -0.1);
	       \draw[red] (3.5,-2) -- (3.5,-0.1);
	       \node at (2.1,-0.4) {$\ldots$};
	       \node at (2.1,-1.6) {$\ldots$};
	       \node[rkhdot] at (2,-1) {};
	       \node at (2,-1) {$\s{2k}$};
	    \end{tikzpicture} }\qquad \qquad \textnormal{ if }k \textnormal{ is odd}
\end{equation}

\begin{equation}
\raisebox{-6ex}{
 \begin{tikzpicture}[scale=0.75]
	       \draw[red] (1.1,-0.1) -- (2,-1);
	       \draw[blue] (1.5,-0.1) -- (2,-1);
	       \draw[blue] (2.9,-0.1) -- (2,-1);
	       \draw[blue] (1.1,-2) -- (2,-1);
	       \draw[red] (1.5,-2) -- (2,-1);
	       \draw[red] (1.1,-0.1) to[out=90,in=0] (0.8,0.5) to[out=-180,in=90] (0.5, -0.1) ;
	       \draw[red] (0.5,-0.1) -- (0.5,-2);
	       \draw[red] (2.9,-2) -- (2,-1);
	       \draw[red] (2.9, -2) to[out=-90,in=-180] (3.2,-2.6) to[out=0,in=-90] (3.5, -2);
	       \draw[red] (3.5,-2) -- (3.5,-0.1);
	       \node at (2.1,-0.4) {$\ldots$};
	       \node at (2.1,-1.6) {$\ldots$};
	       \node[rkhdot] at (2,-1) {};
	       \node at (2,-1) {$\s{2k}$};
	    \end{tikzpicture}\raisebox{1cm}{ \ \  = \ }\raisebox{0.3cm}{\begin{tikzpicture}[scale=0.8]
	       \draw[blue] (1.1,-0.1) -- (2,-1);
	       \draw[red] (1.5,-0.1) -- (2,-1);
	       \draw[red] (2.9,-0.1) -- (2,-1);
	       \draw[red] (1.1,-2) -- (2,-1);
	       \draw[blue] (1.5,-2) -- (2,-1);
	       \draw[blue] (2.9,-2) -- (2,-1);
	       \node at (2.1,-0.4) {$\ldots$};
	       \node at (2.1,-1.6) {$\ldots$};
	       \node[rkhdot] at (2,-1) {};
	       \node at (2,-1) {$\s{2k}$};
	    \end{tikzpicture}}\raisebox{1cm}{ \ \  = \ } \begin{tikzpicture}[scale=0.75]
	       \draw[red] (1.1,-0.1) -- (2,-1);
	       \draw[blue] (1.5,-0.1) -- (2,-1);
	       \draw[blue] (2.9,-0.1) -- (2,-1);
	       \draw[blue] (1.1,-2) -- (2,-1);
	       \draw[red] (1.5,-2) -- (2,-1);
	       \draw[blue] (1.1,-2) to[out=-90,in=0] (0.8,-2.6) to[out=-180,in=-90] (0.5, -2) ;
	       \draw[blue] (0.5,-0.1) -- (0.5,-2);
	       \draw[red] (2.9,-2) -- (2,-1);
	       \draw[blue] (2.9, -0.1) to[out=90,in=-180] (3.2,0.5) to[out=0,in=90] (3.5, -0.1);
	       \draw[blue] (3.5,-2) -- (3.5,-0.1);
	       \node at (2.1,-0.4) {$\ldots$};
	       \node at (2.1,-1.6) {$\ldots$};
	       \node[rkhdot] at (2,-1) {};
	       \node at (2,-1) {$\s{2k}$};
	    \end{tikzpicture} }\qquad \qquad \textnormal{ if }k \textnormal{ is even}
\end{equation}

Rotation invariance will imply that the red $2k-$extvalent morphism is cyclic. Hochschild dots were shown to be cyclic in \cite {M} Section 4, and the other generating morphisms are also cyclic so by Proposition 7.18 in \cite{EMTW20} an isotopy class of diagrams in $\Dext_\infty$ unambiguously represents a morphism in $\Dext_\infty$. As a result, any diagram isotopic or rotationally equivalent to the remaining relations will also  hold in $\Dext_\infty$.

\textbf{4$-$Ext Reduction:}

\begin{equation}
\label{diag4extreduct}
    \begin{tikzpicture}[scale=0.7]
	       \draw[red] (1.2,0) -- (2,-1);
	       \draw[blue] (2.8,0) -- (2,-1);
	       \draw[blue] (1.2,-2) -- (2,-1);
	       \draw[red] (2.8,-2) -- (2,-1);
	       \node[rkhdot] at (2,-1) {};
	       \node at (2,-1) {$4$};
	       \node[rdot] at (2.8,-2) {};
	    \end{tikzpicture}\raisebox{0.6cm}{ \ = \ }
	    \begin{tikzpicture}[scale=0.7]
	       \draw[red] (1.1,0.1) -- (1.1,-1);
	       \draw[blue] (2,0.1) -- (2,-1);
	       \draw[blue] (2,-2) -- (2,-1);
	       \node[rhdot] at (1.1,-1) {};
	\end{tikzpicture}\raisebox{0.6cm}{$\ \ - \ \ $}
\begin{tikzpicture}[scale=0.7]
	       \draw[red] (1.1,0.1) -- (1.1,-1);
	       \draw[blue] (2,0.1) -- (2,-1);
	       \draw[blue] (2,-2) -- (2,-1);
	       \node[rdot] at (1.1,-1) {};
	       \node[bhdot] at (2,-1) {};
	\end{tikzpicture}
\end{equation}

\begin{equation}
\label{diag4extreduct2}
    \begin{tikzpicture}[scale=0.7]
	       \draw[red] (1.2,0) -- (2,-1);
	       \draw[blue] (2.8,0) -- (2,-1);
	       \draw[blue] (1.2,-2) -- (2,-1);
	       \draw[red] (2.8,-2) -- (2,-1);
	       \node[rkhdot] at (2,-1) {};
	       \node at (2,-1) {$4$};
	       \node[bdot] at (1.2,-2) {};
	    \end{tikzpicture}\raisebox{0.6cm}{ \ $= - $\ }
	    \begin{tikzpicture}[scale=0.6]
	       \draw[red] (1.1,0.1) -- (1.1,-2);
	       \draw[blue] (2,0.1) -- (2,-1);
	       \node[bhdot] at (2,-1) {};
	\end{tikzpicture}\raisebox{0.6cm}{$\ + \ $}
\begin{tikzpicture}[scale=0.6]
	       \draw[red] (1.1,0.1) -- (1.1,-2);
	       \draw[blue] (2,0.1) -- (2,-1);
	       \node[bdot] at (2,-1) {};
	       \node[rhdot] at (1.1,-1) {};
	\end{tikzpicture}
\end{equation}

\textbf{Higher Ext Reduction:}

\begin{equation}
\label{highextreduct}
    \raisebox{-4ex}{\begin{tikzpicture}[scale=0.8]
	       \draw[blue] (1.1,-0.1) -- (2,-1);
	       \draw[red] (1.4,-0.1) -- (2,-1);
              \draw[blue] (1.7,-0.1) -- (2,-1);
	       \node[rdot] at (1.4,-0.1) {};
	       \draw[violet] (2.9,-0.1) -- (2,-1);
	       \draw[red] (1.1,-2) -- (2,-1);
	       \draw[blue] (1.5,-2) -- (2,-1);
	       \draw[violet] (2.9,-2) -- (2,-1);
	       \node at (2.1,-0.4) {$\ldots$};
	       \node at (2.1,-1.6) {$\ldots$};
	       \node[rkhdot] at (2,-1) {};
	       \node at (2,-1) {$\s{2k}$};
	    \end{tikzpicture}\raisebox{0.7cm}{ \ = \ }\begin{tikzpicture}[scale=0.8]
	       \draw[blue] (1.5,-0.4) -- (2,-1);
	       \draw[blue] (1.5,-0.4) -- (1.8,0);
	       \draw[blue] (1.5,-0.4) -- (1.2,0);
	       \draw[violet] (2.9,-0.1) -- (2,-1);
	       \draw[red] (1.1,-2) -- (2,-1);
	       \draw[blue] (1.5,-2) -- (2,-1);
	       \draw[violet] (2.9,-2) -- (2,-1);
	       \node at (2.1,-0.4) {$\ldots$};
	       \node at (2.1,-1.6) {$\ldots$};
	       \node[rbkhdot] at (2,-1) {};
	       \node at (2,-1) {$\scalemath{0.6}{2k-2}$};
	    \end{tikzpicture}}\qquad \qquad k \ge 3
\end{equation}

\textbf{Ext Valent Annihilation:}

\begin{equation}
\label{extvalentann}
    \begin{tikzpicture}[scale=0.8]
	       \draw[blue] (1.1,-0.1) -- (2,-1);
	       \draw[red] (1.5,-0.1) -- (2,-1);
	       \draw[violet] (2,-1) -- (4,-1);
	       \draw[red] (1.1,-2) -- (2,-1);
	       \draw[blue] (1.5,-2) -- (2,-1);
	       \draw[violet] (2.9,-2) -- (2,-1);
	       \node at (2.1,-0.4) {$\ldots$};
	       \node at (2.1,-1.6) {$\ldots$};
	       \node[rkhdot] at (2,-1) {};
	       \node at (2,-1) {$\s{2k}$};
	       \draw[red] (4.5,-0.1) -- (4,-1);
	       \draw[blue] (4.9,-0.1) -- (4,-1);
	       \draw[violet] (3.1,-2) -- (4,-1);
	       \draw[blue] (4.5,-2) -- (4,-1);
	       \draw[red] (4.9,-2) -- (4,-1);
	       \node at (3.9,-0.4) {$\ldots$};
	       \node at (3.9,-1.6) {$\ldots$};
	       \node[rkhdot] at (4,-1) {};
	       \node at (4,-1) {$\s{2n}$};
	    \end{tikzpicture}\raisebox{0.7cm}{= \ 0}
\end{equation}

\subsubsection{Further Relations and Morphisms}

\begin{lemma}
The following relations follow from the defining relations above.\\
{\upshape \textbf{More Higher Ext Reduction:}}
\begin{equation}
\label{morehighextreduct}
    \raisebox{-4ex}{\begin{tikzpicture}[scale=0.8]
	       \draw[blue] (1.1,-0.1) -- (2,-1);
	       \draw[red] (1.5,-0.1) -- (2,-1);
	       \node[rdot] at (1.5,-0.1) {};
	       \node[bdot] at (1.1,-0.1) {};
	       \draw[violet] (2.9,-0.1) -- (2,-1);
	       \draw[red] (1.1,-2) -- (2,-1);
	       \draw[blue] (1.5,-2) -- (2,-1);
	       \draw[violet] (2.9,-2) -- (2,-1);
	       \node at (2.1,-0.4) {$\ldots$};
	       \node at (2.1,-1.6) {$\ldots$};
	       \node[rkhdot] at (2,-1) {};
	       \node at (2,-1) {$\s{2k}$};
	    \end{tikzpicture}\raisebox{0.7cm}{ \ = \ }\begin{tikzpicture}[scale=0.8]
	       \draw[blue] (1.2,-0.1) -- (2,-1);
	       \draw[violet] (2.9,-0.1) -- (2,-1);
	       \draw[red] (1.1,-2) -- (2,-1);
	       \draw[blue] (1.5,-2) -- (2,-1);
	       \draw[violet] (2.9,-2) -- (2,-1);
	       \node at (2,-0.4) {$\ldots$};
	       \node at (2.1,-1.6) {$\ldots$};
	       \node[rbkhdot] at (2,-1) {};
	       \node at (2,-1) {$\scalemath{0.6}{2k-2}$};
	    \end{tikzpicture}}\qquad \qquad k \ge 3
\end{equation}

{\upshape \textbf{2$-$color Hochschild Jumping:}}

\begin{equation}
\label{2colorhochjump}
    \raisebox{-4ex}{\begin{tikzpicture}[scale=0.8]
	       \draw[blue] (1.1,-0.1) -- (2,-1);
	       \draw[red] (1.5,-0.1) -- (2,-1);
	       \draw[violet] (2.9,-0.1) -- (2,-1);
	       \draw[red] (1.1,-2) -- (2,-1);
	       \draw[blue] (1.5,-2) -- (2,-1);
	       \draw[violet] (2.9,-2) -- (2,-1);
	       \node at (2.1,-0.4) {$\ldots$};
	       \node at (2.1,-1.6) {$\ldots$};
	       \node[bhdot] at (1.4,-0.43) {};
	       \node[rkhdot] at (2,-1) {};
	       \node at (2,-1) {$\s{2k}$};
	    \end{tikzpicture}\raisebox{0.7cm}{ \ = \ }\begin{tikzpicture}[scale=0.8]
	        \draw[blue] (1.1,-0.1) -- (2,-1);
	       \draw[red] (1.5,-0.1) -- (2,-1);
	       \draw[violet] (2.9,-0.1) -- (2,-1);
	       \draw[red] (1.1,-2) -- (2,-1);
	       \draw[blue] (1.5,-2) -- (2,-1);
	       \draw[violet] (2.9,-2) -- (2,-1);
	       \node[rhdot] at (1.67,-0.43) {};
	       \node at (2.1,-0.4) {$\ldots$};
	       \node at (2.1,-1.6) {$\ldots$};
	       \node[rkhdot] at (2,-1) {};
	       \node at (2,-1) {$\s{2k}$};
	    \end{tikzpicture}}
\end{equation}
{\upshape \textbf{2$-$color Cohomology:}}

\begin{equation}
\label{diacohorelseq}
\raisebox{0.2cm}{\scalemath{1.05}{\boxed{[k]\alpha_s \ }}}
\raisebox{-0.5cm}{
\begin{tikzpicture}[scale=0.8]
	       \draw[red] (1.1,-0.1) -- (2,-1);
	       \draw[blue] (1.5,-0.1) -- (2,-1);
	       \draw[violet] (2.9,-0.1) -- (2,-1);
	       \draw[blue] (1.1,-2) -- (2,-1);
	       \draw[red] (1.5,-2) -- (2,-1);
	       \draw[violet] (2.9,-2) -- (2,-1);
	       \node[rkhdot] at (2,-1) {};
	       \node at (2,-1) {$\s{2k}$};
	       \node at (2.1,-0.4) {$\ldots$};
	       \node at (2.1,-1.6) {$\ldots$};
	    \end{tikzpicture}}
\raisebox{0.2cm}{\scalemath{1.05}{=-\sum_{i=1}^{k-1} \boxed{ [k-i]} \ }}
\raisebox{-0.8cm}{
    \begin{tikzpicture}[scale=0.75]
	       \draw[red] (1,-0.1) -- (2,-1);
	       \draw[blue] (1.5,-0.1) -- (2,-1);
	       \draw[violet] (3,-0.1) -- (2,-1);
	       \draw[blue] (0.8,-2.2) -- (2,-1);
	       \draw[violet] (3.2,-2.2) -- (2,-1);
	       \draw[myyellow] (2,-1) -- (2,-1.6);
	       \draw[violet] (2,-2.2) -- (2,-1.9);
	       \draw[myyellow] (1.7,-2.2) to[out=90,in=180] (2,-1.6) to[out=0,in=90] (2.3,-2.2);
	       \node[pdot] at (2,-1.9) {};
	       \node[rbkhdot] at (2,-1) {};
	       \node at (2,-1) {$\scalemath{0.65}{2k-2}$};
	       \node at (2.1,-0.3) {$\ldots$};
	       \node at (1.4,-1.95) {$\ldots$};
	       \node at (2.6,-1.95) {$\ldots$};
	       \node at (2,-2.5) {$t_i$};
	    \end{tikzpicture}}\raisebox{0.2cm}{ \scalemath{1.05}{+\sum_{i=1}^{k}\boxed{[k+1-i]}}}
	    \raisebox{-0.4cm}{
	    \begin{tikzpicture}[scale=0.75]
	       \draw[red] (0.8,0.1) -- (2,-1);
	       \draw[violet] (3.2,0.1) -- (2,-1);
	       \draw[blue] (1,-2) -- (2,-1);
	       \draw[violet] (3,-2) -- (2,-1);
	       \draw[red] (2,-1) -- (1.5,-2);
	       \draw[myyellow] (2,-1) -- (2,-0.4);
	       \draw[violet] (2,0.2) -- (2,-0.1);
	       \draw[myyellow] (1.7,0.1) to[out=-90,in=180] (2,-0.4) to[out=0,in=-90] (2.3,0.1);
	       \node at (2,0.35) {$s_i$};
	       \node[pdot] at (2,-0.1) {};
	       \node[rbkhdot] at (2,-1) {};
	       \node at (2,-1) {$\scalemath{0.65}{2k-2}$};
	       \node at (2.6,-0.15) {$\ldots$};
	       \node at (1.4,-0.15) {$\ldots$};
	       \node at (2.1,-1.7) {$\ldots$};
	    \end{tikzpicture}} 
\end{equation}

\begin{equation}
\label{diacohorelteq}
\raisebox{0.2cm}{\scalemath{1.05}{\boxed{[k]\alpha_t \ }}}
\raisebox{-0.5cm}{
\begin{tikzpicture}[scale=0.8]
	       \draw[red] (1.1,-0.1) -- (2,-1);
	       \draw[blue] (1.5,-0.1) -- (2,-1);
	       \draw[violet] (2.9,-0.1) -- (2,-1);
	       \draw[blue] (1.1,-2) -- (2,-1);
	       \draw[red] (1.5,-2) -- (2,-1);
	       \draw[violet] (2.9,-2) -- (2,-1);
	       \node[rkhdot] at (2,-1) {};
	       \node at (2,-1) {$\s{2k}$};
	       \node at (2.1,-0.4) {$\ldots$};
	       \node at (2.1,-1.6) {$\ldots$};
	    \end{tikzpicture}}
\raisebox{0.2cm}{\scalemath{1.05}{=-\sum_{i=1}^{k-1} \boxed{ [k-i]} \ }}
\raisebox{-0.4cm}{
	    \begin{tikzpicture}[scale=0.7]
	       \draw[red] (0.8,0.1) -- (2,-1);
	       \draw[violet] (3.2,0.1) -- (2,-1);
	       \draw[blue] (1,-2) -- (2,-1);
	       \draw[violet] (3,-2) -- (2,-1);
	       \draw[red] (2,-1) -- (1.5,-2);
	       \draw[myyellow] (2,-1) -- (2,-0.4);
	       \draw[violet] (2,0.2) -- (2,-0.1);
	       \draw[myyellow] (1.7,0.1) to[out=-90,in=180] (2,-0.4) to[out=0,in=-90] (2.3,0.1);
	       \node at (2,0.35) {$s_i$};
	       \node[pdot] at (2,-0.1) {};
	       \node[rbkhdot] at (2,-1) {};
	       \node at (2,-1) {$\scalemath{0.65}{2k-2}$};
	       \node at (2.6,-0.15) {$\ldots$};
	       \node at (1.4,-0.15) {$\ldots$};
	       \node at (2.1,-1.7) {$\ldots$};
	    \end{tikzpicture}} \raisebox{0.2cm}{ \scalemath{1.05}{+\sum_{i=1}^{k}\boxed{[k+1-i]}}}
	    \raisebox{-0.8cm}{
    \begin{tikzpicture}[scale=0.7]
	       \draw[red] (1,-0.1) -- (2,-1);
	       \draw[blue] (1.5,-0.1) -- (2,-1);
	       \draw[violet] (3,-0.1) -- (2,-1);
	       \draw[blue] (0.8,-2.2) -- (2,-1);
	       \draw[violet] (3.2,-2.2) -- (2,-1);
	       \draw[myyellow] (2,-1) -- (2,-1.6);
	       \draw[violet] (2,-2.2) -- (2,-1.9);
	       \draw[myyellow] (1.7,-2.2) to[out=90,in=180] (2,-1.6) to[out=0,in=90] (2.3,-2.2);
	       \node[pdot] at (2,-1.9) {};
	       \node[rbkhdot] at (2,-1) {};
	       \node at (2,-1) {$\scalemath{0.65}{2k-2}$};
	       \node at (2.1,-0.3) {$\ldots$};
	       \node at (1.4,-1.95) {$\ldots$};
	       \node at (2.6,-1.95) {$\ldots$};
	       \node at (2,-2.5) {$t_i$};
	    \end{tikzpicture}}
\end{equation}
\end{lemma}
\begin{proof}
\cref{morehighextreduct} clearly follows from \cref{highextreduct}. \cref{2colorhochjump} follows from $4-$Ext reduction and Ext Valent Annihilation (see \cref{2colorjumpeq}). \cref{diacohorelseq} and \cref{diacohorelteq}
 follow by using polynomial forcing, i.e. see the proof of \cref{cohorelseq} as \cref{cohoreleq1} also holds in $\Dext_\infty$.
 \end{proof}
 
\begin{definition}
Given an expression $\underline{w}$, let $mc(\underline{w})$ be any non repeating subexpression of $\underline{w}$ whose first and last terms are different such that $|mc(\underline{w})|$ is maximal. If $\underline{w}$ consists of only 1 color, we set $|mc(\underline{w})|=1$.
\end{definition}

For example, we have that $mc(tst)=ts$ or $st$. Note that besides when $|mc(\underline{w})|=1$, $|mc(\underline{w})|$ is always even. 

\begin{definition}
Given any two expressions $\underline{v}$ and $\underline{w}$, we will define a morphism in $\hom^{1, \bullet}_{\Dext_\infty}(\underline{v}, \underline{w})$ as follows. First suppose that $|mc(\underline{v}^{-1}\underline{w})|=2k$ for some $k\in \mathbb{Z}^{\ge 2}$. Then define the morphism
\begin{center}
    \begin{tikzpicture}[scale=0.8]
	       \draw[violet] (1.1,-0.1) -- (2,-1);
	       \draw[violet] (2.9,-0.1) -- (2,-1);
	       \draw[violet] (1.1,-2) -- (2,-1);
	       \draw[violet] (2.9,-2) -- (2,-1);
	       \node at (2,-0.4) {$\underline{w}$};
	       \node at (2,-1.6) {$\underline{v}$};
	       \node[rkhdot] at (2,-1) {};
	       \node at (2,-1) {$\s{2k}$};
	    \end{tikzpicture}\raisebox{0.7cm}{$=$ add merge, splits to}
	    \begin{tikzpicture}[scale=0.8]
	       \draw[blue] (1.1,-0.1) -- (2,-1);
	       \draw[red] (1.5,-0.1) -- (2,-1);
	       \draw[blue] (2.9,-0.1) -- (2,-1);
	       \draw[red] (1.1,-2) -- (2,-1);
	       \draw[blue] (1.5,-2) -- (2,-1);
	       \draw[red] (2.9,-2) -- (2,-1);
	       \node at (2.1,-0.4) {$\ldots$};
	       \node at (2.1,-1.6) {$\ldots$};
	       \node[rkhdot] at (2,-1) {};
	       \node at (2,-1) {$\s{2k}$};
	    \end{tikzpicture}\raisebox{0.7cm}{ and rotate until colors on bottom (top) =$\underline{v} \ (\underline{w})$ respectively}
\end{center}
This is independent of how one chooses to rotate or add merge and splits by Proposition 7.17 in \cite{EMTW20}. If $|mc(\underline{v}^{-1}\underline{w})|<4$, set
\[ \begin{tikzpicture}[scale=0.7]
	       \draw[red] (2,0) -- (2,-1);
	       \draw[blue] (2,-2) -- (2,-1);
	       \node[rkhdot] at (2,-1) {};
	       \node at (2,-1) {$\s{2}$};
	\end{tikzpicture}\raisebox{0.6cm}{$ \quad :=$}\begin{tikzpicture}[scale=0.7]
	       \draw[red] (1.1,-0.1) -- (2,-1);
	       \draw[blue] (2,0.1) -- (2,-1);
	       \draw[red] (2.9,-0.1) -- (2,-1);
	       \draw[blue] (2,-2) -- (2,-1);
	       \node[rkhdot] at (2,-1) {};
	       \node at (2,-1) {$\s{4}$};
	       \node[bdot] at (2,0.1) {};
	       \node[rdot] at (2.9,-0.1) {};
	\end{tikzpicture} \raisebox{0.6cm}{$ \quad =$}\begin{tikzpicture}[scale=0.7]
	       \draw[red] (1.1,-0.1) -- (2,-1);
	       \draw[blue] (2,0.1) -- (2,-1);
	       \draw[red] (2.9,-0.1) -- (2,-1);
	       \draw[blue] (2,-2) -- (2,-1);
	       \node[rkhdot] at (2,-1) {};
	       \node at (2,-1) {$\s{4}$};
	       \node[bdot] at (2,0.1) {};
	       \node[rdot] at (1.1,-0.1) {};
	\end{tikzpicture}  \qquad \quad  \begin{tikzpicture}[scale=0.7]
	       \draw[blue] (2,-2) -- (2,-1);
	       \node[rkhdot] at (2,-1) {};
	       \node at (2,-1) {$\s{1}$};
	\end{tikzpicture}\raisebox{0.6cm}{$ \quad =$}\begin{tikzpicture}[scale=0.7]
	       \draw[red] (1.1,-0.1) -- (2,-1);
	       \draw[blue] (2,0.1) -- (2,-1);
	       \draw[red] (2.9,-0.1) -- (2,-1);
	       \draw[blue] (2,-2) -- (2,-1);
	       \node[rkhdot] at (2,-1) {};
	       \node at (2,-1) {$\s{4}$};
	       \node[bdot] at (2,0.1) {};
	       \node[rdot] at (2.9,-0.1) {};
	       \node[rdot] at (1.1,-0.1) {};
	\end{tikzpicture}   \]
and repeat the same definition as above. 
\end{definition}

\begin{example}
The RHS of \cref{highextreduct} is the morphism on the left below while the RHS of \cref{morehighextreduct} is the morphism on the right below
\[  \begin{tikzpicture}[scale=0.8]
	       \draw[violet] (1.1,-0.1) -- (2,-1);
	       \draw[violet] (2.9,-0.1) -- (2,-1);
	       \draw[violet] (1.1,-2) -- (2,-1);
	       \draw[violet] (2.9,-2) -- (2,-1);
	       \node at (2,-0.3) {$\scalemath{0.6}{(t,\usmt{t}{k-2}{})}$};
	       \node at (1.9,-1.7) {$\s{\usmt{s}{k}{}}$};
	       \node[rbkhdot] at (2,-1) {};
	       \node at (2,-1) {$\scalemath{0.6}{2k-2}$};
	    \end{tikzpicture} \qquad \qquad \qquad \begin{tikzpicture}[scale=0.8]
	       \draw[violet] (1.1,-0.1) -- (2,-1);
	       \draw[violet] (2.9,-0.1) -- (2,-1);
	       \draw[violet] (1.1,-2) -- (2,-1);
	       \draw[violet] (2.9,-2) -- (2,-1);
	       \node at (2,-0.3) {$\s{\usmt{t}{k-2}{}}$};
	       \node at (1.9,-1.7) {$\s{\usmt{s}{k}{}}$};
	       \node[rbkhdot] at (2,-1) {};
	       \node at (2,-1) {$\scalemath{0.6}{2k-2}$};
	    \end{tikzpicture}\]
\end{example}

\begin{example}
Here are possible presentations for two morphisms that we will use later
\begin{equation}
    \begin{tikzpicture}[scale=0.8]
	       \draw[blue] (1.1,-0.1) -- (2,-1);
	       \draw[red] (2.9,-0.1) -- (2,-1);
	       \node at (2,-0.4) {$\s{\usmt{t}{2k}{}}$};
	       \node[rkhdot] at (2,-1) {};
	       \node at (2,-1) {$\s{2k}$};
	    \end{tikzpicture}\raisebox{0.1cm}{= \ \  } \raisebox{-7ex}{\begin{tikzpicture}[scale=0.75]
        \draw[blue] (2.9,-2) to[out=-90,in=0] (1.2,-3.5) to[out=180,in=-90] (-0.5, -2);
	       \draw[blue] (-0.5,-0.1) -- (-0.5, -2);
	       \draw[violet] (1.1,-2) to[out=-90,in=0] (0.8,-2.5) to[out=-180,in=-90] (0.5, -2) ;
	       \draw[violet] (0.5,-0.1) -- (0.5,-2);
	       \draw[myyellow] (1.1,-0.1) -- (2,-1);
	       \draw[red] (2.9,-0.1) -- (2,-1);
	       \draw[violet] (1.1,-2) -- (2,-1);
	       \draw[blue] (2.9,-2) -- (2,-1);
	       \node at (2,-0.3) {$\ldots$};
	       \node at (2,-1.7) {$\ldots$};
	       \node[rkhdot] at (2,-1) {};
	       \node at (2,-1) {$\s{2k}$};
	    \end{tikzpicture}} \qquad \qquad  \begin{tikzpicture}[scale=0.8]
	       \draw[blue] (1.1,-0.1) -- (2,-1);
	       \draw[blue] (2.9,-0.1) -- (2,-1);
	       \node at (2,-0.3) {$\s{\usmt{t}{2k+1}{}}$};
	       \node[rkhdot] at (2,-1) {};
	       \node at (2,-1) {$\s{2k}$};
	    \end{tikzpicture}\raisebox{0.1cm}{= \ \  } \raisebox{-7ex}{\begin{tikzpicture}[scale=0.75]
        \draw[blue] (2.9,-2) to[out=-90,in=0] (1.2,-3.5) to[out=180,in=-90] (-0.5, -2);
	       \draw[blue] (-0.5,-0.1) -- (-0.5, -2);
	       \draw[violet] (1.1,-2) to[out=-90,in=0] (0.8,-2.5) to[out=-180,in=-90] (0.5, -2) ;
	       \draw[violet] (0.5,-0.1) -- (0.5,-2);
	       \draw[myyellow] (1.1,-0.1) -- (2,-1);
	       \draw[red] (2.9,-0.1) -- (2,-1);
	       \draw[violet] (1.1,-2) -- (2,-1);
	       \draw[blue] (2.9,-2) -- (2,-1);
	       \draw[blue] (2.9,-2) -- (3.8,-1);
	       \draw[blue] (3.8,-0.1) -- (3.8,-1);
	       \node at (2,-0.3) {$\ldots$};
	       \node at (2,-1.7) {$\ldots$};
	       \node[rkhdot] at (2,-1) {};
	       \node at (2,-1) {$\s{2k}$};
	    \end{tikzpicture}} 
\end{equation}  
\end{example}

\subsection{Equivalence}
To distinguish between indecomposable objects in the diagrammatic and bimodule categories, we will let $\Bs_w$ denote the indecomposable corresponding to $w\in W_\infty$ in $\kar(\Dext_\infty)$ while $B_w$ will denote the indecomposable Soergel Bimodule corresponding to $w$. For this section we will also let $\bsbimext:=\bsbimext(\hf, W_\infty)$.

\begin{theorem}
\label{affineequiv}
Assume all quantum numbers are invertible and that \cref{assump0}, \cref{assump1}, and \cref{assump2} hold. Define the $\kb-$linear functor $\Fext_{\infty}: \Dext_\infty\to  \bsbimext(\hf, W_\infty) $ on objects by $\Fext_{\infty}((s))=K_s$, $\Fext_{\infty}((t))=K_t$ and extended monoidally. On morphisms $\Fext_{\infty}$ is defined as in \cref{Mmaintheorem} on the subcategories $\Dext_s$ and $\Dext_t$ and additionally sends
\[  \Fext_{\infty}\paren{ \raisebox{-3ex}{\begin{tikzpicture}[scale=0.6]
	       \draw[blue] (1.1,-0.1) -- (2,-1);
	       \draw[red] (1.5,-0.1) -- (2,-1);
	       \draw[violet] (2.9,-0.1) -- (2,-1);
	       \draw[red] (1.1,-2) -- (2,-1);
	       \draw[blue] (1.5,-2) -- (2,-1);
	       \draw[violet] (2.9,-2) -- (2,-1);
	       \node at (2.1,-0.4) {$\scalemath{0.8}{\ldots}$};
	       \node at (2.1,-1.6) {$\scalemath{0.8}{\ldots}$};
	       \node[rhdot, minimum size=3.4mm] at (2,-1) {};
	       \node at (2,-1) {$\scalemath{0.7}{2k}$};
	    \end{tikzpicture}} }=\Omega_{\usmt{s}{k}{}}^{\usmt{t}{k}{}} \qquad \qquad  \Fext_{\infty}\paren{ \raisebox{-3ex}{\begin{tikzpicture}[scale=0.6]
	       \draw[red] (1.1,-0.1) -- (2,-1);
	       \draw[blue] (1.5,-0.1) -- (2,-1);
	       \draw[violet] (2.9,-0.1) -- (2,-1);
	       \draw[blue] (1.1,-2) -- (2,-1);
	       \draw[red] (1.5,-2) -- (2,-1);
	       \draw[violet] (2.9,-2) -- (2,-1);
	       \node at (2.1,-0.4) {$\scalemath{0.8}{\ldots}$};
	       \node at (2.1,-1.6) {$\scalemath{0.8}{\ldots}$};
	       \node[rhdot, minimum size=3.4mm] at (2,-1) {};
	       \node at (2,-1) {$\scalemath{0.7}{2k}$};
	    \end{tikzpicture}} }=\Omega_{\usmt{t}{k}{}}^{\usmt{s}{k}{}} \qquad \qquad \forall k \ge 2  \]
Then $\Fext_\infty$ is well defined and furthermore will be a monoidal equivalence. 
\end{theorem}
\begin{proof}
The defining relations  of $\Dext_\infty$ in \cref{affinediagrel} all hold in $\bsbimext(\hf, W_\infty)$ as shown in \cref{newgensect}. Therefore $\Fext_{\infty}$ is well defined and monoidal by construction. Because both $s$ and $K_s$ have biadjoints in their respective categories, it suffices to check $\hom^{\bullet, \bullet}_{\Dext_\infty}(\varnothing,\underline{w} ) \cong \hom_{\bsbimext}(R, \bs(\underline{w}))$ for all expresssions $\underline{w}$. As all quantum numbers are invertible the Jones-Wenzl projectors $ \jw_{\usmt{t}{2k}{} }$ are defined for all $k$ and as shown in \cite{DC} Section 5.4.2, the object $\underline{w}$ in $\mathscr{D}_\infty$ decomposes in $\kar(\mathscr{D}_\infty)$ into indecomposables exactly as $\bs(\underline{w})$ decomposes in $\bsbim$ into indecomposables. In other words the decomposition in \cref{bsdecompeq} holds with $\bs(\usmt{t}{k}{})$ replaced by $\usmt{t}{k}{}$ and $B_y$ replaced by $\Bs_y$. It follows that we just need to check the isomorphism on indecomposables, aka for all $w\in W_\infty$
\[  \hom^{\bullet, \bullet}_{\Dext_\infty}(\Bs_\varnothing, \Bs_w) \cong \hom_{\bsbimext}(K_\varnothing, K_w)=\ext^{\bullet, \bullet}_{R^e}(R, B_w)  \]
 The exterior forcing relation in $\Dext_\infty$ shows that $\xi\in \Lambda_{st}$ acts freely on $\hom^{\bullet, \bullet}_{\Dext_\infty}(\Bs_\varnothing, \Bs_w)$ in agreement with what happens in the bimodule category. Therefore we can assume that $\boxed{\Lambda_{st}=\kb}$ so that \cref{algextindecomp} gives us a description of the RHS in terms of diagramatics already. WLOG we can assume $w$ starts with $t$ and in fact we can asuume $w=\usmt{t}{2j}{}$ (the odd case will proceed exactly the same, but the color of the last strand below will be blue). It follows that we need to show that 
\begin{align*}
\hom^{0, \bullet}_{\Dext_\infty}(\Bs_\varnothing, \Bs_{\smt{t}{2j}{}})&=  \jw_{\usmt{t}{2j}{} }\circ \bunit \runit \ldots \runit\  R \\
\hom^{1, \bullet}_{\Dext_\infty}(\Bs_\varnothing, \Bs_{\smt{t}{2j}{}})&=  \jw_{\usmt{t}{2j}{} }\circ  \bhunit \runit \ldots \runit\ R\oplus  \jw_{\usmt{t}{2j}{} }\circ  \ 
\raisebox{-0.2cm}{\begin{tikzpicture}[scale=0.6]
	       \draw[blue] (1.1,-0.1) -- (2,-1);
	       \draw[red] (2.9,-0.1) -- (2,-1);
	       \node[rkhdot] at (2,-1) {};
	       \node at (2,-1) {$\s{2j}$};
	       \node at (2,-0.3) {$\scalemath{0.7}{\usmt{t}{2j}{}}$};
	    \end{tikzpicture}} \ R \\
\hom^{2, \bullet}_{\Dext_\infty}(\Bs_\varnothing, \Bs_{\smt{t}{2j}{}})&= \ \jw_{\usmt{t}{2j}{} }\circ \raisebox{-0.2cm}{\begin{tikzpicture}[scale=0.6]
	       \draw[blue] (1.1,-0.1) -- (2,-1);
	       \draw[red] (2.9,-0.1) -- (2,-1);
	       \node[rkhdot] at (2,-1) {};
	       \node at (2,-1) {$\s{2j}$};
	       \node[bhdot] at (1.4, -0.4) {};
	       \node at (2,-0.3) {$\scalemath{0.7}{\usmt{t}{2j}{}}$};
	    \end{tikzpicture}} \ R
\end{align*}
\textbf{Step 1:} The case $\hom^{0, \bullet}_{\Dext_\infty}(\Bs_\varnothing, \Bs_{\usmt{t}{2j}{}})$ was already done in \cite{DC}. 
For $\hom^{1, \bullet}_{\Dext_\infty}(\Bs_\varnothing, \Bs_{\usmt{t}{2j}{}})$, notice that the proof of \cref{algextindecomp} was entirely diagrammatic. It used the equivalence $\mathscr{D}_\infty\to \bsbim(\hf, W_\infty) $ so that diagrammatics can be used to prove results in the bimodule category along with a rotated version of the cohomology relations \cref{cohorelseq}, \cref{cohorelteq} which correspond to \cref{diacohorelseq}, \cref{diacohorelteq} in the diagrammatic category. Therefore the proof of \cref{algextindecomp} can be applied to show $\jw_{\usmt{t}{2j}{} }\circ  \bhunit \runit \ldots \runit$ and $\jw_{\usmt{t}{2j}{} }\circ  \ 
\raisebox{-0.2cm}{\begin{tikzpicture}[scale=0.6]
	       \draw[blue] (1.1,-0.1) -- (2,-1);
	       \draw[red] (2.9,-0.1) -- (2,-1);
	       \node[rkhdot] at (2,-1) {};
	       \node at (2,-1) {$\s{2j}$};
	       \node[bhdot] at (1.4, -0.4) {};
	       \node at (2,-0.3) {$\scalemath{0.7}{\usmt{t}{2j}{}}$};
	    \end{tikzpicture}} $
span $\hom^{1, \bullet}_{\Dext_\infty}(\Bs_\varnothing, \Bs_{\usmt{t}{2j}{}})$ as a right $R$ module if we can show that any morphism in $\hom^{1, \bullet}_{\Dext_\infty}(\varnothing, \,  \usmt{t}{2j}{})$ can be written as a $R$ linear combination of diagrams of the form 
\begin{equation}
\label{ext1lightleaves}
\bhunit \ \  \raisebox{-0.4cm}{\begin{tikzpicture}
\node (pq1) [dashed, shape=rectangle,
                     draw=black, semithick,
                     text width=0.1\linewidth,
                     align=center] {$\mathrm{L}_{  \usmt{s}{j-1}{}, \underline{f} }$};
\end{tikzpicture}} \  \ \textnormal{where} \ r(\underline{f})=\id  \quad   \raisebox{-0.8cm}{\ \  \begin{tikzpicture}
\draw[blue] (0.9,-0.3) to[out=-90,in=-180] (1.45,-1.35) to[out=0, in=-90] (2,-1);
\node[bhdot] at (1.45, -1.35) {};
\node at (2.2,-0.65) [dashed, shape=rectangle,
                     draw=black, semithick,
                     text width=0.1\linewidth,
                     align=center] {$\mathrm{L}_{  \usmt{s}{j-1}{}, \underline{f} }$};
\end{tikzpicture}} \ \  \textnormal{where} \  r(\underline{f})=t \quad \raisebox{-0.8cm}{\ \  
\begin{tikzpicture}
\draw[blue] (0.9,-0.3) to[out=-90,in=-180] (1.55,-1.95) to[out=0, in=-90] (2,-1.6);
\draw[violet] (1.4, -1.1) to[out=-90, in=-180] (2, -1.6);
\draw[violet] (2.6, -1.1) to[out=-90, in=0] (2, -1.6);
\node at (2, -1.3) {$\scalemath{0.8}{r(\underline{f})}$};
\node[rhdot] at (2, -1.6) {};
\node at (2.2,-0.65) [dashed, shape=rectangle,
                     draw=black, semithick,
                     text width=0.1\linewidth,
                     align=center] {$\mathrm{L}_{  \usmt{s}{j-1}{}, \underline{f} }$};
\end{tikzpicture}}\quad r(\underline{f}) \textnormal{ is anything}
\end{equation}
using the relations in \cref{affinediagrel}. Linear independence then follows by applying $\Fext_{\infty}$ and noting the corresponding morphisms in the bimodule category forms an $R$ basis.\\

\textbf{Step 2: }By cohomological degree reasons any diagram in $\hom^{1, \bullet}_{\Dext_\infty}(\varnothing,\,  \usmt{t}{2j}{})$ either has exactly one exterior box $\dboxed{\xi}$ where $\xi\in V$, or a Hochschild dot, or a red $2k-$extvalent map. As we assumed $\Lambda_{st}=\kb$, $V=\kb \alpha_s^\vee \oplus \kb \alpha_t^\vee$. Using the 1 color cohomology relation
\[ \runit \ \dboxed{\alpha_s^\vee}  =\rhunit \ \boxed{\alpha_s} \qquad \qquad \bunit \ \dboxed{\alpha_t^\vee}  =\bhunit \ \boxed{\alpha_t} \]
and \cref{eqn:exterior-boxes-mult} we can reduce any diagram with an exterior box to a right $R$ linear sum of diagrams with Hochschild dots. \\
\textbf{Step 3:} Given a diagram with a hochschild dot or a red $2k-$extvalent morphism at the \underline{bottom}, we claim it can be written as a right $R$ linear sum of diagrams in \cref{ext1lightleaves}. We first do the case of Hochschild dots. Using \cref{4extreduct} any red hochschild dot can be converted to a blue hochschild dot and a $4-$ext valent vertex. Now given a diagram $D$ with a blue hochschild dot at the bottom, the rest of the diagram is some morphism in $\hom_{\mathscr{D}_\infty}(t, \usmt{t}{2k}{} )$. From \cite{SC} we know that double leaves form a right $R$ basis for this. The only possible light leaves for the bottom part of the double leaf map are $\scalemath{0.9}{\bzero}$ and $\raisebox{1.7ex}{\bcounit}$. Thus $D$ is a $R$ linear sum of diagrams of the form
\begin{equation}
  \raisebox{-1cm}{\ \  \begin{tikzpicture}
\draw[blue] (2.2,-1) to[out=-90,in=90] (2.2,-1.5) ;
\node[bhdot] at (2.2, -1.5) {};
\node at (2.2,-0.65) [dashed, shape=rectangle,
                     draw=black, semithick,
                     text width=0.1\linewidth,
                     align=center] {$\mathrm{L}_{  \usmt{t}{j}{}, \underline{f} }$};
\end{tikzpicture}} \ \  \textnormal{where} \  r(\underline{f})=t\qquad \qquad \quad   \raisebox{-1cm}{\begin{tikzpicture}
\node (pq1) [dashed, shape=rectangle,
                     draw=black, semithick,
                     text width=0.1\linewidth,
                     align=center] {$\mathrm{L}_{  \usmt{t}{j}{}, \underline{f} }$};
                     \draw[blue] (0,-1) -- (0,-0.6);
 \node[bdot] at (0,-0.6) {};
  \node[bhdot] at (0,-1) {};
\end{tikzpicture}} \  \ \textnormal{where} \ r(\underline{f})=\id
\end{equation}
Clearly diagrams in the form on the left above are in the span of diagrams in \cref{ext1lightleaves} while for diagrams in the form on the right above we can use 1 color hochschild jumping \cref{eqn:hochjump} to move the blue hochschild dot onto a blue strand in $\mathrm{L}_{  \usmt{t}{j}{}, \underline{f} }$ and then replacing the barbell with $\alpha_t$ so that the result is again clearly in the the span of diagrams in \cref{ext1lightleaves}.\\

Now, given a diagram with a red $2k-$extvalent morphism at the bottom, the rest of the diagram is some morphism in $\hom_{\mathscr{D}_\infty}(\usmt{t}{2k}{}, \usmt{t}{2j}{} )$ which has a right $R$ basis given by double leaves. Because of \cref{highextreduct} and \cref{diag4extreduct}, one can show that pitchforks and therefore generalized pitchforks kill the red $2k-$extvalent morphism. Thus, as in Step 2 of the proof of \cref{algextindecomp} the bottom light leaf of the double leaf map on top of the red $2k-$extvalent consists of only dots and straight lines. But this means that our diagram is exactly in the form of the right most morphism in \cref{ext1lightleaves}.\\
\textbf{Step 4: } Now given a diagram with a Hochschild dot, if it's not at the bottom of the diagram, then it must be trapped by a line, hereafter referred to as the trapping line. If the trapping line has the same color as the Hochschild dot, then we can use 1-color Hochschild Jumping to move the Hocschild dot onto the line. Otherwise if the trapping line is a different color, we can then use $4-$ext reduction \cref{diag4extreduct} to move the Hochschild dot onto the trapping line at the cost of a red $4-$extvalent morphism, as seen below
\begin{center}
	\raisebox{0.05cm}{
\begin{tikzpicture}[scale=0.6]
           \draw[dotted, thick] (2,-0.5) circle (40pt);
	       \draw[blue] (0.9,0) to[out=-90, in=-180] (2,-1.5) to[out=0, in=-90] (3.1, 0);
	       \draw[red] (2,0) -- (2,-0.9);
	       \node[rhdot] at (2,-0.9) {};
	       \node at (1.4,-0.4) {$\cdots$};
	       \node at (2.6,-0.4) {$\cdots$};
	\end{tikzpicture}} \raisebox{0.7cm}{$ \ = \ $}
\begin{tikzpicture}[scale=0.6]
           \draw[dotted, thick] (2,-0.5) circle (40pt);
	       \draw[blue] (0.9,0) to[out=-90, in=-180] (2,-1.5) to[out=0, in=-90] (3.1, 0);
	       \draw[red] (2,0) -- (2,-0.9);
	       \node at (1.4,-0.4) {$\cdots$};
	       \node at (2.6,-0.4) {$\cdots$};
	       \node[bhdot] at (2,-1.5) {};
	       \node[rdot] at (2,-0.9) {};
	\end{tikzpicture}
	 \raisebox{0.7cm}{$ \ + \ $}
	 \raisebox{-1ex}{
	\begin{tikzpicture}[scale=0.6]
	\draw[dotted, thick] (2,-1) circle (50pt);
	       \draw[blue] (0.9,0) to[out=-90, in=-180] (2,-1.2) to[out=0, in=-90] (3.1, 0);
	       \draw[red] (2,0) -- (2,-1.2);
	       \draw[red] (2,-1.2) -- (2,-2.2);
	       \node[rdot] at (2,-2.2) {};
	       \node[rkhdot] at (2,-1.2) {};
	       \node at (2,-1.2) {$4$};
	       \node at (1.4,-0.4) {$\cdots$};
	       \node at (2.6,-0.4) {$\cdots$};
	\end{tikzpicture}}
\end{center}
\begin{center}
	\raisebox{0.05cm}{
\begin{tikzpicture}[scale=0.6]
\draw[dotted, thick] (2,-0.5) circle (40pt);
	       \draw[red] (0.9,0) to[out=-90, in=-180] (2,-1.5) to[out=0, in=-90] (3.1, 0);
	       \draw[blue] (2,0) -- (2,-0.9);
	       \node[bhdot] at (2,-0.9) {};
	       \node at (1.4,-0.4) {$\cdots$};
	       \node at (2.6,-0.4) {$\cdots$};
	\end{tikzpicture}} \raisebox{0.7cm}{$ \ = \ $}
\begin{tikzpicture}[scale=0.6]
\draw[dotted, thick] (2,-0.5) circle (40pt);
	       \draw[red] (0.9,0) to[out=-90, in=-180] (2,-1.5) to[out=0, in=-90] (3.1, 0);
	       \draw[blue] (2,0) -- (2,-0.9);
	       \node at (1.4,-0.4) {$\cdots$};
	       \node at (2.6,-0.4) {$\cdots$};
	       \node[rhdot] at (2,-1.5) {};
	       \node[bdot] at (2,-0.9) {};
	\end{tikzpicture}
	 \raisebox{0.7cm}{$ \ + \ $}
	 \raisebox{-1ex}{
	\begin{tikzpicture}[scale=0.6]
	\draw[dotted, thick] (2,-1) circle (50pt);
	       \draw[red] (0.9,0) to[out=-90, in=-180] (2,-1.2) to[out=0, in=-90] (3.1, 0);
	       \draw[blue] (2,0) -- (2,-1.2);
	       \draw[blue] (2,-1.2) -- (2,-2.2);
	       \node[bdot] at (2,-2.2) {};
	       \node[rkhdot] at (2,-1.2) {};
	       \node at (2,-1.2) {$4$};
	       \node at (1.4,-0.4) {$\cdots$};
	       \node at (2.6,-0.4) {$\cdots$};
	\end{tikzpicture}}
\end{center}
Therefore it suffices to consider the case when a red $2k-$extvalent morphism is trapped by a line. First suppose that the red $2k-$extvalent morphism is on the same connected component as the trapping line, as seen below. We can then apply \cref{highextreduct}, higher ext reduction so that the red $2k-$extvalent absorbs the trapping line and therefore moves further to the bottom.
\begin{equation}
\label{absorbtrap}
\raisebox{-8ex}{
    \begin{tikzpicture}[scale=0.65]
	\draw[dotted, thick] (2,-0.8) circle (60pt);
	       \draw[red] (0.9,0) to[out=-90, in=-180] (2,-1.2) to[out=0, in=-90] (3.1, 0);
	       \draw[blue] (2, -1.2) -- (2, -2.4);
	       \node[rkhdot] at (2,-1.2) {};
	       \node at (2,-1.2) {$\s{2k}$};
	       \node at (2,-0.4) {$\cdots$};
	       \draw[blue] (0.3,0) to[out=-90, in=-180] (2,-2.4) to[out=0, in=-90] (3.7, 0);
	\end{tikzpicture}\raisebox{1.1cm}{$ \ = \ $}
	\begin{tikzpicture}[scale=0.65]
	\draw[dotted, thick] (2,-1) circle (60pt);
	       \draw[red] (0.9,0) to[out=-90, in=-180] (2,-1.2) to[out=0, in=-90] (3.1, 0);
	       \draw[red] (2, -1.2) -- (2, -2.4);
	       \node[rdot] at (2,-2.4) {};
	       \node at (2,-0.4) {$\cdots$};
	       \draw[blue] (0.3,0) to[out=-90, in=-180] (1.1,-2.4) to[out=0, in=-90] (1.9, -1.2);
	       \draw[blue] (2.1,-1.2) to[out=-90, in=-180] (2.9,-2.4) to[out=0, in=-90] (3.7, 0);
	       \node[rbkhdot] at (2,-1.2) {};
	       \node at (2,-1.2) {$\scalemath{0.6}{2k+2}$};
	\end{tikzpicture}}
\end{equation}
On the other hand, if the the red $2k-$extvalent morphism is not on the same connected component as the trapping line, we can first apply \cref{morehighextreduct} to introduce a red and blue dot into the red $2k-$extvalent and then apply fusion to obtain
\begin{equation}
\label{samecompeq}
\raisebox{-8ex}{
    \begin{tikzpicture}[scale=0.65]
	\draw[dotted, thick] (2,-0.8) circle (60pt);
	       \draw[red] (0.9,0) to[out=-90, in=-180] (2,-1.2);
	       \draw[blue]  (2,-1.2) to[out=0, in=-90] (3.1, 0);
	       \node[rkhdot] at (2,-1.2) {};
	       \node at (2,-1.2) {$\s{2k}$};
	       \node at (2,-0.4) {$\cdots$};
	       \draw[blue] (0.3,0) to[out=-90, in=-180] (2,-2.4) to[out=0, in=-90] (3.7, 0);
	\end{tikzpicture}\raisebox{1.1cm}{$ \ = \ $}
	\begin{tikzpicture}[scale=0.65]
	\draw[dotted, thick] (2,-0.8) circle (60pt);
	       \draw[red] (0.9,0) to[out=-90, in=-180] (2,-1.2);
	       \draw[blue]  (2,-1.2) to[out=0, in=-90] (3.1, 0);
	       \draw[blue] (2, -1.2) -- (1.6,-2);
	       \draw[red] (2, -1.2) -- (2.4,-2); 
	       \node[bdot] at (1.6, -2) {};
	       \node[rdot] at (2.4, -2) {};
	       \node[rbkhdot] at (2,-1.2) {};
	       \node at (2,-1.2) {$\scalemath{0.6}{2k+2}$};
	       \node at (2,-0.4) {$\cdots$};
	       \draw[blue] (0.3,0) to[out=-90, in=-180] (2,-2.4) to[out=0, in=-90] (3.7, 0);
	\end{tikzpicture}\raisebox{1.1cm}{$ \ = \ $}
	\begin{tikzpicture}[scale=0.7]
	\draw[dotted, thick] (2,-0.8) circle (57pt);
	       \draw[red] (0.9,0) to[out=-90, in=-180] (2,-1.2);
	       \draw[blue]  (2,-1.2) to[out=0, in=-90] (3.1, 0);
	       \draw[blue] (2, -1.2) -- (1.25,-2.2);
	       \draw[red] (2, -1.2) -- (2.4,-2); 
	       \node[rdot] at (2.4, -2) {};
	       \node[rbkhdot] at (2,-1.2) {};
	       \node at (2,-1.2) {$\scalemath{0.6}{2k+2}$};
	       \node at (2,-0.4) {$\cdots$};
	       \node at (1.1,-1.4) {$ \scalemath{0.7}{\boxed{\rho_t}}$};
	       \draw[blue] (0.3,0) to[out=-90, in=-180] (2,-2.4) to[out=0, in=-90] (3.7, 0);
	\end{tikzpicture}\raisebox{1.1cm}{$ \ - \ $}
	\begin{tikzpicture}[scale=0.7]
	\draw[dotted, thick] (2,-0.8) circle (57pt);
	       \draw[red] (0.9,0) to[out=-90, in=-180] (2,-1.2);
	       \draw[blue]  (2,-1.2) to[out=0, in=-90] (3.1, 0);
	       \draw[blue] (2, -1.2) -- (1.25,-2.2);
	       \draw[red] (2, -1.2) -- (2.6,-1.8); 
	       \node[rdot] at (2.6, -1.8) {};
	       \node[rbkhdot] at (2,-1.2) {};
	       \node at (2,-1.2) {$\scalemath{0.6}{2k+2}$};
	       \node at (2,-0.4) {$\cdots$};
	       \node at (2,-2) {$ \scalemath{0.6}{\boxed{t(\rho_t)}}$};
	       \draw[blue] (0.3,0) to[out=-90, in=-180] (2,-2.4) to[out=0, in=-90] (3.7, 0);
	\end{tikzpicture}}
\end{equation}
Because the red $2k-$extvalent is not on the same connected component as the trapping line, the polynomials $\rho_t$ and $t(\rho_t)$ are free to slide all the way to the top of the diagram. As a result, locally we end up with diagrams that exactly look like the LHS of \cref{absorbtrap} and so we are done. \\

\textbf{Step 5: }Similarly, for $\hom^{2, \bullet}_{\Dext_\infty}(\Bs_\varnothing, \Bs_{\usmt{t}{2k}{}})$, the proof of \cref{algextindecomp} can be applied if we can show any possible morphism in $\hom^{2, \bullet}_{\Dext_\infty}(\Bs_\varnothing, \Bs_{\usmt{t}{2k}{}})$ can be written as a $R$ linear combination of diagrams of the form \cref{ext2lightleaf}. Again by cohomological degree reasons, any diagram in $\hom^{2, \bullet}_{\Dext_\infty}(\Bs_\varnothing, \Bs_{\usmt{t}{2k}{}})$ has exactly two subdiagrams from the list 
\[ \raisebox{2.5ex}{\rhzero} , \quad \raisebox{2.5ex}{\bhzero},\quad  \raisebox{-1ex}{ \begin{tikzpicture}[scale=0.6]
	       \draw[red] (1.1,-0.1) -- (2,-1);
	       \draw[blue] (1.5,-0.1) -- (2,-1);
	       \draw[violet] (2.9,-0.1) -- (2,-1);
	       \draw[blue] (1.1,-2) -- (2,-1);
	       \draw[red] (1.5,-2) -- (2,-1);
	       \draw[violet] (2.9,-2) -- (2,-1);
	       \node at (2.1,-0.4) {$\scalemath{0.8}{\ldots}$};
	       \node at (2.1,-1.6) {$\scalemath{0.8}{\ldots}$};
	       \node[rhdot, minimum size=3.4mm] at (2,-1) {};
	       \node at (2,-1) {$\scalemath{0.7}{2k}$};
	    \end{tikzpicture}}  \]
Using the diagrammatic reductions above, we can assume both subdiagrams are at the bottom of the diagram. Using fusion and $4-$Ext Reduction \cref{diag4extreduct} we can assume that both subdiagrams are on the same connected component. Now Hochschild Annihilation \cref{eqn:hdot-square} and Ext Valent Annihilation \cref{extvalentann} show that the subdiagrams must be distinct. One of those subdiagrams must be a red $2k-$extvalent morphism, as red and blue Hochschild dots cannot be on the same connected component without the presence of a red $2k-$extvalent morphism. But now we are done, as one can use $2-$color Hochschild Jumping \cref{2colorhochjump} to move the Hochschild dot so that our diagram is of the form 
\begin{center}
    \begin{tikzpicture}
\draw[blue] (0.9,-0.3) to[out=-90,in=-180] (1.55,-1.95) to[out=0, in=-90] (2,-1.6);
\draw[violet] (1.4, -0.9) to[out=-90, in=-180] (2, -1.6);
\draw[violet] (2.6, -0.9) to[out=-90, in=0] (2, -1.6);
\node at (2, -1.2) {$\scalemath{0.9}{\underline{v}}$};
\node[rkhdot] at (2, -1.6) {};
\node at (2,-1.6) {$\scalemath{0.7}{2j}$};
\node[bhdot] at (1.55, -1.95) {};
\node at (2.2,-0.65) [dashed, shape=rectangle,
                     draw=black, semithick,
                     text width=0.1\linewidth,
                     align=center] {$\phantom{L}\cdots\phantom{L}$};
\end{tikzpicture} \qquad \quad \raisebox{0.8cm}{where $|mc(t\underline{v})|=2j$}
\end{center}
and since the top portion of the diagram is a morphism in $\mathscr{D}_\infty$, we can proceed as in Step 3. 
\end{proof}

\section{Computations and Relations for $m_{st}<\infty$}
\label{finitesect}

In the finite case we will have that $(ts)^{m_{st}}=1$ and in addition to \cref{assump0},\cref{assump1}, \cref{assump2} we will also add the following two assumptions
\begin{assumption}
 $\hf$ is a faithful realization of the finite dihedral group $W_{m_{st}}$. As noted in \cite{DC} Section 1.3, this means that $q^{2m_{st}}=1$ where $q$ is a primitive $2m_{st}$ root of unity and $[m_{st}]=0$.
\end{assumption}

\begin{assumption}[lesser invertibility]
For all $k<m_{st}$, $[k]$ is invertible in $\kb$. 
\end{assumption}

As before, we can take $\Lambda_{st}=\kb$. As an aside in \cite{DC} Elias assumes

\begin{assumption}[Local non-degeneracy]
Whenever $m_{st}<\infty$, $4-a_{ts}a_{st}$ is invertible in $\kb$.
\end{assumption}

and this will imply \cref{assump1}, as one can then take(in \cite{DC} this was denoted $\omega_s$)
\[ \rho_s=\frac{2\alpha_s-a_{ts}\alpha_t}{4-a_{ts}a_{st}} \]

Recall that now there is a $2m_{st}-$valent morphism which we denote by $v_t^s(m_{st})$ sending
\[\raisebox{4ex}{$v_t^s(m_{st})=$} \begin{tikzpicture}[scale=0.75]
	       \draw[red] (1.1,-0.1) -- (2,-1);
	       \draw[blue] (1.5,-0.1) -- (2,-1);
	       \draw[myyellow] (2.9,-0.1) -- (2,-1);
	       \draw[blue] (1.1,-2) -- (2,-1);
	       \draw[red] (1.5,-2) -- (2,-1);
	       \draw[violet] (2.9,-2) -- (2,-1);
	       \node at (2.1,-0.4) {$\ldots$};
	       \node at (2.1,-1.6) {$\ldots$};
	    \end{tikzpicture}  \qquad \qquad \raisebox{4ex}{\begin{tikzcd}
	    1(\usmt{s}{m_{st}}{})\\
	    1(\usmt{t}{m_{st}}{})\arrow[u]
	    \end{tikzcd}} \]
Outside of $m_{st}=2$, this isn't enough to completely define the $2m_{st}$ valent morphism, but this will suffice for our purposes. Many of the results from the affine case will still hold such as the following analogue of \cref{keycor} 

\begin{lemma}
\label{kerfinitelem}
When $m_{st}<\infty$, the image of $\rhosw: \bsw\to \bsw$ is a free right $R$ module and the kernel is also a free right $R$ module with basis given by $\set{\ll{f} \, | r(\underline{f})=\id \textnormal{ or }t }$.
\end{lemma}
\begin{proof}
\cref{keylemma} holds regardless of what $m_{st}$ is and so in the proof of \cref{keycor} we just need to show that \cref{stsrhos} isn't 0 for $s(ts)^m\neq \id \textnormal{ or }t $ and \cref{tsrhos} isn't 0 for $(ts)^m\neq \id \textnormal{ or }t $. By \cref{vanishingeq} and since $q$ is a primitive $2m_{st}$ root of unity we see that
\[ \cref{stsrhos}=0 \iff m=m_{st}\ell -1 \qquad \quad \cref{tsrhos}=0 \iff m=m_{st}\ell \qquad \ell \in \mathbb{Z}^+\]
But we have that
\[ s(ts)^{m_{st}\ell -1}=t (ts)^{m_{st}\ell}=t\qquad \quad (ts)^{m_{st}\ell}=\id  \]
and thus our claim is proven. 
\end{proof}

The rest of \cref{maincomptutesec} holds mutatis mutandis and so we also have that

\begin{theorem}
\label{maincohoisofin}
When $m_{st}<\infty$ and $\Lambda_{st}=\kb$ we have an isomorphism of right $R-$modules, 
\begin{align*}
\ext^{0, \bullet}_{R^e}(B_t, \mathrm{BS}(\underline{w}))&\cong \ker\rhosw \underline{1 } (-1)\cong \ker\rhosw (-1) \\
\ext^{1, \bullet}_{R^e}(B_t, \mathrm{BS}(\underline{w}))&\cong
\ker\rhosw\underline{\rho_t t(\rho_t)} (-1)\oplus  \widetilde{\Db(\ker\rhosw )}\underline{\rho_s}(-1)\\
&\cong  \ker\rhosw (3)\oplus  \widetilde{\Db(\ker\rhosw )}(1)  \\
\ext^{2, \bullet}_{R^e}(B_t, \mathrm{BS}(\underline{w}))&\cong \Db(\ker\rhosw) \underline{\rho_s \wedge \rho_t t(\rho_t)}(-1) \cong \Db(\ker\rhosw) (5) 
\end{align*}
\end{theorem}

 \begin{remark}
 The statements above for the finite case mirror the affine case word for word, but there's a subtle difference. Namely in the finite case, for a given expression $\underline{w}$ and subexpression $\underline{f}$, the light leaf morphism $\mathrm{L}_{  \underline{w}, \underline{f} }$ is possibly different than the affine case. For example, consider $\underline{w}=(s,t,s,t,s)$ and $\underline{f}=(1,1,1,1,1)$. When $m_{st}=\infty$ The light leaf algorithm then returns the diagram on the left below
 \begin{center}
     \raisebox{0.8cm}{\scalemath{2.1}{\rzero \  \bzero \ \rzero \  \bzero \ \rzero}  \qquad \qquad \qquad}
     \begin{tikzpicture}[scale=0.45, thick]
           \draw[red] (1.1,2) -- (2,1);
           \draw[blue] (2,2) -- (2,1);
           \draw[red] (2.9,2) -- (2,1);
           \draw[blue] (1.1,0) -- (2,1);
           \draw[red] (2,0) -- (2,1);
           \draw[blue] (2.9,0) -- (2,1);
           \draw[blue] (2.9,0) to[out=-90, in=180] (3.2,-0.5) to[out=0, in=-90] (3.5,2);
           \draw[red] (2,0) to[out=-90, in=180] (3,-1) to[out=0, in=-90] (4,2);
           \draw[blue] (1.1,0) --(1.1, -1);
     \end{tikzpicture}
 \end{center}
 while for $m_{st}=3$ the light leaf algorithm returns the diagram on the right above. In particular $r((1,1,1,1,1))=t$ and so by \cref{kerfinitelem} for $m_{st}=3$, $\mathrm{L}_{  (s,t,s,t,s), (1,1,1,1,1) }\in \ker \rho_s^e((s,t,s,t,s))$ and so by \cref{maincohoisofin} corresponds to a morphism in $\ext^{0, \bullet}_{R^e}(B_t, \mathrm{BS}(s,t,s,t,s))$. Of course this is nothing more the $6-$valent morphism, but where 2 of the strands have been twisted up.\\
 
To summarize, in the finite case $\ker \rhosw$ is possibly larger and consequently the groups $\ext^{\bullet, \bullet}_{R^e}(B_t, \mathrm{BS}(\underline{w}))$ are also possibly larger.
 \end{remark}

\subsection{Dimension 1 Calculations}

\begin{lemma}
\label{lowestdeglemfin}
Suppose $m_{st}<\infty$ and suppose $\underline{w}$ is a non repeating expression. If $|\underline{w}|<2m_{st}$, then the lowest internal degree element in $\hom_{R^e}(R, \bsw)$ is of degree 1, 2 if $|\underline{w}|$ is odd, even respectively. If $|\underline{w}|\ge 2m_{st}$ the lowest internal degree element in $\hom_{R^e}(R, \bsw )$ is of degree $2m_{st}-|\underline{w}|$.
\end{lemma}
\begin{proof}
For $|\underline{w}|<2m_{st}$, $ \grk \ \hom_{R^e}(R, \bsw)$ agrees with the affine case (one can use Soergel's Hom formula and then use Lemma 3.19 in \cite{EMTW20} and biadjointness of $b_s$ with the standard form so that one never needs to use the braid relation in the finite Hecke algebra) so this follows from \cref{lowestdeglem}. Now, let $\mathrm{LD}(\bs(\underline{v}), \bs(\underline{u}))$ be the lowest degree morphism in  $\hom_{R^e}(\bs(\underline{v}), \bs(\underline{u}))$. Given any morphism $\hom_{R^e}(B_s, \bs(\underline{u}) )$ we can produce a morphism in $\hom_{R^e}(R, \bs(\underline{u}) )$ by adding a dot to the bottom of $B_s$. As a result, we see that
\[ \mathrm{LD}(R,B_s\otimes_R \bs(\underline{w}))=\mathrm{LD}(B_s, \bs(\underline{u}))\ge \mathrm{LD}(R, \bs(\underline{u}))-1 \]
Therefore each time we tensor with $B_s$ or $B_t$ the lowest degree drops by at most 1. It follows that when $|\underline{w}|\ge 2m_{st}$, the lowest degree element in $\hom_{R^e}(R, \bsw )$ is at least degree $2m_{st}-|\underline{w}|$. One can actually show it's also at most $2m_{st}-|\underline{w}|$ by explicity producing a morphism in $\hom_{R^e}(R, \bsw )$ of this degree which we leave to the reader. For a hint, look at the RHS of \cref{8valentto6examp}.
\end{proof}

\begin{prop}
\label{1dimcritfinite}
For $2<m_{st}<\infty$, $\ext^{1, -(|\underline{w}|+1)}_{R^e}(B_t, \mathrm{BS}(\underline{w}))$ is a 1 dimensional $\kb$ module when $|m(t, \underline{w})|\ge 4$
\end{prop}
\begin{proof}
The proof proceeds similarly to \cref{1dimcriterion} as by \cref{maincohoisofin} we still have the decomposition
\[\ext^{1, -(|\underline{w}|+1)}_{R^e}(B_t, \mathrm{BS}(\underline{w})))\cong \ker\rhosw (3)_{-|\underline{w}|-1}\oplus  \widetilde{\Db(\ker\rhosw )}(1)_{-|\underline{w}|-1}\]
We then have that an analogue of \cref{lowestdeglem2} holds by noting that the proof of \cref{lowestdeglem2} still applies for $m_{st}<\infty$ except we need to replace \cref{lowestdeglem} with \cref{lowestdeglemfin} in the last step. As a result, when $4\le |m(t, \underline{w})|<2m_{st}$ we are in the same situation as in \cref{1dimcriterion} while for $|m(t, \underline{w})|\ge 2m_{st}$, one needs to check
\[ -1-|\underline{w}|+|m(t, \underline{w})|+2m_{st}-|m(t, \underline{w})|-4>-|\underline{w}|-1 \]
which is true when $m_{st}>2$ as desired. 
\end{proof}

\begin{corollary}
\label{stillholdfinite}
All the generators/relations in \cref{newgensect} for $m_{st}=\infty$ still exist/hold for $2<m_{st}<\infty$. 
\end{corollary}

We will deal with the case $m_{st}=2$ separately below.

\subsection{The Case $m_{st}=2$}
By assumption, $[m_{st}]=0$ and thus as $m_{st}=2$ we see that $[2]=a_{st}=a_{ts}=0$. As a result, we see that $s(\alpha_t)=\alpha_t$ and consequently $\rho_tt(\rho_t)\in R^{s,t}$ and therefore $\rho_tt(\rho_t)^e=0$. As a result, ${}_{tt}^0\psi_s^u$ are now chain maps (see \cref{keydoublecomplex}) for any $u\in \bs(\un{w})$.
${}_{tt}^0\psi_s^u$ are now chain maps
\begin{theorem}
When $m_{st}=2$, we have an \underline{equality} of right $R$ modules
\begin{align*}
\ext^{0, \bullet}_{R^e}(B_t, \mathrm{BS}(\underline{w}))&= \ker\rhosw \underline{1 } (-1)\\
\ext^{1, \bullet}_{R^e}(B_t, \mathrm{BS}(\underline{w}))&=
\ker\rhosw\underline{\rho_t t(\rho_t)} (-1)\oplus  \Db(\ker\rhosw )\underline{\rho_s}(-1)\\
\ext^{2, \bullet}_{R^e}(B_t, \mathrm{BS}(\underline{w}))&= \Db(\ker\rhosw ) \underline{\rho_s \wedge \rho_t t(\rho_t)}(-1)
\end{align*}
\end{theorem}

Specifically, the difference between the above theorem and \cref{maincohoisofin} is that the $\widetilde{\Db(\ker\rhosw )}\underline{\rho_s}(-1)$ part of $\ext^{1, \bullet}_{R^e}(B_t, \mathrm{BS}(\underline{w}))$ can now be described as $\sbrac{{}_{tt}^0\psi_s^u}$ for $u\in \bsw$. One computes that 
\[ {}_{tt}^0\psi_s^{1\otimes_s 1}=\overset{ \rhunit}{\bcounit} \implies{}_{tt}^0\psi_s^{1\otimes_s 1 \otimes_t 1} =\raisebox{-3ex}{\begin{tikzpicture}[scale=0.45]
	       \draw[red] (1.1,0.1) -- (1.1,-1);
	       \draw[blue] (2,0.1) -- (2,-1);
	       \draw[blue] (2,-2) -- (2,-1);
	       \node[rhdot] at (1.1,-1) {};
	\end{tikzpicture}} \implies  {}_{tt}^0\psi_s^{1\otimes_s 1 \otimes_t 1\otimes_s 1}=\raisebox{-3ex}{\begin{tikzpicture}[scale=0.45]
	       \draw[blue] (1.1,-0.1) --(2,-1);
	       \draw[red] (2.9,-0.1) -- (2,-1);
	       \draw[blue] (2,-1) -- (2.9, -1.9);
	       \draw[red] (0.8,-0.1) to[out=-90,in=-180] (1.4,-1.75) to[out=0, in=-90] (2,-1);
	       \draw[blue] (2.9,-1.9) --(2.9,-2.9);
	       \node[rhdot] at (1.4,-1.75) {};
	    \end{tikzpicture}} \]
when $m_{st}=2$. In other words, blue and red Hochschild dots, along with the generating morphisms of $\bsbim(\hf, W_2)$, generate all of the morphisms in $\bsbimext(\hf, W_2)$. $\dim_\kb \ext_{R^e}^{1,-4}(B_t, B_sB_tB_s)\neq 1$ when $m_{st}=2$ and one has that
\[ \raisebox{-3ex}{\begin{tikzpicture}[scale=0.45]
	       \draw[blue] (1.1,-0.1) --(2,-1);
	       \draw[red] (2.9,-0.1) -- (2,-1);
	       \draw[blue] (2,-1) -- (2.9, -1.9);
	       \draw[red] (0.8,-0.1) to[out=-90,in=-180] (1.4,-1.75) to[out=0, in=-90] (2,-1);
	       \draw[blue] (2.9,-1.9) --(2.9,-2.9);
	       \node[rhdot] at (1.4,-1.75) {};
	    \end{tikzpicture}} \neq \raisebox{-3ex}{\begin{tikzpicture}[scale=0.6]
	       \draw[red] (1.1,0.1) to[out=-90, in=180] (2,-1);
	       \draw[blue] (2,0.1) -- (2,-1);
	       \draw[red] (2.9,0.1) to[out=-90, in=0](2,-1);
	       \draw[blue] (2,-2) -- (2,-1);
	       \node[rhdot] at (2,-1) {};
	\end{tikzpicture}} \]
\begin{lemma}
For $m_{st}=2$, we have the following relation in $\ext^{1,-4}_{R^e}(B_s B_t, B_t B_s)$
\begin{equation}
\label{2hochslide}
    \begin{tikzpicture}[scale=0.6]
	       \draw[blue] (1.2,-0.1) -- (2,-1);
	       \draw[red] (2.8,-0.1) -- (1.2,-2);
	       \draw[blue] (2.8,-2) -- (2,-1);
	       \node[rhdot] at (1.5,-1.625) {};
	    \end{tikzpicture}\raisebox{0.5cm}{$  =$}
	\begin{tikzpicture}[scale=0.6]
	       \draw[blue] (1.2,-0.1) -- (2,-1);
	       \draw[red] (2.8,-0.1) -- (1.2,-2);
	       \draw[blue] (2.8,-2) -- (2,-1);
	       \node[rhdot] at (2.4,-0.55) {};
	    \end{tikzpicture}
\end{equation}
\end{lemma}
\begin{proof}
Because $[2]=0$, we have that $\rhunit(\widetilde{\tau_t}^1(\urtt\boxtimes 1\otimes 1))=0$ and one can then compute that the following equalities hold
\[  \begin{tikzpicture}[scale=0.6]
	       \draw[blue] (1.1,-0.1) --(2,-1);
	       \draw[red] (2.9,-0.1) -- (2,-1);
	       \draw[blue] (2,-1) -- (2.9, -1.9);
	       \draw[red] (0.8,-0.1) to[out=-90,in=-180] (1.4,-1.75) to[out=0, in=-90] (2,-1);
	       \draw[blue] (2.9,-1.9) --(2.9,-2.9);
	       \node[rhdot] at (1.4,-1.75) {};
	    \end{tikzpicture} \raisebox{4ex}{$= \ {}_{tt}^{0 }\psi_{s}^{ 1\otimes_s 1\otimes_t 1\otimes_s 1}$ = \ } \begin{tikzpicture}[scale=0.6]
	       \draw[red] (1.1,-0.1) --(2,-1);
	       \draw[blue] (2.9,-0.1) -- (2,-1);
	       \draw[blue] (2,-1) -- (1.1, -1.9);
	       \draw[red] (3.2,-0.1) to[out=-90,in=0] (2.6,-1.75) to[out=-180, in=-90] (2,-1);
	       \draw[blue] (1.1,-1.9) --(1.1,-2.9);
	       \node[rhdot] at (2.6,-1.75) {};
	    \end{tikzpicture} \]
and so the lemma follows from rotating the top leftmost red strand down and applying rotation invariance of the 4 valent morphism. 
\end{proof}

\begin{corollary}
We have the following relation in $\ext^{2,-4}_{R^e}(B_t B_s, B_s B_t)$
\begin{equation}
    \begin{tikzpicture}[scale=0.6]
	       \draw[blue] (1.2,-0.1) -- (2,-1);
	       \draw[red] (1.2,-2) -- (1.8,-1.25);
	       \draw[blue] (2.8,-2) -- (2,-1);
	       \node[rhdot] at (1.8,-1.25) {};
	    \end{tikzpicture}\raisebox{0.5cm}{$  =$}
	\begin{tikzpicture}[scale=0.6]
	       \draw[blue] (1.2,-0.1) -- (2,-1);
	       \draw[red] (1.2,-2) -- (2.4,-0.55);
	       \draw[blue] (2.8,-2) -- (2,-1);
	       \node[rhdot] at (2.4,-0.55) {};
	    \end{tikzpicture}
\end{equation}
\end{corollary}
\begin{proof}
Follows from adding a dot to the top right red strand of \cref{2hochslide} and applying two-color dot contraction.
\end{proof}

\subsection{The Case $m_{st}>2$}

Besides the generators and relations from \cref{newgensect}, the presence of the generator $v_t^s(m_{st})$ will result in additional relations in $\bsbimext(\hf, W_{m_{st}})$ not present in  $\bsbimext(\hf, W_\infty)$, namely

\begin{lemma}
\label{2mstabsorbalg}
Suppose that $m_{st}>3$, or when $m_{st}=3$, we have that $\underline{v}$ contains at least one $s$. Then we have the following relation in $\ext_{R^e}^{1, -(m_{st} +|\underline{v}|+1)}(B_t, \bs( \usmt{t}{m_{st}}{} , \underline{v} ) \, )$.
\[ (v_s^t(m_{st})\otimes_R \id_{\underline{v}} )  \circ \,  \Phi_{t}^{ ( \usmt{s}{m_{st}}{} , \underline{v} ) }=   \Phi_{t}^{ ( \usmt{t}{m_{st}}{} , \underline{v} ) } \]
Diagrammatically this will be of the form
\begin{center}
    \raisebox{-0.7cm}{\begin{tikzpicture}[scale=0.7]
          \draw[blue] (0.8,1.5) -- (1.4, 0.8);
           \draw[red] (1.4,0.8) -- (1.1, 1.5);
           \draw[myyellow] (2,1.5) -- (1.4, 0.8);
           \draw[red] (0.8,0.1) -- (1.4, 0.8);
           \draw[blue] (1.1,0.1) -- (1.4, 0.8);
           \draw[violet] (2,0.1) -- (1.4, 0.8);
           \draw[violet] (2, -1) to[out=40, in=-90] (2.2, 1.5);
	       \draw[red] (0.8,0.1) to[out=-90,in=-180] (2,-1);
	       \draw[blue] (1.1,0.1) to[out=-90,in=135] (2,-1);
	       \draw[violet] (2,0.1) to[out=-90,in=90] (2,-1);
	       \draw[violet] (3.2,0.1) to[out=-90,in=0] (2,-1);
	       \draw[violet] (3.2,0.1) -- (3.2,1.5);
	       \foreach \x in {2} \draw[blue] (\x,-1) -- (\x,-2);
	       \node[rhdot] at (2,-1) {};
	       \node at (1.5,1.3) {$\s{\ldots}$};
	       \node at (1.5, 0.2) {$\s{\ldots}$};
	       \node at (2.7, 0.2) {$\underline{v}$};
	    \end{tikzpicture}}=\raisebox{-0.7cm}{\begin{tikzpicture}[scale=0.7]
           \draw[blue] (0.8,0.1) -- (0.8, 1.5);
           \draw[red] (1.1,0.1) -- (1.1, 1.5);
           \draw[myyellow] (2,0.1) -- (2, 1.5);
           \draw[violet] (2, -1) to[out=40, in=-90] (2.2, 1.5);
	       \draw[blue] (0.8,0.1) to[out=-90,in=-180] (2,-1);
	       \draw[red] (1.1,0.1) to[out=-90,in=135] (2,-1);
	       \draw[myyellow] (2,0.1) to[out=-90,in=90] (2,-1);
	       \draw[violet] (3.2,0.1) to[out=-90,in=0] (2,-1);
	       \draw[violet] (3.2,0.1) -- (3.2,1.5);
	       \foreach \x in {2} \draw[blue] (\x,-1) -- (\x,-2);
	       \node[rhdot] at (2,-1) {};
	       \node at (1.5, 0.2) {$\s{\ldots}$};
	       \node at (2.7, 0.2) {$\underline{v}$};
	    \end{tikzpicture}}
\end{center}
\end{lemma}
\begin{proof}
The conditions in the lemma will imply that $|m(t,  \usmt{t}{m_{st}}{} , \underline{v} )|\ge 4$ and so $\dim_\kb \ext_{R^e}^{1, -(m_{st} +|\underline{v}|+1)}(B_t, \bs( \usmt{t}{m_{st}}{} , \underline{v} ) \, )=1 $ by \cref{1dimcritfinite}. Now the lemma will follow by noting that $(v_s^t(m_{st})\otimes_R \id_{\underline{v}} )$ sends $1( \usmt{s}{m_{st}}{} , \underline{v} )$ to $1( \usmt{t}{m_{st}}{} , \underline{v} )$.
\end{proof}

\begin{lemma}
For $m_{st}=3$, we have the following relation in $\ext_{R^e}^{1, -4}(B_t, B_t B_s B_t)$.
\[ v_s^t(3)  \circ \,  \Phi_{t}^{ ( s,t, s ) }=   (\id_t \otimes_R \eta_s^{\ext}\otimes_R \id_t)\circ \delta_t -(\id_t \otimes_R \eta_s\otimes_R \id_t)\circ \delta_t\circ \phi_t  \]
\begin{center}
\raisebox{-0.2cm}{
    \begin{tikzpicture}[scale=0.6]
           \draw[blue] (1.1,2) -- (2,1);
           \draw[red] (2,2) -- (2,1);
           \draw[blue] (2.9,2) -- (2,1);
           \draw[red] (1.1,0) -- (2,1);
           \draw[blue] (2,0) -- (2,1);
           \draw[red] (2.9,0) -- (2,1);
	       \draw[red] (1.1,0) -- (2,-1);
	       \draw[blue] (2,0) -- (2,-1);
	       \draw[red] (2.9,0) -- (2,-1);
	       \draw[blue] (2,-2) -- (2,-1);
	       \node[rhdot] at (2,-1) {};
	\end{tikzpicture}} \raisebox{0.8cm}{$= \ $}
	\raisebox{0.05cm}{
\begin{tikzpicture}[scale=0.6]
	       \draw[blue] (1.1,-0.1) to[out=-90, in=-180] (2,-1.5) to[out=0, in=-90] (2.9, -0.1);
	       \draw[red] (2,-0.1) -- (2,-0.9);
	       \draw[blue] (2,-2.5) -- (2,-1.5);
	       \node[rhdot] at (2,-0.9) {};
	\end{tikzpicture}} \raisebox{0.8cm}{$ \ - \ $}
\begin{tikzpicture}[scale=0.6]
	       \draw[blue] (1.1,-0.1) to[out=-90, in=-180] (2,-1.5) to[out=0, in=-90] (2.9, -0.1);
	       \draw[red] (2,-0.1) -- (2,-0.9);
	       \draw[blue] (2,-2.5) -- (2,-1.5);
	       \node[bhdot] at (2,-2) {};
	       \node[rdot] at (2,-0.9) {};
	\end{tikzpicture}
\end{center}
\end{lemma}
\begin{proof}
  Apply $(\id_{t,s,t}\otimes \epsilon_s)$ to the equality
\[ (v_s^t(3)\otimes_R \id_s )  \circ \,  \Phi_{t}^{ (s,t,s , s ) }=   \Phi_{t}^{ (t,s,t , s ) } \]
from \cref{2mstabsorbalg} and simplify using \cref{4extreduct}.
\end{proof}

Another difference between $\bsbimext(\hf, W_{m_{st}})$ and $\bsbimext(\hf, W_\infty)$ is that in $\bsbimext(\hf, W_{m_{st}})$ for any $k$, the red/blue $2k-$extvalent morphisms can be generated from the $2m_{st}-$extvalent morphism and generators of $\bsbim(\hf, W_{m_{st}})$. We already know that when $k< m_{st}$, then the red $2k-$extvalent morphism can be obtained from the red $2m_{st}-$extvalent morphism by using \cref{newgenreduct}, while for $k>m_{st}$ we have

\begin{lemma}
$\Phi_t^{\usmt{s}{2k-1}{}} $ can be expressed as a composition of $\Phi_t^{\usmt{s}{2m_{st}-1}{}}$, comultiplication, and the $2m_{st}-$valent morphism when $k>m_{st}$. 
\end{lemma}
\begin{proof}
We first give an example in the case when $m_{st}=3$ and $2k-1=7$. We claim that
\begin{equation}
\label{8valentto6examp}
    \begin{tikzpicture}[scale=0.7]
	       \draw[red] (0.8,0.1) to[out=-90,in=-180] (2,-1);
	       \draw[blue] (1.2,0.1) to[out=-90,in=150] (2,-1);
	       \draw[red] (1.6,0.1) to[out=-90,in=120] (2,-1);
	       \draw[blue] (2,0.1) to[out=-90,in=90] (2,-1);
	       \draw[red] (2.4,0.1) to[out=-90,in=60] (2,-1);
	       \draw[blue] (2.7,0.1) to[out=-90,in=30] (2,-1);
	       \draw[red] (3.2,0.1) to[out=-90,in=0] (2,-1);
	       \foreach \x in {2} \draw[blue] (\x,-1) -- (\x,-2);
	       \node[rhdot] at (2,-1) {};
	    \end{tikzpicture}\raisebox{0.6cm}{$= \ $}
	    \raisebox{-0.7cm}{\begin{tikzpicture}[scale=0.7]
	    \draw[blue] (2, 1.3) to[out=40, in=-90] (2.3, 2.5);
	    \draw[red] (0.8,2.5) -- (1.4, 1.9);
           \draw[blue] (1.4,1.9) -- (1.4, 2.5);
           \draw[red] (2,2.5) -- (1.4, 1.9);
           \draw[blue] (0.8,1.3) -- (1.4, 1.9);
           \draw[red] (1.4,1.3) -- (1.4, 1.9);
           \draw[blue] (2,1.3) -- (1.4, 1.9);
          \draw[blue] (0.8,1.3) -- (1.4, 0.7);
           \draw[red] (1.4,0.7) -- (1.4, 1.3);
           \draw[blue] (2,1.3) -- (1.4, 0.7);
           \draw[red] (0.8,0.1) -- (1.4, 0.7);
           \draw[blue] (1.4,0.1) -- (1.4, 0.7);
           \draw[red] (2,0.1) -- (1.4, 0.7);
           \draw[red] (2, 0.1) to[out=40, in=-90] (2.6, 2.5);
	       \draw[red] (0.8,0.1) to[out=-90,in=-180] (2,-1);
	       \draw[blue] (1.4,0.1) to[out=-90,in=135] (2,-1);
	       \draw[red] (2,0.1) to[out=-90,in=90] (2,-1);
	       \draw[blue] (2.9,2.5) to[out=-90,in=45] (2,-1);
	       \draw[red] (3.2,0.1) to[out=-90,in=0] (2,-1);
	       \draw[red] (3.2,0.1) -- (3.2,2.5);
	       \foreach \x in {2} \draw[blue] (\x,-1) -- (\x,-2);
	       \node[rhdot] at (2,-1) {};
	    \end{tikzpicture}}
\end{equation}
This follows from the facts that the morphism above the red $6-$extvalent on the RHS sends $1(s,t,s,t,s)$ to $1(s,t,s,t,s,t,s)$ and $\dim_\kb \ext^{1,-8}_{R^e}(B_t, \bs(s,t,s,t,s,t,s))=1$ by \cref{1dimcritfinite}\\

In general, the red $2k-$extvalent can be realized as the postcomposition of the red $2m_{st}-$extvalent by an internal degree 0 morphism $\alpha$ that sends $1(\usmt{s}{2m_{st}-1}{})$ to $1(\usmt{s}{2k-1}{})$. One can easily generalize the RHS of \cref{8valentto6examp} to produce such an $\alpha$.
\end{proof}


\subsection{Computation of $\ext^{\bullet, \bullet}_{R^e}(R, B_w)$ for $m_{st}<\infty$}

\begin{theorem}
\label{finindecompcohograd}
Assuming lesser invertibility, for all $1\le k\le m_{st}$, there is an isomorphism of right $R$ modules. 
\begin{align*}
\ext^{0, \bullet}_{R^e}(R, B_{\smt{t}{k}{}})&\cong R(-k) \\
\ext^{1, \bullet}_{R^e}(R, B_{\smt{t}{k}{}})&\cong R(4-k)\oplus R(k) \\
\ext^{2, \bullet}_{R^e}(R, B_{\smt{t}{k}{}})&\cong R(k+4)
\end{align*}
\end{theorem}
\begin{proof}
Same as \cref{indecompcohograd} where \cref{maincohoiso} is replaced by \cref{maincohoisofin}.
\end{proof}

\begin{theorem}
\label{finalgextdecomp}
Assuming lesser invertibility, for all $1\le k\le m_{st}$, we have an isomorphism of right $R-$modules
\begin{align*}
\ext^{0, \bullet}_{R^e}(R, B_{\smt{t}{k}{}})&\cong  \jw_{\usmt{t}{k}{} }\circ \bunit \runit \ldots \vunit\  R \\
\ext^{1, \bullet}_{R^e}(R, B_{\smt{t}{k}{}})&\cong  \jw_{\usmt{t}{k}{} }\circ  \bhunit \runit \ldots \vunit\ R\oplus  \jw_{\usmt{t}{k}{} }\circ  \ 
\raisebox{-0.2cm}{\begin{tikzpicture}[scale=0.6]
	       \draw[red] (1.1,-0.1) to[out=-90,in=-180] (2,-1);
	       \draw[violet] (2.9,-0.1) to[out=-90,in=0] (2,-1);
	       \draw[blue] (0.9,-0.1) to[out=-90,in=-180] (1.45,-1.35) to[out=0, in=-90] (2,-1);
	       \node[rhdot] at (2,-1) {};
	       \node at (2,-0.4) {$\scalemath{0.9}{\usmt{s}{k-1}{}}$};
	    \end{tikzpicture}} \ R\\
\ext^{2, \bullet}_{R^e}(R, B_{\smt{t}{k}{}})&\cong \ \jw_{\usmt{t}{k}{} }\circ \raisebox{-0.2cm}{\begin{tikzpicture}[scale=0.6]
	       \draw[red] (1.1,-0.1) to[out=-90,in=-180] (2,-1);
	       \draw[violet] (2.9,-0.1) to[out=-90,in=0] (2,-1);
	       \draw[blue] (0.9,-0.1) to[out=-90,in=-180] (1.45,-1.35) to[out=0, in=-90] (2,-1);
	       \node[rhdot] at (2,-1) {};
	       \node[bhdot] at (1.45,-1.35) {};
	       \node at (2,-0.4) {$\scalemath{0.9}{\usmt{s}{k-1}{}}$};
	    \end{tikzpicture}} \ R
\end{align*}
\end{theorem}
\begin{proof}
When $k\le m_{st}$, the light leaf morphisms $\mathrm{L}_{  \usmt{s}{k-1}{}, \underline{f} }$ appearing in the proof in \cref{algextindecomp} coincide with the light leaf morphisms in the affine case, and so the proof of \cref{algextindecomp} still applies with \cref{indecompcohograd} replaced by \cref{finindecompcohograd}.
\end{proof}

\section{Ext Dihedral  Diagrammatics: $m_{st}<\infty$ }
\label{diagramfinitesect}

For this section let $\mathscr{D}_{m_{st}}:=\mathscr{D}_{m_{st}}(\hf, W_{m_{st}})$ the $2-$color diagrammatic Hecke category as defined in \cite{DC}. 

\begin{definition}
Let $\Dext_{m_{st}}:=\Dext_{m_{st}}(\hf, W_{m_{st}})$ be the strict $\kb$ linear supermonoidal category associated to a realization $\hf$ of the finite dihedral group $(W_{m_{st}}, S)$, where $S=\set{s,t}$ is the set of simple reflections, satisfying \cref{assump1} and \cref{assump2} defined as follows. 
\begin{itemize}
    \item \textbf{Objects} are the same as in $\Dext_\infty$ as defined in \cref{affinedef}.
    \item \begin{enumerate}[(a)]
        \item \textbf{Morphism Spaces when $m_{st}=2$} are again bigraded $\kb$ modules of total degree $(\ell,n)$ where $\ell$ is the cohomological degree and $n$ is the internal or Soergel degree. They are generated by 
        \begin{center}
        \begin{tabular}{c|c|c|c}
            generator &   \raisebox{-2ex}{\begin{tikzpicture}[scale=0.55]
	       \draw[blue] (1.2,-0.1) -- (2,-1);
	       \draw[red] (2.8,-0.1) -- (1.2,-2);
	       \draw[blue] (2.8,-2) -- (2,-1);
	    \end{tikzpicture}} & \raisebox{-2ex}{\begin{tikzpicture}[scale=0.55]
	       \draw[red] (1.2,-0.1) -- (2,-1);
	       \draw[blue] (2.8,-0.1) -- (1.2,-2);
	       \draw[red] (2.8,-2) -- (2,-1);
	    \end{tikzpicture}} \\[2ex]
            name  & $4-$valent & $4-$valent   \\
            bidegree& $(0,0)$  & $(0,0)$
        \end{tabular}
    \end{center}
        along with the generating morphisms in \cref{affinedef} \underline{except} for the red $2k-$extvalent morphisms.\\
        \textbf{Relations in $\Dext_2$} are as follows. All relations in $\mathscr{D}_2$ will be satisfied (see \cite{DC} Section 6.2) and the generating morphisms for the color $s$ will satisfy the relations of $\Dext_s$ as specified in \cite {M} Section 4, and likewise for the color $t$ and $\Dext_t$. We will also have the following additional relation.\\
        
    \textbf{4$-$valent Hochschild Sliding:}
    \begin{equation}
    \label{4valhochslide}
    \raisebox{-4ex}{
    \begin{tikzpicture}[scale=0.7]
	       \draw[blue] (1.2,-0.1) -- (2,-1);
	       \draw[red] (2.8,-0.1) -- (1.2,-2);
	       \draw[blue] (2.8,-2) -- (2,-1);
	       \node[rhdot] at (1.5,-1.625) {};
	    \end{tikzpicture}\raisebox{0.5cm}{$  =$}
	\begin{tikzpicture}[scale=0.7]
	       \draw[blue] (1.2,-0.1) -- (2,-1);
	       \draw[red] (2.8,-0.1) -- (1.2,-2);
	       \draw[blue] (2.8,-2) -- (2,-1);
	       \node[rhdot] at (2.4,-0.55) {};
	    \end{tikzpicture}}
\end{equation}
    \item \textbf{Morphism Spaces when $m_{st}>2$} are bigraded as in the case $m_{st}=2$ above. They are generated by 
    \begin{center}
        \begin{tabular}{c|c|c|c}
            generator &   \raisebox{-2ex}{\begin{tikzpicture}[scale=0.6]
	       \draw[red] (1.1,-0.1) -- (2,-1);
	       \draw[blue] (1.5,-0.1) -- (2,-1);
	       \draw[violet] (2.9,-0.1) -- (2,-1);
	       \draw[blue] (1.1,-2) -- (2,-1);
	       \draw[red] (1.5,-2) -- (2,-1);
	       \draw[violet] (2.9,-2) -- (2,-1);
	       \node at (2.1,-0.4) {$\scalemath{0.8}{\ldots}$};
	       \node at (2.1,-1.6) {$\scalemath{0.8}{\ldots}$};
	    \end{tikzpicture}}  & \raisebox{-2ex}{\begin{tikzpicture}[scale=0.6]
	       \draw[blue] (1.1,-0.1) -- (2,-1);
	       \draw[red] (1.5,-0.1) -- (2,-1);
	       \draw[violet] (2.9,-0.1) -- (2,-1);
	       \draw[red] (1.1,-2) -- (2,-1);
	       \draw[blue] (1.5,-2) -- (2,-1);
	       \draw[violet] (2.9,-2) -- (2,-1);
	       \node at (2.1,-0.4) {$\scalemath{0.8}{\ldots}$};
	       \node at (2.1,-1.6) {$\scalemath{0.8}{\ldots}$};
	    \end{tikzpicture}} \\[2ex]
	    name & $2m_{st}-$valent & $2m_{st}-$valent \\
	    bidegree & $(0,0)$ & $(0,0)$ 
	    \end{tabular}
	    \end{center}
	    along with \underline{all} the generating morphisms in \cref{affinedef}.\\
	   \textbf{Relations in $\Dext_{m_{st}}$} are as follows. All relations in $\mathscr{D}_{m_{st}}$ will be satisfied. The generating morphisms above will satisfy \underline{all} relations in \cref{diagraminftysect} along with the relation.\\
	   
	   \textbf{$2m_{st}$ Absorption($m_{st}>2$)}
	   \begin{equation}
	   \label{2mstabsorb}
    \raisebox{-1.4cm}{\begin{tikzpicture}[scale=0.8]
	       \draw[red] (1.1,0) -- (2,-1);
	       \draw[blue] (1.5,0) -- (2,-1);
	       \draw[myyellow] (2.9,0) -- (2,-1);
	       \draw[blue] (1.1,-2) -- (2,-1);
	       \draw[red] (1.5,-2) -- (2,-1);
	       \draw[violet] (2.9,-2) -- (2,-1);
	       \node at (2.1,0) {$\ldots$};
	       \node at (2.1,-1.6) {$\ldots$};
	       \draw[blue] (1.1,1.9) -- (2,1);
	       \draw[red] (1.5,1.9) -- (2,1);
	       \draw[violet] (2.9,1.9) -- (2,1);
	       \draw[red] (1.1,0) -- (2,1);
	       \draw[blue] (1.5,0) -- (2,1);
	       \draw[myyellow] (2.9,0) -- (2,1);
	       \node at (2.1,1.7) {$\ldots$};
	       \node[rkhdot] at (2,1) {};
	       \node at (2,1) {$\scalemath{0.8}{2k}$};
	    \end{tikzpicture}\raisebox{1.6cm}{\quad = \quad }
	  \raisebox{0.65cm}{
	 \begin{tikzpicture}[scale=1.1]
	       \draw[red] (1.5,-0.1) -- (2,-1);
	       \draw[myyellow] (2.5,-0.1) -- (2,-1);
	       \draw[red] (1.5,-2) -- (2,-1);
	       \draw[myyellow] (2.5,-2) -- (2,-1);
	       \draw[blue] (1.2,-1) -- (2,-1);
	       \draw[blue] (1.2,-0.1) -- (1.2,-2);
	       \draw[violet] (2.8,-1) -- (2,-1);
	       \draw[violet] (2.8,-0.1) -- (2.8,-2);
	       \node at (2,-0.4) {$\ldots$};
	       \node at (2,-1.6) {$\ldots$};
	       \node[rbkhdot] at (2,-1) {};
	       \node at (2,-1) {$\scalemath{0.65}{2k-2}$};
	    \end{tikzpicture}}} \qquad \qquad k > \frac{m_{st}+1}{2}
\end{equation}
	   
    \end{enumerate}
\end{itemize}
\end{definition}

\subsection{Equivalence}

\begin{theorem}
\label{finequiv}
Assume lesser invertibility. Define the functor $\Fext_{m_{st}}: \Dext_{m_{st}}\to  \bsbimext(\hf, W_{m_{st}}) $ where
\[ \Fext_{m_{st}} \paren{ \raisebox{-3ex}{\begin{tikzpicture}[scale=0.6]
	       \draw[red] (1.1,-0.1) -- (2,-1);
	       \draw[blue] (1.5,-0.1) -- (2,-1);
	       \draw[violet] (2.9,-0.1) -- (2,-1);
	       \draw[blue] (1.1,-2) -- (2,-1);
	       \draw[red] (1.5,-2) -- (2,-1);
	       \draw[violet] (2.9,-2) -- (2,-1);
	       \node at (2.1,-0.4) {$\scalemath{0.8}{\ldots}$};
	       \node at (2.1,-1.6) {$\scalemath{0.8}{\ldots}$};
	    \end{tikzpicture}} }= v_t^s(m_{st}) \qquad \qquad \Fext_{m_{st}} \paren{ \raisebox{-3ex}{\begin{tikzpicture}[scale=0.6]
	       \draw[blue] (1.1,-0.1) -- (2,-1);
	       \draw[red] (1.5,-0.1) -- (2,-1);
	       \draw[violet] (2.9,-0.1) -- (2,-1);
	       \draw[red] (1.1,-2) -- (2,-1);
	       \draw[blue] (1.5,-2) -- (2,-1);
	       \draw[violet] (2.9,-2) -- (2,-1);
	       \node at (2.1,-0.4) {$\scalemath{0.8}{\ldots}$};
	       \node at (2.1,-1.6) {$\scalemath{0.8}{\ldots}$};
	    \end{tikzpicture}}  }=v_s^t(m_{st}) \]
and the rest of the generating morphisms and on objects as in \cref{affineequiv}. Then $\Fext_{m_{st}}$ will be well defined and furthermore be a monoidal equivalence. 
\end{theorem}
\begin{proof}
The defining relations  of $\Dext_{m_{st}}$ in \cref{affinediagrel} all hold in $\bsbimext(\hf, W_{m_{st}})$ as shown in \cref{newgensect} and \cref{finitesect}. Therefore $\Fext_{m_{st}}$ is well defined and as in \cref{affineequiv} it suffices to check that for all $w\in W_{m_{st}}$
\[  \hom^{\bullet, \bullet}_{\Dext_{m_{st}}}(\Bs_\varnothing, \Bs_w) \cong \hom_{\bsbimext}(K_\varnothing, K_w)=\ext^{\bullet, \bullet}_{R^e}(R, B_w)  \]
WLOG we can assume that $w=\usmt{t}{2j}{}$ and $2j\le m$ as $\usmt{t}{m_{st}}{}$ is the longest element in $W_{m_{st}}$. Then by \cref{finalgextdecomp} we need to show
\begin{align*}
\hom^{0, \bullet}_{\Dext_{m_{st}}}(\Bs_\varnothing, \Bs_{\smt{t}{2j}{}})&=  \jw_{\usmt{t}{2j}{} }\circ \bunit \runit \ldots \runit\  R \\
\hom^{1, \bullet}_{\Dext_{m_{st}}}(\Bs_\varnothing, \Bs_{\smt{t}{2j}{}})&=  \jw_{\usmt{t}{2j}{} }\circ  \bhunit \runit \ldots \runit\ R\oplus  \jw_{\usmt{t}{2j}{} }\circ  \ 
\raisebox{-0.2cm}{\begin{tikzpicture}[scale=0.6]
	       \draw[blue] (1.1,-0.1) -- (2,-1);
	       \draw[red] (2.9,-0.1) -- (2,-1);
	       \node[rkhdot] at (2,-1) {};
	       \node at (2,-1) {$\s{2j}$};
	       \node at (2,-0.3) {$\scalemath{0.7}{\usmt{t}{2j}{}}$};
	    \end{tikzpicture}} \ R \\
\hom^{2, \bullet}_{\Dext_{m_{st}}}(\Bs_\varnothing, \Bs_{\smt{t}{2j}{}})&= \ \jw_{\usmt{t}{2j}{} }\circ \raisebox{-0.2cm}{\begin{tikzpicture}[scale=0.6]
	       \draw[blue] (1.1,-0.1) -- (2,-1);
	       \draw[red] (2.9,-0.1) -- (2,-1);
	       \node[rkhdot] at (2,-1) {};
	       \node at (2,-1) {$\s{2j}$};
	       \node[bhdot] at (1.4, -0.4) {};
	       \node at (2,-0.3) {$\scalemath{0.7}{\usmt{t}{2j}{}}$};
	    \end{tikzpicture}} \ R
\end{align*}
 As in the affine case in \cref{affineequiv}, we just need to show that any diagram in $\hom^{\bullet, \bullet}_{\Dext_{m_{st}}}(\Bs_\varnothing, \Bs_{\smt{t}{2j}{}})$ with a Hochschild dot or a red $2k-$extvalent morphism (for all $k\ge 2$) can be written as a right $R$ linear combination of diagrams with a Hochschild dot or a red $2k-$extvalent morphism at the bottom of the diagram. If the diagram if it doesn't contain the $2m_{st}-$valent vertex we are done, as we can just use the proof of \cref{affineequiv}. Call such diagrams ext$-\infty$ diagrams. It follows that we just need to show that when the red $2k-$extvalent morphism or Hochschild dot is trapped by the $2m_{st}-$valent morphism we can move the Hochschild dot and red $2k-$valent morphism further down modulo ext$-\infty$ diagrams. \\
 
Suppose we have a red $2k-$extvalent morphism trapped by a $2m_{st}-$valent morphism. We first demonstrate the case when $m=3$. Using \cref{samecompeq} we can assume that the $4-$extvalent morphism is on the same component as the $6-$valent morphism as seen on the left below. Applying fusion, we then see that
\begin{equation}
\label{2kextpast2m1}
\raisebox{-10ex}{
	\begin{tikzpicture}[scale=0.65]
	\draw[dotted, thick] (2.8,-2) circle (87pt);
	       \draw[red] (0.9,0) to[out=-90, in=-180] (2,-1);
	       \draw[blue]  (2,-1) to[out=0, in=-90] (3.1, 0);
	       \draw[blue] (2, -1) -- (2,-2);
	       \draw[red] (2, -1) -- (2.6,-1.5); 
	       \node[rdot] at (2.6, -1.5) {};
	       \node[rkhdot] at (2,-1) {};
	       \node at (2,-1) {$\scalemath{0.7}{2k}$};
	       \node at (2,-0.4) {$\cdots$};
	       \draw[blue] (0.6,0) to[out=-90, in=185] (3.4,-2.2); 
	       \draw[red] (3.4,-2.2) -- (3.4, 0);
	       \draw[blue] (3.4,-2.2) -- (4.8,0); 
	       \draw[red] (3.4,-2.2) -- (4.8,-4.2); 
	       \draw[blue] (3.4,-2.2) -- (3.4,-4.6); 
	       \draw[red] (3.4,-2.2) -- (1.6,-4.6); 
	\end{tikzpicture}\raisebox{1.6cm}{$ \ = \ $}
	\begin{tikzpicture}[scale=0.65]
	\draw[dotted, thick] (2.8,-2) circle (87pt);
	       \draw[red] (0.9,0) to[out=-90, in=-180] (2,-1);
	       \draw[blue]  (2,-1) to[out=0, in=-90] (3.1, 0);
	       \draw[blue] (2, -1) -- (2,-2);
	       \draw[red] (2, -1) -- (3.4,-1.3); 
	       \node at (2.8,-1.7) {$ \scalemath{0.8}{\boxed{\rho_s}}$};
	       \node[rkhdot] at (2,-1) {};
	       \node at (2,-1) {$\scalemath{0.7}{2k}$};
	       \node at (2,-0.4) {$\cdots$};
	       \draw[blue] (0.6,0) to[out=-90, in=185] (3.4,-2.2); 
	       \draw[red] (3.4,-2.2) -- (3.4, 0);
	       \draw[blue] (3.4,-2.2) -- (4.8,0); 
	       \draw[red] (3.4,-2.2) -- (4.8,-4.2); 
	       \draw[blue] (3.4,-2.2) -- (3.4,-4.6); 
	       \draw[red] (3.4,-2.2) -- (1.6,-4.6); 
	\end{tikzpicture}\raisebox{1.6cm}{$ \ - \ $}
	\begin{tikzpicture}[scale=0.65]
	\draw[dotted, thick] (2.8,-2) circle (87pt);
	       \draw[red] (0.9,0) to[out=-90, in=-180] (2,-1);
	       \draw[blue]  (2,-1) to[out=0, in=-90] (3.1, 0);
	       \draw[blue] (2, -1) -- (2,-2);
	       \draw[red] (2, -1) to[out=-85, in=180] (3.4,-2); 
	       \node at (2.9,-1.3) {$ \scalemath{0.65}{\boxed{s(\rho_s)}}$};
	       \node[rkhdot] at (2,-1) {};
	       \node at (2,-1) {$\scalemath{0.7}{2k}$};
	       \node at (2,-0.4) {$\cdots$};
	       \draw[blue] (0.6,0) to[out=-90, in=185] (3.4,-2.2); 
	       \draw[red] (3.4,-2.2) -- (3.4, 0);
	       \draw[blue] (3.4,-2.2) -- (4.8,0); 
	       \draw[red] (3.4,-2.2) -- (4.8,-4.2); 
	       \draw[blue] (3.4,-2.2) -- (3.4,-4.6); 
	       \draw[red] (3.4,-2.2) -- (1.6,-4.6); 
	\end{tikzpicture}}
\end{equation}
Now in the first diagram on the RHS of \cref{2kextpast2m1}, use polynomial forcing to move the $\rho_s$ to the right resulting in 
\begin{equation}
\label{2kextpast2m2}
\raisebox{-10ex}{
\begin{tikzpicture}[scale=0.65]
	\draw[dotted, thick] (2.8,-2) circle (87pt);
	       \draw[red] (0.9,0) to[out=-90, in=-180] (2,-1);
	       \draw[blue]  (2,-1) to[out=0, in=-90] (3.1, 0);
	       \draw[blue] (2, -1) -- (2,-2);
	       \draw[red] (2, -1) -- (3.4,-1.3); 
	       \node at (2.8,-1.7) {$ \scalemath{0.8}{\boxed{\rho_s}}$};
	       \node[rkhdot] at (2,-1) {};
	       \node at (2,-1) {$\scalemath{0.7}{2k}$};
	       \node at (2,-0.4) {$\cdots$};
	       \draw[blue] (0.6,0) to[out=-90, in=185] (3.4,-2.2); 
	       \draw[red] (3.4,-2.2) -- (3.4, 0);
	       \draw[blue] (3.4,-2.2) -- (4.8,0); 
	       \draw[red] (3.4,-2.2) -- (4.8,-4.2); 
	       \draw[blue] (3.4,-2.2) -- (3.4,-4.6); 
	       \draw[red] (3.4,-2.2) -- (1.6,-4.6); 
	\end{tikzpicture}\raisebox{1.6cm}{$ \ = \ $}\begin{tikzpicture}[scale=0.65]
	\draw[dotted, thick] (2.8,-2) circle (87pt);
	       \draw[red] (0.9,0) to[out=-90, in=-180] (2,-1);
	       \draw[blue]  (2,-1) to[out=0, in=-90] (3.1, 0);
	       \draw[blue] (2, -1) -- (2,-2);
	       \draw[red] (2, -1) to[out=-10, in=-90] (3.4,0); 
	       \node[rkhdot] at (2,-1) {};
	       \node at (2,-1) {$\scalemath{0.7}{2k}$};
	       \node at (2,-0.4) {$\cdots$};
	       \draw[red] (3.4, -2.2) -- (3.4, -1.3);
	       \node[rdot] at (3.4,-1.3) {};
	       \draw[blue] (0.6,0) to[out=-90, in=185] (3.4,-2.2); 
	       \draw[blue] (3.4,-2.2) -- (4.8,0); 
	       \draw[red] (3.4,-2.2) -- (4.8,-4.2); 
	       \draw[blue] (3.4,-2.2) -- (3.4,-4.6); 
	       \draw[red] (3.4,-2.2) -- (1.6,-4.6); 
	\end{tikzpicture}\raisebox{1.6cm}{$ \ - \ $}
	\begin{tikzpicture}[scale=0.65]
	\draw[dotted, thick] (2.8,-2) circle (87pt);
	       \draw[red] (0.9,0) to[out=-90, in=-180] (2,-1);
	       \draw[blue]  (2,-1) to[out=0, in=-90] (3.1, 0);
	       \draw[blue] (2, -1) -- (2,-2);
	       \draw[red] (2, -1) -- (3.4,-1.3); 
	       \node at (4,-0.8) {$ \scalemath{0.7}{\boxed{s(\rho_s)}}$};
	       \node[rkhdot] at (2,-1) {};
	       \node at (2,-1) {$\scalemath{0.7}{2k}$};
	       \node at (2,-0.4) {$\cdots$};
	       \draw[blue] (0.6,0) to[out=-90, in=185] (3.4,-2.2); 
	       \draw[red] (3.4,-2.2) -- (3.4, 0);
	       \draw[blue] (3.4,-2.2) to[out=0, in=-90] (4.8,0); 
	       \draw[red] (3.4,-2.2) -- (4.8,-4.2); 
	       \draw[blue] (3.4,-2.2) -- (3.4,-4.6); 
	       \draw[red] (3.4,-2.2) -- (1.6,-4.6); 
	\end{tikzpicture}
}
\end{equation}
The first diagram on the RHS of \cref{2kextpast2m2} reduces to ext$-\infty$ diagrams using two-color dot contraction while the second diagram on the RHS of \cref{2kextpast2m2} is essentially the same as the second diagram on the RHS of \cref{2kextpast2m1} modulo \underline{polynomials that are at the top of the diagram}. Now apply \cref{absorbtrap} and again modulo polynomials both diagrams reduce to
\begin{equation}
    \raisebox{-10ex}{
    \begin{tikzpicture}[scale=0.65]
	\draw[dotted, thick] (2.8,-2) circle (87pt);
	       \draw[red] (0.9,0) to[out=-90, in=-180] (2,-1);
	       \draw[blue]  (2,-1) to[out=0, in=-90] (3.1, 0);
	       \draw[blue] (2, -1) -- (2,-2);
	       \draw[red] (2, -1) to[out=-10, in=-90] (3.4,0); 
	        \draw[red] (3.4, -2.2) -- (2, -1);
	        \draw[blue] (2,-1) -- (2.9,-1.3);
	        \node[bdot] at (2.9,-1.3) {} ;
	        \node[rbkhdot] at (2,-1) {};
	       \node at (2,-1) {$\scalemath{0.65}{2k+2}$};
	       \node at (2,-0.4) {$\cdots$};
	       \draw[blue] (0.6,0) to[out=-90, in=185] (3.4,-2.2); 
	       \draw[blue] (3.4,-2.2) -- (4.8,0); 
	       \draw[red] (3.4,-2.2) -- (4.8,-4.2); 
	       \draw[blue] (3.4,-2.2) -- (3.4,-4.6); 
	       \draw[red] (3.4,-2.2) -- (1.6,-4.6); 
	\end{tikzpicture}
    }
\end{equation}
We are now in same position as the LHS of \cref{2kextpast2m1} except the number of strands of the red $2k-$extvalent morphism connected to the $6-$valent morphism has gone up. We can repeat this argument one more time so that 3 strands of the red $2k-$extvalent morphism are now connected to the $6-$valent morphism, and after another application of \cref{absorbtrap} for the first top blue strand we can apply $2m_{st}-$Absorption \cref{2mstabsorb} because \underline{the only polynomials that appear} are at the top. \\

For general $m_{st}\ge 3$, the argument is the same, where we keep applying the reductions in \cref{2kextpast2m1} and \cref{2kextpast2m2} until we have $m_{st}$ strands connecting the red $2k-$extvalent morphism and the $2m_{st}$ valent morphism such that the only polynomials that appear are at the top of the diagram. Using $4-$ext reductions \cref{diag4extreduct}, \cref{diag4extreduct2}, we can also move Hochschild dots past the $2m_{st}-$valent morphism at the cost of a red $4-$extvalent morphism and so we are done. Note for $m_{st}=2$, we don't have red $2k-$extvalent morphisms but we can always move Hochschild dots past the $4-$valent morphism using \cref{4valhochslide}. 
 \end{proof}

\appendix

\begin{appendices}

\section{Lifting to Chain Level}
\label{chainliftsection}

We will now define (partial) chain lifts of various bimodule morphisms in $\bsbim(\hf, W)$. Specifically, given a morphism $f$ between Bott-Samelson bimodules, we will define a (partial) morphism between the corresponding Bott-Samelson complexes lifting $f$. As our maps are required to be $R^e$ linear it suffices to define the chain maps on the inner tensor factors of $R^e, R^{ee},$ etc. In addition, as all of the bimodule maps we are lifting are of the form $f: R\otimes_{s/t}\ldots \otimes_{s/t}R\to R\otimes_{s/t}\ldots \otimes_{s/t}R $, we can always define $\widetilde{f}^0:  R\otimes\ldots \otimes R\to R\otimes \ldots \otimes R $ by the exact same formula as $f$ but with $\otimes_{s/t}$ replaced by $\otimes$.

\subsection{Unit}
The unit map of $B_s$ will be $\eta_s: R\to B_s$ where $\eta_s(1)=\rho_s\otimes_s 1-1\otimes_s s(\rho_s)$. We want to find $\widetilde{\eta_s}^1$ making the following diagram commute.

\begin{center}
\begin{tikzcd}
\cdots\arrow[r]& \un{V^*}(-2)  \boxtimes R^e  \arrow[r, "d"]\arrow[d,dotted, " \widetilde{\eta_s}^1 "]& \underline{1}\boxtimes R^e\arrow[r] \arrow[d,dotted,"\widetilde{\eta_s}^0"]& R \arrow[d, "\Delta"]\\
\cdots\arrow[r] & (\un{(V^*)^s}(-2) \oplus \kb \un{\rho_s s(\rho_s)}(-4)) \boxtimes R^e(1)  \arrow[r, "d"] & \un{1}\boxtimes R^e(1)\arrow[r] & B_s 
\end{tikzcd}
\end{center}
Recall that $V^*=(V^*)^s\oplus \kb \rho_s$. We then claim that the following definition will make the diagram commute
\begin{align*}
\widetilde{\eta_s}^1(\underline{r}\boxtimes 1\otimes 1)&=\underline{r}\boxtimes \rho_s\otimes 1-\underline{r}\boxtimes 1\otimes s(\rho_s) \quad \textnormal{ if }r\in (V^*)^s \\
\widetilde{\eta_s}^1(\underline{\rho_s} \boxtimes 1\otimes 1)&=(\underline{\rho_s+s(\rho_s)})\boxtimes \rho_s\otimes 1-\underline{\rho_s s(\rho_s)} \boxtimes 1\otimes 1
\end{align*}
This is clear if $r\in (V^*)^s$. For $\rho_s$ note that $\widetilde{\eta_s}^0(d( \underline{\rho_s}\boxtimes 1\otimes 1 ))$ is given by
\begin{center}
\begin{tikzcd}
\underline{\rho_s} \boxtimes 1\otimes 1  \arrow[d]\arrow[r]& \un{1}\boxtimes\paren{\rho_s\otimes 1- 1 \otimes \rho_s } \arrow[d]\\
? \arrow[r]& \un{1}\boxtimes\paren{\rho_s \rho_s\otimes 1 -\rho_s\otimes s(\rho_s)-\rho_s\otimes \rho_s+1\otimes \rho_s s(\rho_s)}
\end{tikzcd}
\end{center}
and one quickly computes that $d$ applied to our definition above yields the expression on the bottom right. 

\subsection{Multiplication}
The multiplication map of $B_s$ is $\mu_s: R\otimes_s R\otimes_s R(2) \to R\otimes_s R(1)$ where $\mu_s(f\otimes_s g\otimes_s h)= \partial_s(g) f\otimes_s h $. We want to find $\widetilde{\mu_s}^1$ making this diagram commute. Let $W=((V^*)^s(-2) \oplus \kb \rho_s s(\rho_s)(-4)) $. 

\begin{center}
\begin{tikzcd}
\cdots\arrow[r]& \un{W^{(1)}} \boxtimes R^{ee}(2) \bigoplus  \un{W^{(2)}} \boxtimes R^{ee}(2)\arrow[r, "d"]\arrow[d,dotted, " \widetilde{\mu_s}^1 "]&  \un{1}\boxtimes R^{ee}(2)\arrow[r] \arrow[d,dotted,"\widetilde{\mu_s}^0"]& B_s\otimes_R B_s \arrow[d, "\mu_s"]\\
\cdots\arrow[r] & \un{W} \boxtimes R^e(1)  \arrow[r, "d"] & \un{1}\boxtimes R^e(1)\arrow[r] & B_s 
\end{tikzcd}
\end{center}
  
We claim that the following definition will work. \vspace{-3ex}

\begin{align*}
\widetilde{\mu_s}^1(\underline{r^{(1)}}\boxtimes 1\otimes h \otimes 1)&=0 \quad \textnormal{ if }r\in (V^*)^s \\
\widetilde{\mu_s}^1(\underline{\rho_s s(\rho_s)^{(1)}} \boxtimes 1\otimes h \otimes 1 )&= 0\\
\widetilde{\mu_s}^1(\underline{r^{(2)}}\boxtimes 1\otimes h\otimes 1 )&= \underline{r}\boxtimes \partial_s(h)\otimes 1  \quad \textnormal{ if }r\in (V^*)^s\\
\widetilde{\mu_s}^1(\underline{\rho_s s(\rho_s)^{(2)}}\boxtimes 1\otimes h\otimes 1 )&= \underline{\rho_s s(\rho_s)}\boxtimes \partial_s(h)\otimes 1
\end{align*}

\subsection{Counit}

The counit map of $B_s$ will be $\epsilon_s: R\otimes_s R(1)\to R$ where $\epsilon_s(f\otimes_s g)=fg$ so we want to find $\widetilde{\epsilon_s}^1$ making the following diagram commute. Again, let $W=((V^*)^s(-2) \oplus \kb \rho_s s(\rho_s)(-4)) $. 
\begin{center}
\begin{tikzcd}
\cdots\arrow[r] & \un{W} \boxtimes R^e(1)  \arrow[r, "d"] \arrow[d,dotted, " \widetilde{\epsilon_s}^1 "]& \un{1}\boxtimes R^e(1)\arrow[r] \arrow[d,dotted,"\widetilde{\epsilon_s}^0"] & B_s \arrow[d, "\epsilon_s"]\\
\cdots\arrow[r]& \un{ V^*}(-2) \boxtimes R^{e}\arrow[r, "d"]&\un{1}\boxtimes R^{e}\arrow[r]& R 
\end{tikzcd}
\end{center}
In \cite {M} it was shown that the following definition makes the diagram commute
\begin{align*}
\widetilde{\epsilon_s}^1(\underline{r}\boxtimes 1\otimes 1)&=\underline{r}\boxtimes 1\otimes 1 \qquad  \qquad\textnormal{ if }r\in (V^*)^s \\
\widetilde{\epsilon_s}^1(\underline{\rho_s s(\rho_s)} \boxtimes 1\otimes 1)&=\underline{\rho_s}\boxtimes s(\rho_s)\otimes 1+\underline{s(\rho_s)} \boxtimes 1\otimes \rho_s
\end{align*}

\subsection{Comultiplication}

The comultiplication map of $B_s$ will be $\delta_s: R\otimes_s R(1)\to R\otimes_sR\otimes_s R(2)$ where $\delta_s(f\otimes_s g)= f\otimes_s 1\otimes_s g $ so we want to find $\widetilde{\delta_s}^1$ making the following diagram commute. Again, let $W=((V^*)^s(-2) \oplus \kb \rho_s s(\rho_s)(-4)) $. 

\begin{center}
\begin{tikzcd}
\cdots\arrow[r] & \un{W} \boxtimes R^e(1)  \arrow[r, "d"] \arrow[d,dotted, " \widetilde{\delta_s}^1 "]& \un{1} \boxtimes R^e(1)\arrow[r] \arrow[d,dotted,"\widetilde{\delta_s}^0"] & B_s \arrow[d, "\delta_s"]\\
\cdots\arrow[r]& \un{W^{(1)}} \boxtimes R^{ee}(2) \bigoplus \un{ W^{(2)}} \boxtimes R^{ee}(2)\arrow[r, "d"]& \un{1} \boxtimes R^{ee}(2)\arrow[r]& B_s\otimes_R B_s 
\end{tikzcd}
\end{center}
 We then claim that the following definition will work.\vspace{-4ex}

\begin{align*}
\widetilde{\delta_s}^1( \underline{r}\boxtimes 1\otimes 1)&=\underline{r^{(1)}}\boxtimes 1\otimes  1\otimes 1+ \underline{r^{(2)}}\boxtimes 1\otimes 1\otimes 1  \qquad \textnormal{ if }r\in (V^*)^s \\
\widetilde{\delta_s}^1(\underline{\rho_s s(\rho_s)} \boxtimes 1\otimes 1 )&=\underline{\rho_s s(\rho_s)^{(1)}}  \boxtimes 1\otimes 1\otimes 1+\underline{\rho_s s(\rho_s)^{(2)}} \boxtimes 1\otimes  1\otimes 1 
\end{align*}
This is an easy check left to the reader.

\subsection{Inverse of Left Unitor}

The inverse of the left unitor $\lambda_s$ for $B_s$ is the map $\tau_s: B_s\to R\otimes_R B_s$ where $\tau_s(f\otimes_s g)=f\otimes_R 1 \otimes_s g$, so we want to find $\widetilde{\tau_s}^1$ making the following diagram commute. Again, let $W=((V^*)^s(-2) \oplus \kb \rho_s s(\rho_s)(-4)) $. 

\begin{center}
\begin{tikzcd}
\cdots\arrow[r] & \un{W} \boxtimes R^e(1)  \arrow[r, "d"] \arrow[d,dotted, " \widetilde{\tau_s}^1 "]& \un{1} \boxtimes  R^e(1)\arrow[r] \arrow[d,dotted,"\widetilde{\tau_s}^0"] & B_s \arrow[d, "\tau_s"]\\
\cdots\arrow[r]& \un{(V^*)^{(1)}}(-2) \boxtimes R^{ee}(1) \bigoplus  W^{(2)} \boxtimes R^{ee}(2)\arrow[r, "d"]& \un{1} \boxtimes  R^{ee}(1)\arrow[r]& R\otimes_R B_s 
\end{tikzcd}
\end{center}
 In \cite {M} it was shown that the following definition will make the diagram commute \vspace{-4ex}

\begin{align*}
\widetilde{\tau_s}^1(\underline{r}\boxtimes 1\otimes 1)&=\underline{r}^{(1)}\boxtimes 1\otimes 1\otimes 1 + \underline{r}^{(2)}\boxtimes 1\otimes 1\otimes 1   \qquad  \qquad \text{ if }r\in (V^*)^s \\
\widetilde{\tau_s}^1(\underline{\rho_s s(\rho_s)}\boxtimes 1\otimes 1)&=\underline{\rho_s^{(1)}}\boxtimes s(\rho_s)\otimes 1\otimes 1+ \underline{s(\rho_s)^{(1)}}\boxtimes 1\otimes \rho_s \otimes 1+\underline{\rho_s s(\rho_s)^{(2)}}\boxtimes 1\otimes 1 \otimes 1
\end{align*}

\subsection{Inverse of Right Unitor}

The inverse of the right unitor $\theta_s$ for $B_s$ is the map $\sigma_s: B_s\to B_s\otimes_R R$ where $\sigma_s(f\otimes_s g)=f\otimes_s 1 \otimes_R g$, so we want to find $\widetilde{\sigma_s}^0,\widetilde{\sigma_s}^1$ making the following diagram commute. Again, let $W=((V^*)^s(-2) \oplus \kb \rho_s s(\rho_s)(-4)) $. 

\begin{center}
\begin{tikzcd}
\cdots\arrow[r] & \un{W} \boxtimes R^e(1)  \arrow[r, "d"] \arrow[d,dotted, " \widetilde{\sigma_s}^1 "]& \un{1} \boxtimes R^e(1)\arrow[r] \arrow[d,dotted,"\widetilde{\sigma_s}^0"] & B_s \arrow[d, "\sigma_s"]\\
\cdots\arrow[r]&  \un{W^{(1)}} \boxtimes R^{ee}(1)\bigoplus \un{(V^*)^{(2)}}(-2) \boxtimes R^{ee}(1)   \arrow[r, "d"]& \un{1} \boxtimes R^{ee}(1)\arrow[r]& B_s\otimes_R R 
\end{tikzcd}
\end{center}

 In \cite {M} it was shown that the following definition will make the diagram commute \vspace{-4ex}

\begin{align*}
\widetilde{\sigma_s}^1(\underline{r}\boxtimes 1\otimes 1)&=\underline{r}^{(1)}\boxtimes 1\otimes 1\otimes 1 + \underline{r}^{(2)}\boxtimes 1\otimes 1\otimes 1   \qquad  \qquad \text{ if }r\in (V^*)^s \\
\widetilde{\sigma_s}^1(\underline{\rho_s s(\rho_s)}\boxtimes 1\otimes 1)&=\underline{\rho_s s(\rho_s)^{(1)}}\boxtimes 1\otimes 1 \otimes 1+ \underline{s(\rho_s)^{(2)}}\boxtimes 1\otimes  \rho_s\otimes 1 +\underline{\rho_s^{(2)}}\boxtimes 1\otimes 1\otimes s(\rho_s)
\end{align*}

\section{HOMFLY Homology Calculations}
\label{appendixhomfly}

We refer the reader to \cite[Chapter 21.6]{EMTW20} for the definition of HOMFLY homology $\hhh$. As shown in \cite[Theorem 1]{Kho07}, the reduced HOMFLY homology $\overline{\hhh}$ is computed using the geometric realization $\hf_{geo}$ of $S_n$. As a result, the complex $F(\beta_L)$ associated to a braid $\beta_L$ lives in $K^b(\sbim(\hf_{geo}, S_n))$ and $R$ is a polynomial ring in $n-1$ variables instead of $n$. In this section we will be computing reduced HOMFLY homology. The reduced HOMFLY homology categorifies the reduced HOMFLY polynomial $\overline{P(a,q)}$ defined via the skein relation
\[ a^{-1}\overline{P(a,q)}(L_+)-a\overline{P(a,q)}(L_-) = (q-q^{-1})\overline{P(a,q)}(L_0) \]
and the normalization $\overline{P(a,q)}(\mathrm{unknot})=1$. 
\begin{definition}
Define $d_s=\frac{1}{2}(\alpha_s\otimes_s 1-1\otimes_s \alpha_s )\in R\otimes_{R^s}R$.
\end{definition}

\begin{example}
Let us review the calculation of the Poincare series for $\overline{\hhh}$ of the Hopf link $L2a1$. One possible braid representative is $\sigma_1^2$. By \cite[Eq 19.32]{EMTW20} the corresponding Rouquier complex is of the form\footnote{Recall we box the term in cohomological degree 0.} \vspace{-1ex}
\[ F(\sigma_1^2)\simeq \boxed{B_s(-1)} \xrightarrow{\ds}B_s(1)\xrightarrow{\rcounit}R(2) \]
As a result, one can calculate that $ \overline{\hhh}^{A=0}(F(\sigma_1^2)), \overline{\hhh}^{A=1}(F(\sigma_1^2))$ will be the cohomology of the following complexes
\[ \overline{\hhh}^{A=0}(F(\sigma_1^2)):\boxed{R(-2)}\xrightarrow{0} R\xrightarrow{\alpha_s}R(2) \qquad \overline{\hhh}^{A=1}(F(\sigma_1^2)):\boxed{R(2)}\xrightarrow{0} R(4)\xrightarrow{\id}R(4) \]
where $R=\kb[\alpha_s]$ as $(\sigma_1)^2$ is a braid on 2 strands. The Poincare series $\overline{\mathscr{P}(A,Q, T)}(F(\sigma_1^2))$ for $\overline{\hhh}(F(\sigma_1^2))$ will be
\[ \overline{\mathscr{P}(A,Q, T)}(F(\sigma_1^2))= \paren{\frac{Q^{2}}{1-Q^2}+Q^{-2}T^2}+A\paren{\frac{Q^{-2}}{1-Q^2}} \]
To recover the HOMFLY polynomial $P(a,q)$  up to a unit in $\mathbb{Z}\dsbrac{A^{\pm 1}, Q^{\pm 1}, T^{\pm 1}}$ we make the following substitutions
\begin{equation}
\label{subeq}
    A= - a^{2}q^2, T=-1, Q=q 
\end{equation}   
Applying this to $\overline{\mathscr{P}(A,Q, T)}(F(\sigma_1^2))$ we obtain
\[ \overline{\mathscr{P}(A,Q, T)}(F(\sigma_1^2))|_{\cref{subeq}}=\frac{q^2+q^{-2}-1-a^{2}}{1-q^2}=\frac{q^2+q^{-2}-1-a^{2}}{-q(q-q^{-1})}  \]
while using the definition of the reduced HOMFLY polynomial above we obtain 
\[ \overline{P(a,q)}(\widehat{\sigma_1^2})=a\left(\frac{q^2+q^{-2}-1-a^{2}}{q-q^{-1}}\right)\]
which after multiplying by the unit $-(aq)^{-1}$ agrees with $\overline{\mathscr{P}(A,Q, T)}(F(\sigma_1^2))|_{\cref{subeq}}$. \\ 
\end{example}

Our next example is a braid on 3 strands. Here $\hf_{geo}$ is the geometric realization of $S_3$ so $R=\kb[\alpha_s, \alpha_t]$ where 
$$s=(1 \ 2), \ t=(2 \ 3)\in S_3\qquad \alpha_s=x_1-x_2, \ \alpha_t=x_2-x_3 \qquad  \alpha_s^\vee=x_1^*-x_2^*, \ \alpha_t^\vee=x_2^*-x_3^*$$ 

\begin{example}
\label{exam:2hopf}
A braid representative for the connect sum of two Hopf links $L2a1\#L2a1$ is given by $\sigma_1^2\sigma_2^2$. The corresponding Rouquier complex is homotopic to
\begin{equation}
\label{roucomplex}
\scalemath{0.92}{
    \boxed{B_sB_t(-2)} \xrightarrow{ \begin{pmatrix}
\rzero \ \dt \\
 \ds \ \bzero
\end{pmatrix} } B_s B_t\oplus B_s B_t \xrightarrow{\begin{pmatrix}
 \rzero \ \bcounit &  0 \\
 \ds \bzero &   -\rzero \dt     \\[1em]
 0       &  \rcounit \bzero        
\end{pmatrix}} B_s(1)\oplus B_sB_t(2)\oplus B_t(1) \xrightarrow{ \begin{pmatrix}
\ds & - \rzero \bcounit & 0 \\
 0  &  \rcounit \bzero  & \dt
\end{pmatrix} } B_s(3)\oplus B_t(3) \xrightarrow{\begin{pmatrix}
 \rcounit &\bcounit 
\end{pmatrix}} R(4)} 
\end{equation}

Using \cref{finalgextdecomp} and \cref{hochbs} we have the following 

\begin{lemma}
\begin{align*}
\hh^{0}(B_s)&=  \runit \, R=R(-1)  & \hh^{0}(B_t)&=  \bunit \, R=R(-1)  \\
\hh^{1}(B_s)&=  \rhunit  \, R\oplus \runit\ \raisebox{-1ex}{ \dbox{\scalemath{0.9}{\alpha_t^\vee}}}  \, R =R(3)\oplus R(1) & \hh^{1}(B_t)&=  \bhunit  \, R\oplus \bunit\ \raisebox{-1ex}{ \dbox{\scalemath{0.9}{\alpha_s^\vee}}}  \, R =R(3)\oplus R(1) \\
\hh^{2}(B_s)&=  \rhunit\ \raisebox{-1ex}{ \dbox{\scalemath{0.9}{\alpha_t^\vee}}}  \, R =R(5) & \hh^{2}(B_t)&=  \bhunit\ \raisebox{-1ex}{ \dbox{\scalemath{0.9}{\alpha_s^\vee}}}  \, R =R(5)
\end{align*}
\end{lemma}

\begin{lemma}
\begin{align*}
\hh^{0}(R)&=  \raisebox{-1ex}{ \dbox{\scalemath{0.9}{1}}} R=R &\hh^{0}(B_sB_t)&=  \runit \bunit \, R =R(-2) \\
\hh^{1}(R)&=  \raisebox{-1ex}{ \dbox{\scalemath{0.9}{\alpha_s^\vee}}} R \oplus  \raisebox{-1ex}{ \dbox{\scalemath{0.9}{\alpha_t^\vee}}} R =R(2)\oplus R(2) &\hh^{1}(B_sB_t)&=  \rhunit \bunit \, R\oplus  \runit \bhunit \, R=R(2)\oplus R(2)  \\
\hh^{2}(R)&=  \raisebox{-1ex}{ \dbox{\scalemath{0.9}{\alpha_s^\vee}}}\raisebox{-1ex}{ \dbox{\scalemath{0.9}{\alpha_t^\vee}}} \, R=R(4)&\hh^{2}(B_sB_t)&=  \rhunit \bhunit \, R=R(6)
\end{align*}
\end{lemma}
The complex for $\overline{\hhh}^{A=0}$ will then be (for notation purposes, let $\overline{ \runit (n)}= \runit \, R(n)$, etc)

\[ \boxed{\overline{\runit\bunit(-2)}} \xrightarrow{ \begin{pmatrix}
0 \\
0
\end{pmatrix} } \overline{\runit\bunit}\oplus \overline{\runit\bunit} \xrightarrow{\begin{pmatrix}
\scalemath{0.9}{\alpha_t} &  0 \\
 0 &   0     \\
 0       &  \scalemath{0.9}{\alpha_s}        
\end{pmatrix}} \overline{\runit(1)}\oplus \overline{\runit\bunit(2)}\oplus \overline{\bunit(1)} \xrightarrow{ \begin{pmatrix}
0 & -  \scalemath{0.9}{\alpha_t}  & 0 \\
 0  &  \scalemath{0.9}{\alpha_s}   & 0
\end{pmatrix} } \overline{\runit(3)}\oplus \overline{\bunit(3)} \xrightarrow{\begin{pmatrix}
 \scalemath{0.9}{\alpha_s} & \scalemath{0.9}{\alpha_t} 
\end{pmatrix}} \overline{ \boxed{1} (4)} \]

As right $R$ modules, the nonzero cohomology will then be
\[ \paren{\overline{\hhh}^{A=0}}^{T=0}=\runit\bunit \, R(-2)=R(-4), \quad \paren{\overline{\hhh}^{A=0}}^{T=2}=\scalemath{0.9}{\frac{\runit R(1)}{(\alpha_t)}}\raisebox{0.5ex}{$\oplus$} \, \scalemath{0.9}{\frac{\bunit R(1)}{(\alpha_s)}} =\scalemath{1.6}{\sfrac{R}{(\alpha_t)}}\raisebox{0.5ex}{$\oplus$} \, \scalemath{1.6}{\sfrac{R}{(\alpha_s)}} , \quad \paren{\overline{\hhh}^{A=0}}^{T=4}=\boxed{1} \kb (4)=\kb(4) \]
with corresponding Poincare series $ \displaystyle \frac{Q^4}{(1-Q^2)^2} +\frac{2T^{2}  }{1-Q^2} +Q^{-4}T^4 $. The complex for $\overline{\hhh}^{A=1}$ will be 
\begin{align*}
\scalemath{0.9}{\boxed{ \overline{\rhunit \bunit (-2)} \oplus \overline{\runit \bhunit (-2)} }} &\xrightarrow{ \begin{pmatrix}
0 & 0 \\
0 & 0 \\
0 & 0\\
0 & 0
\end{pmatrix} } \scalemath{0.9}{\paren{\overline{\rhunit \bunit } \oplus \overline{\runit \bhunit } } }\oplus\scalemath{0.9}{\paren{\overline{\rhunit \bunit } \oplus \overline{\runit \bhunit } } } \xrightarrow{\begin{pmatrix}
\scalemath{0.9}{\alpha_t} &  0 & 0 & 0 \\
0                  &        1  & 0 & 0\\
 0 &   0   & 0 & 0     \\
  0 &   0   & 0 & 0     \\
 0   & 0 & 0    &  \scalemath{0.9}{\alpha_s} \\
 0  &  0 & 1  & 0
\end{pmatrix}} \scalemath{0.9}{\paren{\overline{\rhunit (1) } \oplus \overline{\runit \scalemath{0.8}{\dboxed{\alpha_t^\vee}} (1) } } }\oplus \scalemath{0.9}{\paren{\overline{\rhunit \bunit (2) } \oplus \overline{\runit \bhunit(2) } } }\oplus \scalemath{0.9}{\paren{\overline{\bhunit (1) } \oplus \overline{\bunit \scalemath{0.8}{\dboxed{\alpha_s^\vee}} (1) } }}  \\
&\xrightarrow{ \begin{pmatrix}
0 & 0 & -  \scalemath{0.9}{\alpha_t} &0  & 0 & 0 \\
0 & 0 &0                             &-1  & 0 & 0 \\
 0  & 0 & 0 & \scalemath{0.9}{\alpha_s} & 0   & 0 \\
  0  & 0 & 1 &                       0  & 0   & 0 
\end{pmatrix} }\scalemath{0.9}{\paren{\overline{\rhunit (3) } \oplus \overline{\runit \scalemath{0.8}{\dboxed{\alpha_t^\vee}} (3) } } }\oplus \scalemath{0.9}{\paren{\overline{\bhunit (3) } \oplus \overline{\bunit \scalemath{0.8}{\dboxed{\alpha_s^\vee}} (3)}} } \xrightarrow{\begin{pmatrix}
  1 & 0  & 0 & \scalemath{0.9}{\alpha_t} \\
  0 & \scalemath{0.9}{\alpha_s} & 1 & 0
\end{pmatrix}} \overline{\scalemath{0.8}{\dboxed{\alpha_s^\vee}} (4)}\oplus\overline{\scalemath{0.8}{\dboxed{\alpha_t^\vee}} (4)}
\end{align*} 

As right $R$ modules, the nonzero cohomology will then be
\[ \paren{\overline{\hhh}^{A=1}}^{T=0}= \rhunit \bunit R(-2) \oplus \runit \bhunit R(-2)=R\oplus R, \quad \paren{\overline{\hhh}^{A=1}}^{T=2}= \scalemath{0.9}{\frac{\rhunit R(1)}{(\alpha_t)}}\oplus \scalemath{0.9}{\frac{\bhunit R(1)}{(\alpha_s)}}   =\scalemath{1.6}{\sfrac{R(4)}{(\alpha_t)}}\raisebox{0.5ex}{$\oplus$} \, \scalemath{1.6}{\sfrac{R(4)}{(\alpha_s)}} \]
with corresponding Poincare series  $\displaystyle \frac{2}{(1-Q^2)^2} +\frac{2T^{2}Q^{-4}}{1-Q^2} $. The complex for $\overline{\hhh}^{A=2}$ will  be

\[ \boxed{ \overline{\rhunit \bhunit (-2)} } \xrightarrow{ \begin{pmatrix}
0 \\
0
\end{pmatrix} } \overline{\rhunit \bhunit }  \oplus \overline{\rhunit \bhunit} \xrightarrow{\begin{pmatrix}
1 &  0 \\
 0 &   0     \\
 0       &  1       
\end{pmatrix}}\overline{\rhunit \scalemath{0.8}{\dboxed{\alpha_t^\vee}} (1)}\oplus \overline{\rhunit \bhunit (2)} \oplus\overline{\bhunit \scalemath{0.8}{\dboxed{\alpha_s^\vee}} (1)} \xrightarrow{ \begin{pmatrix}
0 & -  1  & 0 \\
 0  &  1   & 0
\end{pmatrix} } \overline{\rhunit \scalemath{0.8}{\dboxed{\alpha_t^\vee}} (3)}\oplus \overline{\bhunit \scalemath{0.8}{\dboxed{\alpha_s^\vee}} (3)}\xrightarrow{\begin{pmatrix}
 1 & 1
\end{pmatrix}} \overline{\scalemath{0.8}{\dboxed{\alpha_s^\vee}} \scalemath{0.8}{\dboxed{\alpha_t^\vee}} (4)} \]
As right $R$ modules, the nonzero cohomology will then be
\[ \paren{\overline{\hhh}^{A=2}}^{T=0}=\rhunit \bhunit R(-2)=R(4) \quad \textnormal{with corresponding Poincare series } \frac{Q^{-4}}{(1-Q^2)^2}\]
Putting this all together we obtain

\[ \overline{\mathscr{P}(A,Q, T)}(F(\sigma_1^2\sigma_2^2))= \paren{\frac{Q^4}{(1-Q^2)^2} +\frac{2T^{2}  }{1-Q^2} +Q^{-4}T^4}+A\paren{\frac{2}{(1-Q^2)^2} +\frac{2T^{2}Q^{-4}}{1-Q^2}} + A^2\paren{\frac{Q^{-4}}{(1-Q^2)^2}} \]
which upon closer inspection is nothing more than $\overline{\mathscr{P}(A,Q, T)}(F(\sigma_1^2))^2$! This is consistent with the following result of Rasmussen

\end{example}

\begin{prop}[{\cite[Lemma 7.8]{Ras15}}]
\label{raslemma}
Given two braids $\beta_1$ and $\beta_2$, as graded vector spaces, we have that 
\[ \overline{\hhh}(\beta_1\#\beta_2)= \overline{\hhh}(\beta_1)\otimes_{\kb}  \overline{\hhh}(\beta_2)\]
\end{prop}

\subsection{$\overline{\hhh}(T(3,-3))$}

In contrast to the preceding subsection, the diagrammatic computations presented herein involve a greater degree of complexity due to the differentials mapping into and out from $\hh^\bullet(B_{sts})$. It is recommended that the reader possess a solid understanding of the relations established in this paper and in \cite{DC} before continuing.

A braid representative for $T(3,-3)$ is given by $\paren{\sigma_1^{-1}\sigma_2^{-1}}^3$ so $R=\kb[\alpha_s, \alpha_t]$. From \cite[Theorem B]{L3deg} $\overline{\hhh}(T(3,-3))^{T=i}=0$ for $i=-6,-5$ while \cite[Theorem 4.2]{L3deg} says that at $i=-4$ 
\[ \paren{\overline{\hhh}(T(3,-3))^{A=0}}^{T=-4}=  0, \qquad \paren{\overline{\hhh}(T(3,-3))^{A=1}}^{T=-4}= 0, \qquad \paren{\overline{\hhh}(T(3,-3))^{A=2}}^{T=-4}=  \kb(2)  \]

Thus it remains to compute the last 4 $T-$degrees. Before proceeding, we need the following lemmas

\begin{lemma}
\label{hhsts}
For $m_{st}=3$, let \rect denote the Jones-Wenzl projector $\jw_{\usmt{t}{3}{}}$. Then
\begin{align*}
 \hh^{0}(B_{tst})&=  \stackon{\hspace{-0.5ex}\bunit \runit \bunit }{\rect}\, R    =R(-3) \\
\hh^{1}(B_{tst})&=  \stackon{\hspace{-0.5ex}\bhunit \runit \bunit}{\rect} \, R  \oplus \stackon{\hspace{-0.5ex}\bhpitchcupin}{\rect} \, R  = R(1)\oplus R(3) \\
\hh^{2}(B_{tst})&= \stackon{\hspace{-0.5ex}\bhhpitchcupout}{\rect} \, R =R(7)
\end{align*}
\end{lemma}
\begin{proof}
  Apply \cref{finalgextdecomp} and use 4-Ext reduction and that $  \jw_{\usmt{t}{3}{}}\circ\bhpitchcupout=0 $.
\end{proof}

\begin{lemma}[{\cite[Lemma 3.4]{L3deg}, \cite[Corollary 1.12]{GHM19}}]
\label{lem:dualcomplex}
    Let $\beta$ be a braid on $3$ strands. For $F(\beta)\in K^b(\hf_{geo}, S_3)$ we have an isomorphism of complexes 
    \[\hh^k(F(\beta^\vee)) \cong \un{\hom}_R^\bullet(\hh^{2-k}(F(\beta), R) (4) \]
    where $\beta^\vee$ flips all crossings in $\beta$ from positive to negative and vice versa.
\end{lemma}

We will now compute the first 4 terms of the complexes $\hh^k(F((\sigma_1\sigma_2)^3))$ as $R-$modules and then use \cref{lem:dualcomplex} to compute the cohomology for the mirror. \cite[Section 9.2.1]{EH17} computed the minimal complex for the Rouquier complex $F((\sigma_1\sigma_2)^3)$ and so using \cref{hhsts} we see that the complex for $\overline{\hhh}^{A=0}((\sigma_1\sigma_2)^3)$ will be

\begin{align*}
   \boxed{ \scalemath{0.84}{\overline{\stackon{\hspace{-0.5ex}\runit \bunit \runit }{\rect}(-3)}} }\xrightarrow{\begin{pmatrix}
       0 \\
       0
   \end{pmatrix}}  \scalemath{0.84}{\overline{\stackon{\hspace{-0.5ex}\runit \bunit \runit }{\rect}(-1)}} \oplus \scalemath{0.84}{\overline{\stackon{\hspace{-0.5ex}\runit \bunit \runit }{\rect}(-1)}} \xrightarrow{\begin{pmatrix}
    \alpha_t & -\alpha_s \\
    -\alpha_t & \alpha_s \\
    \alpha_s\alpha_t+\alpha_t^2 & 0 \\
    0 & \alpha_s\alpha_t+\alpha_s^2
\end{pmatrix}} \scalemath{0.84}{\overline{\stackon{\hspace{-0.5ex}\runit \bunit \runit }{\rect}(1)}}\oplus \scalemath{0.84}{\overline{\stackon{\hspace{-0.5ex}\runit \bunit \runit }{\rect}(1)}}\oplus \overline{\runit(1)}\oplus\overline{\bunit(1)} \\
\xrightarrow{\begin{pmatrix}
    -\alpha_t & -\alpha_t & 0 & 0 \\
 \alpha_s+\alpha_t & 0 & -1 & 1 \\
 0 & \alpha_s+\alpha_t& 1 & -1
\end{pmatrix} }\scalemath{0.84}{\overline{\stackon{\hspace{-0.5ex}\runit \bunit \runit }{\rect}(3)}}\oplus \overline{\runit \bunit(2)} \oplus \overline{\bunit \runit(2)} \xrightarrow{\begin{pmatrix}
    \alpha_s+\alpha_t & \alpha_t & \alpha_t \\
    \alpha_s+\alpha_t & \alpha_t & \alpha_t
\end{pmatrix}} \overline{\runit \bunit(4)}\oplus \overline{\runit \bunit(4)}
\end{align*} 

Let $(-)^\vee=\hom_R(-, R)$. The nonzero cohomology of $D^\bullet_0$, the $R-$dual\footnote{Reverse the arrows, transpose the matrix, and negate $Q,T$ gradings.} of the complex above, will be
\[  \paren{\overline{\hhh}^{A=0}(D^\bullet_0)}^{T=0}=\paren{\scalemath{0.84}{\stackon{\hspace{-0.5ex}\runit \bunit \runit }{\rect}R(-3)}}^\vee =R(6), \qquad  \paren{\overline{\hhh}^{A=0}(D^\bullet_0)}^{T=-1}=\frac{ \paren{\scalemath{0.84}{\stackon{\hspace{-0.5ex}\runit \bunit \runit }{\rect}R(-1)}}^\vee \oplus \paren{\scalemath{0.84}{\stackon{\hspace{-0.5ex}\runit \bunit \runit }{\rect}R(-1)} }^\vee}{(\alpha_t, -\alpha_s)^T R(4)+(\alpha_s\alpha_t+\alpha_t^2, 0)^T R(4)} \]
\[ \paren{\overline{\hhh}^{A=0}(D^\bullet_0)}^{T=-2}= \paren{\paren{\scalemath{0.84}{\stackon{\hspace{-0.5ex}\runit \bunit \runit }{\rect}}+ \scalemath{0.84}{\stackon{\hspace{-0.5ex}\runit \bunit \runit }{\rect}}} R(1)}^\vee / (\alpha_s, \alpha_t)=\kb(2)   \]

To compute the Poincare series at $T=-1$, notice we have a degree 0 isomorphism. 
\[  R(-2)\oplus R(-4)\xrightarrow{\sim} (\alpha_t, -\alpha_s)^T R+(\alpha_s\alpha_t+\alpha_t^2, 0)^T R,\qquad \qquad (1,0)\mapsto (\alpha_t, -\alpha_s)^T, \ (0,1)\mapsto (\alpha_s\alpha_t+\alpha_t^2, 0)^T \]
Thus the Poincare series at $T=-1$ will then be
\[  \frac{2Q^{-4}}{(1-Q^2)^2}-\frac{Q^{-2}}{(1-Q^2)^2}-\frac{1}{(1-Q^2)^2}=\frac{Q^{-4}(2+Q^2)}{1-Q^2} \]
Thus the total Poincare series for $ \overline{\hhh}(T(3,-3))^{A=2}$ after accounting for the $Q^{-4}$ shift from \cref{lem:dualcomplex} is 
\[ A^2\paren{ T^{-4} Q^{-2}+T^{-2} Q^{-6} + \frac{T^{-1}Q^{-8}(2+Q^2)}{1-Q^2} +\frac{Q^{-10}}{(1-Q^2)^2} } \]
Before proceeding for $A=1$, we need
\begin{equation}
    \stackon{\hspace{-0.5ex}\runit \bhunit \runit }{\rect}=  \stackon{\hspace{-0.5ex}\rhpitchcupin}{\rect} (\alpha_s+\alpha_t)+ \stackon{\hspace{-0.5ex}\rhunit \bunit \runit }{\rect}
\end{equation}

The complex computing $\overline{\hhh}^{A=1}((\sigma_1\sigma_2)^3)$ will then be

\begin{gather*}
   \boxed{ \scalemath{0.84}{\overline{\stackon{\hspace{-0.5ex}\rhunit \bunit \runit }{\rect}(-3)}} \oplus \scalemath{0.84}{\overline{\stackon{\hspace{-0.5ex}\rhpitchcupin}{\rect}(-3)}} }\underset{\hh^1(d^0)}{\xrightarrow{\begin{pmatrix}
       0 & 0 \\
       0 & 0 \\
       0 & 0 \\
       0 & 0 
   \end{pmatrix}}} \scalemath{0.84}{\overline{\stackon{\hspace{-0.5ex}\rhunit \bunit \runit }{\rect}(-1)}}\oplus  \scalemath{0.84}{\overline{\stackon{\hspace{-0.5ex}\rhpitchcupin}{\rect}(-1)}}    \oplus \scalemath{0.84}{\overline{\stackon{\hspace{-0.5ex}\rhunit \bunit \runit }{\rect}(-1)}} \oplus  \scalemath{0.84}{\overline{\stackon{\hspace{-0.5ex}\rhpitchcupin}{\rect}(-1)}}  \underset{\hh^1(d^1)}{\xrightarrow{\begin{pmatrix}
       \alpha_t & 0 & -\alpha_s & 0 \\
       0 & \alpha_t & 0 & -\alpha_s \\
       -\alpha_t & 0 & \alpha_s & 0 \\
       0 & -\alpha_t & 0 & \alpha_s \\
       \alpha_s \alpha_t +\alpha_t^2 & -\alpha_t & 0 & 0 \\
       0 & 1 & 0 & 0 \\
       0 & 0 & 0 & \alpha_s \\
       0 & 0 & \alpha_s+\alpha_t & -1
   \end{pmatrix}}} \\
   \scalemath{0.84}{\overline{\stackon{\hspace{-0.5ex}\rhunit \bunit \runit }{\rect}(1)}}\oplus \scalemath{0.84}{\overline{\stackon{\hspace{-0.5ex}\rhpitchcupin}{\rect}(1)}}\oplus \scalemath{0.84}{\overline{\stackon{\hspace{-0.5ex}\rhunit \bunit \runit }{\rect}(1)}}\oplus \scalemath{0.84}{\overline{\stackon{\hspace{-0.5ex}\rhpitchcupin}{\rect}(1)}}  \oplus \overline{\rhunit (1) } \oplus \overline{\runit \scalemath{0.8}{\dboxed{\alpha_t^\vee}} (1) }   \oplus \overline{\bhunit (1) } \oplus \overline{\bunit \scalemath{0.8}{\dboxed{\alpha_s^\vee}} (1) } \underset{\hh^1(d^2)}{\xrightarrow{\begin{pmatrix}
       -\alpha_t & 0 & -\alpha_t & 0 & 0 & 0& 0 & 0\\
       0 & -\alpha_t & 0 & -\alpha_t & 0 & 0 & 0 & 0 \\
       \alpha_s + \alpha_t & -1 & 0 & 0 & -1 & 0 & 0 & \alpha_s \\
       0 & 1 & 0 & 0 & 0 & -\alpha_t & 1 & 0 \\
       0 & 0 & 0 & 1 & 0 & \alpha_t & -1 & 0 \\
       0 & 0 & \alpha_s +\alpha_t & -1 & 1 & 0 & 0 & -\alpha_s 
   \end{pmatrix}}} \\
   \scalemath{0.84}{\overline{\stackon{\hspace{-0.5ex}\rhunit \bunit \runit }{\rect}(3)}}\oplus \scalemath{0.84}{\overline{\stackon{\hspace{-0.5ex}\rhpitchcupin}{\rect}(3)}} \oplus \overline{\rhunit \bunit (2)} \oplus \overline{\runit \bhunit (2)} \oplus \overline{\bhunit \runit (2)} \oplus \overline{\bunit \rhunit (2)}\underset{\hh^1(d^3)}{\xrightarrow{\begin{pmatrix}
       \alpha_s+\alpha_t & -1 & \alpha_t & 0 & 0 & \alpha_t \\
       0 & 1 & 0 & \alpha_t & \alpha_t & 0 \\
       0 & 1 & 0 & \alpha_t & \alpha_t & 0 \\
       \alpha_s +\alpha_t  & -1 & \alpha_t & 0 & 0 & \alpha_t
   \end{pmatrix}}} \overline{\rhunit \bunit (4)} \oplus \overline{\runit \bhunit (4)} \oplus \overline{\bhunit \runit (4)} \oplus \overline{\bunit \rhunit (4)}
\end{gather*}

Let $D^\bullet_1$ be the $R-$dual of the complex above.

\begin{itemize}
    \item $\paren{\overline{\hhh}^{A=1}(D^\bullet_1)}^{T=0}=\paren{\scalemath{0.84}{\stackon{\hspace{-0.5ex}\rhunit \bunit \runit }{\rect}R(-3)} \oplus \scalemath{0.84}{\stackon{\hspace{-0.5ex}\rhpitchcupin}{\rect}R(-3)}   }^\vee =R(2)\oplus R$. Poincare series will be $\frac{Q^{-2}+1}{(1-Q^2)^2}$. 
    \item  \vspace{-3ex}
    \[ \im \hh^1(d^1)^T=  \begin{pmatrix}
        \alpha_t \\
        0 \\
        -\alpha_s \\
        0
    \end{pmatrix}R\oplus \begin{pmatrix}
        0 \\
        \alpha_t \\
        0 \\
        -\alpha_s
    \end{pmatrix}R\oplus \begin{pmatrix}
        0 \\
        1 \\
        0 \\
        0
    \end{pmatrix}R\oplus \begin{pmatrix}
        0 \\
        0 \\
        \alpha_s+\alpha_t \\
        -1
    \end{pmatrix}R, \qquad \begin{pmatrix}
        0 \\
        \alpha_t \\
        0 \\
        -\alpha_s
    \end{pmatrix}= \alpha_t \begin{pmatrix}
        0 \\
        1 \\
        0 \\
        0
    \end{pmatrix}+\alpha_s \begin{pmatrix}
        0 \\
        0 \\
        \alpha_s+\alpha_t \\
        -1
    \end{pmatrix}-\alpha_s(\alpha_s+\alpha_t)\begin{pmatrix}
        0 \\
        0 \\
        1 \\
        0
    \end{pmatrix} \]
  $$\implies\paren{\overline{\hhh}^{A=1}(D^\bullet_1)}^{T=-1}= \frac{ \paren{\scalemath{0.84}{\stackon{\hspace{-0.5ex}\rhunit \bunit \runit }{\rect}R(-1)}}^\vee \oplus \paren{\scalemath{0.84}{\stackon{\hspace{-0.5ex}\rhunit \bunit \runit }{\rect}R(-1)} }^\vee}{(\alpha_t, -\alpha_s)^T R+(0, -\alpha_s\alpha_t+\alpha_s^2)^T R}  \implies \textnormal{Poincare series}=  \frac{2-Q^2-Q^4}{(1-Q^2)^2}=\frac{2+Q^2}{1-Q^2}$$
  \item \vspace{-1ex}\[ \ker \hh^1(d^1)^T=\set{\begin{pmatrix}
      a-(\alpha_s+\alpha_t)d \\
      b \\
      a\\
      c\\
      d\\
      -\alpha_t(b-c-d)\\
      b-c-d \\
      -\alpha_s d
  \end{pmatrix}}=\begin{pmatrix}
      1 \\
      0 \\
      1 \\
      0 \\
      0 \\
      0 \\
      0 \\
      0 
  \end{pmatrix}R \oplus \begin{pmatrix}
      0 \\
      1 \\
      0 \\
      0 \\
      0 \\
      -\alpha_t \\
      1 \\
      0 
  \end{pmatrix}R \oplus \begin{pmatrix}
      0 \\
      0 \\
      0 \\
      1 \\
      0 \\
      -\alpha_t \\
      -1 \\
      0 
  \end{pmatrix}R \oplus \begin{pmatrix}
      -(\alpha_s+\alpha_t) \\
      0 \\
      0 \\
      0 \\
      1 \\
      \alpha_t \\
      -1 \\
      -\alpha_s 
  \end{pmatrix}R \]
  \[ \im \hh^1(d^2)^T=\begin{pmatrix}
      -\alpha_t \\
      0 \\
      -\alpha_t \\
      0 \\
      0 \\
      0 \\
      0 \\
      0
  \end{pmatrix}R\oplus \begin{pmatrix}
      0 \\
      1 \\
      0 \\
      0 \\
      0 \\
      -\alpha_t \\
      1 \\
      0
  \end{pmatrix}R \oplus \begin{pmatrix}
      0 \\
      0 \\
      0 \\
      1 \\
      0 \\
      \alpha_t \\
      -1 \\
      0
  \end{pmatrix}R\oplus\begin{pmatrix}
      \alpha_s+\alpha_t \\
      -1 \\
      0 \\
      0 \\
      -1 \\
      0 \\
      0 \\
      \alpha_s
  \end{pmatrix}R \oplus \begin{pmatrix}
      0 \\
      0 \\
      \alpha_s+\alpha_t \\
      -1 \\
      1 \\
      0 \\
      0 \\
      -\alpha_s
  \end{pmatrix}R   \]
  Let $\vec{v_1}, \ldots, \vec{v_4}$ be the basis vectors for $\ker \hh^1(d^1)^T$ above and similarly let $\vec{w_1}, \ldots, \vec{w_4}$ be the basis vectors for $\im \hh^1(d^2)^T$. Consider the change of basis in $\im \hh^1(d^2)^T$ given by $\vec{w_4}\mapsto -\vec{w_4}-\vec{w_2}=\vec{v_4}$ and $\vec{w_5}\mapsto \vec{w_5}+\vec{w_4}+\vec{w_3}+\vec{w_2}+\vec{w_1}=\alpha_s \vec{v_1}$. Thus
  \[ \paren{\overline{\hhh}^{A=1}(D^\bullet_1)}^{T=-2}=\paren{\paren{\scalemath{0.84}{\stackon{\hspace{-0.5ex}\rhunit \bunit \runit }{\rect}}+ \scalemath{0.84}{\stackon{\hspace{-0.5ex}\rhunit \bunit \runit }{\rect}}} R(1)}^\vee / (\alpha_s, \alpha_t)=\kb(-2)\implies \textnormal{Poincare series}=Q^2 \]
  \item $\paren{\overline{\hhh}^{A=1}(D^\bullet_1)}^{T=-3}=0$.
\end{itemize}
Thus the total Poincare series for $ \overline{\hhh}(T(3,-3))^{A=1}$ after accounting for the $Q^{-4}$ shift from \cref{lem:dualcomplex} is 
\[ A\paren{ T^{-2}Q^{-2}+\frac{T^{-1} Q^{-4}(2+Q^2)}{1-Q^2}+\frac{Q^{-6}+Q^{-4}}{(1-Q^2)^2} } \]

Before proceeding for $A=2$, we need
\begin{equation}
    \stackon{\hspace{-0.5ex}\rhunit \bhunit \runit }{\rect}=  \stackon{\hspace{-0.5ex}\rhhpitchcupout}{\rect} (\alpha_s+\alpha_t)
\end{equation}

The complex computing $\overline{\hhh}^{A=2}((\sigma_1\sigma_2)^3)$ will then be

\begin{gather*}
    \boxed{\scalemath{0.84}{\overline{\stackon{\hspace{-0.5ex}\rhhpitchcupout}{\rect}(-3)}}}\xrightarrow{\begin{pmatrix}
        0 \\
        0
    \end{pmatrix}}\scalemath{0.84}{\overline{\stackon{\hspace{-0.5ex}\rhhpitchcupout}{\rect}(-1)}}\oplus \scalemath{0.84}{\overline{\stackon{\hspace{-0.5ex}\rhhpitchcupout}{\rect}(-1)}}\xrightarrow{\begin{pmatrix}
        \alpha_t & - \alpha_s \\
        -\alpha_t & \alpha_s \\
        1 & 0 \\
        0 & 1
    \end{pmatrix}} \scalemath{0.84}{\overline{\stackon{\hspace{-0.5ex}\rhhpitchcupout}{\rect}(1)}}\oplus \scalemath{0.84}{\overline{\stackon{\hspace{-0.5ex}\rhhpitchcupout}{\rect}(1)}} \oplus \overline{\rhunit \scalemath{0.8}{\dboxed{\alpha_t^\vee}}(1)}\oplus \overline{\bhunit \scalemath{0.8}{\dboxed{\alpha_s^\vee}}(1)}\xrightarrow{\begin{pmatrix}
        -\alpha_t & -\alpha_t & 0 & 0 \\
        1 & 0 & -\alpha_t & \alpha_s \\
        0 & 1 & \alpha_t & -\alpha_s
    \end{pmatrix}} \\
    \scalemath{0.84}{\overline{\stackon{\hspace{-0.5ex}\rhhpitchcupout}{\rect}(3)}}\oplus \overline{\rhunit\bhunit(2)}\oplus \overline{\bhunit\rhunit(2)}\xrightarrow{\begin{pmatrix}
        1 & \alpha_t & \alpha_t \\
        1 & \alpha_t & \alpha_t
    \end{pmatrix}} \overline{\rhunit\bhunit(4)}\oplus \overline{\bhunit \rhunit (4)}
\end{gather*}
The nonzero cohomology of $D^\bullet_2$, the $R-$dual of the complex above will then be
\[\paren{\overline{\hhh}^{A=2}(D^\bullet_2)}^{T=0}=\paren{\scalemath{0.84}{\stackon{\hspace{-0.5ex}\rhhpitchcupout }{\rect}R(-3)}}^\vee =R(-4)\]
The total Poincare series for $ \overline{\hhh}(T(3,-3))^{A=0}$ after accounting for the $Q^{-4}$ shift from \cref{lem:dualcomplex} is $\frac{1}{(1-Q^2)^2}$. Putting this all together, we obtain
\begin{gather}
\overline{\mathscr{P}(A,Q, T)}(F((\sigma_1^{-1}\sigma_2^{-1})^{3}) )=  \frac{1}{(1-Q^2)^2}+
A\paren{ T^{-2}Q^{-2}+\frac{T^{-1} Q^{-4}(2+Q^2)}{1-Q^2}+\frac{Q^{-6}+Q^{-4}}{(1-Q^2)^2} }\nonumber\\
+A^2\paren{ T^{-4} Q^{-2}+T^{-2} Q^{-6} 
+ \frac{T^{-1}Q^{-8}(2+Q^2)}{1-Q^2} +\frac{Q^{-10}}{(1-Q^2)^2} } 
\end{gather}
The HOMFLY polynomial for $L6n1\set{0,1}$, the braid closure of $(\sigma_1^{-1}\sigma_2^{-1})^3$ is computed in \cite{LM} and is
\[\overline{P(a,q)}(L6n1\set{0,1})= -\frac{3}{v^6} + \frac{3}{v^4} + \frac{1}{v^8 z^2} - \frac{2}{v^6 z^2} + \frac{1}{v^4 z^2} - \frac{z^2}{v^6} + \frac{4 z^2}{v^4} + \frac{z^4}{v^4} \]
where $z=q-q^{-1}$ and $v=a$. After mutliplying by the unit $a^8/q^2$, this agrees with $\overline{\mathscr{P}(A,Q, T)}(F((\sigma_1^{-1}\sigma_2^{-1})^{3}) )|_{\cref{subeq}}$.
\section{Gomi's Trace}
\label{gomisect}

Before going into Gomi's trace, we should mention that there are usually two versions/presentations of the Hecke algebra $\mathbf{H}_W$ of a Coxeter group $W$ in the literature. As originally defined in \cite{KL79}, version 1 of $\mathbf{H}_W$ is the associative algebra over $\Zb[q, q^{-1}]$ on generators $\set{T_s}$ satisfying the braid relation and the quadratic relation 
\[ T_s^2=(q-1)T_s+q\]
In \cite{Soe97}, Soergel gives the second presentation of $\mathbf{H}_W$ as the associative algebra over $\Zb[v, v^{-1}]$ on generators $\set{\delta_s}$ satisfying the braid relation and the quadratic relation 
\[ \delta_s^2=(v^{-1}-v)\delta_s +1 \]
and the two presentations are isomorphic under the $\Zb-$linear map $\delta_w\mapsto v^{\ell(w)} T_w$ and $v\mapsto q^{-1/2}$ after tensoring version 1 by $\otimes_\Zb \Zb[q^{1/2}, q^{-1/2}]$. Under this isomorphism we also have $\delta_s^{-1}=v^{-1}T_s^{-1}$. For $w\in W=W_m$, the KL-basis of $H_{W_m}$ can be explicitly expressed in the standard basis as  
\begin{equation}
\label{bwbasis}
     b_w=\sum_{y\le w} v^{\ell(w)-\ell(y)} \delta_y = v^{\ell(w)} \sum_{y\le w} T_y  
\end{equation}

\cite{Gom06} uses version 1 of $\mathbf{H}_W$ while \cite{EMTW20} uses version 2 and as such we will be using both presentations interchangeably in what follows. Let $\hh^{a,b}(M)=\ext_{R^e}^{a,b}(R,M)$ and $B_w\in \sbim(\hf_{geo}, W_m)$.

\begin{definition}
Define a trace $\epsilon_t: \mathbf{H}_{W_m}\to \Zb[q^{1/2},q^{-1/2}][t]$ by first defining $\epsilon_t$ on the KL-basis via
\[ \epsilon_t(b_w)=(1-q^{-1})^2\sum_{a,b} \dim_\kb \hh^{a,b}(B_w)q^{-b/2}t^a= \mathrm{grk}_R^t(\hh^{\bullet, \bullet}({B_w})) \]
    and extend $\Zb[q^{1/2},q^{-1/2}]$ linearly where $\mathrm{grk}_R^t(M)=av^k t^n$ if $M\cong R(-k)^{\oplus a}[-n]$. 
\end{definition}

\begin{remark}
In order for $\epsilon_t$ to be $\Zb[q^{1/2}, q^{-1/2}]$ linear, at the categorical level we should define $q^{1/2}[B_w]:=[B_w(1)]$. Consequently this means that we should define $v^{-1}[B_w]:=[B_w(1)]$ which is the opposite of what is written in \cite[Section 4.8]{EMTW20}. In fact, Soergel's Hom formula \cite[Theorem 5.27]{EMTW20} is only true when using our convention above, rather than the convention in \cite{EMTW20}. 
\end{remark}

As we are in the geometric realization, \cref{finindecompcohograd} tells us that for $1\le k\le m$

\begin{equation}
\label{epsiloneq}
\epsilon_t(b_{\smt{t}{k}{}})=\epsilon_t(b_{\smt{s}{k}{}})=q^{-k/2} +(q^{(4-k)/2}+q^{k/2})t+q^{(k+4)/2}t^2  
\end{equation}

\begin{lemma}
\label{c2lem}
\begin{enumerate}[(a)]
    \item For all $2\le k \le m-1$
\begin{equation}
\label{partaeq}
    (v+v^{-1})\epsilon_t( b_{\smt{s}{k}{}}) =\epsilon_t(b_{\smt{s}{k+1}{}})+\epsilon_t(b_{\smt{s}{k-1}{}}) 
\end{equation}
    \item $\epsilon_t(\delta_{\smt{t}{k}{}})=\epsilon_t(\delta_{\smt{s}{k}{}})$.
    \item $\epsilon_t( \delta_{\smt{s}{k+1}{}} )=(v^{-1}-v)\epsilon_t( \delta_{\smt{s}{k}{}})+\epsilon_t(\delta_{\smt{s}{k-1}{}})$ for all $2\le k \le m-1$.
\end{enumerate}
\end{lemma}
\begin{proof}
$(a)$ follows from direct computation using \cref{epsiloneq}. $(b)$ follows from \cref{epsiloneq} and induction using \cref{bwbasis}.\\

$(c)$ Using \cref{bwbasis} to expand the LHS of \cref{partaeq} we obtain $v\epsilon_t(b_{\smt{s}{k}{}})+v^{-1} \epsilon_t(\delta_{\smt{s}{k}{}})+\epsilon_t(\delta_{\smt{t}{k-1}{}})+\epsilon_t(b_{\smt{s}{k-1}{}})$ while the RHS of \cref{partaeq} can be written as $\epsilon_t(\delta_{\smt{s}{k+1}{}}) +v\epsilon_t( \delta_{\smt{t}{k}{}})+v\epsilon_t(b_{\smt{s}{k}{}})+\epsilon_t(b_{\smt{s}{k-1}{}})$. Thus it follows that
\[ v^{-1} \epsilon_t(\delta_{\smt{s}{k}{}})+\epsilon_t(\delta_{\smt{t}{k-1}{}})= \epsilon_t(\delta_{\smt{s}{k+1}{}})+v\epsilon_t(\delta_{\smt{t}{k}{}})\]
and using part $(b)$ and moving $v\epsilon_t(\delta_{\smt{t}{k}{}})$ to the LHS gives the result. 
\end{proof}

Let $\tau_{W_m}$ be Gomi's trace for $\mathbf{H}_{W_m}$ as defined in \cite[Section 4.2]{Gom06}. For $W_m$, it turns out we do not need to know the entries of Lusztig's Exotic Fourier transform matrix (although they are explicitly given in \cite[Section 4.6]{Gom06}) or even the characters of $\mathbf{H}_{W_m}$ in order to compute $\tau_{W_m}$ by the following result \cite[Theorem 4.5]{Gom06}

\begin{theorem}
\label{kihtheorem}
Let $\displaystyle z=\frac{(q-1)r}{1+r}$. Then $\tau_{W_m}: \mathbf{H}_{W_m}\to \Zb[q^{1/2}, q^{-1/2}][z]$ is the unique trace\footnote{This means that $\tau_{W_m}$ is a $\Zb[q^{1/2}, q^{-1/2}]$ linear map such that $\tau_{W_m}(ab)=\tau_{W_m}(ba)$ for all $a,b\in \mathbf{H}_{W_m}$} function on $\mathbf{H}_{W_m}$ satisfying
\begin{enumerate}[(1)]
    \item $\tau_{W_m}(1)=1$
    \item $\displaystyle \tau_{W_m}(T_1)=\tau_{W_m}(T_2)=z$
    \item $\tau_{W_m}( \underbrace{T_1T_2\ldots}_{i+1}\underbrace{\ldots T_1^{-1} T_2^{-1}}_{i-1} )=z^2$, for all $\displaystyle 1\le i \le \left\lfloor \frac{m-1}{2}\right \rfloor  $ 
\end{enumerate}
\end{theorem}

\begin{theorem}
For all $x\in \mathbf{H}_{W_m}$
\[ \frac{1}{(1+qt)^2}\epsilon_t(x)=\tau_{W_m}(x) |_{r=qt} \]
\end{theorem}
\begin{proof}
Because we have an isomorphism of doubly graded \un{vector spaces}
\[  \hh^{a,b}(M\otimes_R N)\cong \hh^{a,b}(N\otimes_R M)  \]
the LHS is a trace function on $\mathbf{H}_{W_m}$ and by uniqueness we just need to check the LHS satisfies the conditions in \cref{kihtheorem}. $(1)$ follows from the well known computation $\hh_{\bullet, \bullet}(R)=\mathrm{Sym}(V^*(-2))\otimes_\kb \Lambda^\bullet(V\dsbrac{-1})$, where $\dim V=2$ in this case. $(2)$ is an immediate consequence of \cref{epsiloneq} and that $T_s=q^{1/2}b_s-1$. \\

For $(3)$, under the isomorphism between the two versions of the hecke algebra we have that
\[ \epsilon_t( \underbrace{T_1T_2\ldots}_{i+1}\underbrace{\ldots T_1^{-1} T_2^{-1}}_{i-1} )=v^{-2}\epsilon_t( \underbrace{\delta_1 \delta_2\ldots}_{i+1}\underbrace{\ldots \delta_1^{-1} \delta_2^{-1}}_{i-1} )   \]

For $i\ge 1$, let $\varphi_i= \underbrace{\delta_1 \delta_2\ldots}_{i+1}\underbrace{\ldots \delta_1^{-1} \delta_2^{-1}}_{i-1} $ and $\varphi_i^*= \underbrace{\delta_2 \delta_1\ldots}_{i+1}\underbrace{\ldots \delta_2^{-1} \delta_1^{-1}}_{i-1} $. We claim
\[  \epsilon_t(\varphi_i) =  \epsilon_t(\varphi_{i-1}^*)  \]
and therefore we only need to compute $\epsilon_t(\varphi_1)=\epsilon_t(\varphi_1^*)$. Note $\delta_2^{-1}=\delta_2+v-v^{-1}$ and since $\epsilon_t$ is a trace we see that
\[ \epsilon_t(\varphi_i)=\epsilon_t((\delta_2+v-v^{-1})\delta_1 \varphi_{i-1}^*)=\epsilon_t(\delta_2 \delta_1 \varphi_{i-1}^*)+(v-v^{-1})\epsilon_t( \delta_1 \varphi_{i-1}^*  ) \]
Now, using the quadratic relation, because $\varphi_i$ alternates between $\delta_1, \delta_2$ at the start, $\delta_1^{-1}, \delta_2^{-1}$ at the end with 2 more $\delta_1, \delta_2$ terms than $\delta_1^{-1}, \delta_2^{-1}$ terms, we can write $\displaystyle \varphi_{i-1}^*=\sum_{j=2}^{2i-2} a_j \delta_{\smt{2}{j}{}}$ when $i>1$. Because $\delta_2\delta_1 \delta_{\smt{2}{j}{}}=\delta_{\smt{2}{j+2}{}}$, by \cref{c2lem} part $(c)$ we have
\begin{align*}
 \epsilon_t(\delta_2 \delta_1 \varphi_{i-1}^*)+(v-v^{-1})\epsilon_t( \delta_1 \varphi_{i-1}^*  ) &=(v^{-1}-v)\epsilon_t(\delta_1 \varphi_{i-1}^*)+\epsilon_t(\varphi_{i-1}^*)+(v-v^{-1})\epsilon_t( \delta_1 \varphi_{i-1}^*  )  \\
 &=\epsilon_t(\varphi_{i-1}^*)
\end{align*}
as desired. Finally, using the above and \cref{epsiloneq} we compute
\begin{align*}
\frac{1}{(1+qt)^2} \epsilon_t( \underbrace{T_1 T_2\ldots}_{i+1}\underbrace{\ldots T_1^{-1} T_2^{-1}}_{i-1} )&= \frac{1}{(1+qt)^2}  v^{-2}\epsilon_t( \delta_1 \delta_2 )= \frac{1}{(1+qt)^2} \epsilon_t( T_1 T_2 )     \\
=\frac{1}{(1+qt)^2} \sbrac{q\epsilon_t(b_{12})-\epsilon_t(T_1)-\epsilon_t(T_2)-\epsilon_t(1) } &=\frac{(1+q^2t)^2-2(1+qt)(q-1)qt-(1+qt)^2}{(1+qt)^2}\\
& =\paren{\frac{(1-q)qt}{1+qt}}^2
\end{align*} 
\end{proof}

\begin{corollary}
\label{lihom}
    Let $\omega:\mathbf{H}_{W_m}\to \mathbf{H}_{W_m}$ be the KL anti-involution. Given $B, B^\prime\in \sbim(\hf_{geo}, W_m) $ set $v^{-1}[B]:=[B(1)]$.
    \[ \mathrm{grk}_R^{t}\ext_{R^e}^{\bullet, \bullet}(B,B^\prime)=(1+v^{-2}t)^2\tau_{W_m}\paren{\omega(\ch_{\Delta}(B))  \ch_{\Delta}(B^\prime)} |_{r=v^{-2}t}  \]
    where $\mathrm{grk}_R^t(M)=cv^k t^n$ if $M\cong R(-k)^{\oplus c}[-n]$.
\end{corollary}
\begin{proof}
By adjunction, \cref{cor:adjunct}, and linearity it suffices to show this when $B=R, B^\prime=B_w$, but $\ch_\Delta(B_w)=b_w$.
\end{proof}

By \cite[Remark 4.2]{Gom06}, when $z\to 0$ (equivalently when $t\to 0$), $\tau_{W_m}\to \epsilon$ where $\epsilon$ is the standard/canonical trace on $\mathbf{H}_{W_m}$ and thus the above corollary is a $t-$analog of Soergel's Hom Formula \cite[Theorem 5.27]{EMTW20}.


\end{appendices}

\printbibliography

\noindent
{\textsl \small Cailan Li, Department of Mathematics, Columbia University, New York, NY, USA}

\noindent 
{\tt \small email: ccl@math.columbia.edu}

\end{document}